\documentclass[a4j,12pt]{amsart}

\usepackage{amsmath}
\usepackage{amssymb}
\usepackage{amsfonts}
\usepackage{mathtools}
\usepackage{amsthm}
\usepackage{xcolor}
\usepackage{tikz}
\usepackage{fancyhdr}
\usepackage[colorlinks=true, linkcolor=blue, citecolor=blue, linktoc=page]{hyperref}

\fancypagestyle{mypagestyle1}{
\lhead[]{}
\rhead[]{}
\chead[\footnotesize NAITO, SUZUKI, AND WATANABE \quad\quad\quad\quad\quad\quad\quad\quad]{\footnotesize \quad\quad\quad\quad\quad\quad\quad\quad A proof of the NS conjecture via the branching rule for $\imath$QGs}
\lfoot[]{}
\rfoot[]{}
\cfoot[\textcolor{white}{} \\ \footnotesize \thepage \,\quad\quad\quad\quad\quad\quad\quad\quad\,]{\textcolor{white}{} \\ \footnotesize \,\,\quad\quad\quad\quad\quad\quad\quad\quad \thepage}
\renewcommand{\headrulewidth}{0.0pt}}
\usepackage[top=30truemm,bottom=30truemm,left=28truemm,right=28truemm]{geometry}

\usetikzlibrary{arrows.meta}
\usetikzlibrary{decorations.pathreplacing}
\usetikzlibrary{decorations.pathmorphing}
\numberwithin{equation}{section}

\theoremstyle{plain}
\newtheorem{thm}{Theorem}[section]
\newtheorem{prop}[thm]{Proposition}
\newtheorem{cor}[thm]{Corollary}
\newtheorem{lem}[thm]{Lemma}
\newtheorem*{thmA}{Theorem A}
\newtheorem*{thmB}{Theorem B}

\theoremstyle{definition}

\newtheorem{de}[thm]{Definition}
\newtheorem*{notation}{Notation}
\newtheorem*{acknowledgment}{Acknowledgments}

\theoremstyle{remark}
\newtheorem{rem}[thm]{Remark}
\newtheorem{eg}[thm]{Example}

\newcommand{\setlr}[2]{\left\{#1 \mathrel{}\middle|\mathrel{} #2\right\}}

\title{A proof of the Naito--Sagaki conjecture via the branching rule for $\imath$quantum groups}
\author[1]{Satoshi Naito}
\author[2]{Yujin Suzuki}
\author[3]{Hideya Watanabe}
\date{}
\address[S. Naito]{Department of Mathematics, Institute of Science Tokyo, 2-12-1 Oh-okayama, Meguro-ku, Tokyo, 152-8551, Japan}
\email{naito@math.titech.ac.jp}
\address[Y. Suzuki]{Mizuho-DL Financial Technology Co., Ltd., 2-4-1 Kojimachi, Chiyoda-ku, Tokyo, 102-0083, Japan}
\email{yujin-suzuki@fintec.co.jp}
\address[H. Watanabe]{College of Science, Rikkyo University, 3-34-1 Nishi-Ikebukuro, Toshima-ku, Tokyo, 171-8501, Japan}
\email{watanabehideya@gmail.com}

\begin{document}

\maketitle

\fancypagestyle{mypagestyle2}{
\lhead[]{}
\rhead[]{}
\chead[]{}
\lfoot[]{}
\rfoot[]{}
\cfoot[\textcolor{white}{} \\ \footnotesize \thepage \,\quad\quad\quad\quad\quad\quad\quad\quad\,]{\textcolor{white}{} \\ \footnotesize \,\,\quad\quad\quad\quad\quad\quad\quad\quad \thepage}
\renewcommand{\headrulewidth}{0.0pt}}
\thispagestyle{mypagestyle2}

\begin{abstract}
	The Naito--Sagaki conjecture asserts that the branching rule for the 
	restriction of finite-dimensional, irreducible polynomial representations of
	$GL_{2n}(\mathbb{C})$ to $Sp_{2n}(\mathbb{C})$ amounts to the enumeration of certain ``rational paths'' satisfying specific conditions. 
	This conjecture can be thought of as a non-Levi type analog of the Levi type branching 
	rule, stated in terms of the path model due to Littelmann, and was proved combinatorially in 2018 by Schumann--Torres. 
	In this paper, we give a new proof of the Naito--Sagaki conjecture independently of Schumann--Torres, using the branching rule based on the crystal basis theory for 
	$\imath$quantum groups of type $A\mathrm{II}_{2n-1}$.
	Here, note that $\imath$quantum groups are certain coideal subalgebras of a quantized universal enveloping algebra 
	obtained by $q$-deforming symmetric pairs, and also regarded as a generalization of quantized universal enveloping algebras; 
	these were defined by Letzter in 1999, and since then their representation theory has become an active area of research.
	The main ingredients of our approach are certain combinatorial operations, 
	such as promotion operators and Kashiwara operators, 
	which are well-suited to the representation theory of complex semisimple Lie algebras. 
\end{abstract}

\tableofcontents

\section{Introduction}

The purpose of this paper is to give a new proof
of the branching rule for the restriction of irreducible highest weight representations of general linear Lie algebras to symplectic Lie algebras, 
conjectured by Naito and Sagaki in \cite{NS}.

Let $I$ be the set of vertices of the 
Dynkin diagram of type $A_{r}$,
and $\mathfrak{g}$ the general linear Lie algebra of rank $r$.
In Lie theory,
it is a fundamental problem to describe
how
each irreducible highest weight
$\mathfrak{g}$-module $L(\lambda)$
decomposes
when we restrict the action of $\mathfrak{g}$ to
its certain subalgebra $\widehat{\mathfrak{g}}$,
where $\lambda$ is a dominant weight in the weight lattice $P$ of $\mathfrak{g}$. 
In \cite{L1} and \cite{L2},
Littelmann answered this question in terms of ``rational paths'' lying in 
the real form $P_{\mathbb{R}} = \mathbb{R} \otimes_{\mathbb{Z}} P$
of the dual space of the Cartan subalgebra of $\mathfrak{g}$ 
when $\widehat{\mathfrak{g}}$ is a Levi subalgebra of $\mathfrak{g}$.

Here, let us briefly review Littelmann's theory of the path model. 
Throughout this paper,
a piecewise-linear, continuous map
$\pi : [0,1]_{\mathbb{R}} \to P_{\mathbb{R}}$ is called a \textit{path}
if $\pi(0)=0$ and $\pi(1) \in P$; 
the set of all paths lying in $P_{\mathbb{R}}$ is denoted by $\mathbb{P}$.
For $i \in I$,
Littelmann defined certain operators on $\mathbb{P}$,
called \textit{root operators}, 
which arise from the action of the Chevalley generators of $\mathfrak{g}$.
Let $\mathbb{P}(\lambda)$ denote 
the set of all paths obtained by applying root operators to the
straight-line path connecting the origin and
a dominant integral weight $\lambda \in P$.
Littelmann discovered that the set $\mathbb{P}(\lambda)$, equipped 
with the root operators, 
is deeply connected to the representation $L(\lambda)$.
Let $\widehat{\mathfrak{g}}_J$
be the Levi subalgebra associated with a subset $J \subset I$,
and $\widehat{P}^J_{\mathbb{R}}$ the real form of the dual space of the Cartan subalgebra of $\widehat{\mathfrak{g}}_J$. 
Then, for each path $\pi \in \mathbb{P}(\lambda)$,
a new path $\widehat{\pi}^J : [0,1]_{\mathbb{R}} \to \widehat{P}^J_{\mathbb{R}}$
is defined by the composition with the natural restriction map
$P_{\mathbb{R}} \to \widehat{P}_{\mathbb{R}}^J$; 
the path $\widehat{\pi}^J$
is said to be \textit{$\widehat{\mathfrak{g}}_J$-dominant}
if $\langle\widehat{\pi}^J,\widehat{h}_j^J\rangle \geq 0$ 
for all simple coroots 
$\widehat{h}_j^J$, $j \in J$, of $\widehat{\mathfrak{g}}_J$. 
Littelmann proved that as $\widehat{\mathfrak{g}}_J$-modules, 
\begin{equation}
	\mathsf{Res}^{\mathfrak{g}}_{\widehat{\mathfrak{g}}_J}\,L(\lambda) \cong
		\bigoplus_{\substack{\pi \in \mathbb{P}(\lambda) \\ \widehat{\pi}^J \, : \, \text{$\widehat{\mathfrak{g}}_J$-dominant}}} \widehat{L}^J\left(\widehat{\pi}^J(1)\right) \label{Levi}
\end{equation}
for all dominant integral weight $\lambda$,
where $\widehat{L}^J(\mu)$ denotes the irreducible highest weight $\widehat{\mathfrak{g}}_J$-module of highest weight $\mu$ ; 
to be precise, Littelmann obtained the decomposition (\ref{Levi}) in the more general setting where 
$\mathfrak{g}$ is a 
symmetrizable Kac--Moody algebra. 

Now, a natural question is whether it is possible to generalize the Levi type subalgebra 
$\widehat{\mathfrak{g}}_J$ to a non-Levi type subalgebra. 
One answer to this question is the
\textit{Naito--Sagaki conjecture}.
In the following, let $r$ be an odd number $r=2n-1$,
and $\widehat{\mathfrak{g}}$ the fixed-point subalgebra
under the non-trivial diagram automorphism of $I$.
In this case,
$\widehat{\mathfrak{g}}$ is isomorphic to
$\mathfrak{sp}_{2n}(\mathbb{C})$.
Let $\widehat{I}$ be the set of vertices of the 
Dynkin diagram of type $C_n$,
and $\widehat{P}_{\mathbb{R}}$ the real form of the dual space of 
the Cartan subalgebra of $\widehat{\mathfrak{g}}$.
Also, let $\mathbb{B}(\lambda)$ denote the set of paths
obtained by applying root operators to
a certain ``rational path'' whose endpoint is a non-negative
dominant integral weight
$\lambda \in P^{+}_{\geq 0}$, defined in Section 3.3.
For each $\pi \in \mathbb{B}(\lambda)$,
a new path $\widehat{\pi} : [0,1]_{\mathbb{R}} \to \widehat{P}_{\mathbb{R}}$
is defined by the composition with the natural restriction map
$P_{\mathbb{R}} \to \widehat{P}_{\mathbb{R}}$
in the same way as above.
We say that $\widehat{\pi}$ is
\textit{$\widehat{\mathfrak{g}}$-dominant}
if $\langle\widehat{\pi},\widehat{h}_i\rangle \geq 0$ 
for all simple coroots $\widehat{h}_i$, $i \in \widehat{I}$, 
of $\widehat{\mathfrak{g}}$.
The following is the main result of this paper, which is 
known as the Naito--Sagaki conjecture.
\begin{thmA}[$=$ Theorem \ref{NS2}]
	We have the following decomposition of $L(\lambda)$ as
	a $\widehat{\mathfrak{g}}$-module: 
	\begin{equation}
		\mathsf{Res}^{\mathfrak{g}}_{\widehat{\mathfrak{g}}}\,L(\lambda) \cong
			\bigoplus_{\substack{\pi \in \mathbb{B}(\lambda) \\ \widehat{\pi} \, : \, \text{$\widehat{\mathfrak{g}}$-$\mathrm{dominant}$}}} \widehat{L}\left(\widehat{\pi}(1)\right), 
	\end{equation}
	where $\widehat{L}(\mu)$ denotes the irreducible highest weight
	$\widehat{\mathfrak{g}}$-module of highest weight $\mu \in \widehat{P}^{+}$. 
\end{thmA}
The Naito--Sagaki conjecture was proposed in \cite{NS}. 
After some special cases were verified in \cite{NS} and \cite{T}, 
a complete proof was finally provided by Schumann--Torres \cite{ST}. 
Their strategy for the proof was to compare in a purely combinatorial way the set of $\widehat{\mathfrak{g}}$-dominant paths with Sundaram's branching model in \cite{Su}; recently, a bijection between Sundaram's branching model and Kwon's one in \cite{Kwo2} has been given by Kumar--Torres \cite{kt}. 
In addition, \cite{Mun} has given another combinatorial proof of the 
Naito--Sagaki conjecture, which does not rely on either \cite{ST} or this paper. 

In this paper, 
we prove the Naito--Sagaki conjecture using yet another branching model established in \cite{Wat1} by the third author based on the representation theory of
\textit{$\imath$quantum groups of type $A\mathrm{II}$}. 
One of the advantages of proving Theorem A by this method
is that by employing other types of $\imath$quantum groups,
we can expect further results similar to the decomposition (\ref{Levi}). 

Note that $\imath$quantum groups
of type $A\mathrm{II}$ arise by $q$-deforming 
\textit{symmetric pairs}; 
a symmetric pair of type $A\mathrm{II}$ is the pair
$(\mathfrak{sl}_{2n}(\mathbb{C}),\mathfrak{k})$
consisting of the complex special linear Lie algebra of rank $2n-1$
and its certain subalgebra that is isomorphic to
$\mathfrak{sp}_{2n}(\mathbb{C})$.
In order to construct a $q$-analog of the homogeneous space $GL_{2n}(\mathbb{C})/{Sp_{2n}(\mathbb{C})}$,
Noumi \cite{Nou} defined a certain subalgebra $\mathbf{U}^{\imath}$
of $\mathbf{U}$
as a $q$-deformation of $U(\mathfrak{k})$,
where $\mathbf{U}$ is the 
quantized universal enveloping algebra associated with $\mathfrak{g}$, 
introduced by Drinfeld and Jimbo.
The algebra $\mathbf{U}^{\imath}$ was generalized by Letzter \cite{Let} and Kolb \cite{Kol}
to other types of symmetric pairs as well,
and is now referred to as \textit{$\imath$quantum groups of type $A\mathrm{II}$}. 
See \cite{Wan} for the development of the theory of quantum symmetric pairs and its applications to various branches. 

Although the algebra $\mathbf{U}^{\imath}$ can be thought of as 
a $q$-deformation of $U(\mathfrak{k})$,
unlike the ordinary quantum group $\mathbf{U}_q(\mathfrak{k})$,
$\mathbf{U}^{\imath}$ has a natural embedding
into $\mathbf{U}$.
By focusing on this property,
a new branching rule for $\mathfrak{g}$ and $\mathfrak{k}$
has recently been obtained in \cite{Wat1}.
Based on the new branching rule, we can show that
Theorem A is equivalent to the following theorem
(see Sections 3.3 and 5.3 for more details). 
\begin{thmB}[$=$ Theorem \ref{KeyProposition}]
	Let $\lambda$ be a partition of length at most $2n$.
	Then, there exist two bijections: 
	\begin{equation}
		\begin{aligned}
			\mathsf{\Phi} &: \mathsf{SST}_{2n}^{\text{$\widehat{\mathfrak{g}}$-$\mathrm{dom}$}}(\lambda) \stackrel{\sim}{\longrightarrow} \mathsf{SST}_{2n}^{\text{$\mathfrak{k}$-$\mathrm{hw}$}}(\lambda), &
			\mathsf{\Psi} &: \mathsf{SST}_{2n}^{\text{$\widehat{\mathfrak{g}}$-$\mathrm{dom}$}}(\lambda) \stackrel{\sim}{\longrightarrow} \mathsf{SST}_{2n}^{\text{$\mathfrak{k}$-$\mathrm{lw}$}}(\lambda)
		\end{aligned}
	\end{equation}
	such that for each $T \in \mathsf{SST}_{2n}^{\text{$\widehat{\mathfrak{g}}$-$\mathrm{dom}$}}(\lambda)$, the following equalities hold: 
	\begin{equation}
		\begin{aligned}
			& \mathsf{wt}_{\widehat{\mathfrak{g}}}(T) = \mathsf{wt}_{\mathfrak{k}}\left(\mathsf{\Phi}(T)\right), &
			& \mathsf{wt}_{\widehat{\mathfrak{g}}}(T) = \mathsf{wt}_{\mathfrak{k}}\left(\mathsf{\Psi}(T)\right).
		\end{aligned}
	\end{equation}
\end{thmB}
In this paper,
we prove Theorem B; 
Theorem A immediately follows from Theorem B (see Section 5.3). 

This paper is organized as follows.
In Section 2, we recall some standard combinatorial objects, 
such as partitions, semistandard tableaux, and words, 
which will be used in the subsequent sections.
In Section 3,
we restate the Naito--Sagaki conjecture precisely
and compute some examples.
In Section 4,
we recall the definition and some elementary properties of
quantized universal enveloping algebras, and then briefly review 
the theory of crystal bases.
In Section 5,
we present $\imath$quantum groups and the branching rule for them.
In Sections 6 and 7,
we provide the proof of Theorem B in the cases
$n=2$ and $n\geq3$, respectively.

\begin{acknowledgment}
	The authors would like to express their gratitude to the anonymous referee for valuable comments and suggestions.
	S.N. was partly supported by JSPS Grant-in-Aid for Scientific Research (C) JP21K03198.
	H.W. was partly supported by JSPS Grant-in-Aid for Young Scientists JP24K16903.
\end{acknowledgment}

\begin{notation}
	Throughout this paper, the following notation is used. 
	\begin{itemize}
		\item Let $\mathbb{N}$ and $\mathbb{Z}_{\geq0}$ denote the set of all positive integers and the set of all non-negative integers, respectively.
		\item For $a,b \in \mathbb{R}$ with $a \leq b$,
		let $[a,b]$ and $[a,b]_{\mathbb{R}}$ denote the intervals
		$\setlr{n\in\mathbb{N}}{a \leq n \leq b}$ and
		$\setlr{t\in\mathbb{R}}{a \leq t \leq b}$, respectively.
	\end{itemize}
\end{notation}

\bigskip

\section{Some combinatorial objects}
In this section,
we recall some standard combinatorial objects.
For more details,
refer to \cite{F}.

\subsection{Partitions and Young diagrams}
A weakly decreasing sequence of non-negative integers
$\lambda=\left(\lambda_1,\lambda_2,\dots\right)$
is called a \textit{partition}
if $\lambda_k=0$
for all but finitely many $k \in \mathbb{N}$; 
we denote by $\mathsf{Par}$ the set of all partitions. 
The partition whose parts are all zero is denoted by
$\emptyset=(0,0,\dots)$. 
For $\lambda \in \mathsf{Par}$,
we set 
\begin{equation}
	\begin{aligned}
		 & \ell(\lambda) := \sharp \setlr{k \in \mathbb{N}}{\lambda_k \neq 0},    & 
		 & |\lambda| := \lambda_1 + \lambda_2 + \dots + \lambda_{\ell(\lambda)}.
	\end{aligned}
\end{equation}
For $m \in \mathbb{N}$,
let $\mathsf{Par}_{\leq m}$ denote 
the set of partitions $\lambda$ such that $\ell(\lambda) \leq m$.
For $\lambda,\mu \in \mathsf{Par}$,
we write $\lambda \supset \mu$ if $\lambda_k \geq \mu_k$
holds for all $k \in \mathbb{N}$.

Throughout this paper,
a finite subset of $\mathbb{N} \times \mathbb{N}$
is referred to as a \textit{diagram}.

Hereafter, elements of a diagram will be referred to as \textit{boxes}.
For two boxes $(x_1,y_1)$ and $(x_2,y_2)$,
we say that $(x_1,y_1)$ is farther to the right 
than $(x_2,y_2)$ if $x_1 \geq x_2$,
and that $(x_1,y_1)$ is lower 
than $(x_2,y_2)$ if $y_1 \geq y_2$; 
in the example above, 
$(6,5)$ is lower than $(6,4)$.

Given a partition $\lambda \in \mathsf{Par}$,
the diagram $D(\lambda)$ is defined as follows: 
\begin{equation}
	D(\lambda) := \setlr{ (x,y) \in \mathbb{N}\times\mathbb{N}}{\begin{aligned} & 1 \leq y \leq \ell(\lambda), \\ & 1 \leq x \leq \lambda_y \end{aligned}}; 
\end{equation}
we call the diagram $D(\lambda)$ a \textit{Young diagram}.

\subsection{Tableaux}
For a diagram $D$,
a map $T : D \to \mathbb{N}$ is called a \textit{tableau}.
The diagram $D$ is called the \textit{shape} of the tableau $T$,
and elements of the image of $D$ under the map $T$ are referred to as \textit{entries} of $T$.
A tableau of shape $\emptyset$ is denoted by $\emptyset$.
Also, for $(x,y) \in D$,
the image of $(x,y)$ under $T$ is denoted by $T(x,y)$.
For $m \in \mathbb{N}$,
we define $T[m] = \sharp \, T^{-1}(m) = \sharp \setlr{(x,y) \in D}{T(x,y) = m}$.
A tableau $T$ of shape $D$ is called a \textit{semistandard tableau}
if for every $(x,y) \in D$,
the following conditions are satisfied: 
\begin{itemize}
	\item if $(x+1,y) \in D$, then $T(x,y) \leq T(x+1,y)$; 
	\item if $(x,y+1) \in D$, then $T(x,y) < T(x,y+1)$.
\end{itemize}
Let $m \in \mathbb{N}$.
We denote by $\mathsf{SST}(\lambda)$ the set of semistandard tableaux of shape
$\lambda = D(\lambda) \in \mathsf{Par}$, and by $\mathsf{SST}_m(\lambda)$ 
the set of semistandard tableaux
of shape $\lambda \in \mathsf{Par}$
whose entries are all less than or equal to $m$. 
We set
\begin{equation}
	\begin{aligned}
		 & \mathsf{SST} := \bigsqcup_{\lambda \in \mathsf{Par}} \mathsf{SST}(\lambda),      & 
		 & \mathsf{SST}_m := \bigsqcup_{\lambda \in \mathsf{Par}} \mathsf{SST}_m(\lambda).
	\end{aligned}
\end{equation}
We say that a semistandard tableau $T$ is \textit{symplectic}
if $T(1,y) \geq 2y-1$ for all $y$
(see \cite[Section 4]{Kin}).
We denote by $\mathsf{SpT}(\lambda)$ the set of symplectic tableaux of shape $\lambda \in \mathsf{Par}$, 
and by $\mathsf{SpT}_{m}(\lambda)$ 
the set of symplectic tableaux of shape $\lambda \in \mathsf{Par}$
whose entries are all less than
or equal to $m$.
We set
\begin{equation}
	\begin{aligned}
		 & \mathsf{SpT} := \bigsqcup_{\lambda \in \mathsf{Par}} \mathsf{SpT}(\lambda),      & 
		 & \mathsf{SpT}_m := \bigsqcup_{\lambda \in \mathsf{Par}} \mathsf{SpT}_m(\lambda).
	\end{aligned}
\end{equation}

Let $T \in \mathsf{SST}$ be a semistandard tableau, and
$m$ a positive integer.
Then, a new semistandard tableau $m \rightarrow T$
is defined by 
\textit{Schensted's insertion algorithm}
(see \cite[Chapter 7.1]{BS}).
We define $\varpi_l$
to be the partition whose parts are all $1$
and whose length is $\ell(\varpi_l)=l$.
For $T \in \mathsf{SST}(\varpi_l)$ and $S \in \mathsf{SST}$,
we define $T * S$ to be the semistandard tableau given as: 
\begin{equation}
	T(1,l) \rightarrow T(1,l-1) \rightarrow \cdots \rightarrow T(1,2) \rightarrow T(1,1) \rightarrow S.
\end{equation}

\subsection{Words}
Let $\mathsf{W}$ be the set of
all finite sequences of positive integers,
and refer to its elements as \textit{words}.
For $m \in \mathbb{N}$ and a word $w$,
let $w[m]$ denote the number of occurrences of $m$ in $w$.
Also, we denote by $\mathsf{W}_m$ the set of words consisting of those elements
that are less than or equal to $m$. 
The equivalence relation $\cong_{\mathsf{K}}$ on $\mathsf{W}$
generated by the following relation is called the
\textit{Knuth equivalence relation}:
\begin{itemize}
	\item if $x < y \leq z$, then $\cdots\,y\,x\,z\,\cdots \cong_{\mathsf{K}} \cdots\,y\,z\,x\,\cdots$; 
	\item if $x \leq y < z$, then $\cdots\,x\,z\,y\,\cdots \cong_{\mathsf{K}} \cdots\,z\,x\,y\,\cdots$.
\end{itemize}
Let $D$ be a diagram and $T$ a tableau of shape $D$.
A word $w_r(T)$, called the \textit{row word} of $T$,
is the word obtained by reading the entries of $T$
from the bottom row to the top row, and from left to right in each row. 
Similarly, a word $w_c(T)$, called the \textit{inverse column word} of $T$,
is the word obtained by reading the entries of $T$
from the rightmost column to the leftmost column, and from top to bottom in each column. 
It is well-known (see \cite[Part I, Chapter 2, Theorem]{F}) that for each word $w$,
there exist a unique tableau $T\in\mathsf{SST}$ such that
$w \cong_{\mathsf{K}} w_r(T)$. 

\bigskip

\section{Naito--Sagaki conjecture}

\subsection{General linear Lie algebras}
Let $n \geq 2$ be an integer.
For $1 \leq i, j \leq 2n$, let 
$E_{i,j}:=\left(\delta_{ki}\delta_{lj}\right)_{1\leq k,l\leq2n}$
be the corresponding matrix unit.
We write $\mathfrak{g}:=\mathfrak{gl}_{2n}(\mathbb{C})$,
and set
\begin{equation}
	\begin{aligned}
		 & P^{\vee} :=\sum_{i=1}^{2n}\mathbb{Z}E_{i,i},                                                           & 
		 & P :=\mathrm{Hom}_{\mathbb{Z}}(P^{\vee},\mathbb{Z}); 
	\end{aligned}
\end{equation}
let $\langle\cdot,\cdot\rangle : P \times P^{\vee} \to \mathbb{Z}$ denote the 
canonical pairing. 
Also, we set 
\begin{equation}
	\begin{aligned}
		 & P^{\vee}_{\mathbb{R}}:= \mathbb{R}\otimes_{\mathbb{Z}}P^{\vee}=\sum_{i=1}^{2n}\mathbb{R}E_{ii},                      & 
		 & P_{\mathbb{R}} :=\mathbb{R}\otimes_{\mathbb{Z}}P=\mathrm{Hom}_{\mathbb{R}}(P^{\vee}_{\mathbb{R}},\mathbb{R}). & 
	\end{aligned}
\end{equation}
We define
$\varepsilon_i \in P_{\mathbb{R}}$, $1\leq i \leq 2n$, 
by
$\langle\varepsilon_i,E_{jj}\rangle=\delta_{ij}$.
Let $I$ be the set of vertices of the Dynkin diagram
of type $A_{2n-1}$:
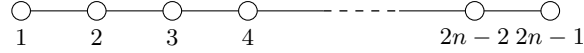
\vspace{10pt}
\begin{center}
	\begin{tikzpicture}[x=5mm,y=5mm]
		\draw (0,-4)--(2,-4);
		\draw (2,-4)--(4,-4);
		\draw (4,-4)--(6,-4);
		\draw (6,-4)--(8,-4);
		\draw[dashed,line width=0.5pt] (8,-4)--(10,-4);
		\draw (10,-4)--(12,-4);
		\draw (12,-4)--(14,-4);
		\draw[fill=white] (0,-4) circle[radius=3.5pt];
		\draw[fill=white] (2,-4) circle[radius=3.5pt];
		\draw[fill=white] (4,-4) circle[radius=3.5pt];
		\draw[fill=white] (6,-4) circle[radius=3.5pt];
		\draw[fill=white] (12,-4) circle[radius=3.5pt];
		\draw[fill=white] (14,-4) circle[radius=3.5pt];
		\node (l) at (0,-4.7) {\scriptsize $1$};
		\node (m) at (2,-4.7) {\scriptsize $2$};
		\node (n) at (4,-4.7) {\scriptsize $3$};
		\node (o) at (6,-4.7) {\scriptsize $4$};
		\node (p) at (12,-4.7) {\scriptsize $2n-2$};
		\node (q) at (14,-4.7) {\scriptsize $2n-1$};
	\end{tikzpicture}
\end{center}
We denote the corresponding Cartan matrix by
$\mathsf{C}=\left(c_{ij}\right)_{i,j\in I}$. 

Let us define the elements (called the Chevalley generators) 
$e_i,f_i,h_i \in \mathfrak{g}$, $i \in  I$, by
$e_i := E_{i+1,i},\,\,f_i := E_{i,i+1},\,\,h_i := E_{i,i} - E_{i+1,i+1}$.
The Lie subalgebra of $\mathfrak{g}$ generated by 
$e_i,f_i,h_i \in \mathfrak{g}$, $i \in  I$, 
is the complex special linear Lie algebra
$\mathfrak{sl}_{2n}(\mathbb{C})$.
Also, 
we define $\alpha_i$ for $i \in I$ by
$\alpha_i := \varepsilon_i - \varepsilon_{i+1}$,
and set
\begin{equation}
	\begin{aligned}
		 & P^+ := \left\{\lambda \in P\mid\langle\lambda,h_i\rangle \geq 0, \,\,i \in  I\right\},                                                    \\
		 & P^+_{\geq0} := \left\{\lambda \in P\mid\langle\lambda,h_i\rangle \geq 0, \,\,i \in  I,\,\,\langle\lambda,E_{2n,2n}\rangle \geq 0\right\}.
	\end{aligned}
\end{equation}
Let $L(\lambda)$
denote the irreducible highest weight $\mathfrak{g}$-module of
highest weight $\lambda \in P^+$.

\subsection{Fixed point subalgebra $\widehat{\mathfrak{g}}$}
Let
$\sigma :  I \to  I$
be the following diagram automorphism:
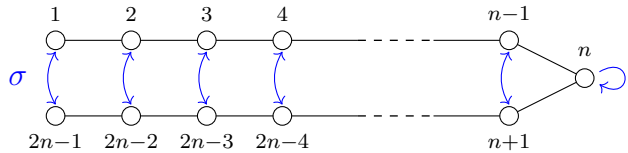
\begin{center}
	\begin{tikzpicture}[x=5mm,y=5mm]
		\node (x) at (-1,-1) {\textcolor{blue}{$\sigma$}};
		\draw (0,0)--(2,0);
		\draw (2,0)--(4,0);
		\draw (4,0)--(6,0);
		\draw (6,0)--(8,0);
		\draw (10,0)--(12,0);
		\draw (12,0)--(14,-1);
		\draw (0,-2)--(2,-2);
		\draw (2,-2)--(4,-2);
		\draw (4,-2)--(6,-2);
		\draw (6,-2)--(8,-2);
		\draw (10,-2)--(12,-2);
		\draw (12,-2)--(14,-1);
		\draw[dashed,line width=0.5pt] (8,0)--(10,0);
		\draw[dashed,line width=0.5pt] (8,-2)--(10,-2);
		\draw[fill=white] (0,0) circle[radius=3.5pt];
		\draw[fill=white] (2,0) circle[radius=3.5pt];
		\draw[fill=white] (4,0) circle[radius=3.5pt];
		\draw[fill=white] (6,0) circle[radius=3.5pt];
		\draw[fill=white] (12,0) circle[radius=3.5pt];
		\draw[fill=white] (14,-1) circle[radius=3.5pt];
		\draw[fill=white] (0,-2) circle[radius=3.5pt];
		\draw[fill=white] (2,-2) circle[radius=3.5pt];
		\draw[fill=white] (4,-2) circle[radius=3.5pt];
		\draw[fill=white] (6,-2) circle[radius=3.5pt];
		\draw[fill=white] (12,-2) circle[radius=3.5pt];
		\node (a) at (0,0.7) {$\scriptstyle 1$};
		\node (b) at (2,0.7) {$\scriptstyle 2$};
		\node (c) at (4,0.7) {$\scriptstyle 3$};
		\node (d) at (6,0.7) {$\scriptstyle 4$};
		\node (e) at (12,0.7) {$\scriptstyle n-1$};
		\node (f) at (14,-0.3) {$\scriptstyle n$};
		\node (g) at (0,-2.7) {$\scriptstyle 2n-1$};
		\node (h) at (2,-2.7) {$\scriptstyle 2n-2$};
		\node (i) at (4,-2.7) {$\scriptstyle 2n-3$};
		\node (j) at (6,-2.7) {$\scriptstyle 2n-4$};
		\node (k) at (12,-2.7) {$\scriptstyle n+1$};
		\draw[<->,color=blue] (0,-0.3) to[bend right] (0,-1.7);
		\draw[<->,color=blue] (2,-0.3) to[bend right] (2,-1.7);
		\draw[<->,color=blue] (4,-0.3) to[bend right] (4,-1.7);
		\draw[<->,color=blue] (6,-0.3) to[bend right] (6,-1.7);
		\draw[<->,color=blue] (12,-0.3) to[bend right] (12,-1.7);
		\node (X) at (14.1,-1) {};
		\draw[->,color=blue] (X) to[out=30,in=-30,loop] ();
	\end{tikzpicture}
\end{center}
The assignment  
$e_i\mapsto e_{\sigma(i)},\,\,f_i\mapsto f_{\sigma(i)},\,\,h_i\mapsto h_{\sigma(i)}$ for $i \in I$ 
induces a Lie algebra automorphism of $\mathfrak{sl}_{2n}(\mathbb{C})$; 
we use the same symbol $\sigma$ 
for this Lie algebra automorphism, 
and set
$\widehat{\mathfrak{g}} :=\left\{x \in \mathfrak{sl}_{2n}(\mathbb{C})\mid\sigma(x)=x\right\}$.
Let $\widehat{ I}$ be the set of vertices of
the Dynkin diagram of type $C_{n}$:
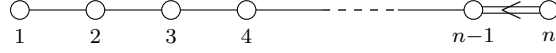
\begin{center}
	\begin{tikzpicture}[x=5mm,y=5mm]
		\draw (0,-4)--(2,-4);
		\draw (2,-4)--(4,-4);
		\draw (4,-4)--(6,-4);
		\draw (6,-4)--(8,-4);
		\draw[dashed,line width=0.5pt] (8,-4)--(10,-4);
		\draw (10,-4)--(12,-4);
		\draw[double distance=2pt] (12,-4)--node{$<$}(14,-4);
		\draw[fill=white] (0,-4) circle[radius=3.5pt];
		\draw[fill=white] (2,-4) circle[radius=3.5pt];
		\draw[fill=white] (4,-4) circle[radius=3.5pt];
		\draw[fill=white] (6,-4) circle[radius=3.5pt];
		\draw[fill=white] (12,-4) circle[radius=3.5pt];
		\draw[fill=white] (14,-4) circle[radius=3.5pt];
		\node (l) at (0,-4.7) {$\scriptstyle 1$};
		\node (m) at (2,-4.7) {$\scriptstyle 2$};
		\node (n) at (4,-4.7) {$\scriptstyle 3$};
		\node (o) at (6,-4.7) {$\scriptstyle 4$};
		\node (p) at (12,-4.7) {$\scriptstyle n-1$};
		\node (q) at (14,-4.7) {$\scriptstyle n$};
	\end{tikzpicture}
\end{center}
We define elements 
$\widehat{e}_i,\widehat{f}_i,\widehat{h}_i \in \widehat{\mathfrak{g}}$, $i \in \widehat{I}$
by:
\begin{equation}
	\begin{aligned}
		        & \widehat{e}_i :=
		\left\{
		\begin{aligned}
			 & e_i + e_{\sigma(i)} &  & \text{if } 1 \leq i < n, \\
			 & e_n                 &  & \text{if } i = n,
		\end{aligned}
		\right. & 
		        & \widehat{f}_i :=
		\left\{
		\begin{aligned}
			 & f_i + f_{\sigma(i)} &  & \text{if } 1 \leq i < n, \\
			 & f_n                 &  & \text{if } i = n,
		\end{aligned}
		\right.                    \\
		        & \widehat{h}_i :=
		\left\{
		\begin{aligned}
			 & h_i + h_{\sigma(i)} &  & \text{if } 1 \leq i < n, \\
			 & h_n                 &  & \text{if } i = n.
		\end{aligned}
		\right. & 
	\end{aligned}
\end{equation}
The Lie subalgebra
$\widehat{\mathfrak{g}}$ is isomorphic to the symplectic Lie algebra $\mathfrak{sp}_{2n}(\mathbb{C})$,
with $\widehat{e}_i,\widehat{f}_i,\widehat{h}_i \in \widehat{\mathfrak{g}}$, $i \in \widehat{I}$, the Chevalley generators.
We set
\begin{equation}
	\begin{aligned}
		 & \widehat{P}^{\vee} :=\sum_{i \in \widehat{ I}}\mathbb{Z}\widehat{h}_i,                                                                                     & 
		 & \widehat{P} :=\mathrm{Hom}_{\mathbb{Z}}(\widehat{P}^{\vee},\mathbb{Z}),                                                                                    & 
		 & \widehat{P}^{+} := \left\{\mu \in \widehat{P} \mathrel{}\middle|\mathrel{} \langle\mu,\widehat{h}_i\rangle \geq 0\,\,(i \in \widehat{ I})\right\}. & 
	\end{aligned}
\end{equation}
Also, we set 
\begin{equation}
	\begin{aligned}
		 & \widehat{P}^{\vee}_{\mathbb{R}} :=\mathbb{R}\otimes_{\mathbb{Z}}\widehat{P}^{\vee} =\sum_{i \in \widehat{I}}\mathbb{R}\widehat{h}_i,                                & 
		 & \widehat{P}_{\mathbb{R}} :=\mathbb{R}\otimes_{\mathbb{Z}}\widehat{P} =\mathrm{Hom}_{\mathbb{R}}(\widehat{P}^{\vee}_{\mathbb{R}},\mathbb{R}). & 
	\end{aligned}
\end{equation}
The inclusion 
$\widehat{P}^{\vee}_{\mathbb{R}} \hookrightarrow P^{\vee}_{\mathbb{R}}$
induces the natural restriction map
$\widehat{\cdot} : P_{\mathbb{R}} \to \widehat{P}_{\mathbb{R}}$.
Let $\overline{i} := 2n-i+1$ for $1 \leq i \leq 2n$; 
it is easy to verify that $\widehat{\varepsilon}_{\overline{i}} = -\widehat{\varepsilon}_i$.
For each $1 \leq i \leq n-1$,
we denote by $\widehat{\mathfrak{g}}_i$
the Lie subalgebra of $\widehat{\mathfrak{g}}$
generated by $\widehat{e}_{i},\widehat{f}_{i},\widehat{e}_{2\widehat{\varepsilon}_i},\widehat{f}_{2\widehat{\varepsilon}_i}$,
where $\widehat{e}_{2\widehat{\varepsilon}_i}$ and
$\widehat{f}_{2\widehat{\varepsilon}_i}$ are the root vectors for the roots $\pm2\widehat{\varepsilon}_i$, respectively.
Note that each $\widehat{\mathfrak{g}}_i$ is isomorphic to $\mathfrak{sp}_4(\mathbb{C})$, and 
$\widehat{\mathfrak{g}}$ is generated as a Lie algebra by $\widehat{\mathfrak{g}}_1 , \widehat{\mathfrak{g}}_2 , \ldots , \widehat{\mathfrak{g}}_{n-1}$:
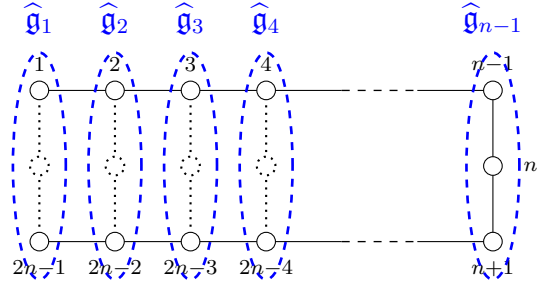
\begin{center}
	\begin{tikzpicture}[x=5mm,y=5mm]
		\draw (0,0)--(2,0);
		\draw (2,0)--(4,0);
		\draw (4,0)--(6,0);
		\draw (6,0)--(8,0);
		\draw (10,0)--(12,0);
		\draw (12,0)--(12,-2);
		\draw (0,-4)--(2,-4);
		\draw (2,-4)--(4,-4);
		\draw (4,-4)--(6,-4);
		\draw (6,-4)--(8,-4);
		\draw (10,-4)--(12,-4);
		\draw (12,-4)--(12,-2);
		\draw[dotted, line width=0.8pt] (0,0)--(0,-2);
		\draw[dotted, line width=0.8pt] (0,-4)--(0,-2);
		\draw[dotted, line width=0.8pt] (2,0)--(2,-2);
		\draw[dotted, line width=0.8pt] (2,-4)--(2,-2);
		\draw[dotted, line width=0.8pt] (4,0)--(4,-2);
		\draw[dotted, line width=0.8pt] (4,-4)--(4,-2);
		\draw[dotted, line width=0.8pt] (6,0)--(6,-2);
		\draw[dotted, line width=0.8pt] (6,-4)--(6,-2);
		\draw[dashed,line width=0.5pt] (8,0)--(10,0);
		\draw[dashed,line width=0.5pt] (8,-4)--(10,-4);
		\draw[fill=white] (0,0) circle[radius=3.5pt];
		\draw[fill=white] (2,0) circle[radius=3.5pt];
		\draw[fill=white] (4,0) circle[radius=3.5pt];
		\draw[fill=white] (6,0) circle[radius=3.5pt];
		\draw[fill=white] (12,0) circle[radius=3.5pt];
		\draw[fill=white] (12,-2) circle[radius=3.5pt];
		\draw[fill=white] (0,-4) circle[radius=3.5pt];
		\draw[fill=white] (2,-4) circle[radius=3.5pt];
		\draw[fill=white] (4,-4) circle[radius=3.5pt];
		\draw[fill=white] (6,-4) circle[radius=3.5pt];
		\draw[fill=white] (12,-4) circle[radius=3.5pt];
		\draw[dotted, fill=white, line width=0.8pt] (0,-2) circle[radius=3.5pt];
		\draw[dotted, fill=white, line width=0.8pt] (2,-2) circle[radius=3.5pt];
		\draw[dotted, fill=white, line width=0.8pt] (4,-2) circle[radius=3.5pt];
		\draw[dotted, fill=white, line width=0.8pt] (6,-2) circle[radius=3.5pt];
		\node (a) at (0,0.7) {$\scriptstyle 1$};
		\node (b) at (2,0.7) {$\scriptstyle 2$};
		\node (c) at (4,0.7) {$\scriptstyle 3$};
		\node (d) at (6,0.7) {$\scriptstyle 4$};
		\node (e) at (12,0.7) {$\scriptstyle n-1$};
		\node (f) at (13,-2) {$\scriptstyle n$};
		\node (g) at (0,-4.7) {$\scriptstyle 2n-1$};
		\node (h) at (2,-4.7) {$\scriptstyle 2n-2$};
		\node (i) at (4,-4.7) {$\scriptstyle 2n-3$};
		\node (j) at (6,-4.7) {$\scriptstyle 2n-4$};
		\node (k) at (12,-4.7) {$\scriptstyle n+1$};
		\draw[dashed, color=blue, line width=1pt] (0,-2) circle[x radius=0.7,y radius=3];
		\draw[dashed, color=blue, line width=1pt] (2,-2) circle[x radius=0.7,y radius=3];
		\draw[dashed, color=blue, line width=1pt] (4,-2) circle[x radius=0.7,y radius=3];
		\draw[dashed, color=blue, line width=1pt] (6,-2) circle[x radius=0.7,y radius=3];
		\draw[dashed, color=blue, line width=1pt] (12,-2) circle[x radius=0.7,y radius=3];
		\node at (0,1.9) {\textcolor{blue}{$\widehat{\mathfrak{g}}_1$}};
		\node at (2,1.9) {\textcolor{blue}{$\widehat{\mathfrak{g}}_2$}};
		\node at (4,1.9) {\textcolor{blue}{$\widehat{\mathfrak{g}}_3$}};
		\node at (6,1.9) {\textcolor{blue}{$\widehat{\mathfrak{g}}_4$}};
		\node at (12,1.9) {\textcolor{blue}{$\widehat{\mathfrak{g}}_{n-1}$}};
	\end{tikzpicture}
\end{center}
Let $\widehat{L}(\mu)$ denote the irreducible highest weight $\widehat{\mathfrak{g}}$-module of
highest weight $\mu \in \widehat{P}^+$.

\subsection{Statement of the Naito--Sagaki conjecture}
In this subsection,
we restate the Naito--Sagaki conjecture precisely; 
for more details, 
refer to \cite{NS}.
The sets
$P^{+}_{\geq0}$ and $\widehat{P}^+$ have the following descriptions:
\begin{equation}
	\begin{aligned}
		P^{+}_{\geq0}       & = \left\{\lambda_1\varepsilon_1 + \cdots + \lambda_{2n}\varepsilon_{2n} \in P\mid\lambda_1 \geq \cdots \geq \lambda_{2n} \geq 0\right\},                                       \\
		\widehat{P}^{+} & = \left\{\mu_1\widehat{\varepsilon}_1 + \cdots + \mu_{n}\widehat{\varepsilon}_{n} \in \widehat{P} \mathrel{}\middle|\mathrel{} \mu_1 \geq \cdots \geq \mu_{n} \geq 0\right\}.
	\end{aligned}
\end{equation}
Therefore, we can identify these sets with
$\mathsf{Par}_{\leq 2n}$ and $\mathsf{Par}_{\leq n}$,
respectively.
From now on, we regard elements of
$P^{+}_{\geq0}$ and $\widehat{P}^+$ as partitions or Young diagrams.

Let us fix a Young diagram $\lambda \in P^{+}_{\geq0}$, 
and set $N := |\lambda|$.
Divide the interval $[0,1]_{\mathbb{R}}$ into $N$ subintervals of the same length,
as $0=t_0<t_1<\dots<t_N=1$.
For each $T \in \mathsf{SST}_{2n}(\lambda)$,
let the inverse column word of $T$ be $w_c(T)=i_1\,i_2\,\cdots\,i_N$,
and define a piecewise-linear, continuous map
$\pi_T : [0,1]_{\mathbb{R}} \to P_{\mathbb{R}}$ as follows:
\begin{equation}
	\begin{aligned}
		 & \pi_T(t) := \sum_{k=1}^{j-1} \varepsilon_{i_k} + N(t-t_{j-1}) \varepsilon_{i_j}, & 
		 & t_{j-1} \leq t \leq t_j, \,\, 1 \leq j \leq N.
	\end{aligned}
\end{equation}
Let $\mathbb{B}(\lambda)$ denote 
the set of all piecewise-linear, continuous maps obtained in this way.
Also, for $\pi_T \in \mathbb{B}(\lambda)$,
define
$\widehat{\pi}_T : [0,1]_{\mathbb{R}} \to \widehat{P}_{\mathbb{R}}$
by: $\widehat{\pi}_T(t) = \widehat{\pi_T(t)}$,
where the right-hand side is the image of $\pi_T(t)$ under
$\widehat{\cdot} : P_{\mathbb{R}} \to \widehat{P}_{\mathbb{R}}$.
In other words,
$\widehat{\pi}_T$ is the piecewise-linear, continuous map defined as follows:
\begin{equation}
	\begin{aligned}
		 & \widehat{\pi}_T(t) = \sum_{k=1}^{j-1} \widehat{\varepsilon}_{i_k} + N(t-t_{j-1}) \widehat{\varepsilon}_{i_j}, & 
		 & t_{j-1} \leq t \leq t_j, \,\, 1 \leq j \leq N.
	\end{aligned}
\end{equation}
We denote by $\widehat{\mathbb{B}}(\widehat{\lambda})$ 
the set of all piecewise-linear, continuous maps
obtained in this manner. 
For $\widehat{\pi}_T \in \widehat{\mathbb{B}}(\widehat{\lambda})$,
we set
$\mathsf{wt}_{\widehat{\mathfrak{g}}}(\widehat{\pi}_T) = \mathsf{wt}_{\widehat{\mathfrak{g}}}(T) := \widehat{\pi}_T(1)$.
A path $\pi_T \in \mathbb{B}(\lambda)$ is said to be
\textit{$\widehat{\mathfrak{g}}$-dominant}
if, for all $i \in \widehat{ I}$ and $t \in [0,1]_{\mathbb{R}}$, the inequality 
$\langle\widehat{\pi}_T(t),\widehat{h}_i\rangle \geq 0$ holds. 
Let $\mathsf{SST}^{\text{$\widehat{\mathfrak{g}}$-$\mathrm{dom}$}}(\lambda)$
denote the set of those tableaux $T \in \mathsf{SST}_{2n}(\lambda)$
for which $\pi_T$ is $\widehat{\mathfrak{g}}$-dominant.
\begin{rem}
	A path
	$\pi_T \in \mathbb{B}(\lambda)$ is $\widehat{\mathfrak{g}}$-dominant
	if and only if
	the image of $\widehat{\pi}_T$ is contained in $\widehat{C}$,
	where $\widehat{C}$ is the (closure of the) fundamental Weyl chamber given by: 
	\begin{equation}
		\widehat{C} :=
		\left\{
		x_1\widehat{\varepsilon}_1 + \cdots + x_n\widehat{\varepsilon}_n \in \widehat{P}_{\mathbb{R}}
		\mathrel{}\middle|\mathrel{}
		x_1 \geq \cdots \geq x_n \geq 0
		\right\}.
	\end{equation}
\end{rem}
The 
\textit{Naito--Sagaki conjecture} is stated as follows. 
\begin{thm}[{\cite[Conjecture 2.5.3]{NS}} $=$ {\cite[Theorem 1]{ST}}]\label{NS2}
	We have the following decomposition of 
	$\mathsf{Res}^{\mathfrak{g}}_{\widehat{\mathfrak{g}}}\,L(\lambda)$ as a $\widehat{\mathfrak{g}}$-module:
	\begin{equation}
		\mathsf{Res}^{\mathfrak{g}}_{\widehat{\mathfrak{g}}}\,L(\lambda) \cong
		\bigoplus_{\substack{\pi \in \mathbb{B}(\lambda) \\ \widehat{\pi} \, : \, \text{$\widehat{\mathfrak{g}}$-$\mathrm{dominant}$}}} \widehat{L}\left(\mathsf{wt}_{\widehat{\mathfrak{g}}}(\widehat{\pi}_T)\right).
	\end{equation}
	Equivalently, for all $\mu \in \widehat{P}^+$,
	the following equality holds:
	\begin{equation}
		\left[\mathsf{Res}^{\mathfrak{g}}_{\widehat{\mathfrak{g}}}\,L(\lambda) : \widehat{L}(\mu)\right]
		=
		\sharp
		\left\{
		T \in \mathsf{SST}_{2n}^{\text{$\widehat{\mathfrak{g}}$-$\mathrm{dom}$}}(\lambda)
		\mathrel{}\middle|\mathrel{}
		\mathsf{wt}_{\widehat{\mathfrak{g}}}(T) = \mu
		\right\}. \label{3.12}
	\end{equation}
\end{thm}

\bigskip

\section{Quantized universal enveloping algebras and Kashiwara crystals}
In this section,
we review the representation theory of
(Drinfeld--Jimbo) quantum groups
and the theory of crystal, or canonical bases, due to Lusztig \cite{Lu1} and Kashiwara \cite{Kas2}; 
as for the notation and some basic results, we mainly use those in \cite{Kas}, \cite{J}, \cite{HK}, or \cite{Lu2}. 
Throughout this paper,
we use the following notation:
\begin{itemize}
	\item $\mathbf{A} = \mathbb{C}[q,q^{-1}]$,
	\item $\mathbf{A}_{1} = \setlr{f \in \mathbb{C}(q)}{\text{$f$ is regular at $q=1$}}$,
	\item $\mathbf{A}_{\infty} = \setlr{f \in \mathbb{C}(q)}{\text{$f$ is regular at $q=\infty$}}$.
\end{itemize}

\subsection{Quantized universal enveloping algebra of $\mathfrak{g}$}
Let $\mathbf{U}$ denote the $\mathbb{C}(q)$-algebra generated by 
$K_h$, $h \in P^{\vee}$, and $E_i, F_i$, $i \in I$, 
subject to the following relations:
\begin{multline}
	\left\{
	\begin{aligned}
		 & K_0 = 1, \,\,\,\,\,\,\,\, K_h K_{h'} = K_{h+h'},                                                                                                         \\
		 & K_h E_i K_{h}^{-1} = q^{\langle\alpha_i,h\rangle}E_i, \,\,\,\,\,\,\,\, K_h F_i K_{h}^{-1} = q^{-\langle\alpha_i,h\rangle}F_i,                            \\
		 & \left[E_i,F_j\right] = \delta_{ij}\frac{K_{h_i} - K_{h_i}^{-1}}{q-q^{-1}},                                                                               \\
		 & \sum_{r+s=1-c_{ij}} (-1)^r E_i^{(r)}E_jE_i^{(s)} = 0, \,\,\,\,\,\,\,\, \sum_{r+s=1-c_{ij}} (-1)^r F_i^{(r)}F_jF_i^{(s)} = 0 \,\,\,\,\,\,\,\, \text{if $i \neq j$}, \\
	\end{aligned}
	\right.
\end{multline}
where $E_i^{(k)}$ and $F_i^{(k)}$ denote the 
\textit{$q$-divided powers}:
\begin{equation}
	\begin{aligned}
		 & E_i^{(k)} := \frac{E_i^k}{[k]!},       & 
		 & F_i^{(k)} := \frac{F_i^k}{[k]!},       & 
		 & [k]! := \prod_{l=0}^k [l],             & 
		 & [l] := \frac{q^l - q^{-l}}{q-q^{-1}}.
	\end{aligned}
\end{equation}
The $\mathbb{C}(q)$-algebra $\mathbf{U}$ is a $q$-analog of the universal enveloping algebra of $\mathfrak{g}$.
For each $i \in  I$,
we denote by $\mathbf{U}_i$ the $\mathbb{C}(q)$-subalgebra of $\mathbf{U}$
generated by $K_{h_i},E_i,F_i$; 
$\mathbf{U}_i$ is a $q$-deformation of
the universal enveloping algebra of $\mathfrak{sl}_2(\mathbb{C})$.

For $\lambda \in P^+_{\geq0}$,
we define the finite-dimensional, irreducible highest weight 
$\mathbf{U}$-module $L_q(\lambda)$, 
with \textit{highest weight vector} $\eta_{\lambda}$, 
as follows:
\begin{equation}
	\begin{aligned}
		 & L_q(\lambda) := \mathbf{U} \left/ \left(\sum_{i \in  I} \mathbf{U}E_i + \sum_{i \in  I} \mathbf{U}F_i^{\langle\lambda,h_i\rangle+1} + \sum_{h \in P^{\vee}} \mathbf{U}\left(K_h-q^{\langle\lambda,h\rangle}\right) \right)\right., & 
		 & \eta_{\lambda} := \overline{1}.
	\end{aligned}
\end{equation}
In the classical limit as $q\to1$,
$L_q(\lambda)$ becomes $L(\lambda)$.
In principle, the structure of $L_q(\lambda)$
can be controlled by stacking the information about 
$\mathsf{Res}^{\mathbf{U}}_{\mathbf{U}_i}L_q(\lambda)$ for $i \in I$, 
as formulated by the theory of \textit{crystal bases}; 
see the next subsection.

\subsection{Crystal bases}
Let $\xi$ be a weight vector of $L_q(\lambda)$.
By applying the representation theory of $\mathfrak{sl}_2(\mathbb{C})$
to $\mathsf{Res}^{\mathbf{U}}_{\mathbf{U}_i}L_q(\lambda)$ for each $i \in I$,
we can show that
there exists a sequence of weight vectors
$\xi_n$, $n \in \mathbb{Z}_{\geq0}$, 
all of which are killed by the action of $E_i$,
such that $\xi$ can be uniquely written as: 
\begin{equation}
	\begin{aligned}
		 & \xi = \sum_{n\, \geq \,\max\{0,-\langle\mathsf{wt}(\xi),h_i\rangle\}} F_i^{(n)}\xi_n,                   & 
		 & \sharp \left\{n \in \mathbb{Z}_{\geq0}\mathrel{}\middle|\mathrel{} \xi_n \neq 0\right\} < +\infty.
	\end{aligned}
\end{equation}
Using this presentation of $\xi$, we define 
$\widetilde{E}_i\xi,\widetilde{F}_i\xi$ as follows:
\begin{equation}
	\widetilde{E}_i\xi = \sum_{n\, \geq \,\max\{0,-\langle\mathsf{wt}(\xi),h_i\rangle\}} F_i^{(n-1)}\xi_n, \,\,\,\,\,\,\,\,
	\widetilde{F}_i\xi = \sum_{n\, \geq \,\max\{0,-\langle\mathsf{wt}(\xi),h_i\rangle\}} F_i^{(n+1)}\xi_n.
\end{equation}
The operators
$\widetilde{E}_i,\widetilde{F}_i$
defined above on each weight space are extended
$\mathbb{C}(q)$-linearly
to the whole of $L_q(\lambda)$,
and the extended operators are called
the \textit{Kashiwara operators}.
Using the operators $\widetilde{F}_i$, $i \in I$, we define
\begin{equation}
	\mathcal{L}(\lambda) := \sum_{r \in \mathbb{Z}_{\geq0}} \sum_{i_1,\ldots,i_r \in  I} \mathbf{A}_{\infty} \widetilde{F}_{i_1}\cdots\widetilde{F}_{i_r}\eta_{\lambda}.
\end{equation}
It is well-known that
$\mathcal{L}(\lambda)$
is closed under the action of $\widetilde{E}_i$ and $\widetilde{F}_i$ for $i \in I$.
Therefore, the action of $\widetilde{E}_i$ and $\widetilde{F}_i$ is
also induced on the
$\mathbb{C}$-vector space
$\mathcal{L}(\lambda)/{q^{-1}\mathcal{L}(\lambda)}$.
We set
\begin{equation}
	\mathcal{B}(\lambda) :=
	\left\{
	\widetilde{F}_{i_1}\cdots\widetilde{F}_{i_r} \eta_{\lambda} \,\,\,\, \mathrm{mod}\,\,q^{-1}\mathcal{L}(\lambda)
	\mathrel{}\middle|\mathrel{}
	r \in \mathbb{Z}_{\geq0},\,\,i_1,\dots,i_r \in  I
	\right\}\,\,\backslash\,\,\{0\}.
\end{equation}
The set $\mathcal{B}(\lambda)$ is called the
\textit{crystal basis} of $L_q(\lambda)$.
The set
$\mathcal{B}(\lambda) \cup \{0\}$
is also closed under the action of
$\widetilde{E}_i$ and $\widetilde{F}_i$ for $i \in I$. 
The crystal basis 
$\mathcal{B}(\lambda)$
aggregates the information about 
$\mathsf{Res}^{\mathbf{U}}_{\mathbf{U}_i}L_q(\lambda)$ for $i \in I$, 
and contains all the representation-theoretic
information about both
$L(\lambda)$ and $L_q(\lambda)$.

The $\mathbb{C}(q)$-algebra
$\mathbf{U}$
admits an involution as a
$\mathbb{C}$-algebra $\overline{\,\cdot\,}$
given by: 
\begin{equation}
	\begin{aligned}
		 & \overline{q} = q^{-1},     & 
		 & \overline{K_h} = K_h^{-1} &
		 & (h \in P^{\vee}),            & 
		 & \overline{E_i} = E_i,      & 
		 & \overline{F_i} = F_i      & 
		 & (i \in  I).
	\end{aligned}
\end{equation}
Similarly, there exists a unique $\mathbb{C}$-linear involution
$\overline{\,\cdot\,}$ on $L_q(\lambda)$ satisfying
$\overline{u\xi} = \overline{u} \cdot \overline{\xi},\,\,u \in \mathbf{U}, \xi \in L_q(\lambda)$.
Also, we set
\begin{equation}
	_{\mathbf{A}} L_q(\lambda) := \sum_{r \in \mathbb{Z}_{\geq0}} \sum_{i_1,\ldots,i_r \in  I} \sum_{n_1,\ldots,n_r \in \mathbb{Z}_{\geq0}} \mathbf{A} F_{i_1}^{(n_1)} \cdots F_{i_r}^{(n_r)} \eta_{\lambda}.
\end{equation}
It is known that the natural
$\mathbb{C}$-linear map
$G^{-1} :\, _{\mathbf{A}} L_q(\lambda) \cap \mathcal{L}(\lambda) \cap \overline{\mathcal{L}(\lambda)} \to \mathcal{L}(\lambda)/{q^{-1}\mathcal{L}(\lambda)}$
is an isomorphism.
We define $G(b)$ as the preimage of $b \in \mathcal{B}(\lambda)$
under the isomorphism $G^{-1}$; 
the set 
$\mathbf{B}(\lambda) := \left\{G(b)\mathrel{}\middle|\mathrel{}b \in \mathcal{B}(\lambda)\right\}$
is called the \textit{global crystal basis} of $L_q(\lambda)$.
Note that 
$\mathbf{B}(\lambda)$ is a $\mathbb{C}(q)$-basis of $L_q(\lambda)$ consisting of weight vectors, while $\mathcal{B}(\lambda)$ is a $\mathbb{C}$-basis of $\mathcal{L}(\lambda)/{q^{-1}\mathcal{L}(\lambda)}$. 

\subsection{Kashiwara crystals}
The notion of a \textit{crystal} is introduced by 
axiomatizing the combinatorial properties of
$\mathcal{B}(\lambda)$. 
Here, for the sake of simplicity,
we restrict our attention to a certain subclass of crystals, 
known as \textit{seminormal crystals}. 
We consider a tuple
$(\mathcal{B},\mathsf{wt},\{\widetilde{e}_i\}_{i \in  I},\{\widetilde{f}_i\}_{i \in  I})$
consisting of the following: 
\begin{itemize}
	\item a set $\mathcal{B}$, together with an element $0$ that is not in $\mathcal{B}$; 
	\item a map $\mathsf{wt} : \mathcal{B} \to P$; 
	\item operators $\widetilde{e}_i,\widetilde{f}_i : \mathcal{B} \to \mathcal{B} \sqcup \{0\}$ indexed by $i \in  I$.
\end{itemize}
For each $b\in\mathcal{B}$, we set
\begin{equation}
	\begin{aligned}
		 & \varepsilon_i(b) := \max\setlr{k \in \mathbb{Z}_{\geq0}}{\widetilde{e}_i^{k}b \neq 0}, & 
		 & \varphi_i(b) := \max\setlr{k \in \mathbb{Z}_{\geq0}}{\widetilde{f}_i^{k}b \neq 0}.
	\end{aligned}
\end{equation}
We call $(\mathcal{B},\mathsf{wt},\{\widetilde{e}_i\}_{i \in  I},\{\widetilde{f}_i\}_{i \in  I})$
a \textit{$\mathfrak{g}$-seminormal crystal}
if, for every $i \in  I$ and $b \in \mathcal{B}$,
the following conditions are satisfied:
\begin{equation}
	\begin{aligned}
		 & \varphi_i(b) = \varepsilon_i(b) + \langle\mathsf{wt}(b),h_i\rangle, \\
		 & 
		\left\{
		\begin{aligned}
			 & \mathsf{wt}(\widetilde{e}_ib) = \mathsf{wt}(b) + \alpha_i &  & \text{if $\widetilde{e}_ib \neq 0$},  \\
			 & \mathsf{wt}(\widetilde{f}_ib) = \mathsf{wt}(b) - \alpha_i &  & \text{if $\widetilde{f}_ib \neq 0$},
		\end{aligned}
		\right.&
		 & 
		\left\{
		\begin{aligned}
			 & \widetilde{f}_i\widetilde{e}_ib = b &  & \text{if $\widetilde{e}_ib \neq 0$},  \\
			 & \widetilde{e}_i\widetilde{f}_ib = b &  & \text{if $\widetilde{f}_ib \neq 0$}.
		\end{aligned}
		\right.
	\end{aligned}
\end{equation}
Throughout this paper,
we write $\widetilde{e}_i^{\mathrm{max}}b := \widetilde{e}_i^{\varepsilon_i(b)}b$,
$\widetilde{f}_i^{\mathrm{max}}b := \widetilde{e}_i^{\varphi_i(b)}b$.
Given two $\mathfrak{g}$-seminormal crystals $\mathcal{B}_1$ and $\mathcal{B}_2$,
a new $\mathfrak{g}$-seminormal crystal $\mathcal{B}_1 \otimes \mathcal{B}_2$,
referred to as their \textit{tensor product}, is defined as follows: 
\begin{itemize}
	\item $\mathcal{B}_1 \otimes \mathcal{B}_2 := \mathcal{B}_1 \times \mathcal{B}_2$ (Cartesian product of sets).
	\item We denote a tuple $(b_1,b_2)$ for $b_1 \in \mathcal{B}_1$ and $b_2 \in \mathcal{B}_2$ by $b_1 \otimes b_2$. 
	\item We set
	      \begin{equation}
		      \begin{aligned}
			       & \widetilde{e}_i(b_1 \otimes b_2):=
			      \left\{
			      \begin{aligned}
				       & (\widetilde{e}_ib_1) \otimes b_2 &  & \text{if $\varepsilon_i(b_1) > \varphi_i(b_2)$},     \\
				       & b_1 \otimes (\widetilde{e}_ib_2) &  & \text{if $\varepsilon_i(b_1) \leq \varphi_i(b_2)$},
			      \end{aligned}
			      \right.
			      \\
			       & \widetilde{f}_i(b_1 \otimes b_2):=
			      \left\{
			      \begin{aligned}
				       & (\widetilde{f}_ib_1) \otimes b_2 &  & \text{if $\varepsilon_i(b_1) \geq \varphi_i(b_2)$}, \\
				       & b_1 \otimes (\widetilde{f}_ib_2) &  & \text{if $\varepsilon_i(b_1) < \varphi_i(b_2)$}.
			      \end{aligned}
			      \right.
		      \end{aligned}\label{4.13}
	      \end{equation}
\end{itemize}
\begin{rem}
	In this paper, we adopt the same convention 
	for the tensor product rule (\ref{4.13}) for crystals 
	as that in \cite{BS}; 
	this convention is 
	opposite to that in \cite{Kas}, \cite{HK}, and \cite{J}.
\end{rem}

The following are typical examples of $\mathfrak{g}$-seminormal crystals.

\begin{eg}[{\cite[Section 4.2]{Kas}}]
	For each $b \in \mathcal{B}(\lambda)$,
	let us denote by $\mathsf{wt}(b)$ the weight of $G(b)$ as an element of $L_q(\lambda)$.
	Then,
	$(\mathcal{B}(\lambda),\mathsf{wt},\{\widetilde{E}_i\}_{i \in  I},\{\widetilde{F}_i\}_{i \in  I})$ is a
	$\mathfrak{g}$-seminormal crystal.
\end{eg}

\begin{eg}[{\cite[Example 2.34]{BS}}]
	For $w \in \mathsf{W}_{2n}$,
	we set 
	$\mathsf{wt}(w) := w[1]\varepsilon_1 + \dots + w[2n]\varepsilon_{2n}$.
	Also, 
	for each $i \in  I$,
	we define
	$\widetilde{e}_i\,w$ and $\widetilde{f}_i\,w$ by the following procedure:
	\begin{enumerate}
		\item[(1)] Remove all letters from $w$ except for ``$i$'' and ``$i+1$''.
		\item[(2)] From the word obtained in (1), remove the subword ``$i+1\,\,i$'' whenever it appears, and transform it into a word of the form ``$i\,\cdots\,i\,\,i+1\,\cdots\,i+1$''.
		\item[(3)] If the word obtained in (2) does not contain ``$i+1$'', set $\widetilde{e}_i\,w :=0$. If ``$i$'' is not present, set $\widetilde{f}_i\,w :=0$.
		\item[(4)] If the word obtained in (2) contains ``$i+1$'', focus on its leftmost occurrence, and define $\widetilde{e}_i\,w$ as the word obtained by changing the corresponding letter in $w$ to ``$i$''.
		      If ``$i$'' is present, focus on its rightmost occurrence, and define $\widetilde{f}_i\,w$ as the word obtained by changing the corresponding letter in $w$ to ``$i+1$''. 
	\end{enumerate}
	Then,
	$(\mathsf{W}_{2n},\mathsf{wt},\{\widetilde{e}_i\}_{i \in  I},\{\widetilde{f}_i\}_{i \in  I})$ is a
	$\mathfrak{g}$-seminormal crystal.
\end{eg}

\begin{eg}[{\cite[Proposition 3.1]{BS}}]
	By the map 
	$T \mapsto w_r(T)$,
	we can embed
	$\mathsf{SST}_{2n}(\lambda)$
	into $\mathsf{W}_{2n}$.
	It is known that the image of $\mathsf{SST}_{2n}(\lambda)$
	under this embedding is closed under the actions of
	$\widetilde{e}_i$ and $\widetilde{f}_i$.
	Therefore, by identifying $\mathsf{SST}_{2n}(\lambda)$ with its image
        under the embedding, 
	we can give $\mathsf{SST}_{2n}(\lambda)$ 
	a structure of $\mathfrak{g}$-seminormal crystal. 
\end{eg}

\begin{rem}
	If $T \in \mathsf{SST}_{2n}(\lambda)$, then we have $\mathsf{wt}(T) = T[1]\varepsilon_1 + \cdots + T[2n]\varepsilon_{2n}$. 
	Hence, the definition of $\mathsf{wt}_{\widehat{\mathfrak{g}}}$
	given above can be rewritten as: $\mathsf{wt}_{\widehat{\mathfrak{g}}}(T) = \widehat{\mathsf{wt}(T)}$,
	where the right-hand side is the image of $\mathsf{wt}(T)$
	under the restriction map $\widehat{\cdot}$ .
	Furthermore, recalling that $\widehat{\varepsilon}_{\overline{i}} = - \widehat{\varepsilon}_i$,
	we can also transform it as follows: 
	\begin{equation}
		\mathsf{wt}_{\widehat{\mathfrak{g}}}(T) = (T[1]-T[\overline{1}])\widehat{\varepsilon}_1 + \cdots + (T[n]-T[\overline{n}])\widehat{\varepsilon}_n. \label{4.14}
	\end{equation}
\end{rem}
Let $\mathcal{B}_1$ and $\mathcal{B}_2$ be $\mathfrak{g}$-seminormal crystals.
We say that $\mathcal{B}_1$ and $\mathcal{B}_2$ are \textit{isomorphic} as crystals if there exists a bijection
$\psi : \mathcal{B}_1 \sqcup \{0\} \to \mathcal{B}_2 \sqcup \{0\}$
that satisfies the following conditions: 
\begin{itemize}
	\item $\psi(0)=0$;
	\item for each $i \in  I, b \in \mathcal{B}_1$, $\mathsf{wt}(b) = \mathsf{wt}(\psi(b))$;
	\item for each $i \in  I$, $\psi \circ \widetilde{e}_i = \widetilde{e}_i \circ \psi,\,\,\psi \circ \widetilde{f}_i = \widetilde{f}_i \circ \psi$.
\end{itemize}
Then, the following proposition holds.
\begin{prop}[{\cite[Theorem 7.3.6]{HK}}]
	$\mathcal{B}(\lambda)$ and $\mathsf{SST}_{2n}(\lambda)$ are isomorphic as $\mathfrak{g}$-seminormal crystals.
\end{prop}
In particular, we can parametrize the elements of $\mathcal{B}(\lambda)$ as: 
$\mathcal{B}(\lambda) = \setlr{b_T}{T \in \mathsf{SST}_{2n}(\lambda)}$.
For later use, we record the following propositions.
\begin{prop}[{\cite[Proposition 2.32]{BS}}]
	Let $\mathcal{B}_1$, $\mathcal{B}_2$, and $\mathcal{B}_3$ be $\mathfrak{g}$-seminormal crystals.
	Then, the correspondence $(b_1 \otimes b_2) \otimes b_3 \mapsto b_1 \otimes (b_2 \otimes b_3)$
	gives an isomorphism of $\mathfrak{g}$-seminormal crystals
	$\left(\mathcal{B}_1 \otimes \mathcal{B}_2\right) \otimes \mathcal{B}_3 \stackrel{\sim}{\longrightarrow} \mathcal{B}_1 \otimes \left(\mathcal{B}_2 \otimes \mathcal{B}_3\right)$.
\end{prop}
\begin{rem}
	Throughout this paper,
	we identify $\left(\mathcal{B}_1 \otimes \mathcal{B}_2\right) \otimes \mathcal{B}_3$ with
	$\mathcal{B}_1 \otimes \left(\mathcal{B}_2 \otimes \mathcal{B}_3\right)$, and write $b_1 \otimes (b_2 \otimes b_3)$ and $(b_1 \otimes b_2) \otimes b_3$ simply as $b_1 \otimes b_2 \otimes b_3$.
\end{rem}
\begin{prop}[see {\cite[Chapter 8]{BS}}]\label{Pieri}
	Let $\lambda \in \mathsf{Par}_{\leq2n}$, and $1 \leq l \leq 2n$ an integer.
	Then, the correspondence $T \otimes S \mapsto T*S$ gives the following isomorphism of $\mathfrak{g}$-seminormal crystals:
	\begin{equation}
		\mathsf{SST}_{2n}(\varpi_l) \otimes \mathsf{SST}_{2n}(\lambda) \stackrel{\sim}{\longrightarrow}
		\bigsqcup_{\lambda \underset{\mathrm{vert}}{\subset} \mu ,\,\, |\mu/{\lambda}| = l} \mathsf{SST}_{2n}(\mu). 
	\end{equation}
\end{prop}
With the above as background,
a significant portion of the study of
$L(\lambda)$ and $L_q(\lambda)$ for $\lambda \in P^{+}_{\geq 0}$ is reduced to
the combinatorics of tableaux and paths
through the crystal basis $\mathcal{B}(\lambda)$.
Indeed, the Naito--Sagaki conjecture is one example of such phenomena; 
other typical examples include the classical Littlewood--Richardson rule.

\bigskip

\section{Branching rule for $\imath$quantum groups of type $A\mathrm{II}$}
In this section,
we recall some properties of $\imath$quantum groups of type $A\mathrm{II}$, which we need in the proof of Theorem B, 
and also explain how to deduce Theorem A from Theorem B.
For more details on $\imath$quantum groups, please refer to
\cite{Wat1} and \cite{Wat2}; 
see also 
\cite{Wan} for a concise summary of
the progress in the study of $\imath$quantum groups.

\subsection{$\imath$quantum groups of type $A\mathrm{II}$}
Let $I_{\bullet}$ denote the
set of odd elements in $I$, and
$I_{\circ}$ the
set of even elements in $I$:
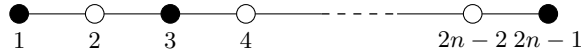
\vspace{10pt}
\begin{center}
	\begin{tikzpicture}[x=5mm,y=5mm]
		\draw (0,-4)--(2,-4);
		\draw (2,-4)--(4,-4);
		\draw (4,-4)--(6,-4);
		\draw (6,-4)--(8,-4);
		\draw[dashed,line width=0.5pt] (8,-4)--(10,-4);
		\draw (10,-4)--(12,-4);
		\draw (12,-4)--(14,-4);
		\draw[fill=black] (0,-4) circle[radius=3.5pt];
		\draw[fill=white] (2,-4) circle[radius=3.5pt];
		\draw[fill=black] (4,-4) circle[radius=3.5pt];
		\draw[fill=white] (6,-4) circle[radius=3.5pt];
		\draw[fill=white] (12,-4) circle[radius=3.5pt];
		\draw[fill=black] (14,-4) circle[radius=3.5pt];
		\node (l) at (0,-4.7) {\scriptsize $1$};
		\node (m) at (2,-4.7) {\scriptsize $2$};
		\node (n) at (4,-4.7) {\scriptsize $3$};
		\node (o) at (6,-4.7) {\scriptsize $4$};
		\node (p) at (12,-4.7) {\scriptsize $2n-2$};
		\node (q) at (14,-4.7) {\scriptsize $2n-1$};
	\end{tikzpicture}
\end{center}
We set $b_j := f_j - \left[e_{j-1},[e_{j+1},e_{j}]\right], \,\,a_j := -\left[f_{j-1},[f_{j+1},b_j]\right]$ for $j \in I_{\circ}$, 
and let $\mathfrak{k}$ denote the Lie subalgebra of $\mathfrak{g}$ 
generated by $h_i,e_i,f_i$, $i \in I_{\bullet}$, and $b_j,a_j$, $j \in  I_{\circ}$. 
In the following, 
let $\widetilde{ I}$ denote the set of vertices of
the Dynkin diagram of type $C_n$, enumerated as follows: 
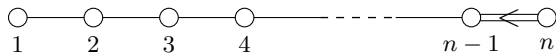
\vspace{10pt}
\begin{center}
	\begin{tikzpicture}[x=5mm,y=5mm]
		\draw (0,-4)--(2,-4);
		\draw (2,-4)--(4,-4);
		\draw (4,-4)--(6,-4);
		\draw (6,-4)--(8,-4);
		\draw[dashed,line width=0.5pt] (8,-4)--(10,-4);
		\draw (10,-4)--(12,-4);
		\draw[double distance=2pt] (12,-4)--node{$<$}(14,-4);
		\draw[fill=white] (0,-4) circle[radius=3.5pt];
		\draw[fill=white] (2,-4) circle[radius=3.5pt];
		\draw[fill=white] (4,-4) circle[radius=3.5pt];
		\draw[fill=white] (6,-4) circle[radius=3.5pt];
		\draw[fill=white] (12,-4) circle[radius=3.5pt];
		\draw[fill=white] (14,-4) circle[radius=3.5pt];
		\node (l) at (0,-4.7) {\scriptsize $1$};
		\node (m) at (2,-4.7) {\scriptsize $2$};
		\node (n) at (4,-4.7) {\scriptsize $3$};
		\node (o) at (6,-4.7) {\scriptsize $4$};
		\node (p) at (12,-4.7) {\scriptsize $n-1$};
		\node (q) at (14,-4.7) {\scriptsize $n$};
	\end{tikzpicture}
\end{center}
Also,
we define
$x_i,y_i,z_i$, $i \in \widetilde{I}$, 
by
\[
	\begin{aligned}
		 & x_i
		:=
		\left\{
		\begin{aligned}
			 & b_{2i} &  & \text{if} \,\, 1 \leq i < n \,\, \text{and $i$ is even},  \\
			 & a_{2i} &  & \text{if} \,\, 1 \leq i < n \,\, \text{and $i$ is odd},  \\
			 & e_{2n-1} &  & \text{if} \,\, i=n \,\, \text{and $n$ is even},  \\
			 & f_{2n-1} &  & \text{if} \,\, i=n \,\, \text{and $n$ is odd}, 
		\end{aligned}
		\right. \,\,\,\,\,\,\,\,
		y_i
		:=
		\left\{
		\begin{aligned}
			 & a_{2i} &  & \text{if} \,\, 1 \leq i < n \,\, \text{and $i$ is even},  \\
			 & b_{2i} &  & \text{if} \,\, 1 \leq i < n \,\, \text{and $i$ is odd},  \\
			 & f_{2n-1} &  & \text{if} \,\, i=n \,\, \text{and $n$ is even},  \\
			 & e_{2n-1} &  & \text{if} \,\, i=n \,\, \text{and $n$ is odd}, 
		\end{aligned}
		\right. \\
		 & z_i
		:=
		\left\{
		\begin{aligned}
			 & (-1)^i(h_{2i-1}+h_{2i+1}) &  & \text{if} \,\, 1 \leq i < n, \\
			 & (-1)^nh_{2n-1}            &  & \text{if} \,\, i=n.
		\end{aligned}
		\right.
	\end{aligned}
\]
It is known that $\mathfrak{k}$ and $\mathfrak{sp}_{2n}(\mathbb{C})$ are isomorphic as Lie algebras, with 
$x_i,y_i,z_i$, $i \in \widetilde{I}$, 
the Chevalley generators of $\mathfrak{k}$.
For each $1 \leq i \leq n-1$,
we denote by $\mathfrak{k}_i$
the Lie subalgebra of $\mathfrak{k}$
generated by $h_{2i\pm1},e_{2i\pm1},f_{2i\pm1},b_{2i},a_{2i}$; 
each $\mathfrak{k}_i$ is isomorphic to $\mathfrak{sp}_4(\mathbb{C})$, and
$\mathfrak{k}$ is generated as a Lie algebra by $\mathfrak{k}_1$, $\mathfrak{k}_2$, ..., $\mathfrak{k}_{n-1}$: 
\begin{center}
	\begin{tikzpicture}[x=5mm,y=5mm]
		\draw (0,-4)--(2,-4);
		\draw (2,-4)--(4,-4);
		\draw (4,-4)--(6,-4);
		\draw (6,-4)--(8,-4);
		\draw (8,-4)--(10,-4);
		\draw[dashed,line width=0.5pt] (14,-4)--(16,-4);
		\draw (10,-4)--(12,-4);
		\draw (12,-4)--(14,-4);
		\draw (16,-4)--(18,-4);
		\draw (18,-4)--(20,-4);
		\draw (20,-4)--(22,-4);
		\draw[fill=black] (0,-4) circle[radius=3.5pt];
		\draw[fill=white] (2,-4) circle[radius=3.5pt];
		\draw[fill=black] (4,-4) circle[radius=3.5pt];
		\draw[fill=white] (6,-4) circle[radius=3.5pt];
		\draw[fill=black] (8,-4) circle[radius=3.5pt];
		\draw[fill=white] (10,-4) circle[radius=3.5pt];
		\draw[fill=black] (12,-4) circle[radius=3.5pt];
		\draw[fill=black] (18,-4) circle[radius=3.5pt];
		\draw[fill=white] (20,-4) circle[radius=3.5pt];
		\draw[fill=black] (22,-4) circle[radius=3.5pt];
		\node (l) at (0,-4.7) {\scriptsize $1$};
		\node (m) at (2,-4.7) {\scriptsize $2$};
		\node (n) at (4,-4.7) {\scriptsize $3$};
		\node (o) at (6,-4.7) {\scriptsize $4$};
		\node (p) at (8,-4.7) {\scriptsize $5$};
		\node (q) at (10,-4.7) {\scriptsize $6$};
		\node at (12,-4.7) {\scriptsize $7$};
		\node at (18,-4.7) {\scriptsize $2n-3$};
		\node at (20,-4.7) {\scriptsize $2n-2$};
		\node at (22,-4.7) {\scriptsize $2n-1$};
		\node at (2,-2.2) {\textcolor{blue}{$\mathfrak{k}_1$}};
		\node at (6,-2.2) {\textcolor{blue}{$\mathfrak{k}_2$}};
		\node at (10,-2.2) {\textcolor{blue}{$\mathfrak{k}_3$}};
		\node at (20,-2.2) {\textcolor{blue}{$\mathfrak{k}_{n-1}$}};
		\draw[dashed, color=blue, line width=1pt] (2,-4) circle[x radius=3,y radius=1];
		\draw[dashed, color=blue, line width=1pt] (6,-4) circle[x radius=3,y radius=1];
		\draw[dashed, color=blue, line width=1pt] (10,-4) circle[x radius=3,y radius=1];
		\draw[dashed, color=blue, line width=1pt] (20,-4) circle[x radius=3,y radius=1];
	\end{tikzpicture}
\end{center}
In addition, we define $\widetilde{\varpi}_i,\widetilde{\varepsilon}_i$, $i \in \widetilde{I}$, 
by $\langle\widetilde{\varpi}_i,z_j\rangle=\delta_{ij}$,
$\widetilde{\varpi}_i = \widetilde{\varepsilon}_1 + \cdots + \widetilde{\varepsilon}_i$,
and set
\begin{equation}
	\begin{aligned}
		& \widetilde{P}^{\vee} := \sum_{i \in \widetilde{I}} \mathbb{Z}\widetilde{z}_i, &
		& \widetilde{P} := \mathrm{Hom}_{\mathbb{Z}}(\widetilde{P}^{\vee},\mathbb{Z}), &
		& \widetilde{P}^+ := \setlr{\mu \in \widetilde{P}}{\langle\mu,\widetilde{z}_i\rangle\geq0, \,\,i \in \widetilde{I}}.
	\end{aligned}
\end{equation}
Since $\widetilde{P}^+$ has the following description: 
\begin{equation}
	\widetilde{P}^{+} = \left\{\mu_1\widetilde{\varepsilon}_1 + \cdots + \mu_{n}\widetilde{\varepsilon}_{n} \in \widetilde{P} \mathrel{}\middle|\mathrel{} \mu_1 \geq \cdots \geq \mu_{n} \geq 0\right\}, 
\end{equation}
we can also identify $\widetilde{P}^+$ with $\mathsf{Par}_{\leq n}$. 
The inclusion 
$\widetilde{P}^{\vee} \hookrightarrow P^{\vee}$
induces the natural restriction map
$\widetilde{\cdot} : P \to \widetilde{P}$.
For $\mu \in \widetilde{P}^+$, let $\widetilde{L}(\mu)$ denote the irreducible highest weight $\mathfrak{k}$-module of highest weight $\mu$. 

The \textit{$\imath$quantum group $\mathbf{U}^{\imath}$ of type $A\mathrm{II}$} defined as below is a 
$q$-analog of the universal enveloping algebra of
$\mathfrak{k}$; 
however,
$\mathbf{U}^{\imath}$ is entirely 
different from the (Drinfeld--Jimbo) quantum group
$\mathbf{U}_q(\mathfrak{k})$.
We set
\begin{equation}
	\begin{aligned}
	  & B_j := F_j - q\left[E_{j-1},[E_{j+1},E_{j}]_{q^{-1}}\right]_{q^{-1}} K_{h_j}^{-1}, \\
		& A_j := -q^{-1}\left[F_{j-1},[F_{j+1},B_j]_q\right]_q,
	\end{aligned}
	\quad (j \in  I_{\circ}), 
\end{equation}
where $[\cdot,\cdot]_{q^{\pm1}}$ denotes the $q$-commutator given by
$[x,y]_{q^{\pm1}} := xy - q^{\pm1}yx $ for $x, y \in \mathbf{U}$. 
Let $\mathbf{U}^{\imath}$ denote
the $\mathbb{C}(q)$-subalgebra of $\mathbf{U}$ generated by
$K_{h_i},E_i,F_i$, $i \in  I_{\bullet}$, and $B_j,A_j$, $j \in I_{\circ}$. 
We call $\mathbf{U}^{\imath}$ the \textit{$\imath$quantum group of type $A\mathrm{II}_{2n-1}$}. 
The $\mathbb{C}(q)$-algebra $\mathbf{U}^{\imath}$ admits an \textit{$\imath$-analog} of the Chevalley generators, 
given as follows:
\[
	\begin{aligned}
		& X_i
		:=
		\left\{
		\begin{aligned}
			 & B_{2i} &  & \text{if} \,\, 1 \leq i < n \,\, \text{and $i$ is even},  \\
			 & A_{2i} &  & \text{if} \,\, 1 \leq i < n \,\, \text{and $i$ is odd},  \\
			 & E_{2n-1} &  & \text{if} \,\, i=n \,\, \text{and $n$ is even},  \\
			 & F_{2n-1} &  & \text{if} \,\, i=n \,\, \text{and $n$ is odd}, 
		\end{aligned}
		\right. \,\,\,\,\,\,\,\,
		Y_i
		:=
		\left\{
		\begin{aligned}
			 & A_{2i} &  & \text{if} \,\, 1 \leq i < n \,\, \text{and $i$ is even},  \\
			 & B_{2i} &  & \text{if} \,\, 1 \leq i < n \,\, \text{and $i$ is odd},  \\
			 & F_{2n-1} &  & \text{if} \,\, i=n \,\, \text{and $n$ is even},  \\
			 & E_{2n-1} &  & \text{if} \,\, i=n \,\, \text{and $n$ is odd}, 
		\end{aligned}
		\right. \\
		 & Z_i
		:=
		\left\{
		\begin{aligned}
			 & (-1)^i \frac{K_{h_{2i-1}+h_{2i+1}}-K_{h_{2i-1}+h_{2i+1}}^{-1}}{q-q^{-1}} &  & \text{if} \,\, 1 \leq i < n, \\
			 & (-1)^n \frac{K_{h_{2n-1}}-K_{h_{2n-1}}^{-1}}{q-q^{-1}}                   &  & \text{if} \,\, i=n.
		\end{aligned}
		\right.
	\end{aligned}
\]
For each $1 \leq i \leq n-1$,
we denote by $\mathbf{U}^{\imath}_i$
the $\mathbb{C}(q)$-subalgebra of $\mathbf{U}^{\imath}$
generated by $K_{h_{2i\pm1}},E_{2i\pm1},F_{2i\pm1},B_{2i},A_{2i}$; 
each $\mathbf{U}^{\imath}_i$ is isomorphic to $\mathbf{U}^{\imath}$ in the case $n=2$,
and 
$\mathbf{U}^{\imath}$ is generated as a $\mathbb{C}(q)$-algebra by $\mathbf{U}^{\imath}_1$, $\mathbf{U}^{\imath}_2$, ..., $\mathbf{U}^{\imath}_{n-1}$. 

For $\mu \in \widetilde{P}^+$,
we define the (finite-dimensional)
irreducible highest weight $\mathbf{U}^{\imath}$-module
$\widetilde{L}^{\imath}(\mu)$ of highest weight $\mu$ as follows:
\begin{equation}
	\widetilde{L}^{\imath}(\mu) := \mathbf{U}^{\imath}
	\left/ \left(\sum_{i \in \widetilde{ I}} \mathbf{U}^{\imath}X_i + \sum_{i \in \widetilde{ I}} \mathbf{U}^{\imath}Y_i^{\langle\mu,w_i\rangle+1} + \sum_{i \in \widetilde{ I}} \mathbf{U}^{\imath}\left(Z_i-q^{\langle\mu,z_i\rangle}\right) \right)\right..
\end{equation}
In the classical limit as $q\to1$,
$\widetilde{L}^{\imath}(\mu)$ becomes $\widetilde{L}(\mu)$.
We summarize the relation 
between irreducible highest weight representation of
some classical (universal enveloping) algebras and their $q$-analogs
in Figure \ref{fig4} below. 

Generally speaking,
the representation theory of $\imath$quantum groups
is much more difficult than that of ordinary quantum groups.
However,
in the case of type $A\mathrm{II}$,
the finite-dimensional
irreducible representations are classified
by Molev \cite{M}.
Moreover,
for those modules
(called \textit{classical weight modules})
on which the actions of $K_{h_i}$, $i \in I_{\bullet}$, 
are simultaneously diagonalizable, with all the eigenvalues
$\in 1+(q-1)\mathbf{A}_1$, we know the following.
\begin{thm}[see {\cite[Proposition 3.1.4, Corollaries 4.3.2 and 4.3.4]{Wat2}}]\label{thm5.1}
	For finite-dimensional classical weight modules,
	the following hold.
	\begin{enumerate}
		\item[(1)] For $\lambda \in P^{+}_{\geq 0}$, $\mathsf{Res}^{\mathbf{U}}_{\mathbf{U}^{\imath}}\,L_q(\lambda)$ is a finite-dimensional classical weight module.
		\item[(2)] Every finite-dimensional classical weight module over $\mathbf{U}^{\imath}$ is completely reducible.
		\item[(3)] $\widetilde{L}^{\imath}(\mu)$, $\mu \in \widetilde{P}^+$, forms a complete set of representatives for the equivalence classes of irreducible finite-dimensional classical weight modules over $\mathbf{U}^{\imath}$.
		\item[(4)] For all $\lambda \in P^+_{\geq0}$ and $\mu \in \widetilde{P}^+$, $\left[\mathsf{Res}^{\mathfrak{g}}_{\mathfrak{k}}\,L(\lambda) : \widetilde{L}(\mu) \right] = \left[\mathsf{Res}^{\mathbf{U}}_{\mathbf{U}^{\imath}}\,L_q(\lambda) : \widetilde{L}^{\imath}(\mu) \right]$.
	\end{enumerate}
\end{thm}
Moreover,
in \cite{Wat1},
an algorithm for computing combinatorially the multiplicity
$\left[\mathsf{Res}^{\mathbf{U}}_{\mathbf{U}^{\imath}}\,L_q(\lambda) : \widetilde{L}^{\imath}(\mu) \right]$
is obtained; 
this is the \textit{quantum Littlewood--Richardson rule of type $A\mathrm{II}$},
which will be explained in the next subsection.
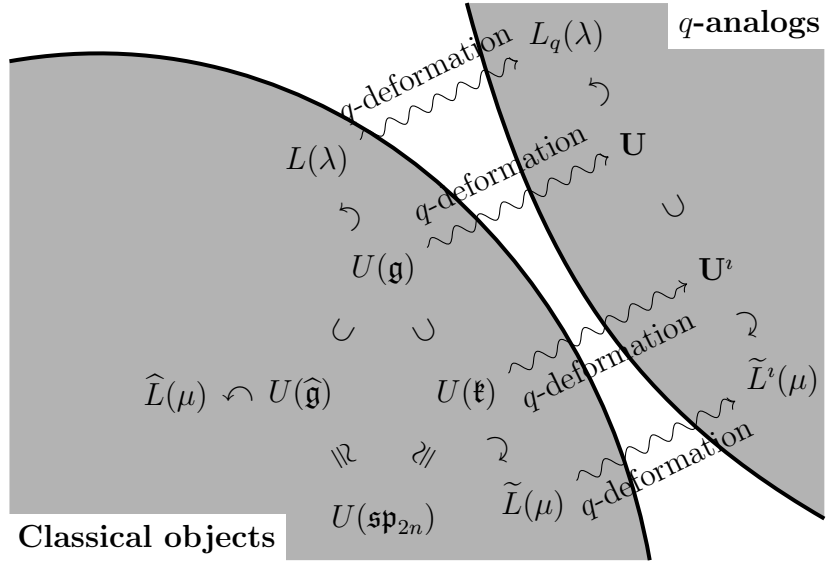
\begin{figure}[ht]
	\begin{center}
		\begin{tikzpicture}[x=11mm,y=11mm]
			\fill[color=black!30] (3.2,-3.5) to[out=100,in=10](-4.5,2.5)--(-4.5,-3.5)--cycle;
			\fill[color=black!30] (1,3.2) to[out=-75,in=150](5.3,-3)--(5.3,3.2)--cycle;
			
			\draw[line width=1.5pt] (3.2,-3.5) to[out=100,in=10](-4.5,2.5);
			\draw[line width=1.5pt] (1,3.2) to[out=-75,in=150](5.3,-3);
			\node[fill=white] at (-2.85,-3.25) {\textbf{Classical objects}};
			\node[fill=white] at (4.4,2.95) {\textbf{$q$-analogs}};
			
			\node (a) at (0,0) {$U(\mathfrak{g})$};
			\node[rotate=-60] at (-0.4,0.65) {$\curvearrowleft$};
			\node (i) at (-0.8,1.3) {$L(\lambda)$};
			\node (b) at (1,-1.5) {$U(\mathfrak{k})$};
			\node[rotate=-60] at (1.4,-2.15) {$\curvearrowright$};
			\node (f) at (1.8,-2.8) {$\widetilde{L}(\mu)$};
			\node (c) at (-1,-1.5) {$U(\widehat{\mathfrak{g}})$};
			\node at (-1.75,-1.5) {$\curvearrowleft$};
			\node (g) at (-2.5,-1.5) {$\widehat{L}(\mu)$};
			\node (d) at (0,-3) {$U(\mathfrak{sp}_{2n})$};
			\node[rotate=120] at (0.5,-0.75) {$\subset$};
			\node[rotate=60] at (-0.5,-0.75) {$\subset$};
			\node[rotate=60] at (0.5,-2.25) {$\cong$};
			\node[rotate=-60] at (-0.5,-2.25) {$\cong$};
			
			\node (d) at (3,1.5) {$\mathbf{U}$};
			\node[rotate=-60] at (2.6,2.15) {$\curvearrowleft$};
			\node (j) at (2.2,2.8) {$L_q(\lambda)$};
			\node (e) at (4,0) {$\mathbf{U}^{\imath}$};
			\node[rotate=-60] at (4.4,-0.65) {$\curvearrowright$};
			\node (h) at (4.8,-1.3) {$\widetilde{L}^{\imath}(\mu)$};
			\node[rotate=120] at (3.5,0.75) {$\subset$};
			
			\draw[->] decorate[decoration=snake] {(a)--(d)};
			\draw[->] decorate[decoration=snake] {(b)--(e)};
			\draw[->] decorate[decoration=snake] {(i)--(j)};
			\draw[->] decorate[decoration=snake] {(f)--(h)};
			
			\node[rotate=26] at (0.5,2.35) {$q$-deformation};
			\node[rotate=26] at (1.4,1.15) {$q$-deformation};
			\node[rotate=26] at (2.7,-1.1) {$q$-deformation};
			\node[rotate=26] at (3.4,-2.4) {$q$-deformation};
		\end{tikzpicture}
	\end{center}
	\caption{The relation between (universal enveloping) algebras and their modules that have appeared so far.} \label{fig4}
\end{figure}
\subsection{Quantum Littlewood--Richardson rule of type $A\mathrm{II}$}
In this subsection, we explain the \textit{quantum Littlewood--Richardson rule},
which gives a combinatorial description of 
$\mathsf{Res}^{\mathbf{U}}_{\mathbf{U}^{\imath}}\,L_q(\lambda)$. 
Now, we present an algorithm 
for describing the branching rule.

\begin{de}[{\cite[Definition 2.5.1]{Wat1}}]
	Let $S \in \mathsf{SST}_{2n}$ be a semistandard tableau
	consisting of a single column,
	and let $l$ be its length 
	(i.e., the shape of $S$ is $\varpi_l$).
	Also, we denote by $S'$
	the tableau obtained from $S$ by removing the box that contains the entry $S(1,l)$,
	and by $S''$ the tableau obtained from $S$ by removing both of the box containing $S(1,l)$
	and the box containing $S(1,l-1)$.
	We define a set $\mathsf{rem}(S)$ inductively on $l$ 
	by the following recurrence relation:
	\[
		\mathsf{rem}(S)=
		\left\{
		\begin{aligned}
			& \emptyset & & \text{if $l=0$ or $l=1$}, \\
			& \mathsf{rem}(S'') \sqcup \left\{S(1,l-1),S(1,l)\right\} & & \begin{aligned} &\text{if $S(1,l)$ is even, $S(1,l)-1=S(1,l-1)$,} \\ &\text{and $S(1,l) < 2l - \sharp \, \mathsf{rem}(S'') - 1$}, \end{aligned} \\
			& \mathsf{rem}(S') & & \text{otherwise}.
		\end{aligned}
		\right.
  \]
	We denote by $\mathsf{red}(S)$
	the tableau obtained from $S$ by removing the boxes whose entry belongs to $\mathsf{rem}(S)$.
\end{de}
\begin{eg}\label{5.3}
	Since $4 < 2\times3 - \sharp \, \emptyset - 1$ and $2 < 2\times2 - \sharp \, \emptyset - 1$, we have
	\begin{center}
		\begin{tikzpicture}[x=5mm,y=5mm]
			\draw (0,0)--(0,3)--(1,3)--(1,0)--cycle;
			\draw (0,1)--(1,1);
			\draw (0,2)--(1,2);
			\node at (0.5,0.5) {4};
			\node at (0.5,1.5) {3};
			\node at (0.5,2.5) {2};
			\node at (-0.25,1.5) {$\mathsf{rem}\left(\begin{aligned} \\ \,\,\,\,\,\,\,\,\, \\ \\ \end{aligned}\right)$};

      \node at (2.3,1.4) {$=$};
			\node at (4.1,1.5) {$\left\{3,4\right\},$};
		\end{tikzpicture}
		\begin{tikzpicture}[x=5mm,y=5mm]
			\draw (0,0)--(0,3)--(1,3)--(1,0)--cycle;
			\draw (0,1)--(1,1);
			\draw (0,2)--(1,2);
			\node at (0.5,0.5) {4};
			\node at (0.5,1.5) {3};
			\node at (0.5,2.5) {2};
			\node at (-0.1,1.5) {$\mathsf{red}\left(\begin{aligned} \\ \,\,\,\,\,\,\,\,\, \\ \\ \end{aligned}\right)$};

      \node at (2.3,1.4) {$=$};

			\draw (3,1)--(3,2)--(4,2)--(4,1)--cycle;
			\node at (3.5,1.5) {2};
			\node at (4.4,1.2) {,};
		\end{tikzpicture}\\
		\begin{tikzpicture}[x=5mm,y=5mm]
			\draw (0,0)--(0,3)--(1,3)--(1,0)--cycle;
			\draw (0,1)--(1,1);
			\draw (0,2)--(1,2);
			\node at (0.5,0.5) {4};
			\node at (0.5,1.5) {2};
			\node at (0.5,2.5) {1};
			\node at (-0.25,1.5) {$\mathsf{rem}\left(\begin{aligned} \\ \,\,\,\,\,\,\,\,\, \\ \\ \end{aligned}\right)$};

      \node at (2.3,1.4) {$=$};

			\draw (5.4,0.5)--(5.4,2.5)--(6.4,2.5)--(6.4,0.5)--cycle;
			\draw (5.4,1.5)--(6.4,1.5);
			\node at (5.9,1) {2};
			\node at (5.9,2) {1};
			\node at (5.15,1.5) {$\mathsf{rem}\left(\begin{aligned} \\ \,\,\,\,\,\,\,\,\, \\ \\ \end{aligned}\right)$};

			\node at (7.7,1.4) {$=$};
			\node at (9.5,1.5) {$\left\{1,2\right\},$};
		\end{tikzpicture}
		\begin{tikzpicture}[x=5mm,y=5mm]
			\draw (0,0)--(0,3)--(1,3)--(1,0)--cycle;
			\draw (0,1)--(1,1);
			\draw (0,2)--(1,2);
			\node at (0.5,0.5) {4};
			\node at (0.5,1.5) {2};
			\node at (0.5,2.5) {1};
			\node at (-0.1,1.5) {$\mathsf{red}\left(\begin{aligned} \\ \,\,\,\,\,\,\,\,\, \\ \\ \end{aligned}\right)$};

      \node at (2.3,1.4) {$=$};

			\draw (3,1)--(3,2)--(4,2)--(4,1)--cycle;
			\node at (3.5,1.5) {4};
			\node at (4.4,1.2) {.};
		\end{tikzpicture}
	\end{center}
\end{eg}
\begin{de}[{\cite[Definition 2.6.1]{Wat1}}]
	Let $S \in \mathsf{SST}_{2n}$. Let $S_{1}$ denote 
	the first column of $S$, and $S_{\geq2}$ the rest of the tableau $S$. 
	We define $\mathsf{suc}(S)$ by
	\begin{equation}
		\mathsf{suc}(S) := \mathsf{red}(S_1) * S_{\geq2}.
	\end{equation}
\end{de}
\begin{eg}\label{5.5}
	By Example \ref{5.3}, we have 
	\begin{center}
		\begin{tikzpicture}[x=5mm,y=5mm]
			\draw (0,0)--(0,3)--(2,3)--(2,1)--(1,1)--(1,0)--cycle;
			\draw (0,1)--(1,1);
			\draw (0,2)--(2,2);
			\draw (1,1)--(1,3);
			\node at (0.5,0.5) {4};
			\node at (0.5,1.5) {2};
			\node at (0.5,2.5) {1};
			\node at (1.5,1.5) {3};
			\node at (1.5,2.5) {2};
			\node at (0.35,1.5) {$\mathsf{suc}\left(\begin{aligned} \\ \,\,\,\,\,\,\,\,\,\,\,\,\,\,\,\, \\ \\ \end{aligned}\right)$};

      \node at (3.3,1.4) {$=$};

			\draw (6,0)--(6,3)--(7,3)--(7,0)--cycle;
			\draw (6,1)--(7,1);
			\draw (6,2)--(7,2);
			\node at (6.5,0.5) {4};
			\node at (6.5,1.5) {2};
			\node at (6.5,2.5) {1};
			\node at (5.9,1.5) {$\mathsf{red}\left(\begin{aligned} \\ \,\,\,\,\,\,\,\,\, \\ \\ \end{aligned}\right)$};

			\node at (8.25,1.5) {$*$};
			
			\draw (8.9,0.5)--(8.9,2.5)--(9.9,2.5)--(9.9,0.5)--cycle;
			\draw (8.9,1.5)--(9.9,1.5);
			\node at (9.4,1) {3};
			\node at (9.4,2) {2};

			\node at (10.7,1.4) {$=$};

			\draw (11.5,0)--(11.5,3)--(12.5,3)--(12.5,0)--cycle;
			\draw (11.5,1)--(12.5,1);
			\draw (11.5,2)--(12.5,2);
			\node at (12,0.5) {4};
			\node at (12,1.5) {3};
			\node at (12,2.5) {2};
			\node at (12.9,1.2) {.};
		\end{tikzpicture}
	\end{center}
\end{eg}
The following proposition ensures that
the procedure above will terminate after a finite number of steps.
\begin{prop}[{\cite[Lemma 2.6.4]{Wat1}}]
	Let $S \in \mathsf{SST}_{2n}(\lambda)$, and let $\mu$ be 
	the shape of $\mathsf{suc}(S)$.
	Then, $\mu \subset \lambda$, and $\lambda/{\mu}$ is vertical strip.
	Moreover, $\lambda=\mu$ if and only if $S=\mathsf{suc}(S)$.
\end{prop}
The following proposition tells us exactly when the procedure above stops. 
\begin{prop}[{\cite[Corollary 4.3.8]{Wat1}}]\label{2024}
	For $S \in \mathsf{SST}_{2n}$, the following are equivalent:
	\begin{enumerate}
		\item[(1)] $S$ is a symplectic tableau; 
		\item[(2)] $\mathsf{suc}(S)=S$.
	\end{enumerate}
\end{prop}

\begin{de}[{\cite[Definition 3.1.1]{Wat1}}]
	Let $S \in \mathsf{SST}_{2n}$,
	and let $N$ be a non-negative integer satisfying $\mathsf{suc}^{N+1}(S) = \mathsf{suc}^{N}(S)$.
	We set $\mathsf{P}^{A\mathrm{II}}(S) := \mathsf{suc}^{N}(S)$.
	Also, 
	let $\mathsf{Q}^{A\mathrm{II}}(S)$ be the tableau
	that records the process of transformations from $S$ to $\mathsf{P}^{\mathrm{II}}(S)$.
	Namely, let the shapes of $S$, $\mathsf{suc}(S)$, $\mathsf{suc}^2(S)$, $\mathsf{suc}^3(S)$, ..., $\mathsf{suc}^N(S)$
	be $\lambda^0$, $\lambda^1$, $\lambda^2$, $\lambda^3$, ...,  $\lambda^N$, respectively, 
	and define $\mathsf{Q}^{A\mathrm{II}}(S)$
	as the tableau obtained by placing the entry $j$ in $\lambda^{j-1}/\lambda^j$ for $j=1,2,3$, ..., $N$. 
	For $\lambda \in \mathsf{Par}_{\leq2n}$ and $\mu \in \mathsf{Par}_{\leq n}$,
	we set
	\begin{equation}
		\mathsf{Rec}_{2n}(\lambda/{\mu}) :=
		\setlr{\mathsf{Q}^{A\mathrm{II}}(S)}{\begin{aligned} & S \in \mathsf{SST}_{2n}(\lambda), \\ & \text{the shape of $\mathsf{P}^{A\mathrm{II}}(S)$ is $\mu$}\end{aligned}}.
	\end{equation}
\end{de}
\begin{eg}\label{5.9}
	By Examples \ref{5.3} and \ref{5.5}, we see that 
	\begin{center}
		\begin{tikzpicture}[x=5mm,y=5mm]
			\draw (0,0)--(0,3)--(2,3)--(2,1)--(1,1)--(1,0)--cycle;
			\draw (0,1)--(1,1);
			\draw (0,2)--(2,2);
			\draw (1,1)--(1,3);
			\node at (0.5,0.5) {4};
			\node at (0.5,1.5) {2};
			\node at (0.5,2.5) {1};
			\node at (1.5,1.5) {3};
			\node at (1.5,2.5) {2};
			\node at (0.15,1.5) {$\mathsf{P}^{A\mathrm{II}}\left(\begin{aligned} \\ \,\,\,\,\,\,\,\,\,\,\,\,\,\,\,\, \\ \\ \end{aligned}\right)$};

      \node at (3.3,1.4) {$=$};

			\draw (4.1,1)--(4.1,2)--(5.1,2)--(5.1,1)--cycle;
			\node at (4.6,1.5) {2};
			\node at (5.5,1.2) {,};
		\end{tikzpicture}
		\begin{tikzpicture}[x=5mm,y=5mm]
			\draw (0,0)--(0,3)--(2,3)--(2,1)--(1,1)--(1,0)--cycle;
			\draw (0,1)--(1,1);
			\draw (0,2)--(2,2);
			\draw (1,1)--(1,3);
			\node at (0.5,0.5) {4};
			\node at (0.5,1.5) {2};
			\node at (0.5,2.5) {1};
			\node at (1.5,1.5) {3};
			\node at (1.5,2.5) {2};
			\node at (0.15,1.5) {$\mathsf{Q}^{A\mathrm{II}}\left(\begin{aligned} \\ \,\,\,\,\,\,\,\,\,\,\,\,\,\,\,\, \\ \\ \end{aligned}\right)$};

      \node at (3.3,1.4) {$=$};

			\draw (4.1,0)--(4.1,2)--(5.1,2)--(5.1,3)--(6.1,3)--(6.1,1)--(5.1,1)--(5.1,0)--cycle;
			\draw (4.1,1)--(5.1,1);
			\draw (5.1,1)--(5.1,3);
			\draw (4.1,2)--(6.1,2);
			\node at (4.6,0.5) {2};
			\node at (4.6,1.5) {2};
			\node at (5.6,1.5) {1};
			\node at (5.6,2.5) {1};
			\node at (6.5,1.2) {.};
		\end{tikzpicture}
	\end{center}
\end{eg}
\begin{thm}[{\cite[Theorem 3.1.4]{Wat1}}]
	Let $\lambda \in \mathsf{Par}_{\leq 2n}$.
	Then, the correspondence $S \mapsto \left(\mathsf{P}^{A\mathrm{II}}(S),\mathsf{Q}^{A\mathrm{II}}(S)\right)$
	gives the following bijection:
	\begin{equation}
		\mathsf{LR}^{A\mathrm{II}} :
		\mathsf{SST}_{2n}(\lambda) \stackrel{\sim}{\longrightarrow}
		\bigsqcup_{\mu \in \mathsf{Par}_{\leq n}}
		\mathsf{SpT}_{2n}(\mu) \times \mathsf{Rec}_{2n}(\lambda/{\mu}).
	\end{equation}
\end{thm}
Let $\mathbb{C}(q)\mathsf{Rec}_{2n}(\lambda/{\mu})$
denote the $\mathbb{C}(q)$-vector space
with basis $\mathsf{Rec}_{2n}(\lambda/{\mu})$.
In \cite{Wat1},
the desired branching rule is described by ``melting'' the bijection above.
\begin{thm}[{\cite[Corollary 7.2.2]{Wat1}}]
	For all $\mu \in \widetilde{P}^+$,
	there exists a $\mathbb{C}(q)$-basis
	\begin{equation}
		\mathbf{B}^{\imath}(\mu) := \setlr{G^{\imath}(b_S^{\imath})}{S \in \mathsf{SpT}_{2n}(\mu)}
	\end{equation}
	of $\widetilde{L}^{\imath}(\mu)$
	such that the following holds:
	for all $\lambda \in P^+_{\geq0}$,
	there exists an isomorphism of $\mathbf{U}^{\imath}$-modules: 
	\begin{equation}\label{LR}
		\mathsf{LR}^{A\mathrm{II}} : \mathsf{Res}^{\mathbf{U}}_{\mathbf{U}^{\imath}}\,L_q(\lambda)
		\stackrel{\sim}{\longrightarrow}
		\bigoplus_{\mu \in \widetilde{P}^+} \widetilde{L}^{\imath}(\mu) \otimes \mathbb{C}(q)\mathsf{Rec}_{2n}(\lambda/{\mu})
	\end{equation}
	that sends $G(b_S)$ to $G^{\imath}\left(b^{\imath}_{\mathsf{P}^{}(S)}\right) \otimes \mathsf{Q}^{}(S)$
	at $q=\infty$.
\end{thm}
In \cite{Wat1}, the isomorphism 
(\ref{LR}) is constructed based on the following. 
\begin{prop}[{\cite[Definition 6.3.3]{Wat1}}]
	We set $\lambda' := (\lambda_1-1,\ldots,\lambda_{\ell(\lambda)}-1)$.
	Then, there exists a $\mathbf{U}^{\imath}$-module homomorphism: 
	\begin{equation}
		\mathsf{suc} : \mathsf{Res}^{\mathbf{U}}_{\mathbf{U}^{\imath}}\,L_q(\lambda)
		\to \bigoplus_{\lambda' \underset{\mathrm{vert}}{\subset} \mu} \mathsf{Res}^{\mathbf{U}}_{\mathbf{U}^{\imath}}\,L_q(\mu)
	\end{equation}
	that sends $G(b_S)$ to $G(b_{\mathsf{suc}(S)})$ at $q=\infty$.
\end{prop}

\subsection{Towards the proof of the Naito--Sagaki conjecture}
In this subsection,
we connect the branching rule for $\imath$quantum groups
stated in the previous subsection with the Naito--Sagaki conjecture,
and explain how to prove the conjecture.

Below, we will define $\mathfrak{k}$-highest (or $\mathfrak{k}$-lowest) weight semistandard tableaux. 
%
%

Let $\lambda \in P^{+}_{\geq 0}$. 
For $S \in \mathsf{SST}_{2n}(\lambda)$, we set
$\mathsf{wt}_{\mathfrak{k}}(S) := \widetilde{\mathsf{wt}(S)}$,
where the right-hand side is the image of $\mathsf{wt}(S)$
under the natural restriction map $\widetilde{\cdot}$ .
It can also be rewritten as follows: 
\begin{equation}
	\begin{aligned}
		& \mathsf{wt}_{\mathfrak{k}}(S) \\
		&= \langle\mathsf{wt}(S),z_1\rangle\widetilde{\varpi}_1 + \langle\mathsf{wt}(S),z_2\rangle\widetilde{\varpi}_2 + \cdots + \langle\mathsf{wt}(S),z_n\rangle\widetilde{\varpi}_n \\
		&= \langle\mathsf{wt}(S),z_1 + \cdots + z_n \rangle\widetilde{\varepsilon}_1 + \langle\mathsf{wt}(S),z_2 + \cdots + z_n\rangle\widetilde{\varepsilon}_2 + \cdots + \langle\mathsf{wt}(S),z_n\rangle\widetilde{\varepsilon}_n \\
		&= \langle\mathsf{wt}(S),-h_1 \rangle\widetilde{\varepsilon}_1 + \langle\mathsf{wt}(S), h_2 \rangle\widetilde{\varepsilon}_2 + \cdots + \langle\mathsf{wt}(S),(-1)^n h_n \rangle\widetilde{\varepsilon}_n \\
		&= \left(S[2]-S[1]\right)\widetilde{\varepsilon}_1 + \left(S[3]-S[4]\right)\widetilde{\varepsilon}_2 + \cdots + (-1)^n\left(S[2n-1]-S[2n]\right)\widetilde{\varepsilon}_n.
	\end{aligned}\label{5.13a}
\end{equation}
If we define two sequences of positive integers
$\{a_k\}_{k \in \mathbb{N}},\{b_k\}_{k \in \mathbb{N}}$
by:
\begin{equation}
	\begin{aligned}
		& a_k := 2k - \frac{1+(-1)^k}{2}, &
		& b_k := 2k - \frac{1+(-1)^{k+1}}{2}. &
		& k=1,2,3,\ldots,n,
	\end{aligned}
\end{equation}
then we have 
\begin{equation}
	\mathsf{wt}_{\mathfrak{k}}(S) = \left(S[a_1]-S[b_1]\right)\widetilde{\varepsilon}_1 + \left(S[a_2]-S[b_2]\right)\widetilde{\varepsilon}_2 + \cdots + \left(S[a_n]-S[b_n]\right)\widetilde{\varepsilon}_n. 
\end{equation}
\begin{eg}
  Let $S$ be the tableau given below.
	Then we have 
	$\mathsf{wt}_{\mathfrak{k}}(S) = 2\widetilde{\varepsilon}_1 + 2\widetilde{\varepsilon}_2$.
	\begin{center}
	\begin{tikzpicture}[x=5mm,y=5mm]
			\draw (0,0)--(0,3)--(6,3)--(6,2)--(5,2)--(5,1)--(3,1)--(3,0)--cycle;
			\draw (0,1)--(3,1);
			\draw (0,2)--(5,2);
			\draw (1,0)--(1,3);
			\draw (2,0)--(2,3);
			\draw (3,1)--(3,3);
			\draw (4,1)--(4,3);
			\draw (5,2)--(5,3);
			\node at (0.5,0.5) {3};
			\node at (0.5,1.5) {2};
			\node at (0.5,2.5) {1};
			\node at (1.5,0.5) {4};
			\node at (1.5,1.5) {2};
			\node at (1.5,2.5) {1};
			\node at (2.5,0.5) {4};
			\node at (2.5,1.5) {3};
			\node at (2.5,2.5) {2};
			\node at (3.5,1.5) {3};
			\node at (3.5,2.5) {2};
			\node at (4.5,1.5) {4};
			\node at (4.5,2.5) {3};
			\node at (5.5,2.5) {3};

      \node at (-1,1.5) {$S=$};
			\node at (6.5,1.3) {$.$};
		\end{tikzpicture}
	\end{center}
\end{eg}
Now,
assume that ``$b^{\imath}_{\mathsf{P}^{A\mathrm{II}}(S)}$''
is the ``highest weight vector'' of $\widetilde{L}^{\imath}(\mu)$.
Then,
the shapes of $\mathsf{P}^{A\mathrm{II}}(S)$
and $\mathsf{wt}_{\mathfrak{k}}(\mathsf{P}^{A\mathrm{II}}(S))$
are both equal to $\mu$. 
Therefore, $\mathsf{P}^{A\mathrm{II}}(S)$
is a tableau of the form shown in Figure \ref{k-hw} below. 
Next,
assume that ``$b^{\imath}_{\mathsf{P}^{A\mathrm{II}}(S)}$''
is the ``lowest weight vector'' of $\widetilde{L}^{\imath}(\mu)$.
Then,
the shapes of $\mathsf{P}^{A\mathrm{II}}(S)$
and $\widetilde{w}_0\,\mathsf{wt}_{\mathfrak{k}}(\mathsf{P}^{A\mathrm{II}}(S))$
are both equal to $\mu$,
where $\widetilde{w}_0$ denotes the longest element of
the Weyl group of $\mathfrak{k}$.
Therefore, $\mathsf{P}^{A\mathrm{II}}(S)$
is a tableau of the form shown in Figure \ref{k-lw} below.
\begin{figure}[th]
	\begin{center}
	\begin{minipage}{0.49\columnwidth}
		\begin{center}
			\begin{tikzpicture}[x=5mm,y=5mm]
				\draw (0,0)--(10,0)--(10,1)--(0,1)--cycle;
				\draw (0,1)--(0,3);
				\draw (10,1)--(11,1)--(11,2)--(12,2)--(12,3);
				\draw (0,3)--(13,3)--(13,4)--(0,4)--cycle;
				\draw (0,4)--(0,5)--(14,5)--(14,4)--(13,4);
				\draw[dotted, line width=1pt] (6,1.7)--(6,2.3);
				\node (a) at (0.5,0.5) {$a_n$};
				\node (b) at (0.5,3.5) {$a_2$};
				\node (c) at (0.5,4.5) {$a_1$};
				\node (d) at (9.5,0.5) {$a_n$};
				\node (e) at (12.5,3.5) {$a_2$};
				\node (f) at (13.5,4.5) {$a_1$};
				\draw[dotted, line width=1pt] (a)--(1.7,0.5);
				\draw[dotted, line width=1pt] (d)--(8.3,0.5);
				\draw[dotted, line width=1pt] (b)--(1.7,3.5);
				\draw[dotted, line width=1pt] (e)--(11.3,3.5);
				\draw[dotted, line width=1pt] (c)--(1.7,4.5);
				\draw[dotted, line width=1pt] (f)--(12.3,4.5);
			\end{tikzpicture}
		\end{center}
		\caption{}\label{k-hw}
	\end{minipage}
	\begin{minipage}{0.49\columnwidth}
		\begin{center}
			\begin{tikzpicture}[x=5mm,y=5mm]
				\draw (0,0)--(10,0)--(10,1)--(0,1)--cycle;
				\draw (0,1)--(0,3);
				\draw (10,1)--(11,1)--(11,2)--(12,2)--(12,3);
				\draw (0,3)--(13,3)--(13,4)--(0,4)--cycle;
				\draw (0,4)--(0,5)--(14,5)--(14,4)--(13,4);
				\draw[dotted, line width=1pt] (6,1.7)--(6,2.3);
				\node (a) at (0.5,0.5) {$b_n$};
				\node (b) at (0.5,3.5) {$b_2$};
				\node (c) at (0.5,4.5) {$b_1$};
				\node (d) at (9.5,0.5) {$b_n$};
				\node (e) at (12.5,3.5) {$b_2$};
				\node (f) at (13.5,4.5) {$b_1$};
				\draw[dotted, line width=1pt] (a)--(1.7,0.5);
				\draw[dotted, line width=1pt] (d)--(8.3,0.5);
				\draw[dotted, line width=1pt] (b)--(1.7,3.5);
				\draw[dotted, line width=1pt] (e)--(11.3,3.5);
				\draw[dotted, line width=1pt] (c)--(1.7,4.5);
				\draw[dotted, line width=1pt] (f)--(12.3,4.5);
			\end{tikzpicture}
		\end{center}
		\caption{}\label{k-lw}
	\end{minipage}
	\end{center}
\end{figure}

Based on the observation above,
we give the following definition.
\begin{de}\label{5.13}
	Let $S \in \mathsf{SST}_{2n}$.
	\begin{enumerate}
		\item $S$ is called a $\mathfrak{k}$-highest weight tableau if $\mathsf{P}^{A\mathrm{II}}(S)$ is a tableau shown in Figure \ref{k-hw}; let $\mathsf{SST}_{2n}^{\text{$\mathfrak{k}$-$\mathrm{hw}$}}(\lambda)$ denote the set of all $\mathfrak{k}$-highest weight tableaux of shape $\lambda$. 
		\item $S$ is called a $\mathfrak{k}$-lowest weight tableau if $\mathsf{P}^{A\mathrm{II}}(S)$ is a tableau shown in Figure \ref{k-lw}; let $\mathsf{SST}_{2n}^{\text{$\mathfrak{k}$-$\mathrm{lw}$}}(\lambda)$ denote the set of all $\mathfrak{k}$-lowest weight tableaux of shape $\lambda$. 
	\end{enumerate}
\end{de}

\begin{eg}
	By Examples \ref{5.3} and \ref{5.9}, 
	the following tableaux $S_1$ and $S_2$ are 
	$\mathfrak{k}$-highest weight tableaux; recall that $a_1 = 2, b_1 = 3$.
	\begin{center}
	\begin{tikzpicture}[x=5mm,y=5mm]
			\draw (0,0)--(0,3)--(2,3)--(2,1)--(1,1)--(1,0)--cycle;
			\draw (0,1)--(1,1);
			\draw (0,2)--(2,2);
			\draw (1,1)--(1,3);
			\node at (0.5,0.5) {4};
			\node at (0.5,1.5) {2};
			\node at (0.5,2.5) {1};
			\node at (1.5,1.5) {3};
			\node at (1.5,2.5) {2};

      \node at (-1.25,1.5) {$S_1=$};
			\node at (2.5,1.3) {$,$};
		\end{tikzpicture}
		\begin{tikzpicture}[x=5mm,y=5mm]
			\draw (0,0)--(0,3)--(4,3)--(4,2)--(3,2)--(3,1)--(2,1)--(2,0)--(0,0)--cycle;
			\draw (0,1)--(2,1);
			\draw (0,2)--(3,2);
			\draw (1,0)--(1,3);
			\draw (2,1)--(2,3);
			\draw (3,2)--(3,3);
			\node at (0.5,0.5) {3};
			\node at (0.5,1.5) {2};
			\node at (0.5,2.5) {1};
			\node at (1.5,0.5) {4};
			\node at (1.5,1.5) {3};
			\node at (1.5,2.5) {2};
			\node at (2.5,1.5) {3};
			\node at (2.5,2.5) {2};
			\node at (3.5,2.5) {2};

      \node at (-1.25,1.5) {$S_2=$};
			\node at (4.5,1.3) {$.$};
		\end{tikzpicture}
	\end{center}
\end{eg}

The branching rule for $\mathbf{U}^{\imath}$ 
can be rewritten as follows.
\begin{cor}
	For all $\lambda \in P^+_{\geq0}$ and $\mu \in \widetilde{P}^+$,
	the following equality holds:
	\begin{equation}\label{5.16}
		\left[\mathsf{Res}^{\mathbf{U}}_{\mathbf{U}^{\imath}}\,L_q(\lambda) : \widetilde{L}^{\imath}(\mu)\right]
		=
		\sharp\left\{
		T \in \mathsf{SST}_{2n}^{\text{$\mathfrak{k}$-$\mathrm{hw}$}}(\lambda)
		\mathrel{}\middle|\mathrel{}
		\mathsf{wt}_{\mathfrak{k}}(T) = \mu
		\right\}.
	\end{equation}
\end{cor}
\begin{proof}
	Let us denote by $S^{H}_{\mu}$
	a $\mathfrak{k}$-highest weight tableau of shape $\lambda$
	such that $\mathsf{wt}_{\mathfrak{k}}(S^{H}_{\mu})=\mu$. 
	Then, we see that 
	\begin{equation}
		\begin{aligned}
			\left[\mathsf{Res}^{\mathbf{U}}_{\mathbf{U}^{\imath}}\,L_q(\lambda) : \widetilde{L}^{\imath}(\mu)\right]
			&= \sharp \, \mathsf{Rec}_{2n}(\lambda/{\mu}) \\
			&= \sharp \left\{ S^{H}_{\mu} \right\} \times \mathsf{Rec}_{2n}(\lambda/{\mu}) \\
			&= \sharp \setlr{ S \in \mathsf{SST}_{2n}(\lambda)}{\mathsf{P}^{}(S)=S^{H}_{\mu}} \\
			&= \sharp
			\left\{ T \in \mathsf{SST}_{2n}^{\text{$\mathfrak{k}$-hw}}(\lambda)
			\mathrel{}\middle|\mathrel{}\mathsf{wt}_{\mathfrak{k}}(T) = \mu \right\}, 
		\end{aligned}
	\end{equation}
as desired. This proves the corollary. 
\end{proof}
Also, the following is well-known.
\begin{prop}[{\cite[Theorem 1.3]{D}}]\label{Dynkin}
	Let $\mathfrak{z}$ be a complex simple Lie algebra,
	and $\iota_1,\iota_2 : \mathfrak{z} \hookrightarrow \mathfrak{sl}_N(\mathbb{C})$ be two embeddings of Lie algebras.
	Denote the vector representation of $\mathfrak{sl}_{N}(\mathbb{C})$ by $\rho : \mathfrak{sl}_{N}(\mathbb{C}) \to \mathrm{End}_{\mathbb{C}}(\mathbb{C}^N)$.
	Then, the following are equivalent:
	\begin{enumerate}
		\item $\rho \circ \iota_1 \cong \rho \circ \iota_2$ as $\mathfrak{z}$-modules; 
		\item for every $\mathfrak{sl}_{N}(\mathbb{C})$-module $\pi$, $\pi \circ \iota_1 \cong \pi \circ \iota_2$ as $\mathfrak{z}$-modules.
	\end{enumerate}
\end{prop}
By applying Proposition \ref{Dynkin} to
$N=2n$,
$\mathfrak{z} = \mathfrak{sp}_{2n}(\mathbb{C})$,
$\mathrm{Image}(\iota_1) = \widehat{\mathfrak{g}}$,
$\mathrm{Image}(\iota_2) = \mathfrak{k}$,
we obtain the following equality for all $\lambda \in P^+_{\geq 0}$, $\mu \in \widehat{P}^+ \cong \widetilde{P}^+$: 
\begin{equation}
 \left[\mathsf{Res}^{\mathfrak{g}}_{\widehat{\mathfrak{g}}}\,L(\lambda) : \widehat{L}(\mu)\right]
	  =
	  \left[\mathsf{Res}^{\mathfrak{g}}_{\mathfrak{k}}\,L(\lambda) : \widetilde{L}(\mu) \right]. \label{5.17}
\end{equation}
By (\ref{3.12}), Theorem \ref{thm5.1}(4), (\ref{5.16}), and (\ref{5.17}),
in order to prove the Naito--Sagaki conjecture, 
it suffices to prove the following.
\begin{thm}\label{KeyProposition}
	For all $\lambda \in \mathsf{Par}_{\leq2n}$,
	there exist two bijections
	\begin{equation}
		\begin{aligned}
			\mathsf{\Phi} &: \mathsf{SST}_{2n}^{\text{$\widehat{\mathfrak{g}}$-$\mathrm{dom}$}}(\lambda) \stackrel{\sim}{\longrightarrow} \mathsf{SST}_{2n}^{\text{$\mathfrak{k}$-$\mathrm{hw}$}}(\lambda), &
			\mathsf{\Psi} &: \mathsf{SST}_{2n}^{\text{$\widehat{\mathfrak{g}}$-$\mathrm{dom}$}}(\lambda) \stackrel{\sim}{\longrightarrow} \mathsf{SST}_{2n}^{\text{$\mathfrak{k}$-$\mathrm{lw}$}}(\lambda)
		\end{aligned}
	\end{equation}
	such that for each $T \in \mathsf{SST}_{2n}^{\text{$\widehat{\mathfrak{g}}$-$\mathrm{dom}$}}(\lambda)$, the following equalities hold: 
	\begin{equation}
		\begin{aligned}
			& \mathsf{wt}_{\widehat{\mathfrak{g}}}(T) = \mathsf{wt}_{\mathfrak{k}}\left(\mathsf{\Phi}(T)\right), &
			& \mathsf{wt}_{\widehat{\mathfrak{g}}}(T) = \mathsf{wt}_{\mathfrak{k}}\left(\mathsf{\Psi}(T)\right). \label{5.19}
		\end{aligned}
	\end{equation}
\end{thm}
Here, $\mathsf{\Psi}$ is an auxiliary map, which we use in order to prove the bijectivity of $\mathsf{\Phi}$. 
The next two sections
are devoted to the proof of this theorem.
We conclude this subsection
by giving the outline of the proof.
First, we prove the theorem in the case $n=2$
by explicitly determining the three sets, 
$\mathsf{SST}_{2n}^{\text{$\widehat{\mathfrak{g}}$-$\mathrm{dom}$}}(\lambda)$,
$\mathsf{SST}_{2n}^{\text{$\mathfrak{k}$-$\mathrm{hw}$}}(\lambda)$, and
$\mathsf{SST}_{2n}^{\text{$\mathfrak{k}$-$\mathrm{lw}$}}(\lambda)$.
Next, we consider the higher rank case.
Recall that
$\widehat{\mathfrak{g}}$, $\mathfrak{k}$, and $\mathbf{U}^{\imath}$
are generated by subalgebras of rank $2$, each of which is associated with $\mathfrak{sp}_4(\mathbb{C})$.
Also, recall that the representation theory of
Kac--Moody algebras and
the theory of crystal bases are built upon the representation theory of $\mathfrak{sl}_2(\mathbb{C})$. 
Hence, we expect that the higher rank case can be reduced to the case of $\mathfrak{sp}_4(\mathbb{C})$
(i.e., to the case $n=2$).
Based on this idea, 
we decompose the global information for the case $n \geq 3$
into local information for each subalgebra associated with $\mathfrak{sp}_4(\mathbb{C})$, 
and then apply the results in the case $n=2$ to each of the subalgebras.
Finally,
by putting together the local information for the subalgebras associated with $\mathfrak{sp}_4(\mathbb{C})$, we obtain the desired global information for the case $n \geq 3$, which completes the proof of the theorem.

\bigskip

\section{Proof of theorem \ref{KeyProposition} (the case $n=2$)}
In this section,
we consider the case $n=2$.
Let us fix an arbitrary partition $\lambda \in \mathsf{Par}_{\leq4}$.
\subsection{Characterizations of certain highest weight tableaux}
In this subsection,
we characterize $\widehat{\mathfrak{g}}$-dominant tableaux,
$\mathfrak{k}$-highest weight tableaux,
and $\mathfrak{k}$-lowest weight tableaux. 

\begin{lem}\label{6.1}
	For each
	$T \in \mathsf{SST}_{4}(\lambda)$,
	the following two conditions are equivalent:
	\begin{enumerate}
		\item[(1)] $T$ is a $\widehat{\mathfrak{g}}$-dominant tableau; 
		\item[(2)] $T$ is a tableau of the following form, with $x \leq y$.
	\end{enumerate}
\end{lem}
\begin{figure}[ht]
	\begin{tikzpicture}[x=5mm,y=5mm]
		\draw (0,0)--(3,0)--(3,1)--(9,1)--(9,2)--(15,2)--(15,3)--(20,3)--(20,4)--(0,4)--cycle;
		\draw (3,1)--(3,4);
		\draw (9,2)--(9,4);
		\draw (15,3)--(15,4);
		\draw (0,3)--(15,3);
		\draw (0,2)--(9,2);
		\draw (0,1)--(3,1);
		\draw (6,1)--(6,4);
		\draw (12,2)--(12,4);
		
		\draw[line width=1pt,dotted] (1.2,0.5)--(1.8,0.5);
		\draw[line width=1pt,dotted] (1.2,1.5)--(1.8,1.5);
		\draw[line width=1pt,dotted] (4.2,1.5)--(4.8,1.5);
		\draw[line width=1pt,dotted] (7.2,1.5)--(7.8,1.5);
		\draw[line width=1pt,dotted] (1.2,2.5)--(1.8,2.5);
		\draw[line width=1pt,dotted] (4.2,2.5)--(4.8,2.5);
		\draw[line width=1pt,dotted] (7.2,2.5)--(7.8,2.5);
		\draw[line width=1pt,dotted] (10.2,2.5)--(10.8,2.5);
		\draw[line width=1pt,dotted] (13.2,2.5)--(13.8,2.5);
		\draw[line width=1pt,dotted] (1.2,3.5)--(1.8,3.5);
		\draw[line width=1pt,dotted] (4.2,3.5)--(4.8,3.5);
		\draw[line width=1pt,dotted] (7.2,3.5)--(7.8,3.5);
		\draw[line width=1pt,dotted] (10.2,3.5)--(10.8,3.5);
		\draw[line width=1pt,dotted] (13.2,3.5)--(13.8,3.5);
		\draw[line width=1pt,dotted] (16.2,3.5)--(16.8,3.5);
		\draw[line width=1pt,dotted] (18.2,3.5)--(18.8,3.5);
		\node (1,4) at (0.5,0.5) {4};
		\node (3,4) at (2.5,0.5) {4};
		\node (1,3) at (0.5,1.5) {3};
		\node (3,3) at (2.5,1.5) {3};
		\node (4,3) at (3.5,1.5) {3};
		\node (6,3) at (5.5,1.5) {3};
		\node (7,3) at (6.5,1.5) {4};
		\node (9,3) at (8.5,1.5) {4};
		\node (1,2) at (0.5,2.5) {2};
		\node (3,2) at (2.5,2.5) {2};
		\node (4,2) at (3.5,2.5) {2};
		\node (6,2) at (5.5,2.5) {2};
		\node (7,2) at (6.5,2.5) {2};
		\node (9,2) at (8.5,2.5) {2};
		\node (10,2) at (9.5,2.5) {2};
		\node (12,2) at (11.5,2.5) {2};
		\node (13,2) at (12.5,2.5) {4};
		\node (15,2) at (14.5,2.5) {4};
		\node (1,1) at (0.5,3.5) {1};
		\node (3,1) at (2.5,3.5) {1};
		\node (4,1) at (3.5,3.5) {1};
		\node (6,1) at (5.5,3.5) {1};
		\node (7,1) at (6.5,3.5) {1};
		\node (9,1) at (8.5,3.5) {1};
		\node (10,1) at (9.5,3.5) {1};
		\node (12,1) at (11.5,3.5) {1};
		\node (13,1) at (12.5,3.5) {1};
		\node (15,1) at (14.5,3.5) {1};
		\node (16,1) at (15.5,3.5) {1};
		\node (20,1) at (19.5,3.5) {1};
		
		\draw decorate[decoration=brace]{ (9,0.7)--(6,0.7) };
		\draw decorate[decoration=brace]{ (20,2.7)--(15.2,2.7) };
		
		\node (long) at (17.6,2) {$y$};
		\node (short) at (7.5,0.1) {$x$};
	\end{tikzpicture}
\caption{A $\widehat{\mathfrak{g}}$-dominant tableau}\label{10000}
\end{figure}

\begin{proof}
	We prove the implication (1) $\implies$ (2). 
	Assume that $T$ is a
	$\widehat{\mathfrak{g}}$-dominant tableau.
	By the comment following {\cite[Definition 2]{ST}},
	the rightmost entry of the first row of the tableau $T$ is $1$, and 
	hence all entries in the first row of $T$ are equal to $1$.
	From this, it also follows that no entry in the second row of $T$ is $3$.
	Therefore, 
	$T$ has the same shape as the one in Figure \ref{vwxyz} below.
	Also, if we replace all entries $3$ and $4$ in $T$ with $7$ and $8$, respectively,
	then $T$ becomes a $\widehat{\mathfrak{gl}}_8(\mathbb{C})$-dominant tableau.
	Hence, applying {\cite[Proposition 56]{ST}} with $i=1$ and $l = 3$,
	we deduce that $x \leq y$.
	\begin{figure}[th]
	\begin{center}
		\begin{tikzpicture}[x=5mm,y=5mm]
			\draw (0,0)--(3,0)--(3,1)--(9,1)--(9,2)--(15,2)--(15,3)--(20,3)--(20,4)--(0,4)--cycle;
			\draw (3,1)--(3,4);
			\draw (9,2)--(9,4);
			\draw (15,3)--(15,4);
			\draw (0,3)--(15,3);
			\draw (0,2)--(9,2);
			\draw (0,1)--(3,1);
			\draw (6,1)--(6,4);
			\draw (12,2)--(12,4);
			
			\draw[line width=1pt,dotted] (1.2,0.5)--(1.8,0.5);
			\draw[line width=1pt,dotted] (1.2,1.5)--(1.8,1.5);
			\draw[line width=1pt,dotted] (4.2,1.5)--(4.8,1.5);
			\draw[line width=1pt,dotted] (7.2,1.5)--(7.8,1.5);
			\draw[line width=1pt,dotted] (1.2,2.5)--(1.8,2.5);
			\draw[line width=1pt,dotted] (4.2,2.5)--(4.8,2.5);
			\draw[line width=1pt,dotted] (7.2,2.5)--(7.8,2.5);
			\draw[line width=1pt,dotted] (10.2,2.5)--(10.8,2.5);
			\draw[line width=1pt,dotted] (13.2,2.5)--(13.8,2.5);
			\draw[line width=1pt,dotted] (1.2,3.5)--(1.8,3.5);
			\draw[line width=1pt,dotted] (4.2,3.5)--(4.8,3.5);
			\draw[line width=1pt,dotted] (7.2,3.5)--(7.8,3.5);
			\draw[line width=1pt,dotted] (10.2,3.5)--(10.8,3.5);
			\draw[line width=1pt,dotted] (13.2,3.5)--(13.8,3.5);
			\draw[line width=1pt,dotted] (16.2,3.5)--(16.8,3.5);
			\draw[line width=1pt,dotted] (18.2,3.5)--(18.8,3.5);
			\node (1,4) at (0.5,0.5) {4};
			\node (3,4) at (2.5,0.5) {4};
			\node (1,3) at (0.5,1.5) {3};
			\node (3,3) at (2.5,1.5) {3};
			\node (4,3) at (3.5,1.5) {3};
			\node (6,3) at (5.5,1.5) {3};
			\node (7,3) at (6.5,1.5) {4};
			\node (9,3) at (8.5,1.5) {4};
			\node (1,2) at (0.5,2.5) {2};
			\node (3,2) at (2.5,2.5) {2};
			\node (4,2) at (3.5,2.5) {2};
			\node (6,2) at (5.5,2.5) {2};
			\node (7,2) at (6.5,2.5) {2};
			\node (9,2) at (8.5,2.5) {2};
			\node (10,2) at (9.5,2.5) {2};
			\node (12,2) at (11.5,2.5) {2};
			\node (13,2) at (12.5,2.5) {4};
			\node (15,2) at (14.5,2.5) {4};
			\node (1,1) at (0.5,3.5) {1};
			\node (3,1) at (2.5,3.5) {1};
			\node (4,1) at (3.5,3.5) {1};
			\node (6,1) at (5.5,3.5) {1};
			\node (7,1) at (6.5,3.5) {1};
			\node (9,1) at (8.5,3.5) {1};
			\node (10,1) at (9.5,3.5) {1};
			\node (12,1) at (11.5,3.5) {1};
			\node (13,1) at (12.5,3.5) {1};
			\node (15,1) at (14.5,3.5) {1};
			\node (16,1) at (15.5,3.5) {1};
			\node (20,1) at (19.5,3.5) {1};
		\end{tikzpicture}
	\end{center}
	\caption{}\label{vwxyz}
  \end{figure}

	We prove the reverse implication (2) $\implies$(1).
	Assume that $T$ is the tableau shown in Figure \ref{10000}.
	It suffices to show that if we replace all entries $3$ and $4$
	in the tableau with $7$ and $8$, respectively, then $T$ becomes a
	$\widehat{\mathfrak{gl}}_8(\mathbb{C})$-dominant tableau.
	Since $T$ has the shape depicted in Figure \ref{10000}, 
	it has the \textit{cancellation property} in the sense of {\cite[Definition 50]{ST}}, 
	and satisfies the assumption of {\cite[Lemma 49]{ST}}. 
	Also, using $x \leq y$,
	we can verify that the inequalities in 
	{\cite[Proposition 56]{ST}} hold; in particular, this is used to verify the inequalities for $l = 3$.
	Therefore, by {\cite[Proposition 56]{ST}}, 
	the tableau obtained from $T$ by the replacement above of entries is
	$\widehat{\mathfrak{gl}}_8(\mathbb{C})$-dominant.
	This proves the lemma. 
\end{proof}

\begin{lem}\label{characterization of k-highest}
	For 
	$S \in \mathsf{SST}_{4}(\lambda)$,
	the following three conditions are equivalent:
	\begin{enumerate}
		\item[(1)] $S$ is a $\mathfrak{k}$-highest weight tableau.
		\item[(2)] $S$ satisfies the following conditions:
		      \begin{multline}
			      \left\{
			      \begin{aligned}
				       & \widetilde{f}_1S = \widetilde{e}_2S = \widetilde{e}_3S = 0,                            \\
				       & \varphi_2(\widetilde{e}_1^{\mathrm{max}}\widetilde{f}_3^{\mathrm{max}}S) \leq \varphi_2(S).
			      \end{aligned}
						\right.\label{400}
		      \end{multline}
		\item[(3)] $S$ is a tableau of one of the following forms, with $x \leq w,\,\,y \leq z$.
	\end{enumerate}
\end{lem}
\begin{center}
	\begin{tikzpicture}[x=5mm,y=5mm]
		\draw (0,0)--(3,0)--(3,1)--(9,1)--(9,2)--(16,2)--(16,3)--(20,3)--(20,4)--(0,4)--cycle;
		\draw (3,1)--(3,4);
		\draw (9,2)--(9,4);
		\draw (16,3)--(16,4);
		\draw (0,3)--(16,3);
		\draw (0,2)--(9,2);
		\draw (0,1)--(3,1);
		\draw (6,1)--(6,4);
		\draw (12,2)--(12,4);
		
		\draw[line width=1pt,dotted] (1.2,0.5)--(1.8,0.5);
		\draw[line width=1pt,dotted] (1.2,1.5)--(1.8,1.5);
		\draw[line width=1pt,dotted] (1.2,2.5)--(1.8,2.5);
		\draw[line width=1pt,dotted] (1.2,3.5)--(1.8,3.5);
		\draw[line width=1pt,dotted] (4.2,1.5)--(4.8,1.5);
		\draw[line width=1pt,dotted] (4.2,2.5)--(4.8,2.5);
		\draw[line width=1pt,dotted] (4.2,3.5)--(4.8,3.5);
		\draw[line width=1pt,dotted] (7.2,1.5)--(7.8,1.5);
		\draw[line width=1pt,dotted] (7.2,2.5)--(7.8,2.5);
		\draw[line width=1pt,dotted] (7.2,3.5)--(7.8,3.5);
		\draw[line width=1pt,dotted] (10.2,2.5)--(10.8,2.5);
		\draw[line width=1pt,dotted] (10.2,3.5)--(10.8,3.5);
		\draw[line width=1pt,dotted] (13.2,2.5)--(13.8,2.5);
		\draw[line width=1pt,dotted] (14.2,2.5)--(14.8,2.5);
		\draw[line width=1pt,dotted] (13.2,3.5)--(13.8,3.5);
		\draw[line width=1pt,dotted] (14.3,3.5)--(14.8,3.5);
		\draw[line width=1pt,dotted] (17.2,3.5)--(17.8,3.5);
		\draw[line width=1pt,dotted] (18.2,3.5)--(18.8,3.5);
		
		\node (1,4) at (0.5,0.5) {4};
		\node (3,4) at (2.5,0.5) {4};
		\node (1,3) at (0.5,1.5) {3};
		\node (3,3) at (2.5,1.5) {3};
		\node (4,3) at (3.5,1.5) {3};
		\node (6,3) at (5.5,1.5) {3};
		\node (7,3) at (6.5,1.5) {4};
		\node (9,3) at (8.5,1.5) {4};
		\node (1,2) at (0.5,2.5) {2};
		\node (3,2) at (2.5,2.5) {2};
		\node (4,2) at (3.5,2.5) {2};
		\node (6,2) at (5.5,2.5) {2};
		\node (7,2) at (6.5,2.5) {2};
		\node (9,2) at (8.5,2.5) {2};
		\node (10,2) at (9.5,2.5) {2};
		\node (12,2) at (11.5,2.5) {2};
		\node (13,2) at (12.5,2.5) {3};
		\node (16,2) at (15.5,2.5) {3};
		\node (1,1) at (0.5,3.5) {1};
		\node (3,1) at (2.5,3.5) {1};
		\node (4,1) at (3.5,3.5) {1};
		\node (6,1) at (5.5,3.5) {1};
		\node (7,1) at (6.5,3.5) {1};
		\node (9,1) at (8.5,3.5) {1};
		\node (10,1) at (9.5,3.5) {1};
		\node (12,1) at (11.5,3.5) {1};
		\node (13,1) at (12.5,3.5) {2};
		\node (16,1) at (15.5,3.5) {2};
		\node (17,1) at (16.5,3.5) {2};
		\node (20,1) at (19.5,3.5) {2};
		
		\draw decorate[decoration=brace]{ (5.9,0.7)--(3.2,0.7) };
		\draw decorate[decoration=brace]{ (9,0.7)--(6.1,0.7) };
		\draw decorate[decoration=brace]{ (20,2.7)--(16.2,2.7) };
		\draw decorate[decoration=brace]{ (16,1.7)--(12.2,1.7) };
		
		\node (long) at (18.2,2) {$w$};
		\node (long') at (14.1,0.8) {$z$};
		\node (short) at (4.5,-0.1) {$x$};
		\node (short') at (7.5,-0.1) {$y$};
	\end{tikzpicture}
\end{center}
\begin{figure}[th]
\begin{center}
	\begin{tikzpicture}[x=5mm,y=5mm]
		\draw (0,0)--(3,0)--(3,1)--(12,1)--(12,2)--(16,2)--(16,3)--(20,3)--(20,4)--(0,4)--cycle;
		\draw (3,1)--(3,4);
		\draw (9,1)--(9,4);
		\draw (16,3)--(16,4);
		\draw (0,3)--(16,3);
		\draw (0,2)--(12,2);
		\draw (0,1)--(3,1);
		\draw (6,1)--(6,4);
		\draw (12,2)--(12,4);
		
		\draw[line width=1pt,dotted] (1.2,0.5)--(1.8,0.5);
		\draw[line width=1pt,dotted] (1.2,1.5)--(1.8,1.5);
		\draw[line width=1pt,dotted] (1.2,2.5)--(1.8,2.5);
		\draw[line width=1pt,dotted] (1.2,3.5)--(1.8,3.5);
		\draw[line width=1pt,dotted] (4.2,1.5)--(4.8,1.5);
		\draw[line width=1pt,dotted] (4.2,2.5)--(4.8,2.5);
		\draw[line width=1pt,dotted] (4.2,3.5)--(4.8,3.5);
		\draw[line width=1pt,dotted] (7.2,1.5)--(7.8,1.5);
		\draw[line width=1pt,dotted] (7.2,2.5)--(7.8,2.5);
		\draw[line width=1pt,dotted] (7.2,3.5)--(7.8,3.5);
		\draw[line width=1pt,dotted] (10.2,1.5)--(10.8,1.5);
		\draw[line width=1pt,dotted] (10.2,2.5)--(10.8,2.5);
		\draw[line width=1pt,dotted] (10.2,3.5)--(10.8,3.5);
		\draw[line width=1pt,dotted] (13.2,2.5)--(13.8,2.5);
		\draw[line width=1pt,dotted] (14.2,2.5)--(14.8,2.5);
		\draw[line width=1pt,dotted] (13.2,3.5)--(13.8,3.5);
		\draw[line width=1pt,dotted] (14.3,3.5)--(14.8,3.5);
		\draw[line width=1pt,dotted] (17.2,3.5)--(17.8,3.5);
		\draw[line width=1pt,dotted] (18.2,3.5)--(18.8,3.5);
		
		\node (1,4) at (0.5,0.5) {4};
		\node (3,4) at (2.5,0.5) {4};
		\node (1,3) at (0.5,1.5) {3};
		\node (3,3) at (2.5,1.5) {3};
		\node (4,3) at (3.5,1.5) {3};
		\node (6,3) at (5.5,1.5) {3};
		\node (7,3) at (6.5,1.5) {4};
		\node (9,3) at (8.5,1.5) {4};
		\node (10,3) at (9.5,1.5) {4};
		\node (12,3) at (11.5,1.5) {4};
		\node (1,2) at (0.5,2.5) {2};
		\node (3,2) at (2.5,2.5) {2};
		\node (4,2) at (3.5,2.5) {2};
		\node (6,2) at (5.5,2.5) {2};
		\node (7,2) at (6.5,2.5) {2};
		\node (9,2) at (8.5,2.5) {2};
		\node (10,2) at (9.5,2.5) {3};
		\node (12,2) at (11.5,2.5) {3};
		\node (13,2) at (12.5,2.5) {3};
		\node (16,2) at (15.5,2.5) {3};
		\node (1,1) at (0.5,3.5) {1};
		\node (3,1) at (2.5,3.5) {1};
		\node (4,1) at (3.5,3.5) {1};
		\node (6,1) at (5.5,3.5) {1};
		\node (7,1) at (6.5,3.5) {1};
		\node (9,1) at (8.5,3.5) {1};
		\node (10,1) at (9.5,3.5) {2};
		\node (12,1) at (11.5,3.5) {2};
		\node (13,1) at (12.5,3.5) {2};
		\node (16,1) at (15.5,3.5) {2};
		\node (17,1) at (16.5,3.5) {2};
		\node (20,1) at (19.5,3.5) {2};
		
		\draw decorate[decoration=brace]{ (5.9,0.7)--(3.2,0.7) };
		\draw decorate[decoration=brace]{ (9,0.7)--(6.1,0.7) };
		\draw decorate[decoration=brace]{ (20,2.7)--(16.2,2.7) };
		\draw decorate[decoration=brace]{ (16,1.7)--(12.2,1.7) };
		
		\node (long) at (18.2,2) {$w$};
		\node (long') at (14.1,0.8) {$z$};
		\node (short) at (4.5,-0.1) {$x$};
		\node (short') at (7.5,-0.1) {$y$};
	\end{tikzpicture}
\end{center}
\caption{A $\mathfrak{k}$-highest weight tableau}\label{30000}
\end{figure}

\begin{proof}
	We prove the implication (1) $\implies$ (2). 
	Assume that $S$ is a $\mathfrak{k}$-highest weight tableau.
	We will show that $S$ satisfies condition (\ref{400}); 
	we proceed by induction on $|\lambda|$.
        If $|\lambda|$ is zero, then condition 
	(\ref{400}) is obviously satisfied.

	Let us consider the case $|\lambda|\geq1$.
	If $S$ is a symplectic tableau, then 
	$S$ is the tableau shown in Figure \ref{20} below since 
	$S=\mathsf{P}^{A\mathrm{II}}(S)$ by Proposition \ref{2024},
	and hence $S$ satisfies condition (\ref{400}).
	\begin{figure}[th]
		\begin{center}
			\begin{tikzpicture}[x=5mm,y=5mm]
				\draw (0,0)--(8,0)--(8,1)--(0,1)--cycle;
				\draw (0,1)--(0,2)--(10,2)--(10,1)--(8,1);
				\node at (0.5,0.5) {$3$};
				\node at (0.5,1.5) {$2$};
				\node at (7.5,0.5) {$3$};
				\node at (9.5,1.5) {$2$};
				\draw[dotted, line width=1pt] (1.2,0.5)--(1.8,0.5);
				\draw[dotted, line width=1pt] (6.2,0.5)--(6.8,0.5);
				\draw[dotted, line width=1pt] (1.2,1.5)--(1.8,1.5);
				\draw[dotted, line width=1pt] (8.2,1.5)--(8.8,1.5);
				\node at (-1,1) {$S=$};
			\end{tikzpicture}
		\end{center}
		\caption{}\label{20}
	\end{figure}

	Now,
	assume that $S$ is not symplectic.
	In this case, we have 
	$\mathsf{suc}(S) \subsetneqq S$
	by Proposition \ref{2024}.
	Therefore, we can apply the induction hypothesis to $\mathsf{suc}(S)$
	to deduce the following: 
	\begin{multline}
		\begin{aligned}
			 & \widetilde{f}_1\,\mathsf{suc}(S) = \widetilde{e}_2\,\mathsf{suc}(S) = \widetilde{e}_3\,\mathsf{suc}(S) = 0, &
			 & \varphi_2(\widetilde{e}_1^{\text{max}}\widetilde{f}_3^{\text{max}}\,\mathsf{suc}(S)) \leq \varphi_2(\mathsf{suc}(S)).
		\end{aligned}
	\end{multline}
	Let $S_1$ denote the first column of $S$, 
	and $S_{\geq2}$ the rest of $S$; 
	in this case, we have $S=S_1*S_{\geq2}$.
	Since $S$ is not symplectic by the assumption, 
        we have the six possibilities for $S_1$: cases (i)--(vi) 
        in Table \ref{2025} below. 
	We will show that condition (\ref{400}) is satisfied 
	in each of cases (i)--(vi). 
	\begin{table}[ht]
		\caption{} \label{2025}
	\begin{center}
		\begin{tabular}{l|rrrrrr} \hline
			\,\,\,\,\,No. & (i)                                                                                                                                                                                                                                                                                                                                         & (ii) & (iii) & (iv) & (v) & (vi) \\ \hline\hline
			
			\,\,\,\,\,\begin{tikzpicture}[x=5mm,y=5mm] \node (0,4.3) at (0,4.3) {}; \node (0,-0.3) at (0,-0.3) {}; \node (0,2) at (0,2) {$S_1$}; \end{tikzpicture}
			              & \begin{tikzpicture}[x=5mm,y=5mm] \node (0,4.3) at (0,4.3) {}; \node (0,-0.3) at (0,-0.3) {}; \draw (0,0)--(1,0)--(1,4)--(0,4)--(0,0); \draw (0,1)--(1,1); \draw (0,2)--(1,2); \draw (0,3)--(1,3); \node (1,1) at (0.5,0.5) {4}; \node (1,2) at (0.5,1.5) {3}; \node (1,3) at (0.5,2.5) {2}; \node (1,4) at (0.5,3.5) {1}; \end{tikzpicture}
			              & \begin{tikzpicture}[x=5mm,y=5mm] \node (0,4.3) at (0,4.3) {}; \node (0,-0.3) at (0,-0.3) {}; \draw (0,0.5)--(1,0.5)--(1,3.5)--(0,3.5)--(0,0.5); \draw (0,1.5)--(1,1.5); \draw (0,2.5)--(1,2.5); \node (1,1) at (0.5,1) {3}; \node (1,2) at (0.5,2) {2}; \node (1,3) at (0.5,3) {1}; \end{tikzpicture}
			              & \begin{tikzpicture}[x=5mm,y=5mm] \node (0,4.3) at (0,4.3) {}; \node (0,-0.3) at (0,-0.3) {}; \draw (0,0.5)--(1,0.5)--(1,3.5)--(0,3.5)--(0,0.5); \draw (0,1.5)--(1,1.5); \draw (0,2.5)--(1,2.5); \node (1,1) at (0.5,1) {4}; \node (1,2) at (0.5,2) {2}; \node (1,3) at (0.5,3) {1}; \end{tikzpicture}
			              & \begin{tikzpicture}[x=5mm,y=5mm] \node (0,4.3) at (0,4.3) {}; \node (0,-0.3) at (0,-0.3) {}; \draw (0,0.5)--(1,0.5)--(1,3.5)--(0,3.5)--(0,0.5); \draw (0,1.5)--(1,1.5); \draw (0,2.5)--(1,2.5); \node (1,1) at (0.5,1) {4}; \node (1,2) at (0.5,2) {3}; \node (1,3) at (0.5,3) {1}; \end{tikzpicture}
			              & \begin{tikzpicture}[x=5mm,y=5mm] \node (0,4.3) at (0,4.3) {}; \node (0,-0.3) at (0,-0.3) {}; \draw (0,0.5)--(1,0.5)--(1,3.5)--(0,3.5)--(0,0.5); \draw (0,1.5)--(1,1.5); \draw (0,2.5)--(1,2.5); \node (1,1) at (0.5,1) {4}; \node (1,2) at (0.5,2) {3}; \node (1,3) at (0.5,3) {2}; \end{tikzpicture}
			              & \begin{tikzpicture}[x=5mm,y=5mm] \node (0,4.3) at (0,4.3) {}; \node (0,-0.3) at (0,-0.3) {}; \draw (0,1)--(1,1)--(1,3)--(0,3)--(0,1); \draw (0,2)--(1,2); \node (1,1) at (0.5,1.5) {2}; \node (1,2) at (0.5,2.5) {1}; \end{tikzpicture}                                                                                                                                        \\
			
			\begin{tikzpicture}[x=5mm,y=5mm] \node (0,-0.3) at (0,-0.3) {}; \node (0,0.5) at (0,0.5) {$\mathsf{red}(S_1)$}; \end{tikzpicture}
			              & \begin{tikzpicture}[x=5mm,y=5mm] \node (0,-0.3) at (0,-0.3) {}; \node (1,1) at (0.5,0.5) {$\emptyset$}; \end{tikzpicture}
			              & \begin{tikzpicture}[x=5mm,y=5mm] \node (0,-0.3) at (0,-0.3) {}; \draw (0,0)--(1,0)--(1,1)--(0,1)--(0,0); \node (1,1) at (0.5,0.5) {3}; \end{tikzpicture}
			              & \begin{tikzpicture}[x=5mm,y=5mm] \node (0,-0.3) at (0,-0.3) {}; \draw (0,0)--(1,0)--(1,1)--(0,1)--(0,0); \node (1,1) at (0.5,0.5) {4}; \end{tikzpicture}
			              & \begin{tikzpicture}[x=5mm,y=5mm] \node (0,-0.3) at (0,-0.3) {}; \draw (0,0)--(1,0)--(1,1)--(0,1)--(0,0); \node (1,1) at (0.5,0.5) {1}; \end{tikzpicture}
			              & \begin{tikzpicture}[x=5mm,y=5mm] \node (0,-0.3) at (0,-0.3) {}; \draw (0,0)--(1,0)--(1,1)--(0,1)--(0,0); \node (1,1) at (0.5,0.5) {2}; \end{tikzpicture}
			              & \begin{tikzpicture}[x=5mm,y=5mm] \node (0,-0.3) at (0,-0.3) {}; \node (1,1) at (0.5,0.5) {$\emptyset$}; \end{tikzpicture}                                                                                                                                                                                                                                                    \\ \hline
		\end{tabular}
	\end{center}
  \end{table} \\
	\textbf{Case (i).}
	In this case, we have 
	$\mathsf{suc}(S) = S_{\geq2}$ and 
	$\widetilde{f}_1S_1 = \widetilde{e}_2S_1 = \widetilde{e}_3S_1 = 0$.
	By the tensor product rule (\ref{4.13})
	and Proposition \ref{Pieri}, we see that 
	\begin{equation}
		\left\{
		\begin{aligned}
			& \varphi_1\!\left(S\right) = \varphi_1\!\left(S_1*S_{\geq2}\right) = \varphi_1\!\left(S_1 \otimes S_{\geq2}\right) = 0, \\
			& \varepsilon_2\!\left(S\right) = \varepsilon_2\!\left(S_1*S_{\geq2}\right) = \varepsilon_2\!\left(S_1 \otimes S_{\geq2}\right) = 0, \\
			& \varepsilon_3\!\left(S\right) = \varepsilon_3\!\left(S_1*S_{\geq2}\right) = \varepsilon_3\!\left(S_1 \otimes S_{\geq2}\right) = 0,
		\end{aligned}
		\right.
	\end{equation}
	where we have also used the equalities 
	$\widetilde{f}_1S_{\geq2} = \widetilde{e}_2S_{\geq2} = \widetilde{e}_3S_{\geq2} = 0$ (which follow by the induction hypothesis) and
	$\widetilde{f}_1S_1 = \widetilde{e}_2S_1 = \widetilde{e}_3S_1 = 0$.
	Hence, we have $\widetilde{f}_1 S=\widetilde{e}_2 S =\widetilde{e}_3 S =0$.
	Also, 
	since 
	$\widetilde{e}_1^{\text{max}}\widetilde{f}_3^{\text{max}}\left(S_1 \otimes S_{\geq2}\right) = S_1 \otimes \left(\widetilde{e}_1^{\text{max}}\widetilde{f}_3^{\text{max}} S_{\geq2}\right)$
	by the tensor product rule (\ref{4.13}),
	we deduce that 
	\begin{equation}
		\begin{aligned}
			&\varphi_2\!\left(S\right) = \varphi_2\!\left(S_1*S_{\geq2}\right) = \varphi_2\!\left(S_1\otimes S_{\geq2}\right) = \varphi_2(S_{\geq2}), \\
			&
			\begin{aligned}
				\varphi_2\!\left(\widetilde{e}_1^{\text{max}}\widetilde{f}_3^{\text{max}}S\right)
			&= \varphi_2\!\left(\widetilde{e}_1^{\text{max}}\widetilde{f}_3^{\text{max}}\left(S_1 * S_{\geq2}\right)\right) \\
			&= \varphi_2\!\left(\widetilde{e}_1^{\text{max}}\widetilde{f}_3^{\text{max}}\left(S_1 \otimes S_{\geq2}\right)\right) \\
			&= \varphi_2\!\left(S_1 \otimes \left(\widetilde{e}_1^{\text{max}}\widetilde{f}_3^{\text{max}} S_{\geq2}\right)\right) \\
			&= \varphi_2\!\left(\widetilde{e}_1^{\text{max}}\widetilde{f}_3^{\text{max}} S_{\geq2}\right),
			\end{aligned}
		\end{aligned}
	\end{equation}
	where the last equality follows by the tensor product rule (\ref{4.13}) since $\widetilde{f}_2 S_1 = 0$.
	Therefore, using the induction hypothesis, we obtain the inequality $\varphi_2\!\left(\widetilde{e}_1^{\text{max}}\widetilde{f}_3^{\text{max}}S \right) \leq \varphi_2\!\left(S\right)$.\\
	\textbf{Case (ii).}
	In this case, we have $\mathsf{suc}(S) = \boxed{3} * S_{\geq2}$
	and $\widetilde{f}_1S_1 = \widetilde{e}_2S_1 = \widetilde{e}_3S_1 = 0$.
	Therefore, by the tensor product rule (\ref{4.13}) and
	Proposition \ref{Pieri}, we see that 
	\begin{equation}
		\left\{
		\begin{aligned}
			& \varphi_1\!\left(S\right) = \varphi_1\!\left(S_1*S_{\geq2}\right) = \varphi_1\!\left(S_1\otimes S_{\geq2}\right) = \varphi_1\!\left(\boxed{3}\otimes S_{\geq2}\right) = \varphi_1(\mathsf{suc}(S)) =0, \\
			& \varepsilon_2\!\left(S\right) = \varepsilon_2\!\left(S_1*S_{\geq2}\right) = \varepsilon_2\!\left(S_1\otimes S_{\geq2}\right) \leq \varepsilon_2\!\left(\boxed{3} \otimes S_{\geq2}\right) = \varepsilon_2\left(\mathsf{suc}(S)\right) = 0, \\
			& \varepsilon_3\!\left(S\right) = \varepsilon_3\!\left(S_1*S_{\geq2}\right) = \varepsilon_3\!\left(S_1\otimes S_{\geq2}\right) = \varepsilon_3\!\left(\boxed{3}\otimes S_{\geq2}\right) = \varepsilon_3(\mathsf{suc}(S)) =0,
		\end{aligned}
		\right.
	\end{equation}
	where we have also used
	the equalities
	$\widetilde{f}_1S_1 = \widetilde{e}_2S_1 = \widetilde{e}_3S_1 = 0$,
	$\widetilde{f}_1\,\boxed{3} = \widetilde{e}_3\,\boxed{3} = 0$, 
	$\widetilde{e}_2\,\boxed{3} = \boxed{2}$.
	Hence, we have $\widetilde{f}_1 S =\widetilde{e}_2 S =\widetilde{e}_3 S =0$.
	Also, by Proposition \ref{Pieri}, we have
	\begin{equation}
	 \begin{aligned}
		 &
		 \begin{aligned}
		  \varphi_2\!\left(\widetilde{e}_1^{\text{max}}\widetilde{f}_3^{\text{max}}\left(S_1 \otimes S_{\geq2}\right)\right)
			 &= \varphi_2\!\left(\widetilde{e}_1^{\text{max}}\left(\boxed{4} \otimes \boxed{2} \otimes \boxed{1} \otimes \left(\widetilde{f}_3^{\text{max}}S_{\geq2}\right)\right)\right) \\
			 &= \varphi_2\!\left(\boxed{4} \otimes \boxed{2} \otimes \boxed{1} \otimes \left(\widetilde{e}_1^{\text{max}}\widetilde{f}_3^{\text{max}}S_{\geq2}\right)\right) \\
			 &= \varphi_2\!\left(\widetilde{e}_1^{\text{max}}\widetilde{f}_3^{\text{max}}S_{\geq2}\right) + 1,
		 \end{aligned} \\
		 & 
		 \begin{aligned}
			\varphi_2\!\left(\widetilde{e}_1^{\text{max}}\widetilde{f}_3^{\text{max}}\left(\boxed{3} \otimes S_{\geq2}\right)\right)
		  &= \varphi_2\!\left(\widetilde{e}_1^{\text{max}}\left(\boxed{4} \otimes S_{\geq2}\right)\right) \\
		  &= \varphi_2\!\left(\boxed{4} \otimes \left(\widetilde{e}_1^{\text{max}}\widetilde{f}_3^{\text{max}}S_{\geq2}\right)\right) \\
			&= \varphi_2\!\left(\widetilde{e}_1^{\text{max}}\widetilde{f}_3^{\text{max}}S_{\geq2}\right),
		 \end{aligned}
	 \end{aligned}\label{6.6}
	\end{equation}
	where we have used 
	the tensor product rule (\ref{4.13})
	and the equalities
	$\widetilde{e}_1\left(\boxed{4}\otimes\boxed{2}\otimes\boxed{1}\right)=\widetilde{e}_1\,\boxed{4}=0$
	for both of the second equalities.
	Therefore, from equations (\ref{6.6}),
	we deduce that
	\begin{equation}
		\begin{aligned}
			\varphi_2\!\left(\widetilde{e}_1^{\text{max}}\widetilde{f}_3^{\text{max}}S\right)
			&= \varphi_2\!\left(\widetilde{e}_1^{\text{max}}\widetilde{f}_3^{\text{max}}\left(S_1*S_{\geq2}\right)\right) \\
			&= \varphi_2\!\left(\widetilde{e}_1^{\text{max}}\widetilde{f}_3^{\text{max}}\left(S_1 \otimes S_{\geq2}\right)\right) \\
			&= \varphi_2\!\left(\widetilde{e}_1^{\text{max}}\widetilde{f}_3^{\text{max}}S_{\geq2}\right) + 1 \\
			&= \varphi_2\!\left(\widetilde{e}_1^{\text{max}}\widetilde{f}_3^{\text{max}}\left(\boxed{3} \otimes S_{\geq2}\right)\right) + 1 \\
			&\leq \varphi_2\!\left(\boxed{3} \otimes S_{\geq2}\right) + 1
			= \varphi_2\!\left(S_{\geq2}\right)
			= \varphi_2\!\left(S\right),
		\end{aligned}\label{6.7}
	\end{equation}
	where the inequality in (\ref{6.7}) follows from the equalities $\widetilde{e}_2\,\boxed{3}=\boxed{2}$, $\widetilde{e}_2\,\boxed{4}=0$.
	Therefore, we obtain the inequality $\varphi_2\!\left(\widetilde{e}_1^{\text{max}}\widetilde{f}_3^{\text{max}}S\right) \leq \varphi_2\!\left(S\right)$. The proofs in cases (iii), (iv), (v), and (vi) are similar.\\

  We prove the implication (2) $\implies$ (3).
	Assume that $S$ satisfies condition (\ref{400}).
	Let $s_{ij}$ denote the number of appearances of $\boxed{j}$ in the $i$th row of $S$.
	Since $\widetilde{f}_1S=0$, we have $s_{2,2} = s_{1,1}$.
	Similarly, we have $s_{1,3}=0$, $s_{1,2} \geq s_{2,3}$, and
	$s_{1,2}+s_{2,2} \geq s_{2,3}+s_{3,3}$ since $\widetilde{e}_2S=0$.
	In addition, we have $s_{1,4}=s_{2,4}=0$ and $a_{2,3} \geq s_{3,4}$ since $\widetilde{e}_3S=0$.
	Therefore, $S$ is a tableau of the form (a) or (b)
	in Figure \ref{60000}.\\
	\begin{figure}[th]
		\begin{center}
			\begin{tikzpicture}[x=5mm,y=5mm]
				\draw (0,0)--(3,0)--(3,1)--(9,1)--(9,2)--(16,2)--(16,3)--(20,3)--(20,4)--(0,4)--cycle;
				\draw (3,1)--(3,4);
				\draw (9,2)--(9,4);
				\draw (16,3)--(16,4);
				\draw (0,3)--(16,3);
				\draw (0,2)--(9,2);
				\draw (0,1)--(3,1);
				\draw (6,1)--(6,4);
				\draw (12,2)--(12,4);
				
				\draw[line width=1pt,dotted] (1.2,0.5)--(1.8,0.5);
				\draw[line width=1pt,dotted] (1.2,1.5)--(1.8,1.5);
				\draw[line width=1pt,dotted] (1.2,2.5)--(1.8,2.5);
				\draw[line width=1pt,dotted] (1.2,3.5)--(1.8,3.5);
				\draw[line width=1pt,dotted] (4.2,1.5)--(4.8,1.5);
				\draw[line width=1pt,dotted] (4.2,2.5)--(4.8,2.5);
				\draw[line width=1pt,dotted] (4.2,3.5)--(4.8,3.5);
				\draw[line width=1pt,dotted] (7.2,1.5)--(7.8,1.5);
				\draw[line width=1pt,dotted] (7.2,2.5)--(7.8,2.5);
				\draw[line width=1pt,dotted] (7.2,3.5)--(7.8,3.5);
				\draw[line width=1pt,dotted] (10.2,2.5)--(10.8,2.5);
				\draw[line width=1pt,dotted] (10.2,3.5)--(10.8,3.5);
				\draw[line width=1pt,dotted] (13.2,2.5)--(13.8,2.5);
				\draw[line width=1pt,dotted] (14.2,2.5)--(14.8,2.5);
				\draw[line width=1pt,dotted] (13.2,3.5)--(13.8,3.5);
				\draw[line width=1pt,dotted] (14.3,3.5)--(14.8,3.5);
				\draw[line width=1pt,dotted] (17.2,3.5)--(17.8,3.5);
				\draw[line width=1pt,dotted] (18.2,3.5)--(18.8,3.5);
				
				\node (1,4) at (0.5,0.5) {4};
				\node (3,4) at (2.5,0.5) {4};
				\node (1,3) at (0.5,1.5) {3};
				\node (3,3) at (2.5,1.5) {3};
				\node (4,3) at (3.5,1.5) {3};
				\node (6,3) at (5.5,1.5) {3};
				\node (7,3) at (6.5,1.5) {4};
				\node (9,3) at (8.5,1.5) {4};
				\node (1,2) at (0.5,2.5) {2};
				\node (3,2) at (2.5,2.5) {2};
				\node (4,2) at (3.5,2.5) {2};
				\node (6,2) at (5.5,2.5) {2};
				\node (7,2) at (6.5,2.5) {2};
				\node (9,2) at (8.5,2.5) {2};
				\node (10,2) at (9.5,2.5) {2};
				\node (12,2) at (11.5,2.5) {2};
				\node (13,2) at (12.5,2.5) {3};
				\node (16,2) at (15.5,2.5) {3};
				\node (1,1) at (0.5,3.5) {1};
				\node (3,1) at (2.5,3.5) {1};
				\node (4,1) at (3.5,3.5) {1};
				\node (6,1) at (5.5,3.5) {1};
				\node (7,1) at (6.5,3.5) {1};
				\node (9,1) at (8.5,3.5) {1};
				\node (10,1) at (9.5,3.5) {1};
				\node (12,1) at (11.5,3.5) {1};
				\node (13,1) at (12.5,3.5) {2};
				\node (16,1) at (15.5,3.5) {2};
				\node (17,1) at (16.5,3.5) {2};
				\node (20,1) at (19.5,3.5) {2};
				
				\draw decorate[decoration=brace]{ (9,0.7)--(6.1,0.7) };
				\draw decorate[decoration=brace]{ (16,1.7)--(12.2,1.7) };
				
				\node (long') at (14.1,0.8) {$z$};
				\node (short') at (7.5,-0.1) {$y$};
				\node at (22,2) {$(y \leq z)$};
				\node at (-2,2) {(a)};
			\end{tikzpicture}
			\begin{tikzpicture}[x=5mm,y=5mm]
				\draw (0,0)--(3,0)--(3,1)--(12,1)--(12,2)--(16,2)--(16,3)--(20,3)--(20,4)--(0,4)--cycle;
				\draw (3,1)--(3,4);
				\draw (9,1)--(9,4);
				\draw (16,3)--(16,4);
				\draw (0,3)--(16,3);
				\draw (0,2)--(12,2);
				\draw (0,1)--(3,1);
				\draw (6,1)--(6,4);
				\draw (12,2)--(12,4);
				
				\draw[line width=1pt,dotted] (1.2,0.5)--(1.8,0.5);
				\draw[line width=1pt,dotted] (1.2,1.5)--(1.8,1.5);
				\draw[line width=1pt,dotted] (1.2,2.5)--(1.8,2.5);
				\draw[line width=1pt,dotted] (1.2,3.5)--(1.8,3.5);
				\draw[line width=1pt,dotted] (4.2,1.5)--(4.8,1.5);
				\draw[line width=1pt,dotted] (4.2,2.5)--(4.8,2.5);
				\draw[line width=1pt,dotted] (4.2,3.5)--(4.8,3.5);
				\draw[line width=1pt,dotted] (7.2,1.5)--(7.8,1.5);
				\draw[line width=1pt,dotted] (7.2,2.5)--(7.8,2.5);
				\draw[line width=1pt,dotted] (7.2,3.5)--(7.8,3.5);
				\draw[line width=1pt,dotted] (10.2,1.5)--(10.8,1.5);
				\draw[line width=1pt,dotted] (10.2,2.5)--(10.8,2.5);
				\draw[line width=1pt,dotted] (10.2,3.5)--(10.8,3.5);
				\draw[line width=1pt,dotted] (13.2,2.5)--(13.8,2.5);
				\draw[line width=1pt,dotted] (14.2,2.5)--(14.8,2.5);
				\draw[line width=1pt,dotted] (13.2,3.5)--(13.8,3.5);
				\draw[line width=1pt,dotted] (14.3,3.5)--(14.8,3.5);
				\draw[line width=1pt,dotted] (17.2,3.5)--(17.8,3.5);
				\draw[line width=1pt,dotted] (18.2,3.5)--(18.8,3.5);
				
				\node (1,4) at (0.5,0.5) {4};
				\node (3,4) at (2.5,0.5) {4};
				\node (1,3) at (0.5,1.5) {3};
				\node (3,3) at (2.5,1.5) {3};
				\node (4,3) at (3.5,1.5) {3};
				\node (6,3) at (5.5,1.5) {3};
				\node (7,3) at (6.5,1.5) {4};
				\node (9,3) at (8.5,1.5) {4};
				\node (10,3) at (9.5,1.5) {4};
				\node (12,3) at (11.5,1.5) {4};
				\node (1,2) at (0.5,2.5) {2};
				\node (3,2) at (2.5,2.5) {2};
				\node (4,2) at (3.5,2.5) {2};
				\node (6,2) at (5.5,2.5) {2};
				\node (7,2) at (6.5,2.5) {2};
				\node (9,2) at (8.5,2.5) {2};
				\node (10,2) at (9.5,2.5) {3};
				\node (12,2) at (11.5,2.5) {3};
				\node (13,2) at (12.5,2.5) {3};
				\node (16,2) at (15.5,2.5) {3};
				\node (1,1) at (0.5,3.5) {1};
				\node (3,1) at (2.5,3.5) {1};
				\node (4,1) at (3.5,3.5) {1};
				\node (6,1) at (5.5,3.5) {1};
				\node (7,1) at (6.5,3.5) {1};
				\node (9,1) at (8.5,3.5) {1};
				\node (10,1) at (9.5,3.5) {2};
				\node (12,1) at (11.5,3.5) {2};
				\node (13,1) at (12.5,3.5) {2};
				\node (16,1) at (15.5,3.5) {2};
				\node (17,1) at (16.5,3.5) {2};
				\node (20,1) at (19.5,3.5) {2};
				
				\draw decorate[decoration=brace]{ (9,0.7)--(6.1,0.7) };
				\draw decorate[decoration=brace]{ (16,1.7)--(12.2,1.7) };
				
				\node (long') at (14.1,0.8) {$z$};
				\node (short') at (7.5,-0.1) {$y$};
				\node at (22,2) {$(y \leq z)$};
				\node at (-2,2) {(b)};
			\end{tikzpicture}
		\end{center}
		\caption{}\label{60000}
	\end{figure}

	In case (a),
	we deduce that
	$\widetilde{e}_1^{\text{max}}\widetilde{f}_3^{\text{max}}$
	is the tableau shown in Figure \ref{8765} below. 
	Hence, it follows that 
	$\varphi_2\!\left(\widetilde{e}_1^{\text{max}}\widetilde{f}_3^{\text{max}}S\right) = s_{2,2} - s_{4,4}$.
	Here, we see that
	$\varphi_2(S)=s_{2,2}-s_{3,3}+s_{1,2}-s_{2,3}$.
	Therefore, the inequality 
	$\varphi_2\!\left(\widetilde{e}_1^{\text{max}}\widetilde{f}_3^{\text{max}}S\right) \leq \varphi_2\!\left(S\right)$
	implies that
	$s_{3,3}-s_{4,4} \leq s_{1,2} - s_{2,3}$, 
	which means that $x \leq w$.\\
	\begin{figure}[th]
		\begin{center}
			\begin{tikzpicture}[x=5mm,y=5mm]
				\draw (0,0)--(3,0)--(3,1)--(9,1)--(9,2)--(18,2)--(18,3)--(21,3)--(21,4)--(0,4)--cycle;
				\draw (3,1)--(3,4);
				\draw (9,2)--(9,4);
				\draw (18,3)--(18,4);
				\draw (0,3)--(18,3);
				\draw (0,2)--(9,2);
				\draw (0,1)--(3,1);
				\draw (12,2)--(12,4);
				\draw (15,2)--(15,4);
				
				\draw[line width=1pt,dotted] (1.2,0.5)--(1.8,0.5);
				\draw[line width=1pt,dotted] (1.2,1.5)--(1.8,1.5);
				\draw[line width=1pt,dotted] (1.2,2.5)--(1.8,2.5);
				\draw[line width=1pt,dotted] (1.2,3.5)--(1.8,3.5);
				\draw[line width=1pt,dotted] (4.2,1.5)--(4.8,1.5);
				\draw[line width=1pt,dotted] (4.2,2.5)--(4.8,2.5);
				\draw[line width=1pt,dotted] (4.2,3.5)--(4.8,3.5);
				\draw[line width=1pt,dotted] (7.2,1.5)--(7.8,1.5);
				\draw[line width=1pt,dotted] (7.2,2.5)--(7.8,2.5);
				\draw[line width=1pt,dotted] (7.2,3.5)--(7.8,3.5);
				\draw[line width=1pt,dotted] (10.2,2.5)--(10.8,2.5);
				\draw[line width=1pt,dotted] (10.2,3.5)--(10.8,3.5);
				\draw[line width=1pt,dotted] (13.2,2.5)--(13.8,2.5);
				\draw[line width=1pt,dotted] (13.2,3.5)--(13.8,3.5);
				\draw[line width=1pt,dotted] (16.2,2.5)--(16.8,2.5);
				\draw[line width=1pt,dotted] (16.2,3.5)--(16.8,3.5);
				\draw[line width=1pt,dotted] (19.2,3.5)--(19.8,3.5);
				
				\node (1,4) at (0.5,0.5) {4};
				\node (3,4) at (2.5,0.5) {4};
				\node (1,3) at (0.5,1.5) {3};
				\node (3,3) at (2.5,1.5) {3};
				\node (4,3) at (3.5,1.5) {4};
				\node (9,3) at (8.5,1.5) {4};
				\node (1,2) at (0.5,2.5) {2};
				\node (3,2) at (2.5,2.5) {2};
				\node (4,2) at (3.5,2.5) {2};
				\node (9,2) at (8.5,2.5) {2};
				\node (10,2) at (9.5,2.5) {2};
				\node (12,2) at (11.5,2.5) {2};
				\node (13,2) at (12.5,2.5) {3};
				\node (16,2) at (14.5,2.5) {3};
				\node (1,1) at (0.5,3.5) {1};
				\node (3,1) at (2.5,3.5) {1};
				\node (4,1) at (3.5,3.5) {1};
				\node (9,1) at (8.5,3.5) {1};
				\node (10,1) at (9.5,3.5) {1};
				\node (12,1) at (11.5,3.5) {1};
				\node (13,1) at (12.5,3.5) {1};
				\node (16,1) at (14.5,3.5) {1};
				\node (17,1) at (15.5,3.5) {1};
				\node (20,1) at (17.5,3.5) {1};
				\node (17,1) at (15.5,2.5) {4};
				\node (20,1) at (17.5,2.5) {4};
				\node (17,1) at (18.5,3.5) {1};
				\node (20,1) at (20.5,3.5) {1};
				
				\draw decorate[decoration=brace]{ (9,0.7)--(3.1,0.7) };
				\draw decorate[decoration=brace]{ (12,1.7)--(9.1,1.7) };
				\draw decorate[decoration=brace]{ (3,4.3)--(12,4.3) };
				
				\node (long') at (7.5,5.1) {$s_{2,2}-s_{4,4}$};
				\node (short') at (6.5,-0.1) {$s_{3,3}-s_{4,4}+s_{3,4}$};
				\node at (12.5,0.9) {$s_{2,2}-s_{3,3}-s_{3,4}$};
				\node at (-2.8,2) {$\widetilde{e}_1^{\text{max}}\widetilde{f}_3^{\text{max}}\,S =$};
			\end{tikzpicture}
		\end{center}
		\caption{}\label{8765}
	\end{figure}
The proof in case (b) is similar. \\

	The implication (3) $\implies$ (1) is straightforward. 
This proves the lemma.
\end{proof}

Similarly,
we can show the following lemma.
\begin{lem}\label{characterization of k-lowest}
	For 
	$S \in \mathsf{SST}_{4}(\lambda)$,
	the following three conditions are equivalent:
	\begin{enumerate}
		\item[(1)] $S$ is a $\mathfrak{k}$-lowest weight tableau.
		\item[(2)] $S$ satisfies the following conditions:
		      \begin{multline}
			      \left\{
			      \begin{aligned}
				       & \widetilde{e}_1S = \widetilde{f}_2S = \widetilde{f}_3S = 0,                            \\
				       & \varphi_2(\widetilde{f}_1^{\mathrm{max}}\widetilde{e}_3^{\mathrm{max}}S) \geq \varphi_2(S).
			      \end{aligned}
						\right.\label{401}
		      \end{multline}
		\item[(3)] $S$ is a tableau of one of the following forms, with $0 \leq x-y \leq z$: 
	\end{enumerate}
\end{lem}

\begin{figure}[th]
	\begin{center}
		\begin{tikzpicture}[x=5mm,y=5mm]
			\draw (0,0)--(3,0)--(3,1)--(7,1)--(7,2)--(16,2)--(16,3)--(20,3)--(20,4)--(0,4)--cycle;
			\draw (3,1)--(3,4);
			\draw (10,2)--(10,4);
			\draw (16,3)--(16,4);
			\draw (0,3)--(16,3);
			\draw (0,2)--(7,2);
			\draw (0,1)--(3,1);
			\draw (7,2)--(7,4);
			\draw (13,2)--(13,4);
			
			\draw[line width=1pt,dotted] (1.2,0.5)--(1.8,0.5);
			\draw[line width=1pt,dotted] (1.2,1.5)--(1.8,1.5);
			\draw[line width=1pt,dotted] (1.2,2.5)--(1.8,2.5);
			\draw[line width=1pt,dotted] (1.2,3.5)--(1.8,3.5);
			\draw[line width=1pt,dotted] (4.2,1.5)--(4.8,1.5);
			\draw[line width=1pt,dotted] (4.2,2.5)--(4.8,2.5);
			\draw[line width=1pt,dotted] (4.2,3.5)--(4.8,3.5);
			\draw[line width=1pt,dotted] (5.2,1.5)--(5.8,1.5);
			\draw[line width=1pt,dotted] (5.2,2.5)--(5.8,2.5);
			\draw[line width=1pt,dotted] (5.2,3.5)--(5.8,3.5);
			\draw[line width=1pt,dotted] (8.2,2.5)--(8.8,2.5);
			\draw[line width=1pt,dotted] (8.2,3.5)--(8.8,3.5);
			\draw[line width=1pt,dotted] (11.2,2.5)--(11.8,2.5);
			\draw[line width=1pt,dotted] (11.2,3.5)--(11.8,3.5);
			\draw[line width=1pt,dotted] (14.2,2.5)--(14.8,2.5);
			\draw[line width=1pt,dotted] (14.2,3.5)--(14.8,3.5);
			\draw[line width=1pt,dotted] (17.2,3.5)--(17.8,3.5);
			\draw[line width=1pt,dotted] (18.2,3.5)--(18.8,3.5);
			
			\node (1,4) at (0.5,0.5) {4};
			\node (3,4) at (2.5,0.5) {4};
			\node (1,3) at (0.5,1.5) {3};
			\node (3,3) at (2.5,1.5) {3};
			\node (4,3) at (3.5,1.5) {4};
			\node (6,3) at (6.5,1.5) {4};
			\node (1,2) at (0.5,2.5) {2};
			\node (3,2) at (2.5,2.5) {2};
			\node (4,2) at (3.5,2.5) {2};
			\node (6,2) at (6.5,2.5) {2};
			\node (7,2) at (7.5,2.5) {2};
			\node (9,2) at (9.5,2.5) {2};
			\node (10,2) at (10.5,2.5) {3};
			\node (12,2) at (12.5,2.5) {3};
			\node (13,2) at (13.5,2.5) {4};
			\node (16,2) at (15.5,2.5) {4};
			\node (1,1) at (0.5,3.5) {1};
			\node (3,1) at (2.5,3.5) {1};
			\node (4,1) at (3.5,3.5) {1};
			\node (6,1) at (6.5,3.5) {1};
			\node (7,1) at (7.5,3.5) {1};
			\node (9,1) at (9.5,3.5) {1};
			\node (10,1) at (10.5,3.5) {1};
			\node (12,1) at (12.5,3.5) {1};
			\node (13,1) at (13.5,3.5) {1};
			\node (16,1) at (15.5,3.5) {1};
			\node (17,1) at (16.5,3.5) {1};
			\node (20,1) at (19.5,3.5) {1};
			
			\draw decorate[decoration=brace]{ (6.9,0.7)--(3.2,0.7) };
			\draw decorate[decoration=brace]{ (20,2.7)--(16.2,2.7) };
			\draw decorate[decoration=brace]{ (16,1.7)--(13.2,1.7) };
			
			\node (long) at (18.2,2) {$z$};
			\node (long') at (14.6,0.8) {$y$};
			\node (short') at (5,-0.1) {$x$};
		\end{tikzpicture}
	\end{center}
	
	\begin{center}
		\begin{tikzpicture}[x=5mm,y=5mm]
			\draw (0,0)--(3,0)--(3,1)--(10,1)--(10,2)--(16,2)--(16,3)--(20,3)--(20,4)--(0,4)--cycle;
			\draw (3,1)--(3,4);
			\draw (10,2)--(10,4);
			\draw (16,3)--(16,4);
			\draw (0,3)--(16,3);
			\draw (0,2)--(10,2);
			\draw (0,1)--(3,1);
			\draw (7,1)--(7,4);
			\draw (13,2)--(13,4);
			
			\draw[line width=1pt,dotted] (1.2,0.5)--(1.8,0.5);
			\draw[line width=1pt,dotted] (1.2,1.5)--(1.8,1.5);
			\draw[line width=1pt,dotted] (1.2,2.5)--(1.8,2.5);
			\draw[line width=1pt,dotted] (1.2,3.5)--(1.8,3.5);
			\draw[line width=1pt,dotted] (4.2,1.5)--(4.8,1.5);
			\draw[line width=1pt,dotted] (4.2,2.5)--(4.8,2.5);
			\draw[line width=1pt,dotted] (4.2,3.5)--(4.8,3.5);
			\draw[line width=1pt,dotted] (5.2,1.5)--(5.8,1.5);
			\draw[line width=1pt,dotted] (5.2,2.5)--(5.8,2.5);
			\draw[line width=1pt,dotted] (5.2,3.5)--(5.8,3.5);
			\draw[line width=1pt,dotted] (8.2,1.5)--(8.8,1.5);
			\draw[line width=1pt,dotted] (8.2,2.5)--(8.8,2.5);
			\draw[line width=1pt,dotted] (8.2,3.5)--(8.8,3.5);
			\draw[line width=1pt,dotted] (11.2,2.5)--(11.8,2.5);
			\draw[line width=1pt,dotted] (11.2,3.5)--(11.8,3.5);
			\draw[line width=1pt,dotted] (14.2,2.5)--(14.8,2.5);
			\draw[line width=1pt,dotted] (14.2,3.5)--(14.8,3.5);
			\draw[line width=1pt,dotted] (17.2,3.5)--(17.8,3.5);
			\draw[line width=1pt,dotted] (18.2,3.5)--(18.8,3.5);
			
			\node at (7.5,1.5) {4};
			\node at (9.5,1.5) {4};
			\node (1,4) at (0.5,0.5) {4};
			\node (3,4) at (2.5,0.5) {4};
			\node (1,3) at (0.5,1.5) {3};
			\node (3,3) at (2.5,1.5) {3};
			\node (4,3) at (3.5,1.5) {4};
			\node (6,3) at (6.5,1.5) {4};
			\node (1,2) at (0.5,2.5) {2};
			\node (3,2) at (2.5,2.5) {2};
			\node (4,2) at (3.5,2.5) {2};
			\node (6,2) at (6.5,2.5) {2};
			\node (7,2) at (7.5,2.5) {3};
			\node (9,2) at (9.5,2.5) {3};
			\node (10,2) at (10.5,2.5) {3};
			\node (12,2) at (12.5,2.5) {3};
			\node (13,2) at (13.5,2.5) {4};
			\node (16,2) at (15.5,2.5) {4};
			\node (1,1) at (0.5,3.5) {1};
			\node (3,1) at (2.5,3.5) {1};
			\node (4,1) at (3.5,3.5) {1};
			\node (6,1) at (6.5,3.5) {1};
			\node (7,1) at (7.5,3.5) {1};
			\node (9,1) at (9.5,3.5) {1};
			\node (10,1) at (10.5,3.5) {1};
			\node (12,1) at (12.5,3.5) {1};
			\node (13,1) at (13.5,3.5) {1};
			\node (16,1) at (15.5,3.5) {1};
			\node (17,1) at (16.5,3.5) {1};
			\node (20,1) at (19.5,3.5) {1};
			
			\draw decorate[decoration=brace]{ (6.9,0.7)--(3.2,0.7) };
			\draw decorate[decoration=brace]{ (20,2.7)--(16.2,2.7) };
			\draw decorate[decoration=brace]{ (16,1.7)--(13.2,1.7) };
			
			\node (long) at (18.2,2) {$z$};
			\node (long') at (14.6,0.8) {$y$};
			\node (short') at (5,-0.1) {$x$};
		\end{tikzpicture}
	\end{center}
	\caption{A $\mathfrak{k}$-lowest weight tableau}\label{402}
\end{figure}

\subsection{Construction of the bijections $\mathsf{\Phi}$ and $\mathsf{\Psi}$}
In order to construct the bijections $\mathsf{\Phi}$ and $\mathsf{\Psi}$,
we need some combinatorial operations.

First, we recall the 
so-called \textit{jeu-de-taquin slide}.
Let $T$ be a semistandard tableau of shape $D$.
Here, $D$ is a general diagram,
not necessarily a Young diagram. 
We fix a box $(x_0,y_0) \in D$, and
consider the following algorithm.
\begin{enumerate}
	\item[$(j1)$] If there is no box to the right of $(x_0,y_0)$
	nor below $(x_0, y_0)$, then end the algorithm.
	\item[$(j2)$] If there exist boxes both to the right of $(x_0,y_0)$
	and below $(x_0, y_0)$, and if $T(x_0+1,y_0) \leq T(x_0,y_0+1)$,
	then define a new tableau $S$ of shape $D$ by:
	$S(x_0,y_0)=T(x_0+1,y_0), S(x_0+1,y_0)=T(x_0,y_0)$
	(all other entries are the same as those of $T$); 
	rewrite $(x_0+1,y_0)$ as $(x_0,y_0)$ and $S$ as $T$,
	and return to $(j1)$.
	\item[$(j3)$] If there exist boxes both to the right of $(x_0,y_0)$
	and below $(x_0, y_0)$, and if $T(x_0+1,y_0) > T(x_0,y_0+1)$,
	then define a new tableau $S$ of shape $D$ by:
	$S(x_0,y_0)=T(x_0,y_0+1), S(x_0,y_0+1)=T(x_0,y_0)$
	(all other entries are the same as those of $T$); 
	rewrite $(x_0,y_0+1)$ as $(x_0,y_0)$ and $S$ as $T$,
	and return to $(j1)$.
	\item[$(j4)$] If there exists a box either to the right of $(x_0,y_0)$
	or below $(x_0,y_0)$ (but not both of these),
	then define a new tableau $S$ of shape $D$ by
	exchanging the entry in $(x_0,y_0)$
	with the entry in the adjacent box 
	(all other entries are the same as those of $T$);
	rewrite the adjacent box as $(x_0,y_0)$
	and $S$ as $T$, and return to $(j1)$.
\end{enumerate}
We call the tableau constructed in this way 
\textit{the tableau obtained by applying the jeu-de-taquin slide to $(x_0,y_0)$}.
By using this operation,
we define \textit{promotion operators}, as follows.
\begin{de}\label{defpr}
	Let $1 \leq a \leq b \leq 2n$ be integers,
	and $T \in \mathsf{SST}_{2n}(\lambda)$.
	We define
	a new  semistandard tableau $\mathsf{pr}_{a,b}^{-1}(T)$
	by the following algorithm: 
	\begin{enumerate}
		\item Remove all boxes with entries not contained $[a,b]$ from $T$, and denote by $S$ the remaining tableau.
		\item Decrease each entry of $S$, except for $a$, by $1$, and replace the entries $a$ by $b$.
		\item Apply the jeu-de-taquin slide to the boxes whose entry equals $b$, from right to left.
		\item Let $\mathsf{pr}_{a,b}^{-1}(T)$ be the tableau obtained by adding the boxes removed in (1) to the tableau obtained in (3).
	\end{enumerate}
\end{de}
\begin{rem}
	The operator $\mathsf{pr}_{a,b}^{-1} : \mathsf{SST}_{2n}(\lambda) \to \mathsf{SST}_{2n}(\lambda)$ is bijective
	since its inverse can be constructed
	by applying steps (1)--(4) in reverse order; 
	we write the inverse operator as $\mathsf{pr}_{a,b}$,
	and call it \textit{the promotion between $a$ and $b$}.
	The usual promotion operator $\mathrm{pr}$, 
	used in many references such as \cite{Sh}, 
	agrees with $\mathsf{pr}_{1,2n}$ in our notation.
\end{rem}
\begin{eg}
  Let $a=2$, $b=4$. Then, we have
	\vspace{10pt}
	\begin{center}
		\begin{tikzpicture}[x=5mm,y=5mm]
			\draw (0,0)--(3,0)--(3,-1)--(2,-1)--(2,-2)--(1,-2)--(1,-3)--(0,-3)--cycle;
			\draw (1,0)--(1,-2);
			\draw (2,0)--(2,-1);
			\draw (0,-1)--(2,-1);
			\draw (0,-2)--(1,-2);
			\node at (0.5,-0.5) {1};
			\node at (1.5,-0.5) {2};
			\node at (2.5,-0.5) {3};
			\node at (0.5,-1.5) {2};
			\node at (1.5,-1.5) {4};
			\node at (0.5,-2.5) {4};

			\node at (4.5,-1.5) {$\longmapsto$};
			\node at (4.5,-0.7) {\scriptsize $(1),(2)$};

			\draw (6,-1)--(7,-1)--(7,0)--(9,0)--(9,-1)--(8,-1)--(8,-2)--(7,-2)--(7,-3)--(6,-3)--cycle;
			\draw (7,-1)--(7,-2);
			\draw (8,0)--(8,-1);
			\draw (7,-1)--(8,-1);
			\draw (6,-2)--(7,-2);
			\node at (6.5,-0.5) {};
			\node at (7.5,-0.5) {4};
			\node at (8.5,-0.5) {2};
			\node at (6.5,-1.5) {4};
			\node at (7.5,-1.5) {3};
			\node at (6.5,-2.5) {3};

			\node at (10.5,-1.5) {$\longmapsto$};
			\node at (10.5,-0.7) {\scriptsize $(3)$};

			\draw (12,-1)--(13,-1)--(13,0)--(15,0)--(15,-1)--(14,-1)--(14,-2)--(13,-2)--(13,-3)--(12,-3)--cycle;
			\draw (13,-1)--(13,-2);
			\draw (14,0)--(14,-1);
			\draw (13,-1)--(14,-1);
			\draw (12,-2)--(13,-2);
			\node at (12.5,-0.5) {};
			\node at (13.5,-0.5) {2};
			\node at (14.5,-0.5) {4};
			\node at (12.5,-1.5) {3};
			\node at (13.5,-1.5) {3};
			\node at (12.5,-2.5) {4};

			\node at (16.5,-1.5) {$\longmapsto$};
			\node at (16.5,-0.7) {\scriptsize $(4)$};

			\draw (18,0)--(21,0)--(21,-1)--(20,-1)--(20,-2)--(19,-2)--(19,-3)--(18,-3)--cycle;
			\draw (19,0)--(19,-2);
			\draw (20,0)--(20,-1);
			\draw (18,-1)--(20,-1);
			\draw (18,-2)--(19,-2);
			\node at (18.5,-0.5) {1};
			\node at (19.5,-0.5) {2};
			\node at (20.5,-0.5) {4};
			\node at (18.5,-1.5) {3};
			\node at (19.5,-1.5) {3};
			\node at (18.5,-2.5) {4};

			\node at (21.3,-1.8) {$.$};
		\end{tikzpicture}
	\end{center}
	\vspace{10pt}
	Therefore,
	\begin{center}
		\begin{tikzpicture}[x=5mm,y=5mm]
			\draw (0,0)--(3,0)--(3,-1)--(2,-1)--(2,-2)--(1,-2)--(1,-3)--(0,-3)--cycle;
			\draw (1,0)--(1,-2);
			\draw (2,0)--(2,-1);
			\draw (0,-1)--(2,-1);
			\draw (0,-2)--(1,-2);
			\node at (0.5,-0.5) {1};
			\node at (1.5,-0.5) {2};
			\node at (2.5,-0.5) {4};
			\node at (0.5,-1.5) {3};
			\node at (1.5,-1.5) {3};
			\node at (0.5,-2.5) {4};

			\node at (0.65,-1.5) {$\mathsf{pr}_{2,4}\left(\begin{aligned} \\ \quad \quad \quad \quad \\ \\ \end{aligned}\right)$};

			\node at (4.5,-1.5) {$=$};

			\draw (5.5,0)--(8.5,0)--(8.5,-1)--(7.5,-1)--(7.5,-2)--(6.5,-2)--(6.5,-3)--(5.5,-3)--cycle;
			\draw (6.5,0)--(6.5,-2);
			\draw (7.5,0)--(7.5,-1);
			\draw (5.5,-1)--(7.5,-1);
			\draw (5.5,-2)--(6.5,-2);
			\node at (6,-0.5) {1};
			\node at (7,-0.5) {2};
			\node at (8,-0.5) {3};
			\node at (6,-1.5) {2};
			\node at (7,-1.5) {4};
			\node at (6,-2.5) {4};

			\node at (8.8,-1.8) {$.$};
		\end{tikzpicture}
	\end{center}
	\vspace{10pt}
\end{eg}
Now, we can state and prove the following. 
\begin{prop}\label{n=2thm}
	We set $\mathsf{\Phi}:=\mathsf{pr}_{1,4}$ and 
	$\mathsf{\Psi}:=\mathsf{pr}_{2,4}\circ\mathsf{pr}_{2,3}$.
	Then, for each $T \in \mathsf{SST}_4(\lambda)$, the equalities 
	$\mathsf{wt}_{\widehat{\mathfrak{g}}}(T) = \mathsf{wt}_{\mathfrak{k}}\left(\Phi(T)\right)$ and
	$\mathsf{wt}_{\widehat{\mathfrak{g}}}(T) = \mathsf{wt}_{\mathfrak{k}}\left(\Psi(T)\right)$ hold.
	Moreover, the following hold: 
	\begin{enumerate}
		\item $\mathsf{\Phi}\left(\mathsf{SST}_4^{\text{$\widehat{\mathfrak{g}}$-$\mathrm{dom}$}}(\lambda)\right) = \mathsf{SST}_4^{\text{$\mathfrak{k}$-$\mathrm{hw}$}}(\lambda)$,
		\item $\mathsf{\Psi}\left(\mathsf{SST}_4^{\text{$\widehat{\mathfrak{g}}$-$\mathrm{dom}$}}(\lambda)\right) = \mathsf{SST}_4^{\text{$\mathfrak{k}$-$\mathrm{lw}$}}(\lambda)$.
	\end{enumerate}
	In other words,
	the statement of Theorem \ref{KeyProposition}
	holds true for $n=2$.
\end{prop}
\begin{proof}
	We prove part (1); the proof of part (2) is similar.
	Let us take an arbitrary $T \in \mathsf{SST}_4^{\text{$\widehat{\mathfrak{g}}$-dom}}(\lambda)$.
	Then, by Lemma \ref{6.1},
	$T$ is the tableau shown in Figure \ref{10000}.
	Therefore, by Definition \ref{defpr},
	$\mathsf{\Phi}(T)$ is the tableau shown in Figure \ref{30000};
	see Figure \ref{2525625} below:
	\begin{figure}[ht]
	\begin{tikzpicture}[x=5mm,y=5mm]
		\node at (-2,2) {\textcolor{white}{$\longmapsto$}};
		\node at (-1.5,2) {$T=$};
		\draw (0,0)--(3,0)--(3,1)--(9,1)--(9,2)--(15,2)--(15,3)--(20,3)--(20,4)--(0,4)--cycle;
		\draw (3,1)--(3,4);
		\draw (9,2)--(9,4);
		\draw (15,3)--(15,4);
		\draw (0,3)--(15,3);
		\draw (0,2)--(9,2);
		\draw (0,1)--(3,1);
		\draw (6,1)--(6,4);
		\draw (12,2)--(12,4);
		
		\draw[line width=1pt,dotted] (1.2,0.5)--(1.8,0.5);
		\draw[line width=1pt,dotted] (1.2,1.5)--(1.8,1.5);
		\draw[line width=1pt,dotted] (4.2,1.5)--(4.8,1.5);
		\draw[line width=1pt,dotted] (7.2,1.5)--(7.8,1.5);
		\draw[line width=1pt,dotted] (1.2,2.5)--(1.8,2.5);
		\draw[line width=1pt,dotted] (4.2,2.5)--(4.8,2.5);
		\draw[line width=1pt,dotted] (7.2,2.5)--(7.8,2.5);
		\draw[line width=1pt,dotted] (10.2,2.5)--(10.8,2.5);
		\draw[line width=1pt,dotted] (13.2,2.5)--(13.8,2.5);
		\draw[line width=1pt,dotted] (1.2,3.5)--(1.8,3.5);
		\draw[line width=1pt,dotted] (4.2,3.5)--(4.8,3.5);
		\draw[line width=1pt,dotted] (7.2,3.5)--(7.8,3.5);
		\draw[line width=1pt,dotted] (10.2,3.5)--(10.8,3.5);
		\draw[line width=1pt,dotted] (13.2,3.5)--(13.8,3.5);
		\draw[line width=1pt,dotted] (16.2,3.5)--(16.8,3.5);
		\draw[line width=1pt,dotted] (18.2,3.5)--(18.8,3.5);
		\node (1,4) at (0.5,0.5) {4};
		\node (3,4) at (2.5,0.5) {4};
		\node (1,3) at (0.5,1.5) {3};
		\node (3,3) at (2.5,1.5) {3};
		\node (4,3) at (3.5,1.5) {3};
		\node (6,3) at (5.5,1.5) {3};
		\node (7,3) at (6.5,1.5) {4};
		\node (9,3) at (8.5,1.5) {4};
		\node (1,2) at (0.5,2.5) {2};
		\node (3,2) at (2.5,2.5) {2};
		\node (4,2) at (3.5,2.5) {2};
		\node (6,2) at (5.5,2.5) {2};
		\node (7,2) at (6.5,2.5) {2};
		\node (9,2) at (8.5,2.5) {2};
		\node (10,2) at (9.5,2.5) {2};
		\node (12,2) at (11.5,2.5) {2};
		\node (13,2) at (12.5,2.5) {4};
		\node (15,2) at (14.5,2.5) {4};
		\node (1,1) at (0.5,3.5) {1};
		\node (3,1) at (2.5,3.5) {1};
		\node (4,1) at (3.5,3.5) {1};
		\node (6,1) at (5.5,3.5) {1};
		\node (7,1) at (6.5,3.5) {1};
		\node (9,1) at (8.5,3.5) {1};
		\node (10,1) at (9.5,3.5) {1};
		\node (12,1) at (11.5,3.5) {1};
		\node (13,1) at (12.5,3.5) {1};
		\node (15,1) at (14.5,3.5) {1};
		\node (16,1) at (15.5,3.5) {1};
		\node (20,1) at (19.5,3.5) {1};
		
		\draw decorate[decoration=brace]{ (9,0.7)--(6,0.7) };
		\draw decorate[decoration=brace]{ (20,2.7)--(15.2,2.7) };
		
		\node (long) at (17.6,2) {$y$};
		\node (short) at (7.5,0.1) {$x$};
		\node at (24.2,1) {};
	\end{tikzpicture}
	\begin{tikzpicture}[x=5mm,y=5mm]
		\node at (-2,2) {$\longmapsto$};
		\draw (0,0)--(3,0)--(3,1)--(9,1)--(9,2)--(15,2)--(15,3)--(20,3)--(20,4)--(0,4)--cycle;
		\draw (3,1)--(3,4);
		\draw (9,2)--(9,4);
		\draw (15,3)--(15,4);
		\draw (0,3)--(15,3);
		\draw (0,2)--(9,2);
		\draw (0,1)--(3,1);
		\draw (6,1)--(6,4);
		\draw (12,2)--(12,4);
		
		\draw[line width=1pt,dotted] (1.2,0.5)--(1.8,0.5);
		\draw[line width=1pt,dotted] (1.2,1.5)--(1.8,1.5);
		\draw[line width=1pt,dotted] (4.2,1.5)--(4.8,1.5);
		\draw[line width=1pt,dotted] (7.2,1.5)--(7.8,1.5);
		\draw[line width=1pt,dotted] (1.2,2.5)--(1.8,2.5);
		\draw[line width=1pt,dotted] (4.2,2.5)--(4.8,2.5);
		\draw[line width=1pt,dotted] (7.2,2.5)--(7.8,2.5);
		\draw[line width=1pt,dotted] (10.2,2.5)--(10.8,2.5);
		\draw[line width=1pt,dotted] (13.2,2.5)--(13.8,2.5);
		\draw[line width=1pt,dotted] (1.2,3.5)--(1.8,3.5);
		\draw[line width=1pt,dotted] (4.2,3.5)--(4.8,3.5);
		\draw[line width=1pt,dotted] (7.2,3.5)--(7.8,3.5);
		\draw[line width=1pt,dotted] (10.2,3.5)--(10.8,3.5);
		\draw[line width=1pt,dotted] (13.2,3.5)--(13.8,3.5);
		\draw[line width=1pt,dotted] (16.2,3.5)--(16.8,3.5);
		\draw[line width=1pt,dotted] (18.2,3.5)--(18.8,3.5);
		\node (1,4) at (0.5,0.5) {1};
		\node (3,4) at (2.5,0.5) {1};
		\node (1,3) at (0.5,1.5) {4};
		\node (3,3) at (2.5,1.5) {4};
		\node (4,3) at (3.5,1.5) {4};
		\node (6,3) at (5.5,1.5) {4};
		\node (7,3) at (6.5,1.5) {1};
		\node (9,3) at (8.5,1.5) {1};
		\node (1,2) at (0.5,2.5) {3};
		\node (3,2) at (2.5,2.5) {3};
		\node (4,2) at (3.5,2.5) {3};
		\node (6,2) at (5.5,2.5) {3};
		\node (7,2) at (6.5,2.5) {3};
		\node (9,2) at (8.5,2.5) {3};
		\node (10,2) at (9.5,2.5) {3};
		\node (12,2) at (11.5,2.5) {3};
		\node (13,2) at (12.5,2.5) {1};
		\node (15,2) at (14.5,2.5) {1};
		\node (1,1) at (0.5,3.5) {2};
		\node (3,1) at (2.5,3.5) {2};
		\node (4,1) at (3.5,3.5) {2};
		\node (6,1) at (5.5,3.5) {2};
		\node (7,1) at (6.5,3.5) {2};
		\node (9,1) at (8.5,3.5) {2};
		\node (10,1) at (9.5,3.5) {2};
		\node (12,1) at (11.5,3.5) {2};
		\node (13,1) at (12.5,3.5) {2};
		\node (15,1) at (14.5,3.5) {2};
		\node (16,1) at (15.5,3.5) {2};
		\node (20,1) at (19.5,3.5) {2};
		
		\draw decorate[decoration=brace]{ (9,0.7)--(6,0.7) };
		\draw decorate[decoration=brace]{ (20,2.7)--(15.2,2.7) };
		
		\node (long) at (17.6,2) {$y$};
		\node (short) at (7.5,0.1) {$x$};
		\node at (24.2,1) {};
	\end{tikzpicture}
	\begin{tikzpicture}[x=5mm,y=5mm]
		\node at (-2,2) {$\longmapsto$};
		\draw (0,0)--(3,0)--(3,1)--(9,1)--(9,2)--(15,2)--(15,3)--(20,3)--(20,4)--(0,4)--cycle;
		\draw (3,1)--(3,4);
		\draw (9,2)--(9,4);
		\draw (15,3)--(15,4);
		\draw (0,3)--(15,3);
		\draw (0,2)--(9,2);
		\draw (0,1)--(3,1);
		\draw (6,1)--(6,4);
		\draw (12,2)--(12,4);
		
		\draw[line width=1pt,dotted] (1.2,0.5)--(1.8,0.5);
		\draw[line width=1pt,dotted] (1.2,1.5)--(1.8,1.5);
		\draw[line width=1pt,dotted] (4.2,1.5)--(4.8,1.5);
		\draw[line width=1pt,dotted] (7.2,1.5)--(7.8,1.5);
		\draw[line width=1pt,dotted] (1.2,2.5)--(1.8,2.5);
		\draw[line width=1pt,dotted] (4.2,2.5)--(4.8,2.5);
		\draw[line width=1pt,dotted] (7.2,2.5)--(7.8,2.5);
		\draw[line width=1pt,dotted] (10.2,2.5)--(10.8,2.5);
		\draw[line width=1pt,dotted] (13.2,2.5)--(13.8,2.5);
		\draw[line width=1pt,dotted] (1.2,3.5)--(1.8,3.5);
		\draw[line width=1pt,dotted] (4.2,3.5)--(4.8,3.5);
		\draw[line width=1pt,dotted] (7.2,3.5)--(7.8,3.5);
		\draw[line width=1pt,dotted] (10.2,3.5)--(10.8,3.5);
		\draw[line width=1pt,dotted] (13.2,3.5)--(13.8,3.5);
		\draw[line width=1pt,dotted] (16.2,3.5)--(16.8,3.5);
		\draw[line width=1pt,dotted] (18.2,3.5)--(18.8,3.5);
		\node (1,4) at (0.5,0.5) {4};
		\node (3,4) at (2.5,0.5) {4};
		\node (1,3) at (0.5,1.5) {3};
		\node (3,3) at (2.5,1.5) {3};
		\node (4,3) at (3.5,1.5) {4};
		\node (6,3) at (5.5,1.5) {4};
		\node (7,3) at (6.5,1.5) {1};
		\node (9,3) at (8.5,1.5) {1};
		\node (1,2) at (0.5,2.5) {2};
		\node (3,2) at (2.5,2.5) {2};
		\node (4,2) at (3.5,2.5) {3};
		\node (6,2) at (5.5,2.5) {3};
		\node (7,2) at (6.5,2.5) {3};
		\node (9,2) at (8.5,2.5) {3};
		\node (10,2) at (9.5,2.5) {3};
		\node (12,2) at (11.5,2.5) {3};
		\node (13,2) at (12.5,2.5) {1};
		\node (15,2) at (14.5,2.5) {1};
		\node (1,1) at (0.5,3.5) {1};
		\node (3,1) at (2.5,3.5) {1};
		\node (4,1) at (3.5,3.5) {2};
		\node (6,1) at (5.5,3.5) {2};
		\node (7,1) at (6.5,3.5) {2};
		\node (9,1) at (8.5,3.5) {2};
		\node (10,1) at (9.5,3.5) {2};
		\node (12,1) at (11.5,3.5) {2};
		\node (13,1) at (12.5,3.5) {2};
		\node (15,1) at (14.5,3.5) {2};
		\node (16,1) at (15.5,3.5) {2};
		\node (20,1) at (19.5,3.5) {2};
		
		\draw decorate[decoration=brace]{ (9,0.7)--(6,0.7) };
		\draw decorate[decoration=brace]{ (20,2.7)--(15.2,2.7) };
		
		\node (long) at (17.6,2) {$y$};
		\node (short) at (7.5,0.1) {$x$};
		\node at (24.2,1) {};
	\end{tikzpicture}
	\begin{tikzpicture}[x=5mm,y=5mm]
		\node at (-2,2) {$\longmapsto$};

		\draw (0,0)--(3,0)--(3,1)--(9,1)--(9,2)--(15,2)--(15,3)--(20,3)--(20,4)--(0,4)--cycle;
		\draw (3,1)--(3,4);
		\draw (9,2)--(9,4);
		\draw (15,3)--(15,4);
		\draw (0,3)--(15,3);
		\draw (0,2)--(9,2);
		\draw (0,1)--(3,1);
		\draw (6,1)--(6,4);
		\draw (12,2)--(12,4);
		
		\draw[line width=1pt,dotted] (1.2,0.5)--(1.8,0.5);
		\draw[line width=1pt,dotted] (1.2,1.5)--(1.8,1.5);
		\draw[line width=1pt,dotted] (4.2,1.5)--(4.8,1.5);
		\draw[line width=1pt,dotted] (7.2,1.5)--(7.8,1.5);
		\draw[line width=1pt,dotted] (1.2,2.5)--(1.8,2.5);
		\draw[line width=1pt,dotted] (4.2,2.5)--(4.8,2.5);
		\draw[line width=1pt,dotted] (7.2,2.5)--(7.8,2.5);
		\draw[line width=1pt,dotted] (10.2,2.5)--(10.8,2.5);
		\draw[line width=1pt,dotted] (13.2,2.5)--(13.8,2.5);
		\draw[line width=1pt,dotted] (1.2,3.5)--(1.8,3.5);
		\draw[line width=1pt,dotted] (4.2,3.5)--(4.8,3.5);
		\draw[line width=1pt,dotted] (7.2,3.5)--(7.8,3.5);
		\draw[line width=1pt,dotted] (10.2,3.5)--(10.8,3.5);
		\draw[line width=1pt,dotted] (13.2,3.5)--(13.8,3.5);
		\draw[line width=1pt,dotted] (16.2,3.5)--(16.8,3.5);
		\draw[line width=1pt,dotted] (18.2,3.5)--(18.8,3.5);
		\node (1,4) at (0.5,0.5) {4};
		\node (3,4) at (2.5,0.5) {4};
		\node (1,3) at (0.5,1.5) {3};
		\node (3,3) at (2.5,1.5) {3};
		\node (4,3) at (3.5,1.5) {3};
		\node (6,3) at (5.5,1.5) {3};
		\node (7,3) at (6.5,1.5) {4};
		\node (9,3) at (8.5,1.5) {4};
		\node (1,2) at (0.5,2.5) {2};
		\node (3,2) at (2.5,2.5) {2};
		\node (4,2) at (3.5,2.5) {2};
		\node (6,2) at (5.5,2.5) {2};
		\node (7,2) at (6.5,2.5) {3};
		\node (9,2) at (8.5,2.5) {3};
		\node (10,2) at (9.5,2.5) {3};
		\node (12,2) at (11.5,2.5) {3};
		\node (13,2) at (12.5,2.5) {1};
		\node (15,2) at (14.5,2.5) {1};
		\node (1,1) at (0.5,3.5) {1};
		\node (3,1) at (2.5,3.5) {1};
		\node (4,1) at (3.5,3.5) {1};
		\node (6,1) at (5.5,3.5) {1};
		\node (7,1) at (6.5,3.5) {2};
		\node (9,1) at (8.5,3.5) {2};
		\node (10,1) at (9.5,3.5) {2};
		\node (12,1) at (11.5,3.5) {2};
		\node (13,1) at (12.5,3.5) {2};
		\node (15,1) at (14.5,3.5) {2};
		\node (16,1) at (15.5,3.5) {2};
		\node (20,1) at (19.5,3.5) {2};
		
		\draw decorate[decoration=brace]{ (6,0.7)--(3.1,0.7) };
		\draw decorate[decoration=brace]{ (20,2.7)--(15.2,2.7) };
		\draw decorate[decoration=brace]{ (9,0.7)--(6.1,0.7) };
		\draw decorate[decoration=brace]{ (12,1.7)--(9.1,1.7) };
		\draw decorate[decoration=brace]{ (15,1.7)--(12.1,1.7) };

		\node (long) at (17.6,2) {$y$};
		\node (short) at (4.6,0.1) {$x$};
		\node (short) at (7.5,0.1) {$t$};
		\node (short) at (10.5,1.1) {$u$};
		\node (short) at (13.5,1.1) {$v$};

		\node at (24.2,1) {};
	\end{tikzpicture}
	\begin{tikzpicture}[x=5mm,y=5mm]
		\node at (-2,-0.5) {$\longmapsto$};
		\node at (-0.7,-0.5) {$\left\{\begin{aligned} \\ \\ \\ \\ \\ \\ \\ \\\end{aligned}\right.$};

		\draw (0,0)--(3,0)--(3,1)--(12,1)--(12,2)--(16,2)--(16,3)--(20,3)--(20,4)--(0,4)--cycle;
		\draw (3,1)--(3,4);
		\draw (9,1)--(9,4);
		\draw (16,3)--(16,4);
		\draw (0,3)--(16,3);
		\draw (0,2)--(12,2);
		\draw (0,1)--(3,1);
		\draw (6,1)--(6,4);
		\draw (12,2)--(12,4);
		
		\draw[line width=1pt,dotted] (1.2,0.5)--(1.8,0.5);
		\draw[line width=1pt,dotted] (1.2,1.5)--(1.8,1.5);
		\draw[line width=1pt,dotted] (1.2,2.5)--(1.8,2.5);
		\draw[line width=1pt,dotted] (1.2,3.5)--(1.8,3.5);
		\draw[line width=1pt,dotted] (4.2,1.5)--(4.8,1.5);
		\draw[line width=1pt,dotted] (4.2,2.5)--(4.8,2.5);
		\draw[line width=1pt,dotted] (4.2,3.5)--(4.8,3.5);
		\draw[line width=1pt,dotted] (7.2,1.5)--(7.8,1.5);
		\draw[line width=1pt,dotted] (7.2,2.5)--(7.8,2.5);
		\draw[line width=1pt,dotted] (7.2,3.5)--(7.8,3.5);
		\draw[line width=1pt,dotted] (10.2,1.5)--(10.8,1.5);
		\draw[line width=1pt,dotted] (10.2,2.5)--(10.8,2.5);
		\draw[line width=1pt,dotted] (10.2,3.5)--(10.8,3.5);
		\draw[line width=1pt,dotted] (13.2,2.5)--(13.8,2.5);
		\draw[line width=1pt,dotted] (14.2,2.5)--(14.8,2.5);
		\draw[line width=1pt,dotted] (13.2,3.5)--(13.8,3.5);
		\draw[line width=1pt,dotted] (14.3,3.5)--(14.8,3.5);
		\draw[line width=1pt,dotted] (17.2,3.5)--(17.8,3.5);
		\draw[line width=1pt,dotted] (18.2,3.5)--(18.8,3.5);
		
		\node (1,4) at (0.5,0.5) {4};
		\node (3,4) at (2.5,0.5) {4};
		\node (1,3) at (0.5,1.5) {3};
		\node (3,3) at (2.5,1.5) {3};
		\node (4,3) at (3.5,1.5) {3};
		\node (6,3) at (5.5,1.5) {3};
		\node (7,3) at (6.5,1.5) {4};
		\node (9,3) at (8.5,1.5) {4};
		\node (10,3) at (9.5,1.5) {4};
		\node (12,3) at (11.5,1.5) {4};
		\node (1,2) at (0.5,2.5) {2};
		\node (3,2) at (2.5,2.5) {2};
		\node (4,2) at (3.5,2.5) {2};
		\node (6,2) at (5.5,2.5) {2};
		\node (7,2) at (6.5,2.5) {2};
		\node (9,2) at (8.5,2.5) {2};
		\node (10,2) at (9.5,2.5) {3};
		\node (12,2) at (11.5,2.5) {3};
		\node (13,2) at (12.5,2.5) {3};
		\node (16,2) at (15.5,2.5) {3};
		\node (1,1) at (0.5,3.5) {1};
		\node (3,1) at (2.5,3.5) {1};
		\node (4,1) at (3.5,3.5) {1};
		\node (6,1) at (5.5,3.5) {1};
		\node (7,1) at (6.5,3.5) {1};
		\node (9,1) at (8.5,3.5) {1};
		\node (10,1) at (9.5,3.5) {2};
		\node (12,1) at (11.5,3.5) {2};
		\node (13,1) at (12.5,3.5) {2};
		\node (16,1) at (15.5,3.5) {2};
		\node (17,1) at (16.5,3.5) {2};
		\node (20,1) at (19.5,3.5) {2};
		
		\draw decorate[decoration=brace]{ (5.9,0.7)--(3.2,0.7) };
		\draw decorate[decoration=brace]{ (9,0.7)--(6.1,0.7) };
		\draw decorate[decoration=brace]{ (12,0.7)--(9.1,0.7) };
		\draw decorate[decoration=brace]{ (20,2.7)--(16.2,2.7) };
		\draw decorate[decoration=brace]{ (16,1.7)--(12.2,1.7) };
		
		\node (long) at (18.2,2) {$y$};
		\node (long') at (14.1,0.8) {$u+v$};
		\node (short) at (4.5,-0.1) {$x$};
		\node (short') at (7.5,-0.1) {$v$};
		\node (shot') at (10.5,-0.1) {$t-v$};
		
		\node (if1) at (23,2) {if $v \leq t$};
		\node (if2) at (23,-3) {if $v \geq t$};

		\draw (0,-5)--(3,-5)--(3,-4)--(9,-4)--(9,-3)--(16,-3)--(16,-2)--(20,-2)--(20,-1)--(0,-1)--cycle;
		\draw (3,-4)--(3,-1);
		\draw (9,-3)--(9,-1);
		\draw (16,-2)--(16,-1);
		\draw (0,-2)--(16,-2);
		\draw (0,-3)--(9,-3);
		\draw (0,-4)--(3,-4);
		\draw (6,-4)--(6,-1);
		\draw (12,-3)--(12,-1);
		
		\draw[line width=1pt,dotted] (1.2,-4.5)--(1.8,-4.5);
		\draw[line width=1pt,dotted] (1.2,-3.5)--(1.8,-3.5);
		\draw[line width=1pt,dotted] (1.2,-2.5)--(1.8,-2.5);
		\draw[line width=1pt,dotted] (1.2,-1.5)--(1.8,-1.5);
		\draw[line width=1pt,dotted] (4.2,-3.5)--(4.8,-3.5);
		\draw[line width=1pt,dotted] (4.2,-2.5)--(4.8,-2.5);
		\draw[line width=1pt,dotted] (4.2,-1.5)--(4.8,-1.5);
		\draw[line width=1pt,dotted] (7.2,-3.5)--(7.8,-3.5);
		\draw[line width=1pt,dotted] (7.2,-2.5)--(7.8,-2.5);
		\draw[line width=1pt,dotted] (7.2,-1.5)--(7.8,-1.5);
		\draw[line width=1pt,dotted] (10.2,-2.5)--(10.8,-2.5);
		\draw[line width=1pt,dotted] (10.2,-1.5)--(10.8,-1.5);
		\draw[line width=1pt,dotted] (13.2,-2.5)--(13.8,-2.5);
		\draw[line width=1pt,dotted] (14.2,-2.5)--(14.8,-2.5);
		\draw[line width=1pt,dotted] (13.2,-1.5)--(13.8,-1.5);
		\draw[line width=1pt,dotted] (14.3,-1.5)--(14.8,-1.5);
		\draw[line width=1pt,dotted] (17.2,-1.5)--(17.8,-1.5);
		\draw[line width=1pt,dotted] (18.2,-1.5)--(18.8,-1.5);
		
		\node (1,4) at (0.5,-4.5) {4};
		\node (3,4) at (2.5,-4.5) {4};
		\node (1,3) at (0.5,-3.5) {3};
		\node (3,3) at (2.5,-3.5) {3};
		\node (4,3) at (3.5,-3.5) {3};
		\node (6,3) at (5.5,-3.5) {3};
		\node (7,3) at (6.5,-3.5) {4};
		\node (9,3) at (8.5,-3.5) {4};
		\node (1,2) at (0.5,-2.5) {2};
		\node (3,2) at (2.5,-2.5) {2};
		\node (4,2) at (3.5,-2.5) {2};
		\node (6,2) at (5.5,-2.5) {2};
		\node (7,2) at (6.5,-2.5) {2};
		\node (9,2) at (8.5,-2.5) {2};
		\node (10,2) at (9.5,-2.5) {2};
		\node (12,2) at (11.5,-2.5) {2};
		\node (13,2) at (12.5,-2.5) {3};
		\node (16,2) at (15.5,-2.5) {3};
		\node (1,1) at (0.5,-1.5) {1};
		\node (3,1) at (2.5,-1.5) {1};
		\node (4,1) at (3.5,-1.5) {1};
		\node (6,1) at (5.5,-1.5) {1};
		\node (7,1) at (6.5,-1.5) {1};
		\node (9,1) at (8.5,-1.5) {1};
		\node (10,1) at (9.5,-1.5) {1};
		\node (12,1) at (11.5,-1.5) {1};
		\node (13,1) at (12.5,-1.5) {2};
		\node (16,1) at (15.5,-1.5) {2};
		\node (17,1) at (16.5,-1.5) {2};
		\node (20,1) at (19.5,-1.5) {2};
		
		\draw decorate[decoration=brace]{ (5.9,-4.3)--(3.2,-4.3) };
		\draw decorate[decoration=brace]{ (9,-4.3)--(6.1,-4.3) };
		\draw decorate[decoration=brace]{ (12,-3.3)--(9.1,-3.3) };
		\draw decorate[decoration=brace]{ (20,-2.3)--(16.2,-2.3) };
		\draw decorate[decoration=brace]{ (16,-3.3)--(12.2,-3.3) };
		
		\node (long) at (18.2,-3) {$y$};
		\node (long') at (14.1,-4.2) {$u+t$};
		\node (short) at (4.5,-5.1) {$x$};
		\node (short') at (7.5,-5.1) {$t$};
		\node (sho') at (10.6,-4.2) {$v-t$};
	\end{tikzpicture}

\caption{The action of the operator $\mathsf{pr}_{1,4}$}\label{2525625}
\end{figure}

  \noindent
	Thus, $\mathsf{\Phi}(T)$ is an element of
	$\mathsf{SST}_4^{\text{$\mathfrak{k}$-hw}}(\lambda)$.
	Next, let us take an arbitrary $S \in \mathsf{SST}_4^{\text{$\mathfrak{k}$-hw}}(\lambda)$.
	Then, by Lemma \ref{characterization of k-highest},
	$S$ is the tableau shown in Figure \ref{30000}.
	Therefore, by Definition \ref{defpr},
	$\mathsf{\Phi}^{-1}(S)$ is the tableau shown in Figure \ref{10000}.
	Thus, $\mathsf{\Phi}^{-1}(S)$ is an element of
	$\mathsf{SST}_4^{\text{$\widehat{\mathfrak{g}}$-dom}}(\lambda)$.
	This proves part (1). 
\end{proof}

\bigskip

\section{Proof of theorem \ref{KeyProposition} (the case $n \geq 3$)}
In this section,
we consider the case $n\geq3$,
and complete the proof of Theorem \ref{KeyProposition}. 
Let us fix an arbitrary partition $\lambda \in \mathsf{Par}_{\leq2n}$.

\subsection{Restriction operators and promotion operators}
In this subsection,
we define two kinds of operators, i.e., restriction operators and promotion operators. 

\begin{de}
	Let $T \in \mathsf{SST}_{2n}$,
	and let $a \leq b \leq c \leq d$ be integers.
        We denote by $T|_{a,b;c,d}$ the semistandard tableau
	obtained from $T$ by removing all boxes whose entry does not belong
	to $[a,b]$ nor $[c,d]$.
	When $b=c=d$,
	we simply write $T|_{a,b}$ for $T|_{a,b;c,d}$.
\end{de}

\begin{rem}
	The shape of $T|_{a,b;c,d}$ is not a Young diagram in general.
\end{rem}

\begin{de}\label{7.3}
	Let $\lambda^1 \supset \lambda^2 \supset \lambda^3$ be
	a decreasing sequence of partitions,
	and $T$ a semistandard tableau of shape
	$(\lambda^1 \setminus \lambda^2) \sqcup \lambda^3$.
	We denote by $\mathsf{Rect}(T)$ 
        the tableau given by the following algorithm:
	\begin{enumerate}
		\item Let $\bullet$ be a formal symbol. 
		We define a new tableau $T_{\bullet}$ of shape $\lambda^1 \setminus \lambda^3$
		by: 
		\[
			T_{\bullet}(x,y) := \left\{
			\begin{aligned}
				& T(x,y) & & \text{if $(x,y) \in (\lambda^1 \setminus \lambda^2) \sqcup \lambda^3$}, \\
				& \bullet & & \text{if $(x,y) \in \lambda^2 \setminus \lambda^3$}.
			\end{aligned}
			\right.
		\]
		\item Apply the jeu-de-taquin slide to the boxes whose entry equals $\bullet$ lying at inside corners of $T_{\bullet}$, 
		in any order; a box $(x_0,y_0)$ is said to be 
		at an inside corner
		if there exist boxes both to the right of $(x_0,y_0)$ and below $(x_0,y_0)$.
		\item Remove all boxes whose entry equals $\bullet$ .
	\end{enumerate}
\end{de}

\begin{rem}
	$\mathsf{Rect}(T)$ is independent of the order
	in which the boxes are chosen in step (2).
	This follows from the fact that the Knuth
	equivalence for row words is preserved
	in the process of applying the jeu-de-taquin slide to $T_{\bullet}$, 
	which is verified  by an argument 
	similar to that in \cite[Section 1.2, Claim 2]{F}.
\end{rem}

\begin{de}
	Let $T \in \mathsf{SST}_{2n}$, 
	and let $a \leq b \leq c \leq d $ be integers.
	We set $\mathsf{Res}_{a,b;c,d}(T):=\mathsf{Rect}(T|_{a,b;c,d})$.
	When $b=c=d$, we simply write $\mathsf{Res}_{a,b}(T)$ for $\mathsf{Res}_{a,b;c,d}(T)$.
\end{de}

\begin{eg}\label{example}
	Let $a=1$, $b=2$, $c=5$, $d=6$. Then, we have
	
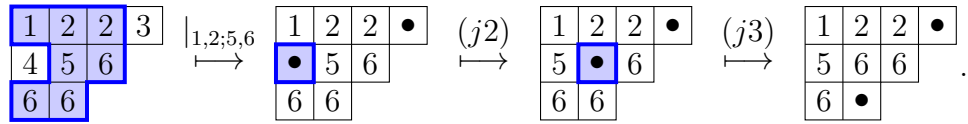
\vspace{10pt}
	\begin{center}
		\begin{tikzpicture}[x=5mm,y=5mm]
			\draw[fill=blue!20] (0,0)--(3,0)--(3,-2)--(2,-2)--(2,-3)--(0,-3)--(0,-2)--(1,-2)--(1,-1)--(0,-1)--cycle;
			\draw (0,0)--(4,0)--(4,-1)--(3,-1)--(3,-2)--(2,-2)--(2,-3)--(0,-3)--cycle;
			\draw (1,0)--(1,-3);
			\draw (2,0)--(2,-2);
			\draw (3,0)--(3,-1);
			\draw (0,-1)--(3,-1);
			\draw (0,-2)--(2,-2);
			\node at (0.5,-0.5) {1};
			\node at (1.5,-0.5) {2};
			\node at (2.5,-0.5) {2};
			\node at (3.5,-0.5) {3};
			\node at (0.5,-1.5) {4};
			\node at (1.5,-1.5) {5};
			\node at (2.5,-1.5) {6};
			\node at (0.5,-2.5) {6};
			\node at (1.5,-2.5) {6};
			\draw[line width=1.5pt, color=blue] (0,0)--(3,0)--(3,-2)--(2,-2)--(2,-3)--(0,-3)--(0,-2)--(1,-2)--(1,-1)--(0,-1)--cycle;

			\node at (5.5,-1.5) {$\longmapsto$};
			\node at (5.5,-0.7) {$|_{1,2;5,6}$};

			\fill[color=blue!20] (7,-1)--(8,-1)--(8,-2)--(7,-2)--cycle;
			\draw (7,0)--(11,0)--(11,-1)--(10,-1)--(10,-2)--(9,-2)--(9,-3)--(7,-3)--cycle;
			\draw (8,0)--(8,-3);
			\draw (9,0)--(9,-2);
			\draw (10,0)--(10,-1);
			\draw (7,-1)--(10,-1);
			\draw (7,-2)--(9,-2);
			\node at (7.5,-0.5) {1};
			\node at (8.5,-0.5) {2};
			\node at (9.5,-0.5) {2};
			\node at (10.5,-0.5) {$\bullet$};
			\node at (7.5,-1.5) {$\bullet$};
			\node at (8.5,-1.5) {5};
			\node at (9.5,-1.5) {6};
			\node at (7.5,-2.5) {6};
			\node at (8.5,-2.5) {6};
			\draw[line width=1.5pt, color=blue] (7,-1)--(8,-1)--(8,-2)--(7,-2)--cycle;

			\node at (12.5,-1.5) {$\longmapsto$};
			\node at (12.5,-0.7) {$(j2)$};

			\fill[color=blue!20] (15,-1)--(16,-1)--(16,-2)--(15,-2)--cycle;
			\draw (14,0)--(18,0)--(18,-1)--(17,-1)--(17,-2)--(16,-2)--(16,-3)--(14,-3)--cycle;
			\draw (15,0)--(15,-3);
			\draw (16,0)--(16,-2);
			\draw (17,0)--(17,-1);
			\draw (14,-1)--(17,-1);
			\draw (14,-2)--(16,-2);
			\node at (14.5,-0.5) {1};
			\node at (15.5,-0.5) {2};
			\node at (16.5,-0.5) {2};
			\node at (17.5,-0.5) {$\bullet$};
			\node at (14.5,-1.5) {5};
			\node at (15.5,-1.5) {$\bullet$};
			\node at (16.5,-1.5) {6};
			\node at (14.5,-2.5) {6};
			\node at (15.5,-2.5) {6};
			\draw[line width=1.5pt, color=blue] (15,-1)--(16,-1)--(16,-2)--(15,-2)--cycle;

			\node at (19.5,-1.5) {$\longmapsto$};
			\node at (19.5,-0.7) {$(j3)$};

			\draw (21,0)--(25,0)--(25,-1)--(24,-1)--(24,-2)--(23,-2)--(23,-3)--(21,-3)--cycle;
			\draw (22,0)--(22,-3);
			\draw (23,0)--(23,-2);
			\draw (24,0)--(24,-1);
			\draw (21,-1)--(24,-1);
			\draw (21,-2)--(23,-2);
			\node at (21.5,-0.5) {1};
			\node at (22.5,-0.5) {2};
			\node at (23.5,-0.5) {2};
			\node at (24.5,-0.5) {$\bullet$};
			\node at (21.5,-1.5) {5};
			\node at (22.5,-1.5) {6};
			\node at (23.5,-1.5) {6};
			\node at (21.5,-2.5) {6};
			\node at (22.5,-2.5) {$\bullet$};
			\node at (25.2,-1.8) {.};
		\end{tikzpicture}
	\end{center}
	\vspace{10pt}
	Note that the three partitions
	$\lambda^1$, $\lambda^2$, $\lambda^3$ in Definition \ref{7.3} are
	$(4,3,2)$, $(4,1)$ and $(3)$, respectively.
	Therefore,
	\begin{center}
		\begin{tikzpicture}[x=5mm,y=5mm]
			\draw (0,0)--(4,0)--(4,-1)--(3,-1)--(3,-2)--(2,-2)--(2,-3)--(0,-3)--cycle;
			\draw (1,0)--(1,-3);
			\draw (2,0)--(2,-2);
			\draw (3,0)--(3,-1);
			\draw (0,-1)--(3,-1);
			\draw (0,-2)--(2,-2);
			\node at (0.5,-0.5) {1};
			\node at (1.5,-0.5) {2};
			\node at (2.5,-0.5) {2};
			\node at (3.5,-0.5) {3};
			\node at (0.5,-1.5) {4};
			\node at (1.5,-1.5) {5};
			\node at (2.5,-1.5) {6};
			\node at (0.5,-2.5) {6};
			\node at (1.5,-2.5) {6};
			\node at (-2.5,-1.5) {$\mathsf{Res}_{1,2;5,6}$};
			\node at (2,-1.5) {$\left(\begin{aligned} \\ \quad\quad\quad\quad\quad\, \\ \\ \end{aligned}\right)$};

			\node at (5.5,-1.5) {$=$};

			\draw (6.5,0)--(9.5,0)--(9.5,-2)--(7.5,-2)--(7.5,-3)--(6.5,-3)--cycle;
			\draw (7.5,0)--(7.5,-3);
			\draw (8.5,0)--(8.5,-2);
			\draw (9.5,0)--(9.5,-1);
			\draw (6.5,-1)--(9.5,-1);
			\draw (6.5,-2)--(8.5,-2);
			\node at (7,-0.5) {1};
			\node at (8,-0.5) {2};
			\node at (9,-0.5) {2};
			\node at (7,-1.5) {5};
			\node at (8,-1.5) {6};
			\node at (9,-1.5) {6};
			\node at (7,-2.5) {6};
			\node at (9.8,-1.8) {.};
		\end{tikzpicture}
	\end{center}
\end{eg}

As in the case $n=2$, we construct the bijections
$\mathsf{\Phi}$ and $\mathsf{\Psi}$, 
as certain composites of promotion operators. 

For $1 \leq a < b \leq 2n$,
we define $\mathsf{pr}_{a,b}$ as in Definition \ref{defpr}.
The following lemma follows from the definition of promotion operators. 
\begin{lem}\label{prrel}
	Let $1 \leq a < b < c < d \leq 2n$ be integers.
	Then, we have the following:
	\begin{enumerate}
		\item $\mathsf{pr}_{a,a+1}^2 = \mathsf{id}$,
		\item $\mathsf{pr}_{a,b} \circ \mathsf{pr}_{c,d} = \mathsf{pr}_{c,d} \circ \mathsf{pr}_{a,b}$ if $|b-c|\geq2$,
		\item $\mathsf{pr}_{a,b} \circ \mathsf{pr}_{b,c} = \mathsf{pr}_{a,c}$.
	\end{enumerate}
\end{lem}

\begin{proof}
Parts (1) and (2) immediately follow from the definition of $\mathsf{pr}_{a,b}$. Part (3) follows from the equality
$\mathsf{pr}_{a,a+1} \circ \mathsf{pr}_{a+1,a+2} = \mathsf{pr}_{a,a+2}$.
This proves the lemma. 
\end{proof}

In the next two subsections, we give several lemmas, which are needed to connect the global information concerning $\widehat{\mathfrak{g}}$, $\mathfrak{k}$ 
with the local information concerning $\widehat{\mathfrak{g}}_i$, $\mathfrak{k}_i$ for $1 \leq i \leq n-1$. 

\subsection{The local-global principle for $\widehat{\mathfrak{g}}$-dominant tableaux}
For each $1 \leq i \leq n-1$,
we set
\begin{equation}
	\widehat{C}_i :=
	\left\{
	x_1\widehat{\varepsilon}_1 + \cdots + x_n\widehat{\varepsilon}_n \in \widehat{P}_{\mathbb{R}}
	\mathrel{}\middle|\mathrel{}
	x_i \geq x_{i+1} \geq 0
	\right\}.
\end{equation}
Note that $\widehat{C}=\widehat{C}_1 \cap \cdots \cap \widehat{C}_{n-1}$.

\begin{de}
	Let $T$ be a semistandard tableau with all entries in $[1,2n]$.
	\begin{enumerate}
		\item We say that $T$ is a
		\textit{$\widehat{\mathfrak{g}}_i$-dominant tableau}
		if the image of $\widehat{\pi}_T$ is contained in $\widehat{C}_i$.
		\item We say that $T$ is a
		\textit{$\widehat{\mathfrak{g}}_{\geq i}$-dominant tableau}
		if $T$ is a $\widehat{\mathfrak{g}}_j$-dominant tableau for all $i \leq j \leq n-1$.
	\end{enumerate}
\end{de}
\begin{rem}
	A semistandard tableau $T \in \mathsf{SST}_{2n}$ is a $\widehat{\mathfrak{g}}$-dominant tableau
	if and only if $T$ is a $\widehat{\mathfrak{g}}_i$-dominant tableau for all $1 \leq i \leq n-1$.
\end{rem}

Recall that $\overline{i}=2n-i+1$.

\begin{lem}
	Let $T \in \mathsf{SST}_{2n}^{\text{$\widehat{\mathfrak{g}}$-$\mathrm{dom}$}}(\lambda)$
	be a $\widehat{\mathfrak{g}}$-dominant tableau.
	Then,
	$\mathsf{Res}_{i,i+1;\overline{i+1},\overline{i}}(T)$ is
	a $\widehat{\mathfrak{g}}_i$-dominant tableau
	for all $1 \leq i \leq n-1$.
\end{lem}

\begin{proof}
	Since $T$ is $\widehat{\mathfrak{g}}$-dominant,
	$T|_{i,i+1;\overline{i+1},\overline{i}}$
	is a $\widehat{\mathfrak{g}}_i$-dominant tableau for all $1 \leq i \leq n-1$. 
	Hence, 
	it suffices to show that
	$\widehat{\mathfrak{g}}_i$-dominance is preserved
	in the process of applying the jeu-de-taquin slide to
	$T|_{i,i+1;\overline{i+1},\overline{i}}$;
	here, 
	$T|_{i,i+1;\overline{i+1},\overline{i}}$
	is the tableau shown in Figure \ref{omega} below.\\
	\begin{figure}[th]
	\begin{center}
		\begin{tikzpicture}[x=5mm,y=5mm]
			\fill[fill=black!30] (0,-2)--(0,-10)--(5,-10)--(5,-9)--(6,-9)--(6,-8)--(20,-8)--(20,-5)--(24,-5)--(24,-3)--(28,-3)--(28,-1)--(18,-1)--(18,-2)--cycle;
			\draw (0,0)--(30,0)--(30,-1)--(18,-1)--(18,-2)--(0,-2)--cycle;
			\draw (0,-1)--(18,-1);
			\node (a) at (0.5,-0.5) {$i$};
			\node (b) at (29.5,-0.5) {$i$};
			\draw[dotted, line width=1pt] (1.2,-0.5)--(1.8,-0.5);
			\draw[dotted, line width=1pt] (28.2,-0.5)--(28.8,-0.5);
			\node (c) at (1.2,-1.5) {$i+1$};
			\node (d) at (16.8,-1.5) {$i+1$};
			\draw[dotted, line width=1pt] (2.4,-1.5)--(3,-1.5);
			\draw[dotted, line width=1pt] (15,-1.5)--(15.6,-1.5);
			\draw (23,-5)--(23,-6)--(20,-6);
			\node (e) at (22.5,-5.5) {$\overline{i}$};
			\node (f) at (20.5,-5.5) {$\overline{i}$};
			\draw[dotted, line width=1pt] (21.2,-5.5)--(21.8,-5.5);
			\draw (27,-3)--(27,-4)--(24,-4);
			\node (g) at (26.5,-3.5) {$\overline{i}$};
			\node (h) at (24.5,-3.5) {$\overline{i}$};
			\draw[dotted, line width=1pt] (25.2,-3.5)--(25.8,-3.5);
			\draw (28,-3)--(24,-3)--(24,-5)--(20,-5)--(20,-7);
			\draw (17,-8)--(16,-8)--(16,-9)--(13,-9)--(5,-9)--(5,-10);
			\node (i) at (15.5,-8.5) {$\overline{i}$};
			\node (j) at (13.5,-8.5) {$\overline{i}$};
			\draw[dotted, line width=1pt] (14.2,-8.5)--(14.8,-8.5);
			\draw (13,-9)--(13,-8);
			\node (k) at (11.7,-8.5) {$\overline{i+1}$};
			\node (l) at (7.25,-8.5) {$\overline{i+1}$};
			\draw[dotted, line width=1pt] (8.5,-8.5)--(9.1,-8.5);
			\draw[dotted, line width=1pt] (9.85,-8.5)--(10.45,-8.5);
			\draw (6,-9)--(6,-8)--(16,-8);
			\draw[line width=1pt, color=blue, line join=round] (21.5,-8.2)--(18.2,-8.2)--(18.2,-1.2)--(28.2,-1.2)--(28.2,-4.5)--cycle;
			\node at (25.5,-7) {\textcolor{blue}{$\Omega$}};
		\end{tikzpicture}
	\end{center}
	\caption{}\label{omega}
  \end{figure}

	We have the following possibilities (i)--(iv) for 
	the slides that occur during the process of applying
	the jeu-de-taquin slide to $T|_{i,i+1;\overline{i+1},\overline{i}}$: 
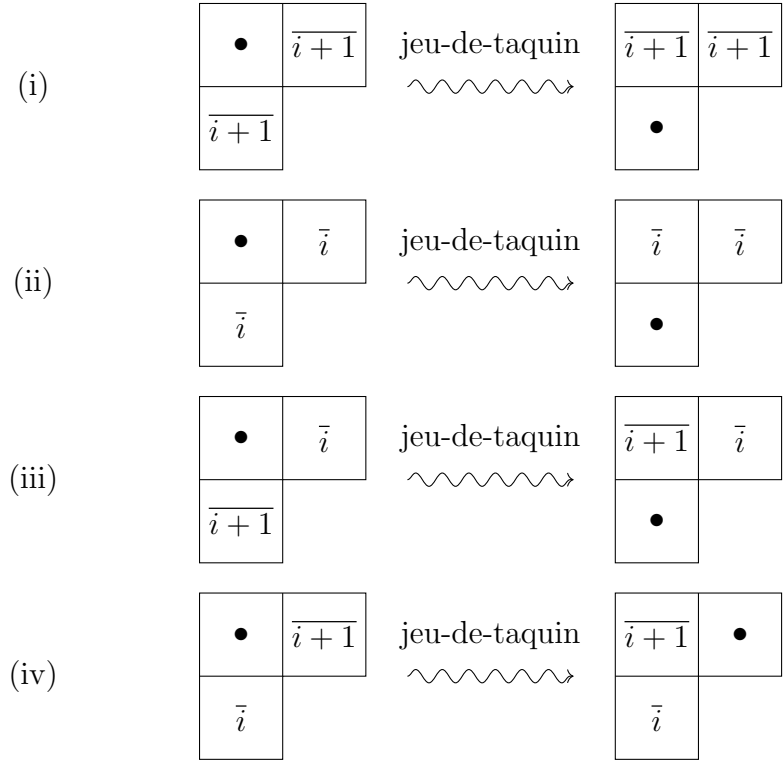
	\begin{center}
		\begin{tikzpicture}[x=11mm,y=11mm]
			\draw (0,0)--(1,0)--(1,1)--(2,1)--(2,2)--(0,2)--cycle;
			\draw (0,1)--(1,1);
			\draw (1,1)--(1,2);
			\node at (0.5,0.5) {$\overline{i+1}$};
			\node at (1.5,1.5) {$\overline{i+1}$};
			\node at (0.5,1.5) {$\bullet$};
			\draw[->] decorate[decoration=snake] {(2.5,1)--(4.5,1)};
			\node at (3.5,1.5) {jeu-de-taquin};
			
			\draw (5,0)--(6,0)--(6,1)--(7,1)--(7,2)--(5,2)--cycle;
			\draw (5,1)--(6,1);
			\draw (6,1)--(6,2);
			\node at (5.5,1.5) {$\overline{i+1}$};
			\node at (6.5,1.5) {$\overline{i+1}$};
			\node at (5.5,0.5) {$\bullet$};
			
			\node at (-2,1) {\textcolor{white}{(iii)}};
			\node at (-2,1) {(i)};
			\node at (-2,-0.2) {};
			\node at (9,0) {};
		\end{tikzpicture} \\
		\begin{tikzpicture}[x=11mm,y=11mm]
			\draw (0,0)--(1,0)--(1,1)--(2,1)--(2,2)--(0,2)--cycle;
			\draw (0,1)--(1,1);
			\draw (1,1)--(1,2);
			\node at (0.5,0.5) {$\overline{i}$};
			\node at (1.5,1.5) {$\overline{i}$};
			\node at (0.5,1.5) {$\bullet$};
			\draw[->] decorate[decoration=snake] {(2.5,1)--(4.5,1)};
			\node at (3.5,1.5) {jeu-de-taquin};
			
			\draw (5,0)--(6,0)--(6,1)--(7,1)--(7,2)--(5,2)--cycle;
			\draw (5,1)--(6,1);
			\draw (6,1)--(6,2);
			\node at (5.5,1.5) {$\overline{i}$};
			\node at (6.5,1.5) {$\overline{i}$};
			\node at (5.5,0.5) {$\bullet$};
			
			\node at (-2,1) {\textcolor{white}{(iii)}};
			\node at (-2,1) {(ii)};
			\node at (-2,-0.2) {};
			\node at (9,0) {};
		\end{tikzpicture} \\
		\begin{tikzpicture}[x=11mm,y=11mm]
			\draw (0,0)--(1,0)--(1,1)--(2,1)--(2,2)--(0,2)--cycle;
			\draw (0,1)--(1,1);
			\draw (1,1)--(1,2);
			\node at (0.5,0.5) {$\overline{i+1}$};
			\node at (1.5,1.5) {$\overline{i}$};
			\node at (0.5,1.5) {$\bullet$};
			\draw[->] decorate[decoration=snake] {(2.5,1)--(4.5,1)};
			\node at (3.5,1.5) {jeu-de-taquin};
			
			\draw (5,0)--(6,0)--(6,1)--(7,1)--(7,2)--(5,2)--cycle;
			\draw (5,1)--(6,1);
			\draw (6,1)--(6,2);
			\node at (5.5,1.5) {$\overline{i+1}$};
			\node at (6.5,1.5) {$\overline{i}$};
			\node at (5.5,0.5) {$\bullet$};
			
			\node at (-2,1) {(iii)};
			\node at (-2,-0.2) {};
			\node at (9,0) {};
		\end{tikzpicture} \\
		\begin{tikzpicture}[x=11mm,y=11mm]
			\draw (0,0)--(1,0)--(1,1)--(2,1)--(2,2)--(0,2)--cycle;
			\draw (0,1)--(1,1);
			\draw (1,1)--(1,2);
			\node at (0.5,0.5) {$\overline{i}$};
			\node at (1.5,1.5) {$\overline{i+1}$};
			\node at (0.5,1.5) {$\bullet$};
			\draw[->] decorate[decoration=snake] {(2.5,1)--(4.5,1)};
			\node at (3.5,1.5) {jeu-de-taquin};
			
			\draw (5,0)--(6,0)--(6,1)--(7,1)--(7,2)--(5,2)--cycle;
			\draw (5,1)--(6,1);
			\draw (6,1)--(6,2);
			\node at (5.5,0.5) {$\overline{i}$};
			\node at (5.5,1.5) {$\overline{i+1}$};
			\node at (6.5,1.5) {$\bullet$};
			
			\node at (-2,1) {\textcolor{white}{(iii)}};
			\node at (-2,1) {(iv)};
			\node at (9,0) {};
		\end{tikzpicture}
	\end{center}
	In cases (i), (ii), and (iii),
	the column word is preserved by the jeu-de-taquin slide, and hence 
	$\widehat{\mathfrak{g}}_k$-dominance is also preserved.
	Now, let us 
	consider case (iv).
	Since the area $\Omega$ in Figure \ref{omega}
	does not contain $\boxed{\overline{i+1}}$, case 
	(iv) occurs only in either of cases (a) and (b) below;
	in either case, it is easily verified that 
	$\widehat{\mathfrak{g}}$-dominance is preserved.
	This proves the lemma.\\
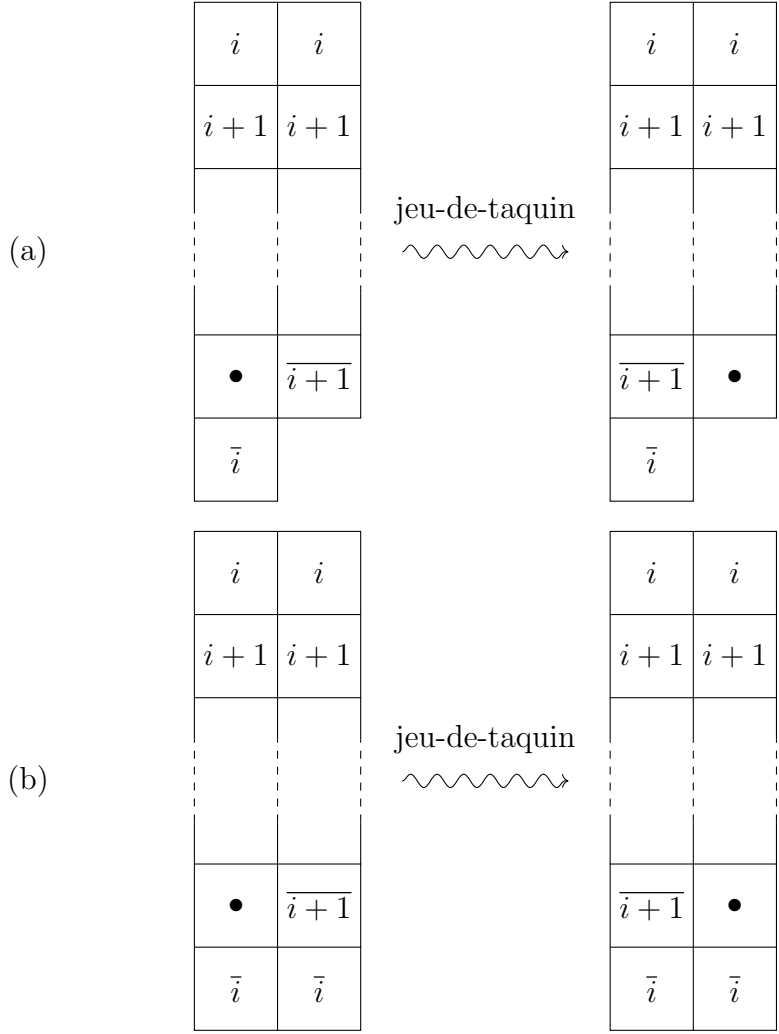
	\begin{tikzpicture}[x=11mm,y=11mm]
			\draw (0,0)--(1,0)--(1,1)--(2,1)--(2,2)--(0,2)--cycle;
			\draw (0,1)--(1,1);
			\draw (1,1)--(1,2);
			\node at (0.5,0.5) {$\overline{i}$};
			\node at (1.5,1.5) {$\overline{i+1}$};
			\node at (0.5,1.5) {$\bullet$};
			\draw (0,2)--(0,2.5);
			\draw (0,4)--(0,3.5);
			\draw[dashed] (0,2.5)--(0,3.5);
			\draw (1,2)--(1,2.5);
			\draw (1,4)--(1,3.5);
			\draw[dashed] (1,2.5)--(1,3.5);
			\draw (2,2)--(2,2.5);
			\draw (2,4)--(2,3.5);
			\draw[dashed] (2,2.5)--(2,3.5);
			\draw (0,4)--(0,6)--(2,6)--(2,4)--cycle;;
			\draw (0,5)--(2,5);
			\draw (1,4)--(1,6);
			\node at (0.5,4.5) {$i+1$};
			\node at (1.5,4.5) {$i+1$};
			\node at (0.5,5.5) {$i$};
			\node at (1.5,5.5) {$i$};
			\draw[->] decorate[decoration=snake] {(2.5,3)--(4.5,3)};
			\node at (3.5,3.5) {jeu-de-taquin};
			
			\draw (5,0)--(6,0)--(6,1)--(7,1)--(7,2)--(5,2)--cycle;
			\draw (5,1)--(6,1);
			\draw (6,1)--(6,2);
			\node at (5.5,0.5) {$\overline{i}$};
			\node at (5.5,1.5) {$\overline{i+1}$};
			\node at (6.5,1.5) {$\bullet$};
			\draw (5,2)--(5,2.5);
			\draw (5,4)--(5,3.5);
			\draw[dashed] (5,2.5)--(5,3.5);
			\draw (6,2)--(6,2.5);
			\draw (6,4)--(6,3.5);
			\draw[dashed] (6,2.5)--(6,3.5);
			\draw (7,2)--(7,2.5);
			\draw (7,4)--(7,3.5);
			\draw[dashed] (7,2.5)--(7,3.5);
			\draw (5,4)--(5,6)--(7,6)--(7,4)--cycle;
			\draw (5,5)--(7,5);
			\draw (6,4)--(6,6);
			\node at (5.5,4.5) {$i+1$};
			\node at (6.5,4.5) {$i+1$};
			\node at (5.5,5.5) {$i$};
			\node at (6.5,5.5) {$i$};
			
			\node at (-2,3) {\textcolor{white}{(b)}};
			\node at (-2,3) {(a)};
			\node at (9,0) {};
			\node at (-3.6,0) {};
			\node at (-2,-0.2) {};
	\end{tikzpicture} \\
	\begin{tikzpicture}[x=11mm,y=11mm]
			\draw (0,0)--(1,0)--(1,1)--(2,1)--(2,2)--(0,2)--cycle;
			\draw (0,1)--(1,1);
			\draw (1,1)--(1,2);
			\node at (0.5,0.5) {$\overline{i}$};
			\node at (1.5,1.5) {$\overline{i+1}$};
			\node at (0.5,1.5) {$\bullet$};
			\node at (1.5,0.5) {$\overline{i}$};
			\draw (1,0)--(2,0)--(2,1);
			\draw (0,2)--(0,2.5);
			\draw (0,4)--(0,3.5);
			\draw[dashed] (0,2.5)--(0,3.5);
			\draw (1,2)--(1,2.5);
			\draw (1,4)--(1,3.5);
			\draw[dashed] (1,2.5)--(1,3.5);
			\draw (2,2)--(2,2.5);
			\draw (2,4)--(2,3.5);
			\draw[dashed] (2,2.5)--(2,3.5);
			\draw (0,4)--(0,6)--(2,6)--(2,4)--cycle;;
			\draw (0,5)--(2,5);
			\draw (1,4)--(1,6);
			\node at (0.5,4.5) {$i+1$};
			\node at (1.5,4.5) {$i+1$};
			\node at (0.5,5.5) {$i$};
			\node at (1.5,5.5) {$i$};
			\draw[->] decorate[decoration=snake] {(2.5,3)--(4.5,3)};
			\node at (3.5,3.5) {jeu-de-taquin};
			
			\draw (5,0)--(6,0)--(6,1)--(7,1)--(7,2)--(5,2)--cycle;
			\draw (5,1)--(6,1);
			\draw (6,1)--(6,2);
			\node at (5.5,0.5) {$\overline{i}$};
			\node at (5.5,1.5) {$\overline{i+1}$};
			\node at (6.5,1.5) {$\bullet$};
			\node at (6.5,0.5) {$\overline{i}$};
			\draw (6,0)--(7,0)--(7,1);
			\draw (5,2)--(5,2.5);
			\draw (5,4)--(5,3.5);
			\draw[dashed] (5,2.5)--(5,3.5);
			\draw (6,2)--(6,2.5);
			\draw (6,4)--(6,3.5);
			\draw[dashed] (6,2.5)--(6,3.5);
			\draw (7,2)--(7,2.5);
			\draw (7,4)--(7,3.5);
			\draw[dashed] (7,2.5)--(7,3.5);
			\draw (5,4)--(5,6)--(7,6)--(7,4)--cycle;
			\draw (5,5)--(7,5);
			\draw (6,4)--(6,6);
			\node at (5.5,4.5) {$i+1$};
			\node at (6.5,4.5) {$i+1$};
			\node at (5.5,5.5) {$i$};
			\node at (6.5,5.5) {$i$};
			
			\node at (-2,3) {\textcolor{white}{(a)}};
			\node at (-2,3) {(b)};
			\node at (-3.6,0) {};
			\node at (9,0) {};
	\end{tikzpicture}
\end{proof}

\begin{lem}\label{gLGP}
	Let $T \in \mathsf{SST}_{2n}(\lambda)$ be a semistandard tableau.
	If 
	$\mathsf{Res}_{1,2;\overline{2},\overline{1}}(T)$
	is a $\widehat{\mathfrak{g}}_1$-dominant tableau and
	$T|_{2,\overline{2}}$
	is a $\widehat{\mathfrak{g}}_{\geq2}$-dominant tableau,
	then $T$ is a $\widehat{\mathfrak{g}}$-dominant tableau.
\end{lem}

\begin{proof}
	It suffices to show that
	$T|_{1,2;\overline{2},\overline{1}}$
	is a $\widehat{\mathfrak{g}}_1$-dominant tableau.
	By the same reasoning as in the proof of the previous lemma,
	we are reduced to showing that
	if a tableau becomes $\widehat{\mathfrak{g}}_1$-dominant
	after the following slide,
	then the original tableau before the slide is also $\widehat{\mathfrak{g}}_1$-dominant.\\
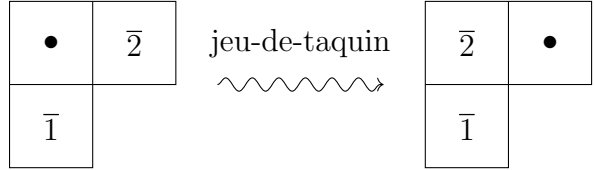
	\begin{center}
		\begin{tikzpicture}[x=11mm,y=11mm]
			\draw (0,0)--(1,0)--(1,1)--(2,1)--(2,2)--(0,2)--cycle;
			\draw (0,1)--(1,1);
			\draw (1,1)--(1,2);
			\node at (0.5,0.5) {$\overline{1}$};
			\node at (1.5,1.5) {$\overline{2}$};
			\node at (0.5,1.5) {$\bullet$};
			\draw[->] decorate[decoration=snake] {(2.5,1)--(4.5,1)};
			\node at (3.5,1.5) {jeu-de-taquin};
			
			\draw (5,0)--(6,0)--(6,1)--(7,1)--(7,2)--(5,2)--cycle;
			\draw (5,1)--(6,1);
			\draw (6,1)--(6,2);
			\node at (5.5,0.5) {$\overline{1}$};
			\node at (5.5,1.5) {$\overline{2}$};
			\node at (6.5,1.5) {$\bullet$};
			
			\node at (-2,1) {\textcolor{white}{(iii)}};
			\node at (9,0) {};
		\end{tikzpicture}
	\end{center}
	Since
	$\mathsf{Res}_{1,2;\overline{2},\overline{1}}(T)$ and
	$T|_{2,\overline{2}}$ are $\widehat{\mathfrak{g}}_{\geq2}$-dominant,
	the slide above occurs only in either of cases (a) and (b) below.
	
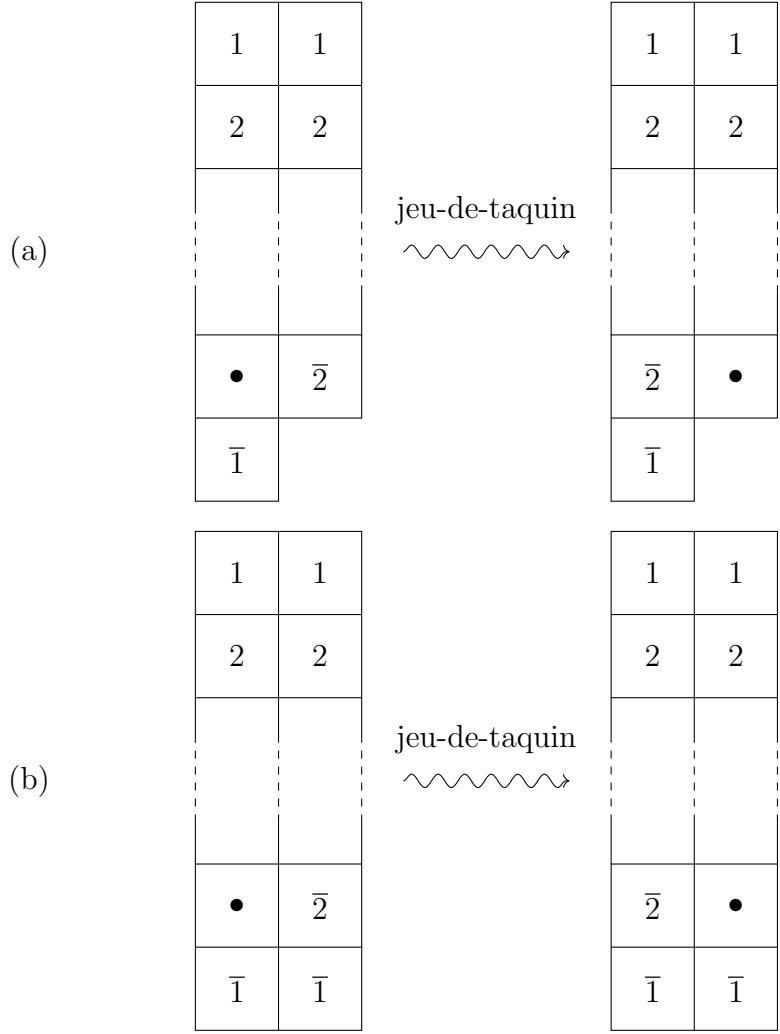
\begin{center}
	\begin{tikzpicture}[x=11mm,y=11mm]
			\draw (0,0)--(1,0)--(1,1)--(2,1)--(2,2)--(0,2)--cycle;
			\draw (0,1)--(1,1);
			\draw (1,1)--(1,2);
			\node at (0.5,0.5) {$\overline{1}$};
			\node at (1.5,1.5) {$\overline{2}$};
			\node at (0.5,1.5) {$\bullet$};
			\draw (0,2)--(0,2.5);
			\draw (0,4)--(0,3.5);
			\draw[dashed] (0,2.5)--(0,3.5);
			\draw (1,2)--(1,2.5);
			\draw (1,4)--(1,3.5);
			\draw[dashed] (1,2.5)--(1,3.5);
			\draw (2,2)--(2,2.5);
			\draw (2,4)--(2,3.5);
			\draw[dashed] (2,2.5)--(2,3.5);
			\draw (0,4)--(0,6)--(2,6)--(2,4)--cycle;;
			\draw (0,5)--(2,5);
			\draw (1,4)--(1,6);
			\node at (0.5,4.5) {$2$};
			\node at (1.5,4.5) {$2$};
			\node at (0.5,5.5) {$1$};
			\node at (1.5,5.5) {$1$};
			\draw[->] decorate[decoration=snake] {(2.5,3)--(4.5,3)};
			\node at (3.5,3.5) {jeu-de-taquin};
			
			\draw (5,0)--(6,0)--(6,1)--(7,1)--(7,2)--(5,2)--cycle;
			\draw (5,1)--(6,1);
			\draw (6,1)--(6,2);
			\node at (5.5,0.5) {$\overline{1}$};
			\node at (5.5,1.5) {$\overline{2}$};
			\node at (6.5,1.5) {$\bullet$};
			\draw (5,2)--(5,2.5);
			\draw (5,4)--(5,3.5);
			\draw[dashed] (5,2.5)--(5,3.5);
			\draw (6,2)--(6,2.5);
			\draw (6,4)--(6,3.5);
			\draw[dashed] (6,2.5)--(6,3.5);
			\draw (7,2)--(7,2.5);
			\draw (7,4)--(7,3.5);
			\draw[dashed] (7,2.5)--(7,3.5);
			\draw (5,4)--(5,6)--(7,6)--(7,4)--cycle;
			\draw (5,5)--(7,5);
			\draw (6,4)--(6,6);
			\node at (5.5,4.5) {$2$};
			\node at (6.5,4.5) {$2$};
			\node at (5.5,5.5) {$1$};
			\node at (6.5,5.5) {$1$};
			
			\node at (-2,3) {\textcolor{white}{(b)}};
			\node at (-2,3) {(a)};
			\node at (9,0) {};
			\node at (-2,-0.2) {};
		\end{tikzpicture} \\
		\begin{tikzpicture}[x=11mm,y=11mm]
			\draw (0,0)--(1,0)--(1,1)--(2,1)--(2,2)--(0,2)--cycle;
			\draw (0,1)--(1,1);
			\draw (1,1)--(1,2);
			\node at (0.5,0.5) {$\overline{1}$};
			\node at (1.5,1.5) {$\overline{2}$};
			\node at (0.5,1.5) {$\bullet$};
			\node at (1.5,0.5) {$\overline{1}$};
			\draw (1,0)--(2,0)--(2,1);
			\draw (0,2)--(0,2.5);
			\draw (0,4)--(0,3.5);
			\draw[dashed] (0,2.5)--(0,3.5);
			\draw (1,2)--(1,2.5);
			\draw (1,4)--(1,3.5);
			\draw[dashed] (1,2.5)--(1,3.5);
			\draw (2,2)--(2,2.5);
			\draw (2,4)--(2,3.5);
			\draw[dashed] (2,2.5)--(2,3.5);
			\draw (0,4)--(0,6)--(2,6)--(2,4)--cycle;;
			\draw (0,5)--(2,5);
			\draw (1,4)--(1,6);
			\node at (0.5,4.5) {$2$};
			\node at (1.5,4.5) {$2$};
			\node at (0.5,5.5) {$1$};
			\node at (1.5,5.5) {$1$};
			\draw[->] decorate[decoration=snake] {(2.5,3)--(4.5,3)};
			\node at (3.5,3.5) {jeu-de-taquin};
			
			\draw (5,0)--(6,0)--(6,1)--(7,1)--(7,2)--(5,2)--cycle;
			\draw (5,1)--(6,1);
			\draw (6,1)--(6,2);
			\node at (5.5,0.5) {$\overline{1}$};
			\node at (5.5,1.5) {$\overline{2}$};
			\node at (6.5,1.5) {$\bullet$};
			\node at (6.5,0.5) {$\overline{1}$};
			\draw (6,0)--(7,0)--(7,1);
			\draw (5,2)--(5,2.5);
			\draw (5,4)--(5,3.5);
			\draw[dashed] (5,2.5)--(5,3.5);
			\draw (6,2)--(6,2.5);
			\draw (6,4)--(6,3.5);
			\draw[dashed] (6,2.5)--(6,3.5);
			\draw (7,2)--(7,2.5);
			\draw (7,4)--(7,3.5);
			\draw[dashed] (7,2.5)--(7,3.5);
			\draw (5,4)--(5,6)--(7,6)--(7,4)--cycle;
			\draw (5,5)--(7,5);
			\draw (6,4)--(6,6);
			\node at (5.5,4.5) {$2$};
			\node at (6.5,4.5) {$2$};
			\node at (5.5,5.5) {$1$};
			\node at (6.5,5.5) {$1$};
			
			\node at (-2,3) {\textcolor{white}{(a)}};
			\node at (-2,3) {(b)};
			\node at (9,0) {};
			\node at (-2,-0.2) {};
		\end{tikzpicture}
		\end{center}
	In either case, it is easily verified that 
	if a tableau becomes 
	$\widehat{\mathfrak{g}}_1$-dominant after the slide above,
	then the original tableau before the slide is also
	$\widehat{\mathfrak{g}}_1$-dominant.
	Therefore,
	$T|_{1,2;\overline{2},\overline{1}}$
	is $\widehat{\mathfrak{g}}_1$-dominant
	since $\widehat{\mathfrak{g}}_1$-dominance is preserved
	in the process of getting back from $\mathsf{Res}_{1,2;\overline{2},\overline{1}}(T)$ to $T|_{1,2;\overline{2},\overline{1}}$.   
	This is what we had to prove. This proves the lemma.
\end{proof}

\subsection{The local-global principle for $\mathfrak{k}$-highest weight tableaux}
In this subsection,
$J$ denotes either of the following two subsets of $I$:
\begin{equation}
	J = \left\{1,2,3\right\},
	\left\{3,4,5,\ldots,2n-1\right\}.
\end{equation}
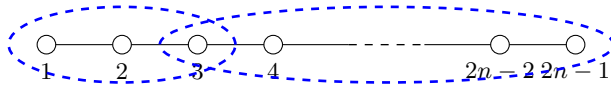
\begin{center}
	\begin{tikzpicture}[x=5mm,y=5mm]
		\draw (0,-4)--(2,-4);
		\draw (2,-4)--(4,-4);
		\draw (4,-4)--(6,-4);
		\draw (6,-4)--(8,-4);
		\draw[dashed,line width=0.5pt] (8,-4)--(10,-4);
		\draw (10,-4)--(12,-4);
		\draw (12,-4)--(14,-4);
		\draw[fill=white] (0,-4) circle[radius=3.5pt];
		\draw[fill=white] (2,-4) circle[radius=3.5pt];
		\draw[fill=white] (4,-4) circle[radius=3.5pt];
		\draw[fill=white] (6,-4) circle[radius=3.5pt];
		\draw[fill=white] (12,-4) circle[radius=3.5pt];
		\draw[fill=white] (14,-4) circle[radius=3.5pt];
		\node (l) at (0,-4.7) {\scriptsize $1$};
		\node (m) at (2,-4.7) {\scriptsize $2$};
		\node (n) at (4,-4.7) {\scriptsize $3$};
		\node (o) at (6,-4.7) {\scriptsize $4$};
		\node (p) at (12,-4.7) {\scriptsize $2n-2$};
		\node (q) at (14,-4.7) {\scriptsize $2n-1$};
		\draw[dashed, color=blue, line width=1pt] (2,-4) circle[x radius=3,y radius=1];
		\draw[dashed, color=blue, line width=1pt] (9,-4) circle[x radius=6,y radius=1];
	\end{tikzpicture}
\end{center}
Let $\mathbf{U}_J$ denote the Levi subalgebra of $\mathbf{U}$
generated by $E_i,F_i,K_{\pm h_i}$, $i \in J$, and let 
$\mathbf{U}^{\imath}_J$ denote the $\mathbb{C}(q)$-subalgebra of
$\mathbf{U}^{\imath}$ generated by
$E_i,F_i$, $K_{\pm h_i}$, $i \in J \cap I_{\bullet}$, and $B_j,A_j$, $j \in J \cap I_{\circ}$.
We set
\begin{equation}
	\begin{aligned}
		& P^{\vee}_J := \sum_{j \in J} \mathbb{Z}h_j,\quad
		P_J := \mathrm{Hom}_{\mathbb{Z}}\left(P^{\vee}_J,\mathbb{Z}\right), \\
		& P_J^+ := \setlr{\nu \in P_J}{\langle\nu,h_j\rangle \in \mathbb{Z}_{\geq 0}, \,\, j \in J}.
	\end{aligned}
\end{equation}
Also, let $L^J_q(\nu)$ denote the irreducible highest weight module of highest weight $\nu \in P^+_J$
over $\mathbf{U}_J$.
For each $T \in \mathsf{SST}_{2n}(\lambda)$, let 
$\mathsf{wt}_J(T)$ denote the image of $\mathsf{wt}(T) \in P$
under the natural restriction map $P \to P_J$.
In the following,
we write $\mathsf{Res}_J$ for $\mathsf{Res}_{1,4}$ or $\mathsf{Res}_{3,\overline{1}}$, 
and write $\,\cdot\,\,\,|_{J}$ for $\,\cdot\,\,\,|_{1,4}$ or $\,\cdot\,\,\,|_{3,\overline{1}}$.
We also denote by $\mathsf{SST}_J$ the set of all semistandard tableau of normal shape
with all entries in $\{1,\ldots,4\}$ or $\{3,\ldots,2n\}$. 
\begin{lem}
	There exists an isomorphism of $\mathbf{U}_J$-modules: 
	\begin{equation}
		\mathsf{Res}_{\lambda} : \mathsf{Res}^{\mathbf{U}}_{\mathbf{U}_{J}} L_q(\lambda) \stackrel{\sim}{\longrightarrow} \bigoplus_{\substack{T \in \mathsf{SST}_{2n}(\lambda) \\ \widetilde{e}_j T \,=\, 0\, (j \in J)}} L^{J}_q\left(\mathsf{wt}_{J}(T)\right)
	\end{equation}
	that sends
	$b_S$ for $S \in \mathsf{SST}_{2n}(\lambda)$ to
	$b_{\mathsf{Res}_{J}(S)}$ at $q=\infty$.
\end{lem}

\begin{proof}
	By the branching rule for Levi subalgebras,
	we have an isomorphism of $\mathbf{U}_{J}$-modules:
	\begin{equation}
		\psi : \mathsf{Res}^{\mathbf{U}}_{\mathbf{U}_J} L_q(\lambda) \stackrel{\sim}{\longrightarrow} \bigoplus_{\substack{T \in \mathsf{SST}_{2n}(\lambda) \\ \widetilde{e}_j T \,=\, 0\, (j \in J)}} L^{J}_q\left(\mathsf{wt}_{J}(T)\right).
	\end{equation}
	For each $S\in\mathsf{SST}_{2n}(\lambda)$,
	there exists $T_S \in \mathsf{SST}_{J}$
	such that $\psi(b_S)=b_{T_S}$ at $q=\infty$
	by the uniqueness of crystal basis.
	We will show that
	$T_S = \mathsf{Res}_J(S)$.
	
	First, we consider the case that
	$\widetilde{e}_j S =0$ for all $j \in J$.
	In this case, we have 
	$\mathsf{wt}_{J}(T_S) = \mathsf{wt}_{J}(S)$, and 
	$\widetilde{e}_j T_S =0$ for all $j \in J$.
	Also, we have 
	$\mathsf{wt}_{J}\left(\mathsf{Res}_{J}(S)\right) = \mathsf{wt}_{J}(S)$.
	In addition, 
	since
	$w_{r}\left(\mathsf{Res}_{J}(S)\right) \cong_{\mathsf{K}} w_{r}(S|_{J})$,
	we see that
	$ \widetilde{e}_j w_{r}\left(\mathsf{Res}_{J}(S)\right) = 0 $
	for all $j \in J$ by \cite[Theorem 4.2]{Kwo}.
	Therefore, we deduce that 
	$\widetilde{e}_j \mathsf{Res}_{J}(S) =0$ for all $j \in J$.
	Hence it follows that $T_S = \mathsf{Res}_{J}(S)$.

	Next, we consider the case that
	there exists $j \in J$ such that $\widetilde{e}_j\,S \neq0$.
	In this case, there exist
	$r\in\mathbb{N}, j_1,\dots,j_r\in J, S'\in\mathsf{SST}_{2n}(\lambda)$
	such that
	\begin{equation}
		\begin{aligned}
			 & S = \widetilde{f}_{j_1}\cdots\widetilde{f}_{j_r} S', & 
			 & \widetilde{e}_j S' =0, \,\, j \in J.
		\end{aligned}
	\end{equation}
	Since we have $\widetilde{e}_j\,S'=0, \,\, j \in J$,
	we see that
	$T_S = \widetilde{f}_{j_1}\cdots\widetilde{f}_{j_r} T_S' = \widetilde{f}_{j_1}\cdots\widetilde{f}_{j_r} \mathsf{Res}_{J}(S')$. 
	Hence it follows that  
	\begin{equation}
		\begin{aligned}
			w_{r}\left(T_S\right)
			= w_{r}\left(\widetilde{f}_{j_1}\cdots\widetilde{f}_{j_r} \mathsf{Res}_{J}(S')\right).
		\end{aligned}
	\end{equation}
	Also, 
	by \cite[Theorem 4.2]{Kwo}, we have
	\begin{equation}
		\begin{aligned}
			\widetilde{f}_{j_1}\cdots\widetilde{f}_{j_r}w_{r}\left(\mathsf{Res}_{J}(S')\right)
			\cong_{\mathsf{K}} \widetilde{f}_{j_1}\cdots\widetilde{f}_{j_r} w_{r}\left(S'|_{J}\right). \label{7.9}
		\end{aligned}
	\end{equation}
	The left-hand side of (\ref{7.9}) can be written as: 
	\begin{equation}
		\begin{aligned}
			\widetilde{f}_{j_1}\cdots\widetilde{f}_{j_r} w_{r}\left(S'|_{J}\right)
			=w_{r}\left(\widetilde{f}_{j_1}\cdots\widetilde{f}_{j_r}(S'|_{J})\right)
			=w_{r}\left(S|_{J}\right).
		\end{aligned}
	\end{equation}
	Thus, we deduce that 
	$w_{r}(T_S)\cong_{\mathsf{K}}w_{r}\left(S|_{J}\right)$. 
	Therefore, by \cite[Chapter 2, Corollary 1]{F}, 
	we can conclude that $T_S = \mathsf{Res}_{J}(S)$. 
	This proves the lemma. 
\end{proof}

\begin{de}
	Let $1 \leq a < b \leq 2n$ be integers,
	and $S \in \mathsf{SST}_{2n}(\lambda)$ a semistandard tableau with
	all entries in $[a,b]$.
	We define a new tableau $\mathsf{P}^{A\mathrm{II}}_{a,b}(S)$
	as follows: 
	\begin{enumerate}
		\item Decrease each entry of $S$ by $a-1$.
		\item Apply the operation $\mathsf{P}^{A\mathrm{II}}$ in the case $n=b-a+1$ to the tableau obtained in (1).
		\item Finally, increase each entry of the tableau obtained in (2) by $a-1$.
	\end{enumerate}
\end{de}

\begin{lem}\label{formula}
	For all $T \in \mathsf{SST}_{2n}(\lambda)$,
	the following equality holds: 
	\begin{equation}
		\mathsf{P}^{A\mathrm{II}}_{J} \left( \mathsf{Res}_{J}\,\,\mathsf{P}^{A\mathrm{II}}(T) \right) = \mathsf{P}^{A\mathrm{II}}_{J}\left(\mathsf{Res}_{J}(T)\right). \label{7.11}
	\end{equation}
\end{lem}

\begin{proof}
	Let $T_{L}$ (resp., $T_{R}$) denote the tableau on the left-hand (resp., right-hand) side of (\ref{7.11}); 
	also, let $\mu_L$ and $\mu_R$ denote the shapes of the tableaux $T_{L}$ and $T_{R}$, respectively. 
	We now set
	\begin{equation}
		N_{\lambda} := \min\setlr{N \in \mathbb{Z}_{\geq0}}{\mathsf{suc}^{N+1}(S) = \mathsf{suc}^N(S), \,\,S \in \mathsf{SST}_{2n}(\lambda)}.
	\end{equation}
	We have a $\mathbf{U}^{\imath}_{J}$-module homomorphism
	$\psi : \widetilde{L}_q^{\imath,J}(\mu_{R})\to\widetilde{L}_q^{\imath,J}(\mu_{L})$
	that makes the following diagram commutative: 
	\begin{center}
		\begin{tikzpicture}
			\node (A) at (0,0) {$L_q(\lambda)$};
			\node (B) at (-4.2,-1) {$\bigoplus_{\nu} L_q(\nu)$};
			\node (C) at (-4.2,-3) {$\bigoplus_{\nu} \bigoplus_{\xi_{\nu}} L_q^{J}(\xi_{\nu})$};
			\node (D) at (-4.2,-5) {$\bigoplus_{\nu} \bigoplus_{\xi_{\nu}} \bigoplus_{\mu_{\xi_{\nu}}} \widetilde{L}_q^{\imath,J}(\mu_{\xi_{\nu}})\otimes\mathbb{C}(q)\mathsf{Rec}_4(\xi_{\nu}/\mu_{\xi_{\nu}})$};
			\node (G) at (-4.2,-7) {$\widetilde{L}_q^{\imath,J}(\mu_L)$};
			\node (H) at (4.2,-7) {$\widetilde{L}_q^{\imath,J}(\mu_R)$};
			\node (E) at (4.2,-2) {$\bigoplus_{\xi_{\lambda}}L_q^{J}(\xi_{\lambda})$};
			\node (F) at (4.2,-4.5) {$\bigoplus_{\xi_{\lambda}} \bigoplus_{\mu_{\xi_{\lambda}}} \widetilde{L}_q^{\imath,J}(\mu_{\xi_{\lambda}})\otimes\mathbb{C}(q)\mathsf{Rec}_4(\xi_{\lambda}/\mu_{\xi_{\lambda}})$};
			\draw[->] (A) -- node[auto=right] {$\mathsf{suc}^{N_{\lambda}}$} (B);
			\draw[->] (B) -- node[auto=right] {$\bigoplus_{\nu} \mathsf{Res}_{\nu}$} (C);
			\draw[->] (C) -- node[auto=right] {$\bigoplus_{\nu} \bigoplus_{\xi_{\nu}} \mathsf{LR}^{A\mathrm{II}}$} (D);
			\draw[->] (A) -- node[auto=left] {$\mathsf{Res}_{\lambda}$} (E);
			\draw[->] (E) -- node[auto=left] {$\bigoplus_{\xi_{\lambda}} \mathsf{LR}^{A\mathrm{II}}$} (F);
			\draw[->>] (D)--(G);
			\draw[{Hooks[right]}-{>}] (H)--(F);
			\draw[->,dashed] (H)-- node[auto=left]{$\psi$} (G);
		\end{tikzpicture}
	\end{center}
	From the construction of $\psi$,
	it follows that
	$\psi\left( b^{\imath}_{T_R} \right) = b^{\imath}_{T_L}$ at $q=\infty$.
	In particular, we have 
	\begin{equation}
		\mathrm{Hom}_{\mathbf{U}^{\imath}_{J}}\left(\widetilde{L}_q^{\imath,J}(\mu_R),\widetilde{L}_q^{\imath,J}(\mu_L)\right)\neq \{0\}.
	\end{equation}
	Therefore,
	by Schur's lemma, we see that $\mu_L = \mu_R$
	(see \cite[Chapter 5.5]{J}).
  Moreover,
	there exists $c(q)\in\mathbf{A}_{\infty}$
	such that
	$\psi=c(q)\mathrm{id}$ with $\lim_{q\to\infty}c(q)=1$.
	By taking the crystal limit as 
	$q \to \infty$ in $L_q^{J}(\mu_L)=L_q^{J}(\mu_R)$, we deduce that 
	$b_{T_L} = b_{T_R}$ from the construction of $\mathbf{B}^{\imath}(\mu_L)$
	in \cite[Proposition 6.4.1]{Wat1}.
	Hence, if the tableau $T_{L}$ were not equal to the tableau $T_{R}$, 
	then we would get a contradiction, 
	since 
	$b_{T_{L}}$ and $b_{T_{R}}$ 
	are linearly independent over $\mathbb{C}$.
	Thus, equation (\ref{7.11}) is proved. This proves the lemma. 
\end{proof}

\begin{lem}\label{kLGP}
	Let
	$S \in \mathsf{SST}_{2n}(\lambda)$
	be a semistandard tableau.
	\begin{enumerate}
		\item The followings are equivalent:
		      \begin{enumerate}
			      \item[(i)] $S$ is a $\mathfrak{k}$-highest weight tableau;
			      \item[(ii)] we have 
			            \begin{center}
				            \begin{tikzpicture}[x=5mm,y=5mm]
					            \draw (0,0)--(8,0)--(8,1)--(0,1)--cycle;
					            \draw (0,1)--(0,2)--(10,2)--(10,1)--(8,1);
					            \node at (0.5,0.5) {$a_2$};
					            \node at (0.5,1.5) {$a_1$};
					            \node at (7.5,0.5) {$a_2$};
					            \node at (9.5,1.5) {$a_1$};
					            \draw[dotted, line width=1pt] (1.2,0.5)--(1.8,0.5);
											\draw[dotted, line width=1pt] (6.2,0.5)--(6.8,0.5);
					            \draw[dotted, line width=1pt] (1.2,1.5)--(1.8,1.5);
											\draw[dotted, line width=1pt] (8.2,1.5)--(8.8,1.5);
					            \node at (-3.5,1) {$\mathsf{P}^{A\mathrm{II}}_{1,4} \left( \mathsf{Res}_{1,4}\,\, S \right)  =$};
					            \node at (19.6,1) {};
											\node at (10.5,0.5) {,};
				            \end{tikzpicture}
				            
				            \vspace{5pt}
				            
				            \begin{tikzpicture}[x=5mm,y=5mm]
					            \draw (0,0)--(15,0)--(15,1)--(0,1)--cycle;
					            \draw (0,3)--(18,3)--(18,4)--(0,4)--cycle;
					            \draw (0,4)--(0,5)--(19,5)--(19,4)--(18,4);
											\draw (0,1)--(0,3);
											\draw (15,1)--(16,1)--(16,2)--(17,2)--(17,3);
					            \draw[dotted, line width=1pt] (8,1.7)--(8,2.3);
					            \node (a) at (0.5,0.5) {$a_n$};
					            \node (b) at (0.5,3.5) {$a_3$};
					            \node (c) at (0.5,4.5) {$a_2$};
					            \node (d) at (14.5,0.5) {$a_n$};
					            \node (e) at (17.5,3.5) {$a_3$};
					            \node (f) at (18.5,4.5) {$a_2$};
					            \draw[dotted, line width=1pt] (1.2,0.5)--(1.8,0.5);
											\draw[dotted, line width=1pt] (13.2,0.5)--(13.8,0.5);
					            \draw[dotted, line width=1pt] (1.2,3.5)--(1.8,3.5);
											\draw[dotted, line width=1pt] (16.2,3.5)--(16.8,3.5);
					            \draw[dotted, line width=1pt] (1.2,4.5)--(1.8,4.5);
											\draw[dotted, line width=1pt] (17.2,4.5)--(17.8,4.5);
					            \node at (-3.5,2.5) {$\mathsf{P}^{A\mathrm{II}}_{3,\overline{1}} \left( \mathsf{Res}_{3,\overline{1}}\,\,S \right)  =$};
											\node at (19.5,2) {.};
				            \end{tikzpicture}
			            \end{center}
		      \end{enumerate}
		\item The followings are equivalent: 
		      \begin{enumerate}
			      \item[(iii)] $S$ is a $\mathfrak{k}$-lowest weight tableau; 
			      \item[(iv)] we have
			      \begin{center}
							\begin{tikzpicture}[x=5mm,y=5mm]
								\draw (0,0)--(8,0)--(8,1)--(0,1)--cycle;
								\draw (0,1)--(0,2)--(10,2)--(10,1)--(8,1);
								\node at (0.5,0.5) {$b_2$};
								\node at (0.5,1.5) {$b_1$};
								\node at (7.5,0.5) {$b_2$};
								\node at (9.5,1.5) {$b_1$};
								\draw[dotted, line width=1pt] (1.2,0.5)--(1.8,0.5);
								\draw[dotted, line width=1pt] (6.2,0.5)--(6.8,0.5);
								\draw[dotted, line width=1pt] (1.2,1.5)--(1.8,1.5);
								\draw[dotted, line width=1pt] (8.2,1.5)--(8.8,1.5);
								\node at (-3.5,1) {$\mathsf{P}^{A\mathrm{II}}_{1,4} \left( \mathsf{Res}_{1,4}\,\, S \right)  =$};
					      \node at (19.6,1) {};
								\node at (10.5,0.5) {,};
							\end{tikzpicture}
							
							\vspace{5pt}
							
							\begin{tikzpicture}[x=5mm,y=5mm]
								\draw (0,0)--(15,0)--(15,1)--(0,1)--cycle;
								\draw (0,3)--(18,3)--(18,4)--(0,4)--cycle;
								\draw (0,4)--(0,5)--(19,5)--(19,4)--(18,4);
								\draw (0,1)--(0,3);
								\draw (15,1)--(16,1)--(16,2)--(17,2)--(17,3);
								\draw[dotted, line width=1pt] (8,1.7)--(8,2.3);
								\node (a) at (0.5,0.5) {$b_n$};
								\node (b) at (0.5,3.5) {$b_3$};
								\node (c) at (0.5,4.5) {$b_2$};
								\node (d) at (14.5,0.5) {$b_n$};
								\node (e) at (17.5,3.5) {$b_3$};
								\node (f) at (18.5,4.5) {$b_2$};
								\draw[dotted, line width=1pt] (1.2,0.5)--(1.8,0.5);
								\draw[dotted, line width=1pt] (13.2,0.5)--(13.8,0.5);
								\draw[dotted, line width=1pt] (1.2,3.5)--(1.8,3.5);
								\draw[dotted, line width=1pt] (16.2,3.5)--(16.8,3.5);
								\draw[dotted, line width=1pt] (1.2,4.5)--(1.8,4.5);
								\draw[dotted, line width=1pt] (17.2,4.5)--(17.8,4.5);
								\node at (-3.5,2.5) {$\mathsf{P}^{A\mathrm{II}}_{3,\overline{1}} \left( \mathsf{Res}_{3,\overline{1}}\,\,S \right)  =$};
								\node at (19.5,2) {.};
							\end{tikzpicture}
						\end{center}
		      \end{enumerate}
	\end{enumerate}
\end{lem}

\begin{proof}
	We prove the implication (i)$\implies$(ii). 
	Since $S$ is a $\mathfrak{k}$-highest weight tableau, 
	by Definition \ref{5.13}, we have
	\begin{center}
		\begin{tikzpicture}[x=5mm,y=5mm]
			\draw (0,0)--(15,0)--(15,1)--(0,1)--cycle;
			\draw (0,3)--(18,3)--(18,4)--(0,4)--cycle;
			\draw (0,4)--(0,5)--(19,5)--(19,4)--(18,4);
			\draw (0,1)--(0,3);
			\draw (15,1)--(16,1)--(16,2)--(17,2)--(17,3);
			\draw[dotted, line width=1pt] (8,1.7)--(8,2.3);
			\node (a) at (0.5,0.5) {$a_n$};
			\node (b) at (0.5,3.5) {$a_2$};
			\node (c) at (0.5,4.5) {$a_1$};
			\node (d) at (14.5,0.5) {$a_n$};
			\node (e) at (17.5,3.5) {$a_2$};
			\node (f) at (18.5,4.5) {$a_1$};
			\draw[dotted, line width=1pt] (1.2,0.5)--(1.8,0.5);
			\draw[dotted, line width=1pt] (13.2,0.5)--(13.8,0.5);
			\draw[dotted, line width=1pt] (1.2,3.5)--(1.8,3.5);
			\draw[dotted, line width=1pt] (16.2,3.5)--(16.8,3.5);
			\draw[dotted, line width=1pt] (1.2,4.5)--(1.8,4.5);
			\draw[dotted, line width=1pt] (17.2,4.5)--(17.8,4.5);
			\node at (-2.3,2.5) {$\mathsf{P}^{A\mathrm{II}}\left( S \right)  =$};
			\node at (-8,2.5) {(7.13)};
			\node at (19.5,2) {.};
			\node at (25,2.5) {};
		\end{tikzpicture}\label{7.13}
	\end{center}
	By applying $\mathsf{Res}_{1,4}$ and
	$\mathsf{Res}_{3,\overline{1}}$	to (7.13),
	we get
	\begin{center}
		\begin{tikzpicture}[x=5mm,y=5mm]
			\draw (0,0)--(8,0)--(8,1)--(0,1)--cycle;
			\draw (0,1)--(0,2)--(10,2)--(10,1)--(8,1);
			\node at (0.5,0.5) {$a_2$};
			\node at (0.5,1.5) {$a_1$};
			\node at (7.5,0.5) {$a_2$};
			\node at (9.5,1.5) {$a_1$};
			\draw[dotted, line width=1pt] (1.2,0.5)--(1.8,0.5);
			\draw[dotted, line width=1pt] (6.2,0.5)--(6.8,0.5);
			\draw[dotted, line width=1pt] (1.2,1.5)--(1.8,1.5);
			\draw[dotted, line width=1pt] (8.2,1.5)--(8.8,1.5);
			\node at (-3.7,1) {$\mathsf{Res}_{1,4}\left(\mathsf{P}^{A\mathrm{II}}(S)\right)=$};
			\node at (19.6,1) {};
			\node at (10.5,0.5) {,};
			\node at (-10,1) {(7.14)};
			\node at (20,2.5) {};
		\end{tikzpicture}\label{7.14}
		
		\vspace{5pt}
		
		\begin{tikzpicture}[x=5mm,y=5mm]
			\draw (0,0)--(15,0)--(15,1)--(0,1)--cycle;
			\draw (0,3)--(18,3)--(18,4)--(0,4)--cycle;
			\draw (0,4)--(0,5)--(19,5)--(19,4)--(18,4);
			\draw (0,1)--(0,3);
			\draw (15,1)--(16,1)--(16,2)--(17,2)--(17,3);
			\draw[dotted, line width=1pt] (8,1.7)--(8,2.3);
			\node (a) at (0.5,0.5) {$a_n$};
			\node (b) at (0.5,3.5) {$a_3$};
			\node (c) at (0.5,4.5) {$a_2$};
			\node (d) at (14.5,0.5) {$a_n$};
			\node (e) at (17.5,3.5) {$a_3$};
			\node (f) at (18.5,4.5) {$a_2$};
			\draw[dotted, line width=1pt] (1.2,0.5)--(1.8,0.5);
			\draw[dotted, line width=1pt] (13.2,0.5)--(13.8,0.5);
			\draw[dotted, line width=1pt] (1.2,3.5)--(1.8,3.5);
			\draw[dotted, line width=1pt] (16.2,3.5)--(16.8,3.5);
			\draw[dotted, line width=1pt] (1.2,4.5)--(1.8,4.5);
			\draw[dotted, line width=1pt] (17.2,4.5)--(17.8,4.5);
			\node at (-3.7,2.5) {$\mathsf{Res}_{3,\overline{1}}\left(\mathsf{P}^{A\mathrm{II}}(S)\right)=$};
			\node at (19.5,2) {.};
			\node at (-10,2.5) {(7.15)};
			\node at (20,2.5) {};
		\end{tikzpicture}\label{7.15}
	\end{center}
	Furthermore,
	by applying $\mathsf{P}^{A\mathrm{II}}_{1,4}$ to (7.14)
	and applying $\mathsf{P}^{A\mathrm{II}}_{3,\overline{1}}$ to (7.15),
	we obtain 
	\begin{center}
		\begin{tikzpicture}[x=5mm,y=5mm]
			\draw (0,0)--(8,0)--(8,1)--(0,1)--cycle;
			\draw (0,1)--(0,2)--(10,2)--(10,1)--(8,1);
			\node at (0.5,0.5) {$a_2$};
			\node at (0.5,1.5) {$a_1$};
			\node at (7.5,0.5) {$a_2$};
			\node at (9.5,1.5) {$a_1$};
			\draw[dotted, line width=1pt] (1.2,0.5)--(1.8,0.5);
			\draw[dotted, line width=1pt] (6.2,0.5)--(6.8,0.5);
			\draw[dotted, line width=1pt] (1.2,1.5)--(1.8,1.5);
			\draw[dotted, line width=1pt] (8.2,1.5)--(8.8,1.5);
			\node at (-4.75,1) {$\mathsf{P}^{A\mathrm{II}}_{1,4}\left(\mathsf{Res}_{1,4}\left(\mathsf{P}^{A\mathrm{II}}(S)\right)\right)=$};
			\node at (19.6,1) {};
			\node at (10.5,0.5) {,};
			\node at (20,2.5) {};
		\end{tikzpicture}
		
		\vspace{5pt}
		
		\begin{tikzpicture}[x=5mm,y=5mm]
			\draw (0,0)--(15,0)--(15,1)--(0,1)--cycle;
			\draw (0,3)--(18,3)--(18,4)--(0,4)--cycle;
			\draw (0,4)--(0,5)--(19,5)--(19,4)--(18,4);
			\draw (0,1)--(0,3);
			\draw (15,1)--(16,1)--(16,2)--(17,2)--(17,3);
			\draw[dotted, line width=1pt] (8,1.7)--(8,2.3);
			\node (a) at (0.5,0.5) {$a_n$};
			\node (b) at (0.5,3.5) {$a_3$};
			\node (c) at (0.5,4.5) {$a_2$};
			\node (d) at (14.5,0.5) {$a_n$};
			\node (e) at (17.5,3.5) {$a_3$};
			\node (f) at (18.5,4.5) {$a_2$};
			\draw[dotted, line width=1pt] (1.2,0.5)--(1.8,0.5);
			\draw[dotted, line width=1pt] (13.2,0.5)--(13.8,0.5);
			\draw[dotted, line width=1pt] (1.2,3.5)--(1.8,3.5);
			\draw[dotted, line width=1pt] (16.2,3.5)--(16.8,3.5);
			\draw[dotted, line width=1pt] (1.2,4.5)--(1.8,4.5);
			\draw[dotted, line width=1pt] (17.2,4.5)--(17.8,4.5);
			\node at (-4.75,2.5) {$\mathsf{P}^{A\mathrm{II}}_{1,4}\left(\mathsf{Res}_{3,\overline{1}}\left(\mathsf{P}^{A\mathrm{II}}(S)\right)\right)=$};
			\node at (19.5,2) {.};
			\node at (20,2.5) {};
		\end{tikzpicture}
	\end{center}
	Therefore, we deduce assertion (ii) by Lemma \ref{formula}.
	The proof of the implication (iii)$\implies$(iv) is similar. \\
	
  We prove the reverse implication (ii)$\implies$(i). 
	We will prove equation (7.16) below 
	for $k=2,3,\ldots,n$ by induction on $k$,
	where the entries of the boxes contained in the area $\Omega_1$
	lie in $[2k+1,\overline{1}]$;	by definition \ref{5.13},
	assertion (i) follows from the case $k=n$.
	\begin{center}
		\begin{tikzpicture}[x=5mm,y=5mm]
			\fill[fill=blue!20] (0,0)--(0,-4)--(8,-4)--(8,-3)--(16,-3)--(16,-2)--(18,-2)--(18,0)--(19,0)--(19,2)--(21,2)--(21,3)--(22,3)--(22,5)--(19,5)--(19,4)--(18,4)--(18,3)--(17,3)--(17,2)--(16,2)--(16,1)--(15,1)--(15,0)--cycle;
			\draw (0,0)--(15,0)--(15,1)--(0,1)--cycle;
			\draw (0,3)--(18,3)--(18,4)--(0,4)--cycle;
			\draw (0,4)--(0,5)--(19,5)--(19,4)--(18,4);
			\draw[dotted, line width=1pt] (8,1.7)--(8,2.3);
			\node (a) at (0.5,0.5) {$a_k$};
			\node (b) at (0.5,3.5) {$a_2$};
			\node (c) at (0.5,4.5) {$a_1$};
			\node (d) at (14.5,0.5) {$a_k$};
			\node (e) at (17.5,3.5) {$a_2$};
			\node (f) at (18.5,4.5) {$a_1$};
			\draw[dotted, line width=1pt] (1.2,0.5)--(1.8,0.5);
			\draw[dotted, line width=1pt] (13.2,0.5)--(13.8,0.5);
			\draw[dotted, line width=1pt] (1.2,3.5)--(1.8,3.5);
			\draw[dotted, line width=1pt] (16.2,3.5)--(16.8,3.5);
			\draw[dotted, line width=1pt] (1.2,4.5)--(1.8,4.5);
			\draw[dotted, line width=1pt] (17.2,4.5)--(17.8,4.5);

			\node at (-2.3,0.5) {$\mathsf{P}^{A\mathrm{II}}(S)=$};
			\draw (0,1)--(0,3);
			\draw (15,1)--(16,1)--(16,2)--(17,2)--(17,3);
			\draw (0,0)--(0,-4)--(8,-4)--(8,-3)--(16,-3)--(16,-2)--(18,-2)--(18,0)--(19,0)--(19,2)--(21,2)--(21,3)--(22,3)--(22,5)--(19,5);
			\node at (8,-1.5) {$2k+1 , \ldots , \overline{1}$};
			\node at (22.5,0) {.};
			\node at (-6,0.5) {(7.16)};
			\draw[line width=1pt, color=blue, line join=round] (0,0)--(0,-4)--(8,-4)--(8,-3)--(16,-3)--(16,-2)--(18,-2)--(18,0)--(19,0)--(19,2)--(21,2)--(21,3)--(22,3)--(22,5)--(19,5)--(19,4)--(18,4)--(18,3)--(17,3)--(17,2)--(16,2)--(16,1)--(15,1)--(15,0)--cycle;
			\node at (19.3,-2.5) {\textcolor{blue}{$\Omega_1$}};
			\node at (26,0) {};
		\end{tikzpicture}
	\end{center}\label{7.16}
	First,
	we consider the case $k=2$.
	From assertion (ii), we see by Lemma \ref{formula} that 
	\[
		\mathsf{P}^{A\mathrm{II}}_{1,4} \left( \mathsf{Res}_{1,4}\,\,\mathsf{P}^{A\mathrm{II}}(S) \right)
		=
		\mathsf{P}^{A\mathrm{II}}_{1,4} \left( \mathsf{Res}_{1,4}\,\,S \right)
	\]
	\begin{center}
		\begin{tikzpicture}[x=5mm,y=5mm]
			\draw (0,0)--(8,0)--(8,1)--(0,1)--cycle;
			\draw (0,1)--(0,2)--(10,2)--(10,1)--(8,1);
			\node at (0.5,0.5) {$a_2$};
			\node at (0.5,1.5) {$a_1$};
			\node at (7.5,0.5) {$a_2$};
			\node at (9.5,1.5) {$a_1$};
			
			\draw[dotted, line width=1pt] (1.2,0.5)--(1.8,0.5);
			\draw[dotted, line width=1pt] (6.8,0.5)--(6.2,0.5);
			\draw[dotted, line width=1pt] (1.2,1.5)--(1.8,1.5);
			\draw[dotted, line width=1pt] (8.2,1.5)--(8.8,1.5);
			\node at (-0.7,1) {$=$};
			\node at (-14.2,1) {};
			\node at (10.5,0.5) {.};
		\end{tikzpicture}
	\end{center}
	Therefore, from the implication (1) $\implies$ (3) in Lemma \ref{characterization of k-highest}, 
	we deduce that
	$\mathsf{Res}_{1,4}\,\,\mathsf{P}^{A\mathrm{II}}(S)$
	is a tableau of the following form 
	since $\mathsf{P}^{A\mathrm{II}}(S)$ is symplectic and 
	$a_1=2, a_2=3, a_3=6$: 
	\begin{center}
		\begin{tikzpicture}[x=5mm,y=5mm]
			\draw (0,0)--(8,0)--(8,1)--(0,1)--cycle;
			\draw (0,1)--(0,2)--(10,2)--(10,1)--(8,1);
			\node at (0.5,0.5) {$a_2$};
			\node at (0.5,1.5) {$a_1$};
			\node at (7.5,0.5) {$a_2$};
			\node at (9.5,1.5) {$a_1$};
			\draw[dotted, line width=1pt] (1.2,0.5)--(1.8,0.5);
			\draw[dotted, line width=1pt] (6.2,0.5)--(6.8,0.5);
			\draw[dotted, line width=1pt] (1.2,1.5)--(1.8,1.5);
			\draw[dotted, line width=1pt] (8.2,1.5)--(8.8,1.5);
			\node at (-3.5,1) {$\mathsf{Res}_{1,4}\,\,\mathsf{P}^{A\mathrm{II}}(S) =$};
			\node at (10.5,0.5) {.};
		\end{tikzpicture}
	\end{center}
	Hence, 
	$\mathsf{P}^{A\mathrm{II}}(S)$
	is a tableau of the following form,
	where the entries of the boxes contained in the area $\Omega_2$
	lie in $[5,\overline{1}]$: 
	\begin{center}
		\begin{tikzpicture}[x=5mm,y=5mm]
			\fill[fill=blue!20] (0,0)--(0,-4)--(8,-4)--(8,-3)--(16,-3)--(16,-2)--(18,-2)--(18,0)--(19,0)--(19,2)--(21,2)--(21,3)--(22,3)--(22,5)--(19,5)--(19,4)--(18,4)--(18,3)--(0,3)--cycle;
			\draw (0,3)--(18,3)--(18,4)--(0,4)--cycle;
			\draw (0,4)--(0,5)--(19,5)--(19,4)--(18,4);
			\node (b) at (0.5,3.5) {$a_2$};
			\node (c) at (0.5,4.5) {$a_1$};
			\node (e) at (17.5,3.5) {$a_2$};
			\node (f) at (18.5,4.5) {$a_1$};
			\draw[dotted, line width=1pt] (1.2,3.5)--(1.8,3.5);
			\draw[dotted, line width=1pt] (16.2,3.5)--(16.8,3.5);
			\draw[dotted, line width=1pt] (1.2,4.5)--(1.8,4.5);
			\draw[dotted, line width=1pt] (17.2,4.5)--(17.8,4.5);
			\node at (-2.2,0.5) {$\mathsf{P}^{A\mathrm{II}}(S)=$};
			\draw (0,0)--(0,3);
			\draw (0,0)--(0,-4)--(8,-4)--(8,-3)--(16,-3)--(16,-2)--(18,-2)--(18,0)--(19,0)--(19,2)--(21,2)--(21,3)--(22,3)--(22,5)--(19,5);
			\node at (8,0) {$5 , \ldots , \overline{1}$};
			\node at (22.5,0) {.};
			\draw[line width=1pt, color=blue, line join=round] (0,0)--(0,-4)--(8,-4)--(8,-3)--(16,-3)--(16,-2)--(18,-2)--(18,0)--(19,0)--(19,2)--(21,2)--(21,3)--(22,3)--(22,5)--(19,5)--(19,4)--(18,4)--(18,3)--(0,3)--cycle;
			\node at (19.3,-2.5) {\textcolor{blue}{$\Omega_2$}};
		\end{tikzpicture}
	\end{center}
	
  Now, let $k \geq 3$.
  The proof is divided into two cases according as $k$ is even or odd. 
	First, assume that $k$ is even.
	By using the implication 
	(i)$\implies$(ii) and (iii)$\implies$(iv) repeatedly, 
	we deduce that 
	\begin{center}
		\begin{tikzpicture}[x=5mm,y=5mm]
			\draw (0,0)--(8,0)--(8,1)--(0,1)--cycle;
			\draw (0,1)--(0,2)--(10,2)--(10,1)--(8,1);
			\node at (0.9,0.5) {$a_{k+1}$};
			\node at (0.5,1.5) {$a_k$};
			\node at (7.1,0.5) {$a_{k+1}$};
			\node at (9.5,1.5) {$a_k$};
			\draw[dotted, line width=1pt] (2,0.5)--(2.6,0.5);
			\draw[dotted, line width=1pt] (5.4,0.5)--(6,0.5);
			\draw[dotted, line width=1pt] (1.2,1.5)--(1.8,1.5);
			\draw[dotted, line width=1pt] (8.2,1.5)--(8.8,1.5);
			\node at (-6.7,1) {$\mathsf{P}^{A\mathrm{II}}_{2k-1,2k+2}\left(\mathsf{Res}_{2k-1,2k+2}\,\,\mathsf{P}^{A\mathrm{II}}(S)\right) =$};
			\node at (10.5,0.5) {.};
		\end{tikzpicture}
	\end{center}
	Here, by the induction hypothesis, 
	$\mathsf{P}^{A\mathrm{II}}(S)$ does not contain $\boxed{2k}$.
	Therefore, from the implication (1) $\implies$ (3) in Lemma \ref{characterization of k-highest} 
	and the equalities $a_k = 2k-1, a_{k+1} = 2k+2$,
	we conclude that
	$\mathsf{Res}_{2k-1,2k+2}\,\,\mathsf{P}^{A\mathrm{II}}(S)$
	is a tableau of the following form: 
	\begin{center}
		\begin{tikzpicture}[x=5mm,y=5mm]
			\draw (0,0)--(8,0)--(8,1)--(0,1)--cycle;
			\draw (0,1)--(0,2)--(8,2)--(8,1);
			\draw (8,2)--(16,2)--(16,1)--(8,1);
			\draw (0,2)--(0,3)--(18,3)--(18,2)--(16,2);
			\node at (0.9,0.5) {$a_{k+1}$};
			\node at (1.3,1.5) {$2k+1$};
			\node at (8.9,1.5) {$a_{k+1}$};
			\node at (7.1,0.5) {$a_{k+1}$};
			\node at (6.7,1.5) {$2k+1$};
			\node at (15.1,1.5) {$a_{k+1}$};
			\node at (0.6,2.5) {$a_k$};
			\node at (17.4,2.5) {$a_k$};
			\draw[dotted, line width=1pt] (2,0.5)--(2.6,0.5);
			\draw[dotted, line width=1pt] (5.4,0.5)--(6,0.5);
			\draw[dotted, line width=1pt] (2.8,1.5)--(3.4,1.5);
			\draw[dotted, line width=1pt] (4.6,1.5)--(5.2,1.5);
			\draw[dotted, line width=1pt] (10,1.5)--(10.6,1.5);
			\draw[dotted, line width=1pt] (13.4,1.5)--(14,1.5);
			\draw[dotted, line width=1pt] (1.4,2.5)--(2,2.5);
			\draw[dotted, line width=1pt] (16,2.5)--(16.6,2.5);
			\node at (-4.5,1.5) {$\mathsf{Res}_{2k-1,2k+2}\,\,\mathsf{P}^{A\mathrm{II}}(S) =$};
			\node at (18.5,1) {.};
		\end{tikzpicture}
	\end{center}
  Also, by Definition \ref{7.3},
	we deduce that
	$\boxed{2k+1}$ and $\boxed{a_{k+1}}$
	lie in a row of $\mathsf{P}^{A\mathrm{II}}(S)$ which is lower than or equal to the $(k+1)$st row. 
	In addition, from the definition of the jeu-de-taquin slide,
	we see that
	if $\mathsf{P}^{A\mathrm{II}}(S)$ contained $\boxed{2k+1}$, then 
	$\boxed{a_{k+1}}$ and $\boxed{2k+1}$ would lie in the $(k+1)$st row since 
	$\mathsf{P}^{A\mathrm{II}}(S)$ is symplectic.
	However, this is a contradiction since, in such a case, 
	$\mathsf{Res}_{2k-1,2k+2}\,\,\mathsf{P}^{A\mathrm{II}}(S)$
	is not the tableau above.
	
  Next, assume that $k$ is odd.
	By using the implication 
	(i)$\implies$(ii) and
	(iii)$\implies$(iv) repeatedly, we deduce that 
	\begin{center}
		\begin{tikzpicture}[x=5mm,y=5mm]
			\draw (0,0)--(8,0)--(8,1)--(0,1)--cycle;
			\draw (0,1)--(0,2)--(10,2)--(10,1)--(8,1);
			\node at (0.9,0.5) {$a_{k+1}$};
			\node at (0.5,1.5) {$a_k$};
			\node at (7.1,0.5) {$a_{k+1}$};
			\node at (9.5,1.5) {$a_k$};
			\draw[dotted, line width=1pt] (2,0.5)--(2.6,0.5);
			\draw[dotted, line width=1pt] (5.4,0.5)--(6,0.5);
			\draw[dotted, line width=1pt] (1.2,1.5)--(1.8,1.5);
			\draw[dotted, line width=1pt] (8.2,1.5)--(8.8,1.5);
			\node at (-6.7,1) {$\mathsf{P}^{A\mathrm{II}}_{2k-1,2k+2}\left(\mathsf{Res}_{2k-1,2k+2}\,\,\mathsf{P}^{A\mathrm{II}}(S)\right) =$};
			\node at (10.5,0.5) {.};
		\end{tikzpicture}
	\end{center}
	Here, by the induction hypothesis, 
	$\mathsf{P}^{A\mathrm{II}}(S)$ does not contain
	$\boxed{2k+1}$.
	Therefore, from the implication (1) $\implies$ (3) in Lemma \ref{characterization of k-lowest} 
	and the equalities $a_{k}=2k, a_{k}=2k+1$,
	we see that 
	$\mathsf{Res}_{2k-1,2k+2}\,\,\mathsf{P}^{A\mathrm{II}}(S)$
	is a tableau of the following form: 
	\begin{center}
		\begin{tikzpicture}[x=5mm,y=5mm]
			\draw (0,0)--(8,0)--(8,1)--(0,1)--cycle;
			\draw (0,1)--(0,2)--(8,2);
			\draw (8,2)--(16,2)--(16,1)--(8,1);
			\draw (0,2)--(0,3)--(18,3)--(18,2)--(16,2);
			\node at (0.9,1.5) {$a_{k+1}$};
			\node at (1.3,0.5) {$2k+2$};
			\node at (6.7,0.5) {$2k+2$};
			\node at (15.1,1.5) {$a_{k+1}$};
			\node at (0.6,2.5) {$a_k$};
			\node at (17.4,2.5) {$a_k$};
			
			\draw[dotted, line width=1pt] (2,1.5)--(2.6,1.5);
			\draw[dotted, line width=1pt] (13.4,1.5)--(14,1.5);
			\draw[dotted, line width=1pt] (2.8,0.5)--(3.4,0.5);
			\draw[dotted, line width=1pt] (4.6,0.5)--(5.2,0.5);
			\draw[dotted, line width=1pt] (1.4,2.5)--(2,2.5);
			\draw[dotted, line width=1pt] (16,2.5)--(16.6,2.5);
			\node at (-4.5,1.5) {$\mathsf{Res}_{2k-1,2k+2}\,\,\mathsf{P}^{A\mathrm{II}}(S) =$};
			\node at (18.5,1) {.};
		\end{tikzpicture}
	\end{center}
	Hence, in $\mathsf{P}^{A\mathrm{II}}(S)$,
	$\boxed{a_{k+1}}$ lies only in the $(k+1)$st row, and
	$\boxed{2k+2}$ lies only in the $(k+2)$nd row.
	However, since $\mathsf{P}^{A\mathrm{II}}(T)$
	is a symplectic tableau,
	all the entries in the $(k+2)$nd row 
	must be greater than or equal to $2k+3$.
	Thus, we see that 
	$\mathsf{P}^{A\mathrm{II}}(S)$
	does not contain $\boxed{2k+2}$.
	The proof of the implication
	(iv) $\implies$ (iii) is similar.
  This proves the lemma. 
\end{proof}

\setcounter{equation}{16}

\subsection{Construction of the bijections $\mathsf{\Phi}$ and $\mathsf{\Psi}$}
In this subsection,
we construct the bijections
$\mathsf{\Phi}$ and $\mathsf{\Psi}$
as certain composites of promotion operators.

We begin with a simple observation on restriction operators. 
Let $a \leq b \leq c \leq d$ be integers, and
$T \in \mathsf{SST}_{2n}(\lambda)$ a semistandard tableau
with all entries in $[a,2n]$.
We can describe $\mathsf{Res}_{a,b;c,d}(T)$
in terms of promotion operators as follows.
Let $\widetilde{T}$ be a semistandard tableau obtained by
placing the numbers in the interval $[b+1,c-1]$
into the empty boxes of $T|_{a,b;c,d}$; 
for example, $\widetilde{T}$ can be taken to be $T$, but any other choice is also acceptable.
Then, it follows from the definition of $\mathsf{Rect}$ that 
$\mathsf{Res}_{a,b;c,d}(T)$ is the tableau obtained
by simultaneously increasing by $c-b-1$ all entries in $[b+1,b+1+(d-c)]$
of the tableau below: 
\begin{equation}
	\left.\overbrace{\mathsf{pr}_{b+1,d+b-c+2}^{-1}\circ\cdots\circ\mathsf{pr}_{c-3,d-2}^{-1}\circ\mathsf{pr}_{c-2,d-1}^{-1}\circ\mathsf{pr}_{c-1,d}^{-1}}^{\text{($c-b-1$) factors}}(\widetilde{T})\right|_{a,b+1+(d-c)}. \label{7.18}
\end{equation}
\begin{eg}
	Let $a=1$, $b=2$, $c=5$, $d=6$, and take
	\vspace{5pt}
	\begin{center}
		\begin{tikzpicture}[x=5mm,y=5mm]
			\draw (0,0)--(4,0)--(4,-1)--(3,-1)--(3,-2)--(2,-2)--(2,-3)--(0,-3)--cycle;
			\draw (1,0)--(1,-3);
			\draw (2,0)--(2,-2);
			\draw (3,0)--(3,-1);
			\draw (0,-1)--(3,-1);
			\draw (0,-2)--(2,-2);
			\node at (0.5,-0.5) {1};
			\node at (1.5,-0.5) {2};
			\node at (2.5,-0.5) {2};
			\node at (3.5,-0.5) {3};
			\node at (0.5,-1.5) {4};
			\node at (1.5,-1.5) {5};
			\node at (2.5,-1.5) {6};
			\node at (0.5,-2.5) {6};
			\node at (1.5,-2.5) {6};
			\node at (-1.1,-1.5) {$T=$};
			\node at (4.2,-1.8) {,};
		\end{tikzpicture}
		\hspace{10pt}
		\begin{tikzpicture}[x=5mm,y=5mm]
			\draw (0,0)--(4,0)--(4,-1)--(3,-1)--(3,-2)--(2,-2)--(2,-3)--(0,-3)--cycle;
			\draw (1,0)--(1,-3);
			\draw (2,0)--(2,-2);
			\draw (3,0)--(3,-1);
			\draw (0,-1)--(3,-1);
			\draw (0,-2)--(2,-2);
			\node at (0.5,-0.5) {1};
			\node at (1.5,-0.5) {2};
			\node at (2.5,-0.5) {2};
			\node at (3.5,-0.5) {4};
			\node at (0.5,-1.5) {3};
			\node at (1.5,-1.5) {5};
			\node at (2.5,-1.5) {6};
			\node at (0.5,-2.5) {6};
			\node at (1.5,-2.5) {6};
			\node at (-1.1,-1.3) {$\widetilde{T}=$};
			\node at (4.2,-1.8) {.};
		\end{tikzpicture}
	\end{center}
	In this case, we have
	\vspace{5pt}
	\begin{center}
		\begin{tikzpicture}[x=5mm,y=5mm]
			\fill[color=blue!20] (0,-3)--(0,-2)--(1,-2)--(1,-1)--(3,-1)--(3,0)--(4,0)--(4,-1)--(3,-1)--(3,-2)--(2,-2)--(2,-3)--cycle;
			\draw (0,0)--(4,0)--(4,-1)--(3,-1)--(3,-2)--(2,-2)--(2,-3)--(0,-3)--cycle;
			\draw (1,0)--(1,-3);
			\draw (2,0)--(2,-2);
			\draw (3,0)--(3,-1);
			\draw (0,-1)--(3,-1);
			\draw (0,-2)--(2,-2);
			\node at (0.5,-0.5) {1};
			\node at (1.5,-0.5) {2};
			\node at (2.5,-0.5) {2};
			\node at (3.5,-0.5) {4};
			\node at (0.5,-1.5) {3};
			\node at (1.5,-1.5) {5};
			\node at (2.5,-1.5) {6};
			\node at (0.5,-2.5) {6};
			\node at (1.5,-2.5) {6};
			\draw[line width=1.5pt, color=blue] (0,-3)--(0,-2)--(1,-2)--(1,-1)--(3,-1)--(3,0)--(4,0)--(4,-1)--(3,-1)--(3,-2)--(2,-2)--(2,-3)--cycle;

			\node at (5.5,-1.5) {$\longmapsto$};
			\node at (5.5,-0.7) {$\mathsf{pr}_{4,6}^{-1}$};

			\fill[color=blue!20] (7,-3)--(7,-1)--(10,-1)--(10,-2)--(9,-2)--(9,-3)--cycle;
			\draw (7,0)--(11,0)--(11,-1)--(10,-1)--(10,-2)--(9,-2)--(9,-3)--(7,-3)--cycle;
			\draw (8,0)--(8,-3);
			\draw (9,0)--(9,-2);
			\draw (10,0)--(10,-1);
			\draw (7,-1)--(10,-1);
			\draw (7,-2)--(9,-2);
			\node at (7.5,-0.5) {1};
			\node at (8.5,-0.5) {2};
			\node at (9.5,-0.5) {2};
			\node at (10.5,-0.5) {6};
			\node at (7.5,-1.5) {3};
			\node at (8.5,-1.5) {4};
			\node at (9.5,-1.5) {5};
			\node at (7.5,-2.5) {5};
			\node at (8.5,-2.5) {5};
			\draw[line width=1.5pt, color=blue] (7,-3)--(7,-1)--(10,-1)--(10,-2)--(9,-2)--(9,-3)--cycle;
			
			\node at (12.5,-1.5) {$\longmapsto$};
			\node at (12.5,-0.7) {$\mathsf{pr}_{3,5}^{-1}$};

			\fill[color=blue!20] (14,0)--(14,-3)--(15,-3)--(15,-2)--(17,-2)--(17,0)--cycle;
			\draw (14,0)--(18,0)--(18,-1)--(17,-1)--(17,-2)--(16,-2)--(16,-3)--(14,-3)--cycle;
			\draw (15,0)--(15,-3);
			\draw (16,0)--(16,-2);
			\draw (17,0)--(17,-1);
			\draw (14,-1)--(17,-1);
			\draw (14,-2)--(16,-2);
			\node at (14.5,-0.5) {1};
			\node at (15.5,-0.5) {2};
			\node at (16.5,-0.5) {2};
			\node at (17.5,-0.5) {6};
			\node at (14.5,-1.5) {3};
			\node at (15.5,-1.5) {4};
			\node at (16.5,-1.5) {4};
			\node at (14.5,-2.5) {4};
			\node at (15.5,-2.5) {5};
			\draw[line width=1.5pt, color=blue] (14,0)--(14,-3)--(15,-3)--(15,-2)--(17,-2)--(17,0)--cycle;

			\node at (19.5,-1.5) {$\longmapsto$};
			\node at (19.5,-0.7) {$|_{1,4}$};

			\fill[color=blue!20] (21,-1)--(21,-3)--(22,-3)--(22,-2)--(24,-2)--(24,-1)--cycle;
			\draw (21,0)--(24,0)--(24,-2)--(22,-2)--(22,-3)--(21,-3)--cycle;
			\draw (22,0)--(22,-3);
			\draw (23,0)--(23,-2);
			\draw (24,0)--(24,-1);
			\draw (21,-1)--(24,-1);
			\draw (21,-2)--(23,-2);
			\node at (21.5,-0.5) {1};
			\node at (22.5,-0.5) {2};
			\node at (23.5,-0.5) {2};
			\node at (21.5,-1.5) {3};
			\node at (22.5,-1.5) {4};
			\node at (23.5,-1.5) {4};
			\node at (21.5,-2.5) {4};
			\draw[line width=1.5pt, color=blue] (21,-1)--(21,-3)--(22,-3)--(22,-2)--(24,-2)--(24,-1)--cycle;

			\node at (25.5,-1.5) {$\longmapsto$};
			\node at (25.5,-0.7) {$+2$};

			\draw (27,0)--(30,0)--(30,-2)--(28,-2)--(28,-3)--(27,-3)--cycle;
			\draw (28,0)--(28,-3);
			\draw (29,0)--(29,-2);
			\draw (30,0)--(30,-1);
			\draw (27,-1)--(30,-1);
			\draw (27,-2)--(29,-2);
			\node at (27.5,-0.5) {1};
			\node at (28.5,-0.5) {2};
			\node at (29.5,-0.5) {2};
			\node at (27.5,-1.5) {5};
			\node at (28.5,-1.5) {6};
			\node at (29.5,-1.5) {6};
			\node at (27.5,-2.5) {6};
			\node at (30.3,-1.8) {.};
		\end{tikzpicture}
	\end{center}
	\vspace{5pt}
	By Example \ref{example}, we can verify that the last tableau agrees with $\mathsf{Res}_{1,2;5,6}\,(T)$. 
\end{eg}
Using this observation, we can show the following lemmas.
\begin{lem}\label{lem1}
	Assume that all the entries of
	$T \in \mathsf{SST}_{2n}(\lambda)$
	are greater than or equal to $k$.
	Then, the equality 
	\begin{equation}
		\left.\mathsf{pr}_{k+3,\overline{k}} \circ \mathsf{pr}_{k+2,\overline{k+1}} (T) \right|_{k,k+3}
		=
		\left.\left(\mathsf{Res}_{k,k+1;\overline{k+1},\overline{k}}\,\,T\right)\right|_{\overline{k+1}\to k+2,\,\,\overline{k}\to k+3}
	\end{equation}
	holds, where the operation $\,\cdot\,\,\,|_{\overline{k+1}\to k+2,\,\,\overline{k}\to k+3}$
	means replacing $\overline{k+1}$ and $\overline{k}$
	with $k+2$ and $k+3$, respectively.
\end{lem}
\begin{proof}
	We take $a=k$, $b=k+1$, $c=\overline{k+1}$, $d=\overline{k}$ and $\widetilde{T}=T$.
	Then, by Lemma \ref{prrel} together with the observation above, we see that 
	\[
		\begin{aligned}
			& \left.\left(\mathsf{Res}_{k,k+1;\overline{k+1},\overline{k}}\,\,T\right)\right|_{\overline{k+1}\to k+2,\,\,\overline{k}\to k+3} \\
			&= \left.\mathsf{pr}_{k+2,k+4}^{-1} \circ \cdots \circ \mathsf{pr}_{\overline{k+3},\overline{k+1}}^{-1} \circ \mathsf{pr}_{\overline{k+2},\overline{k}}^{-1}(T)\right|_{k,k+3} \\
			&= \left.\left(\mathsf{pr}_{k+3,k+4} \circ \mathsf{pr}_{k+2,k+3}\right) \circ \cdots \circ \left(\mathsf{pr}_{\overline{k+2},\overline{k+1}} \circ \mathsf{pr}_{\overline{k+3},\overline{k+2}}\right) \circ \left(\mathsf{pr}_{\overline{k+1},\overline{k}} \circ \mathsf{pr}_{\overline{k+2},\overline{k+1}}\right) (T) \right|_{k,k+3} \\
			&= \left.\left(\mathsf{pr}_{k+3,k+4} \circ \cdots \circ \mathsf{pr}_{\overline{k+2},\overline{k+1}} \circ \mathsf{pr}_{\overline{k+1},\overline{k}}\right)\circ\left(\mathsf{pr}_{k+2,k+3} \circ \cdots \circ \mathsf{pr}_{\overline{k+3},\overline{k+2}} \circ \mathsf{pr}_{\overline{k+2},\overline{k+1}}\right)\right|_{k,k+3} \\
			&= \left.\mathsf{pr}_{k+3,\overline{k}} \circ \mathsf{pr}_{k+2,\overline{k+1}} (T) \right|_{k,k+3}.
		\end{aligned}
	\]
This proves the lemma. 
\end{proof}

\begin{lem}\label{lem2}
	Let
	$T \in \mathsf{SST}_{2n}(\lambda)$ be a semistandard tableau.
	Then, $\mathsf{Res}_{2,\overline{2}}\,\,\mathsf{pr}^{-1}_{2,\overline{1}}(T)$
	is the tableau obtained from $\mathsf{Res}_{3,\overline{1}}(T)$ 
        by simultaneously decreasing all entries by $1$. 
\end{lem}

\begin{proof}
	By the observation above,
	we deduce that
	$\mathsf{Res}_{3,\overline{1}}(T)$ is the tableau obtained
        from $\mathsf{pr}_{1,\overline{2}}^{-1}\circ\mathsf{pr}_{2,\overline{1}}^{-1}(T)|_{1,\overline{3}}$
	by simultaneously increasing all entries by $2$. 
	Also, 
	$\mathsf{Res}_{2,\overline{2}}\,\,\mathsf{pr}^{-1}_{2,\overline{1}}(T)$
	is the tableau obtained from $\mathsf{pr}_{1,\overline{2}}^{-1}\circ\mathsf{pr}_{2,\overline{1}}^{-1}(T)|_{1,\overline{3}}$ 
	by simultaneously increasing all entries by $1$. 
	This proves the lemma.
\end{proof}

Now, we define maps 
$\mathsf{\Phi}$ and $\mathsf{\Psi}$ as follows.

\begin{de}
	We define two sequences of operators 
	$\mathsf{\Phi}_1,\ldots,\mathsf{\Phi}_n,\mathsf{\Psi}_1,\ldots,\mathsf{\Psi}_n : \mathsf{SST}_{2n}(\lambda) \to \mathsf{SST}_{2n}(\lambda)$
	by the following recurrence relations: 
	\[
		\mathsf{\Phi}_n :=
		\left\{\,
		\begin{aligned}
			 & \mathrm{id}                  &  & \text{if $n$ is even}, \\
			 & \mathsf{pr}_{n,\overline{n}} &  & \text{if $n$ is odd},
		\end{aligned}
		\right.
		\,\,\,\,\,\,
		\mathsf{\Phi}_k :=
		\left\{\,
		\begin{aligned}
			 & \mathsf{pr}_{k+1,\overline{k}} \circ \mathsf{\Phi}_{k+1} &  & \text{if $k = 2, 4, 6, \ldots, 2\left\lfloor\frac{n}{2}\right\rfloor$}, \\
			 & \mathsf{pr}_{k,\overline{k}} \circ \mathsf{\Phi}_{k+1}   &  & \text{if $k = 1, 3, 5, \ldots, 2\left\lfloor\frac{n-1}{2}\right\rfloor + 1$},
		\end{aligned}
		\right.
	\]
	\[
		\mathsf{\Psi}_n :=
		\left\{\,
		\begin{aligned}
			 & \mathsf{pr}_{n,\overline{n}} &  & \text{if $n$ is even}, \\
			 & \mathrm{id}                  &  & \text{if $n$ is odd},
		\end{aligned}
		\right.
		\,\,\,\,\,\,
		\mathsf{\Psi}_k :=
		\left\{\,
		\begin{aligned}
			 & \mathsf{pr}_{k,\overline{k}} \circ \mathsf{\Psi}_{k+1}   &  & \text{if $k = 2, 4, 6, \ldots, 2\left\lfloor\frac{n}{2}\right\rfloor$}, \\
			 & \mathsf{pr}_{k+1,\overline{k}} \circ \mathsf{\Psi}_{k+1} &  & \text{if $k = 1, 3, 5, \ldots, 2\left\lfloor\frac{n-1}{2}\right\rfloor + 1$}.
		\end{aligned}
		\right.
	\]
	Then, 
	we set
	$\mathsf{\Phi} := \mathsf{\Phi}_1$, $\mathsf{\Psi} := \mathsf{\Psi}_1$.
\end{de}

\begin{eg}
	If $n=3$,
	then we have
	$\mathsf{\Phi}_3 = \mathsf{pr}_{3,4}$,
	$\mathsf{\Phi}_2 = \mathsf{pr}_{3,5} \circ \mathsf{pr}_{3,4}$, and
	$\mathsf{\Phi}_1 = \mathsf{pr}_{1,6} \circ \mathsf{pr}_{3,5} \circ \mathsf{pr}_{3,4}$.
	Hence, $\mathsf{\Phi} = \mathsf{pr}_{1,6} \circ \mathsf{pr}_{3,5} \circ \mathsf{pr}_{3,4}$.
	Thus, 
	\begin{center}
		\begin{tikzpicture}[x=5mm,y=5mm]
			\draw (0,0)--(0,3)--(2,3)--(2,1)--(1,1)--(1,0)--cycle;
			\draw (0,1)--(1,1);
			\draw (0,2)--(2,2);
			\draw (1,1)--(1,3);
			\node at (0.5,0.5) {5};
			\node at (0.5,1.5) {2};
			\node at (0.5,2.5) {1};
			\node at (1.5,1.5) {6};
			\node at (1.5,2.5) {1};

			\node at (3.5,1.5) {$\longmapsto$};
			\node at (3.5,2.3) {$\mathsf{pr}_{3,4}$};

			\fill[color=blue!20] (5,0)--(6,0)--(6,1)--(5,1)--cycle;
			\draw (5,0)--(5,3)--(7,3)--(7,1)--(6,1)--(6,0)--cycle;
			\draw (5,1)--(6,1);
			\draw (5,2)--(7,2);
			\draw (6,1)--(6,3);
			\node at (5.5,0.5) {5};
			\node at (5.5,1.5) {2};
			\node at (5.5,2.5) {1};
			\node at (6.5,1.5) {6};
			\node at (6.5,2.5) {1};
			\draw[line width=1.5pt, color=blue] (5,0)--(6,0)--(6,1)--(5,1)--cycle;

			\node at (8.5,1.5) {$\longmapsto$};
			\node at (8.5,2.3) {$\mathsf{pr}_{3,5}$};

			\fill[color=blue!20] (10,0)--(10,3)--(12,3)--(12,1)--(11,1)--(11,0)--cycle;
			\draw (10,0)--(10,3)--(12,3)--(12,1)--(11,1)--(11,0)--cycle;
			\draw (10,1)--(11,1);
			\draw (10,2)--(12,2);
			\draw (11,1)--(11,3);
			\node at (10.5,0.5) {3};
			\node at (10.5,1.5) {2};
			\node at (10.5,2.5) {1};
			\node at (11.5,1.5) {6};
			\node at (11.5,2.5) {1};
			\draw[line width=1.5pt, color=blue] (10,0)--(10,3)--(12,3)--(12,1)--(11,1)--(11,0)--cycle;

			\node at (13.5,1.5) {$\longmapsto$};
			\node at (13.5,2.3) {$\mathsf{pr}_{1,6}$};

			\draw (15,0)--(15,3)--(17,3)--(17,1)--(16,1)--(16,0)--cycle;
			\draw (15,1)--(17,1);
			\draw (15,2)--(17,2);
			\draw (16,1)--(16,3);
			\node at (15.5,0.5) {4};
			\node at (15.5,1.5) {2};
			\node at (15.5,2.5) {1};
			\node at (16.5,1.5) {3};
			\node at (16.5,2.5) {2};

			\draw (20.2,0)--(20.2,3)--(22.2,3)--(22.2,1)--(21.2,1)--(21.2,0)--cycle;
			\draw (20.2,1)--(21.2,1);
			\draw (20.2,2)--(22.2,2);
			\draw (21.2,1)--(21.2,3);
			\node at (20.7,0.5) {5};
			\node at (20.7,1.5) {2};
			\node at (20.7,2.5) {1};
			\node at (21.7,1.5) {6};
			\node at (21.7,2.5) {1};
			\node at (20.8,1.5) {$\mathsf{\Phi}\left(\begin{aligned} \\ \,\,\,\,\,\,\,\,\,\,\,\,\,\,\,\, \\ \\ \end{aligned}\right)$};
			\node at (23.2,1.2) {.};
			\node at (17.85,1.4) {$=$};
		\end{tikzpicture}
	\end{center}
	It is easy to see that the leftmost tableau is a
	$\widehat{\mathfrak{g}}$-dominant tableau.
	Also, 
	the rightmost tableau is the 
	$\mathfrak{k}$-highest weight tableau
	in Example \ref{5.9}.
\end{eg}

\begin{prop}
	We have $\mathsf{\Phi} \left( \mathsf{SST}_{2n}^{\text{$\widehat{\mathfrak{g}}$-$\mathrm{dom}$}}(\lambda) \right) \subset \mathsf{SST}_{2n}^{\text{$\mathfrak{k}$-$\mathrm{hw}$}}(\lambda)$.
\end{prop}

\begin{proof}
	Take an arbitrary 
	$T \in \mathsf{SST}_{2n}^{\text{$\widehat{\mathfrak{g}}$-dom}}(\lambda)$.
	We will prove 
	the following equation for $k=1,2,\ldots,n-1$
	by descending induction on $k$; 
	by Definition \ref{5.13},
	the assertion 
	$\mathsf{\Phi}(T) \in \mathsf{SST}_{2n}^{\text{$\mathfrak{k}$-hw}}(\lambda)$
	follows from the case $k=1$.
	\begin{center}
		\begin{tikzpicture}[x=5mm,y=5mm]
			\draw (0,0)--(15,0)--(15,1)--(0,1)--cycle;
			\draw (0,3)--(18,3)--(18,4)--(0,4)--cycle;
			\draw (0,4)--(0,5)--(19,5)--(19,4)--(18,4);
			\draw (0,1)--(0,3);
			\draw (15,1)--(16,1)--(16,2)--(17,2)--(17,3);
			\draw[dotted, line width=1pt] (8,1.7)--(8,2.3);
			\node (a) at (2.3,0.5) {$a_n - k + 1$};
			\node (b) at (2.5,3.5) {$a_{k+1} - k + 1$};
			\node (c) at (2.3,4.5) {$a_k - k + 1$};
			\node (d) at (12.9,0.5) {$a_n - k + 1$};
			\node (e) at (15.5,3.5) {$a_{k+1} - k + 1$};
			\node (f) at (16.9,4.5) {$a_k - k + 1$};
			\draw[dotted, line width=1pt] (4.7,0.5)--(5.3,0.5);
			\draw[dotted, line width=1pt] (9.9,0.5)--(10.5,0.5);
			\draw[dotted, line width=1pt] (5,3.5)--(5.6,3.5);
			\draw[dotted, line width=1pt] (12.4,3.5)--(13,3.5);
			\draw[dotted, line width=1pt] (4.7,4.5)--(5.3,4.5);
			\draw[dotted, line width=1pt] (13.9,4.5)--(14.5,4.5);
			\node at (-4,2.5) {$\mathsf{P}^{A\mathrm{II}}_{k,\overline{k}}\left(\mathsf{\Phi}_k\left(T|_{k,\overline{k}}\right)\right) = $};
			\node at (19.5,2) {.};
		\end{tikzpicture}
	\end{center}
	First,
	we prove the equation in the case $k=n-1$.
	We can easily verify the following: 
	\[
		\begin{aligned}
		 & a_{n-1} - (n-1) + 1 =
		 \left\{\,
		 \begin{aligned}
			 & n   &  & \text{$n$ : even}, \\
			 & n-1 &  & \text{$n$ : odd},
		 \end{aligned}
		 \right.&
		 & a_{n} - (n-1) + 1 =
		 \left\{\,
		 \begin{aligned}
			 & \overline{n}   &  & \text{$n$ : even}, \\
			 & \overline{n-1} &  & \text{$n$ : odd}.
		 \end{aligned}
		 \right.
	  \end{aligned}
	\]
	Since $T$ is $\widehat{\mathfrak{g}}$-dominant,
	$T|_{n-1,\overline{n-1}}$
	is $\widehat{\mathfrak{g}}_n$-dominant.
	Therefore, by Proposition \ref{n=2thm} (1), with $n-1$, $n$, $\overline{n}$, and $\overline{n-1}$ in place of $1$, $2$, $3$, and $4$, respectively, 
  we see that 
	\[
		\mathsf{P}^{A\mathrm{II}}_{n-1,\overline{n-1}}\left(\mathsf{\Phi}_{n-1}\left(T|_{n-1,\overline{n-1}}\right)\right)
		=
		\left\{\,
		\begin{aligned}
			 & \mathsf{P}^{A\mathrm{II}}_{n-1,\overline{n-1}} \left( \mathsf{pr}_{n-1,\overline{n-1}} (T|_{n-1,\overline{n-1}}) \right)                                  &  & \text{$n$ : even,} \\
			 & \mathsf{P}^{A\mathrm{II}}_{n-1,\overline{n-1}} \left( \mathsf{pr}_{n,\overline{n-1}} \circ \mathsf{pr}_{n,\overline{n}} (T|_{n-1,\overline{n-1}}) \right) &  & \text{$n$ : odd,}
		\end{aligned}
		\right.
	\]
	\begin{tikzpicture}[x=5mm,y=5mm]
		\draw (0,0)--(8,0)--(8,1)--(0,1)--cycle;
		\draw (0,1)--(0,2)--(10,2)--(10,1)--(8,1);
		\node at (0.5,0.5) {$\overline{n}$};
		\node at (0.5,1.5) {$n$};
		\node at (7.5,0.5) {$\overline{n}$};
		\node at (9.5,1.5) {$n$};
		\draw[dotted, line width=1pt] (1.2,0.5)--(1.8,0.5);
		\draw[dotted, line width=1pt] (6.2,0.5)--(6.8,0.5);
		\draw[dotted, line width=1pt] (1.2,1.5)--(1.8,1.5);
		\draw[dotted, line width=1pt] (8.2,1.5)--(8.8,1.5);
		
		\draw (0,-3)--(8,-3)--(8,-2)--(0,-2)--cycle;
		\draw (0,-2)--(0,-1)--(10,-1)--(10,-2)--(8,-2);
		\node at (1.2,-2.5) {$\overline{n-1}$};
		\node at (1.2,-1.5) {$n-1$};
		\node at (6.8,-2.5) {$\overline{n-1}$};
		\node at (8.8,-1.5) {$n-1$};
		\draw[dotted, line width=1pt] (2.5,-2.5)--(3.1,-2.5);
		\draw[dotted, line width=1pt] (4.9,-2.5)--(5.5,-2.5);
		\draw[dotted, line width=1pt] (2.5,-1.5)--(3.1,-1.5);
		\draw[dotted, line width=1pt] (6.9,-1.5)--(7.5,-1.5);
		\node at (-1,-0.5) {$= \left\{\begin{aligned} \\ \\ \\ \\ \\ \end{aligned}\right.$};
		\node at (12.5,1) {$n$ : even,};
		\node at (12.5,-2) {$n$ : odd,};
		\node at (-13,0) {};
	\end{tikzpicture} \\
	\begin{tikzpicture}[x=5mm,y=5mm]
		\draw (0,0)--(13.7,0)--(13.7,1)--(0,1)--cycle;
		\draw (0,1)--(0,2)--(15,2)--(15,1)--(13.7,1);
		\node at (3.1,0.5) {$a_n - (n-1) + 1$};
		\node at (3.5,1.5) {$a_{n-1} - (n-1) + 1$};
		\node at (10.5,0.5) {$a_n - (n-1) + 1$};
		\node at (11.5,1.5) {$a_{n-1} - (n-1) + 1$};
		\draw[dotted, line width=1pt] (6.5,0.5)--(7.1,0.5);
		\draw[dotted, line width=1pt] (7.2,1.5)--(7.8,1.5);
		
		\draw (0,-3)--(13.7,-3)--(13.7,-2)--(0,-2)--cycle;
		\draw (0,-2)--(0,-1)--(15,-1)--(15,-2)--(13.7,-2);
		\node at (3.1,-2.5) {$a_n - (n-1) + 1$};
		\node at (3.5,-1.5) {$a_{n-1} - (n-1) + 1$};
		\node at (10.5,-2.5) {$a_n - (n-1) + 1$};
		\node at (11.5,-1.5) {$a_{n-1} - (n-1) + 1$};
		\draw[dotted, line width=1pt] (6.5,-2.5)--(7.1,-2.5);
		\draw[dotted, line width=1pt] (7.2,-1.5)--(7.8,-1.5);
		\node at (-1,-0.5) {$= \left\{\begin{aligned} \\ \\ \\ \\ \\ \end{aligned}\right.$};
		\node at (17,1) {$n$ : even,};
		\node at (17,-2) {$n$ : odd.};
		\node at (-13,0) {};
	\end{tikzpicture} \\
	This completes the proof of the equation in the case $k=n-1$.
	
  Let $k \leq n-2$.
  The proof is divided into two cases according as $k$ is even or odd. 
	First, assume that $k$ is even.
	Note that the following equality holds by Lemma \ref{lem1}.
	\begin{multline}\label{7.20}
		\left.\mathsf{pr}_{k+3,\overline{k}} \circ \mathsf{pr}_{k+2,\overline{k+1}} \circ \mathsf{\Phi}_{k+2} \left( T|_{k,\overline{k}} \right) \right| _{k,k+3}
		\\
		= \left.\left( \mathsf{Res}_{k,k+1;\overline{k+1},\overline{k}}\,\,\mathsf{\Phi}_{k+2}\left(T|_{k,\overline{k}}\right) \right)\right| _{\overline{k+1} \to k+2,\,\,\overline{k} \to k+3}.
	\end{multline}
	Here,
	since $T$ is $\widehat{\mathfrak{g}}$-dominant, we have 
	$\mathsf{Res}_{k,k+3} \,\, \mathsf{\Phi}_k(T) = \mathsf{\Phi}_k(T)|_{k,k+3}$.
	Therefore, by equation (\ref{7.20}),
	we see that 
	\[
		\begin{aligned}
			\mathsf{Res}_{k,k+3} \,\, \mathsf{\Phi}_k(T|_{k,\overline{k}}) & = \mathsf{\Phi}_k(T|_{k,\overline{k}})|_{k,k+3}                                                                                                                                                                                                                                 \\
			                                           & = \left.\mathsf{pr}_{k+1,\overline{k}} \circ \mathsf{\Phi}_{k+1} \left( T|_{k,\overline{k}} \right) \right|_{k,k+3}                                                                                                                             \\
			                                           & = \left.\mathsf{pr}_{k+1,\overline{k}} \circ \mathsf{pr}_{k+1,\overline{k+1}} \circ \mathsf{\Phi}_{k+2} \left( T|_{k,\overline{k}} \right) \right|_{k,k+3}                                                                               \\
			                                           & = \left.\mathsf{pr}_{k+1,k+3} \circ \mathsf{pr}_{k+3,\overline{k}} \circ \mathsf{pr}_{k+1,k+2} \circ \mathsf{pr}_{k+2,\overline{k+1}} \circ \mathsf{\Phi}_{k+2} \left( T|_{k,\overline{k}} \right) \right|_{k,k+3}                \\
			                                           & = \left.\mathsf{pr}_{k+1,k+3} \circ \mathsf{pr}_{k+1,k+2} \circ \mathsf{pr}_{k+3,\overline{k}} \circ \mathsf{pr}_{k+2,\overline{k+1}} \circ \mathsf{\Phi}_{k+2} \left( T|_{k,\overline{k}} \right) \right|_{k,k+3}                \\
			                                           & = \mathsf{pr}_{k+1,k+3} \circ \mathsf{pr}_{k+1,k+2} \left(\left.\mathsf{pr}_{k+3,\overline{k}} \circ \mathsf{pr}_{k+2,\overline{k+1}} \circ \mathsf{\Phi}_{k+2} \left( T|_{k,\overline{k}} \right)\right|_{k,k+3}\right)          \\
			                                           & = \mathsf{pr}_{k+1,k+3} \circ \mathsf{pr}_{k+1,k+2} \left(\left.\left( \mathsf{Res}_{k,k+1;\overline{k+1},\overline{k}}\,\,\mathsf{\Phi}_{k+2}\left(T|_{k,\overline{k}}\right) \right)\right| _{\overline{k+1} \to k+2,\,\,\overline{k} \to k+3}\right) \\
			                                           & = \mathsf{pr}_{k+1,k+3} \circ \mathsf{pr}_{k+1,k+2} \left(\left.\left( \mathsf{Res}_{k,k+1;\overline{k+1},\overline{k}}\,\,T \right)\right| _{\overline{k+1} \to k+2,\,\,\overline{k} \to k+3}\right),
		\end{aligned}
	\]
	where the last equality follows since 
	$\mathsf{\Phi}_{k+2}(T|_{k,\overline{k}})|_{k,k+1;\overline{k+1},\overline{k}} = T|_{k,k+1;\overline{k+1},\overline{k}}$.
	Since $T$ is a $\widehat{\mathfrak{g}}$-dominant tableau,
	$\mathsf{Res}_{k,k+1;\overline{k+1},\overline{k}}\,\,T$ is $\widehat{\mathfrak{g}}_k$-dominant.
	Hence, we deduce that
	\begin{multline}
		\mathsf{P}^{A\mathrm{II}}_{k,k+3} \left(\mathsf{Res}_{k,k+3} \,\, \mathsf{\Phi}_k (T|_{k,\overline{k}})\right)
		\\ = \mathsf{P}^{A\mathrm{II}}_{k,k+3} \left(\mathsf{pr}_{k+1,k+3} \circ \mathsf{pr}_{k+1,k+2}\left(\left.\left(\mathsf{Res}_{k,k+1;\overline{k+1},\overline{k}}\,\,T\right)\right|_{\overline{k+1}\to k+2, \,\, \overline{k} \to k+3}\right)\right)
	\end{multline}
	\begin{tikzpicture}[x=5mm,y=5mm]
		\draw (0,0)--(8,0)--(8,1)--(0,1)--cycle;
		\draw (0,1)--(0,2)--(10,2)--(10,1)--(8,1);
		\node at (1.2,0.5) {$k+3$};
		\node at (0.5,1.5) {$k$};
		\node at (6.8,0.5) {$k+3$};
		\node at (9.5,1.5) {$k$};
		\draw[dotted, line width=1pt] (2.6,0.5)--(3.2,0.5);
		\draw[dotted, line width=1pt] (4.8,0.5)--(5.4,0.5);
		\draw[dotted, line width=1pt] (1.2,1.5)--(1.8,1.5);
		\draw[dotted, line width=1pt] (8.2,1.5)--(8.8,1.5);
		\node at (-0.8,1) {$=$};
		\node at (-7.3,1) {};
	\end{tikzpicture} \\
	
	\vspace{0.5pt}
	
	\begin{tikzpicture}[x=5mm,y=5mm]
		\draw (0,0)--(13.7,0)--(13.7,1)--(0,1)--cycle;
		\draw (0,1)--(0,2)--(15,2)--(15,1)--(13.7,1);
		\node at (2.5,0.5) {$a_{k+1} - k + 1$};
		\node at (2.1,1.5) {$a_k - k + 1$};
		\node at (11.2,0.5) {$a_{k+1} - k + 1$};
		\node at (12.9,1.5) {$a_k - k + 1$};
		\draw[dotted, line width=1pt] (5.2,0.5)--(5.8,0.5);
		\draw[dotted, line width=1pt] (8,0.5)--(8.6,0.5);
		\draw[dotted, line width=1pt] (4.4,1.5)--(5,1.5);
		\draw[dotted, line width=1pt] (9.9,1.5)--(10.5,1.5);
		\node at (-0.8,1) {$=$};
		\node at (-6.4,0) {};
		\node at (15.5,0.5) {,};
	\end{tikzpicture} \\
	where the second equality follows from Proposition \ref{n=2thm} (1).
	Also,
	the induction hypothesis implies that 
	\begin{center}
		\begin{tikzpicture}[x=5mm,y=5mm]
			\draw (0,0)--(15,0)--(15,1)--(0,1)--cycle;
			\draw (0,3)--(18,3)--(18,4)--(0,4)--cycle;
			\draw (0,4)--(0,5)--(19,5)--(19,4)--(18,4);
			\draw (0,1)--(0,3);
			\draw (15,1)--(16,1)--(16,2)--(17,2)--(17,3);
			\draw[dotted, line width=1pt] (8,1.7)--(8,2.3);
			\node (a) at (1.5,0.5) {$a_n - k$};
			\node (b) at (1.8,3.5) {$a_{k+2} - k$};
			\node (c) at (1.8,4.5) {$a_{k+1} - k$};
			\node (d) at (13.5,0.5) {$a_n - k$};
			\node (e) at (16.2,3.5) {$a_{k+2} - k$};
			\node (f) at (17.2,4.5) {$a_{k+1} - k$};
			\draw[dotted, line width=1pt] (3,0.5)--(3.6,0.5);
			\draw[dotted, line width=1pt] (11.4,0.5)--(12,0.5);
			\draw[dotted, line width=1pt] (3.7,3.5)--(4.3,3.5);
			\draw[dotted, line width=1pt] (3.7,4.5)--(4.3,4.5);
			\draw[dotted, line width=1pt] (13.7,3.5)--(14.3,3.5);
			\draw[dotted, line width=1pt] (14.7,4.5)--(15.3,4.5);
			\node at (-6,2.5) {$\mathsf{P}^{A\mathrm{II}}_{k+1,\overline{k+1}}\left(\mathsf{\Phi}_{k+1}\left(T|_{k+1,\overline{k+1}}\right)\right) = $};
			\node at (19.5,2) {.};
		\end{tikzpicture}
	\end{center}
	Since $\mathsf{Res}_{k+2,\overline{k}}\,\,\mathsf{\Phi}_k(T|_{k,\overline{k}})$
	is the tableau obtained from $\mathsf{\Phi}_{k+1}(T|_{k+1,\overline{k+1}})$
        by increasing all entries by $1$, 
	it follows that 
	\begin{center}
		\begin{tikzpicture}[x=5mm,y=5mm]
			\draw (0,0)--(15,0)--(15,1)--(0,1)--cycle;
			\draw (0,3)--(18,3)--(18,4)--(0,4)--cycle;
			\draw (0,4)--(0,5)--(19,5)--(19,4)--(18,4);
			\draw (0,1)--(0,3);
			\draw (15,1)--(16,1)--(16,2)--(17,2)--(17,3);
			\draw[dotted, line width=1pt] (8,1.7)--(8,2.3);
			\node (a) at (2.3,0.5) {$a_n - k + 1$};
			\node (b) at (2.5,3.5) {$a_{k+2} - k + 1$};
			\node (c) at (2.5,4.5) {$a_{k+1} - k + 1$};
			\node (d) at (12.9,0.5) {$a_n - k + 1$};
			\node (e) at (15.5,3.5) {$a_{k+2} - k + 1$};
			\node (f) at (16.5,4.5) {$a_{k+1} - k + 1$};
			\draw[dotted, line width=1pt] (9.8,0.5)--(10.4,0.5);
			\draw[dotted, line width=1pt] (4.8,0.5)--(5.4,0.5);
			\draw[dotted, line width=1pt] (5.3,3.5)--(5.9,3.5);
			\draw[dotted, line width=1pt] (5.3,4.5)--(5.9,4.5);
			\draw[dotted, line width=1pt] (12.1,3.5)--(12.7,3.5);
			\draw[dotted, line width=1pt] (13.1,4.5)--(13.7,4.5);
			\node at (-5.7,2.5) {$\mathsf{P}^{A\mathrm{II}}_{k+2,\overline{k}}\left(\mathsf{Res}_{k+2,\overline{k}}\,\,\mathsf{\Phi}_k\left(T|_{k,\overline{k}}\right)\right) = $};
			\node at (19.5,2) {.};
		\end{tikzpicture}
	\end{center}
	Hence, by Lemma \ref{kLGP} (1), with $k$, $k+1$, $\ldots$, $\overline{k}$
  in place of $1$, $2$, $\ldots$, $2n$, respectively, 
	the proof of the equation is completed in the case that $k$ is even.

	Next,
	assume that $k$ is odd.
	Note that
	the following equality holds by Lemma \ref{lem1}.
	\begin{multline}\label{7.22}
		\left.\mathsf{pr}_{k+3,\overline{k}} \circ \mathsf{pr}_{k+2,\overline{k+1}} \circ \mathsf{\Phi}_{k+2} \left(T|_{k,\overline{k}} \right) \right| _{k,k+3}
		\\ =
		\left. \left( \mathsf{Res}_{k,k+1;\overline{k+1},\overline{k}} \,\, \mathsf{\Phi}_{k+2}\left( T|_{k,\overline{k}} \right) \right) \right| _{\overline{k+1} \to k+2,\,\,\overline{k} \to k+3}.
	\end{multline}
	Here, since $T$ is $\widehat{\mathfrak{g}}$-dominant, we have 
	$\mathsf{Res}_{k,k+3} \,\, \mathsf{\Phi}_k(T) = \mathsf{\Phi}_k(T)|_{k,k+3}$. 
	Therefore, by equation (\ref{7.22}), we see that 
	\begin{equation}
		\begin{aligned}
			\mathsf{Res}_{k,k+3} \,\, \mathsf{\Phi}_k (T|_{k,\overline{k}})
			 & =\left.\mathsf{\Phi}_k\left(T|_{k,\overline{k}}\right)\right|_{k,k+3}                                                                                                                                                          \\
			 & =\left.\mathsf{pr}_{k,\overline{k}} \circ \mathsf{\Phi}_{k+1}\left(T|_{k,\overline{k}}\right) \right|_{k,k+3}                                                                                                           \\
			 & =\left.\mathsf{pr}_{k,k+3} \circ \mathsf{pr}_{k+3,\overline{k}} \circ \mathsf{\Phi}_{k+1}\left(T|_{k,\overline{k}}\right) \right|_{k,k+3}                                                                        \\
			 & =\mathsf{pr}_{k,k+3}\left(\left.\mathsf{pr}_{k+3,\overline{k}} \circ \mathsf{\Phi}_{k+1}\left(T|_{k,\overline{k}}\right) \right|_{k,k+3}\right)                                                                  \\
			 & =\mathsf{pr}_{k,k+3}\left(\left.\mathsf{pr}_{k+3,\overline{k}} \circ \mathsf{pr}_{k+2,\overline{k+1}} \circ \mathsf{\Phi}_{k+2}\left(T|_{k,\overline{k}}\right) \right|_{k,k+3}\right)                    \\
			 & =\mathsf{pr}_{k,k+3}\left(\left.\left(\mathsf{Res}_{k,k+1;\overline{k+1},\overline{k}}\,\,\mathsf{\Phi}_{k+2}\left(T|_{k,\overline{k}}\right) \right)\right|_{\overline{k+1}\to k+2, \,\, \overline{k} \to k+3}\right) \\
			 & =\mathsf{pr}_{k,k+3}\left(\left.\left(\mathsf{Res}_{k,k+1;\overline{k+1},\overline{k}}\,\,T\right)\right|_{\overline{k+1}\to k+2, \,\, \overline{k} \to k+3}\right),
		\end{aligned}
	\end{equation}
	where the last equality follows since 
	$\mathsf{\Phi}_{k+2}(T|_{k,\overline{k}})|_{k,k+1;\overline{k+1},\overline{k}} = T|_{k,k+1;\overline{k+1},\overline{k}}$.
	Since $T$ is $\widehat{\mathfrak{g}}$-dominant,
	$\mathsf{Res}_{k,k+1;\overline{k+1},\overline{k}}\,\,T$ is $\widehat{\mathfrak{g}}_k$-dominant.
	Hence, we deduce that 
	\[
		\mathsf{P}^{A\mathrm{II}}_{k,k+3} \left(\mathsf{Res}_{k,k+3} \,\, \mathsf{\Phi}_k (T|_{k,\overline{k}})\right) = \mathsf{P}^{A\mathrm{II}}_{k,k+3} \left(\mathsf{pr}_{k,k+3}\left(\left.\left(\mathsf{Res}_{k,k+1;\overline{k+1},\overline{k}}\,\,T\right)\right|_{\overline{k+1}\to k+2, \,\, \overline{k} \to k+3}\right)\right)
	\]
	\begin{tikzpicture}[x=5mm,y=5mm]
		\draw (0,0)--(8,0)--(8,1)--(0,1)--cycle;
		\draw (0,1)--(0,2)--(10,2)--(10,1)--(8,1);
		\node at (1.2,0.5) {$k+2$};
		\node at (1.2,1.5) {$k+1$};
		\node at (6.8,0.5) {$k+2$};
		\node at (8.8,1.5) {$k+1$};
		\draw[dotted, line width=1pt] (2.6,0.5)--(3.2,0.5);
		\draw[dotted, line width=1pt] (4.8,0.5)--(5.4,0.5);
		\draw[dotted, line width=1pt] (2.6,1.5)--(3.2,1.5);
		\draw[dotted, line width=1pt] (6.8,1.5)--(7.4,1.5);
		\node at (-0.8,1) {$=$};
		\node at (-11.6,1) {};
	\end{tikzpicture} \\
	
	\vspace{0.5pt}
	
	\begin{tikzpicture}[x=5mm,y=5mm]
		\draw (0,0)--(13.7,0)--(13.7,1)--(0,1)--cycle;
		\draw (0,1)--(0,2)--(15,2)--(15,1)--(13.7,1);
		\node at (2.5,0.5) {$a_{k+1} - k + 1$};
		\node at (2.1,1.5) {$a_k - k + 1$};
		\node at (11.2,0.5) {$a_{k+1} - k + 1$};
		\node at (12.9,1.5) {$a_k - k + 1$};
		\draw[dotted, line width=1pt] (5.2,0.5)--(5.8,0.5);
		\draw[dotted, line width=1pt] (8,0.5)--(8.6,0.5);
		\draw[dotted, line width=1pt] (4.4,1.5)--(5,1.5);
		\draw[dotted, line width=1pt] (9.9,1.5)--(10.5,1.5);
		\node at (-0.8,1) {$=$};
		\node at (-10.8,1) {};
		\node at (15.5,0.5) {,};
	\end{tikzpicture} \\
	where the second equality follows from Proposition \ref{n=2thm} (1).
	Also, by the same argument as in the case that $k$ is even, 
	we get 
	\begin{center}
		\begin{tikzpicture}[x=5mm,y=5mm]
			\draw (0,0)--(15,0)--(15,1)--(0,1)--cycle;
			\draw (0,3)--(18,3)--(18,4)--(0,4)--cycle;
			\draw (0,4)--(0,5)--(19,5)--(19,4)--(18,4);\
			\draw (0,1)--(0,3);
			\draw (15,1)--(16,1)--(16,2)--(17,2)--(17,3);
			\draw[dotted, line width=1pt] (8,1.7)--(8,2.3);
			\node (a) at (2.3,0.5) {$a_n - k + 1$};
			\node (b) at (2.5,3.5) {$a_{k+2} - k + 1$};
			\node (c) at (2.5,4.5) {$a_{k+1} - k + 1$};
			\node (d) at (12.9,0.5) {$a_n - k + 1$};
			\node (e) at (15.5,3.5) {$a_{k+2} - k + 1$};
			\node (f) at (16.5,4.5) {$a_{k+1} - k + 1$};
			\draw[dotted, line width=1pt] (9.8,0.5)--(10.4,0.5);
			\draw[dotted, line width=1pt] (4.8,0.5)--(5.4,0.5);
			\draw[dotted, line width=1pt] (5.3,3.5)--(5.9,3.5);
			\draw[dotted, line width=1pt] (5.3,4.5)--(5.9,4.5);
			\draw[dotted, line width=1pt] (12.1,3.5)--(12.7,3.5);
			\draw[dotted, line width=1pt] (13.1,4.5)--(13.7,4.5);
			\node at (-5.8,2.5) {$\mathsf{P}^{A\mathrm{II}}_{k+2,\overline{k}}\left(\mathsf{Res}_{k+2,\overline{k}}\,\,\mathsf{\Phi}_k\left(T|_{k,\overline{k}}\right)\right) = $};
			\node at (19.5,2) {.};
		\end{tikzpicture}
	\end{center}
	Hence, by Lemma \ref{kLGP} (1), with $k$, $k+1$, $\ldots$, $\overline{k}$ in place of $1$, $2$, $\ldots$, $2n$, respectively, 
        the proof of the equation is completed in the case that $k$ is odd. 
	This proves the proposition.
\end{proof}

\begin{prop}
	We have $\mathsf{\Psi} \left( \mathsf{SST}_{2n}^{\text{$\widehat{\mathfrak{g}}$-$\mathrm{dom}$}}(\lambda) \right) \subset \mathsf{SST}_{2n}^{\text{$\mathfrak{k}$-$\mathrm{lw}$}}(\lambda)$.
\end{prop}

\begin{proof}
	Take an arbitrary 
	$T \in \mathsf{SST}_{2n}^{\text{$\widehat{\mathfrak{g}}$-dom}}(\lambda)$.
	We will prove 
	the following equation for $k=1,2,\ldots,n-1$ 
	by descending induction on $k$;
	by Definition \ref{5.13},
	the assertion 
	$\mathsf{\Psi}(T) \in \mathsf{SST}_{2n}^{\text{$\mathfrak{k}$-lw}}(\lambda)$
	follows from the case $k=1$.
	\begin{center}
		\begin{tikzpicture}[x=5mm,y=5mm]
			\draw (0,0)--(15,0)--(15,1)--(0,1)--cycle;
			\draw (0,3)--(18,3)--(18,4)--(0,4)--cycle;
			\draw (0,4)--(0,5)--(19,5)--(19,4)--(18,4);
			\draw (0,1)--(0,3);
			\draw (15,1)--(16,1)--(16,2)--(17,2)--(17,3);
			\draw[dotted, line width=1pt] (8,1.7)--(8,2.3);
			\node (a) at (2.3,0.5) {$b_n - k + 1$};
			\node (b) at (2.5,3.5) {$b_{k+1} - k + 1$};
			\node (c) at (2.3,4.5) {$b_k - k + 1$};
			\node (d) at (12.9,0.5) {$b_n - k + 1$};
			\node (e) at (15.5,3.5) {$b_{k+1} - k + 1$};
			\node (f) at (16.9,4.5) {$b_k - k + 1$};
			\draw[dotted, line width=1pt] (4.7,0.5)--(5.3,0.5);
			\draw[dotted, line width=1pt] (9.9,0.5)--(10.5,0.5);
			\draw[dotted, line width=1pt] (5,3.5)--(5.6,3.5);
			\draw[dotted, line width=1pt] (12.4,3.5)--(13,3.5);
			\draw[dotted, line width=1pt] (4.7,4.5)--(5.3,4.5);
			\draw[dotted, line width=1pt] (13.9,4.5)--(14.5,4.5);
			\node at (-4,2.5) {$\mathsf{P}^{A\mathrm{II}}_{k,\overline{k}}\left(\mathsf{\Phi}_k\left(T|_{k,\overline{k}}\right)\right) = $};
			\node at (19.5,2) {.};
		\end{tikzpicture}
	\end{center}
	First,
	we prove the equation in the case $k=n-1$.
        We can easily verify the following: 
	\[
		\begin{aligned}
		 & b_{n-1} - (n-1) + 1 =
		 \left\{\,
		 \begin{aligned}
			 & n-1   &  & \text{$n$ : even}, \\
			 & n &  & \text{$n$ : odd},
		 \end{aligned}
		 \right.&
		 & b_{n} - (n-1) + 1 =
		 \left\{\,
		 \begin{aligned}
			 & \overline{n-1}   &  & \text{$n$ : even}, \\
			 & \overline{n} &  & \text{$n$ : odd}.
		 \end{aligned}
		 \right.
	  \end{aligned}
	\]
	Since $T$ is $\widehat{\mathfrak{g}}$-dominant,
	$T|_{n-1,n;\overline{n},\overline{n-1}}$
	is $\widehat{\mathfrak{g}}_n$-dominant.
	Therefore, by Proposition \ref{n=2thm} (2), with $n-1$, $n$, $\overline{n}$, and $\overline{n-1}$ in place of $1$, $2$, $3$, and $4$, respectively, 
we see that 
	\[
		\mathsf{P}^{A\mathrm{II}}_{n-1,\overline{n-1}}\left(\mathsf{\Psi}_{n-1}\left(T|_{n-1,\overline{n-1}}\right)\right)
		=
		\left\{\,
		\begin{aligned}
			 & \mathsf{P}^{A\mathrm{II}}_{n-1,\overline{n-1}} \left( \mathsf{pr}_{n,\overline{n-1}} \circ \mathsf{pr}_{n,\overline{n}} (T|_{n-1,\overline{n-1}}) \right)  &  & \text{$n$ : even,} \\
			 & \mathsf{P}^{A\mathrm{II}}_{n-1,\overline{n-1}} \left( \mathsf{pr}_{n-1,\overline{n-1}}  (T|_{n-1,\overline{n-1}}) \right) &  & \text{$n$ : odd,}
		\end{aligned}
		\right.
	\]
	\begin{tikzpicture}[x=5mm,y=5mm]
		\draw (0,0)--(8,0)--(8,1)--(0,1)--cycle;
		\draw (0,1)--(0,2)--(10,2)--(10,1)--(8,1);
		\node at (1.2,0.5) {$\overline{n-1}$};
		\node at (1.2,1.5) {$n-1$};
		\node at (6.8,0.5) {$\overline{n-1}$};
		\node at (8.8,1.5) {$n-1$};
		\draw[dotted, line width=1pt] (2.5,0.5)--(3.1,0.5);
		\draw[dotted, line width=1pt] (4.9,0.5)--(5.5,0.5);
		\draw[dotted, line width=1pt] (2.5,1.5)--(3.1,1.5);
		\draw[dotted, line width=1pt] (6.9,1.5)--(7.5,1.5);
		
		\draw (0,-3)--(8,-3)--(8,-2)--(0,-2)--cycle;
		\draw (0,-2)--(0,-1)--(10,-1)--(10,-2)--(8,-2);
		\node at (0.5,-2.5) {$\overline{n}$};
		\node at (0.5,-1.5) {$n$};
		\node at (7.5,-2.5) {$\overline{n}$};
		\node at (9.5,-1.5) {$n$};
		\draw[dotted, line width=1pt] (1.2,-2.5)--(1.8,-2.5);
		\draw[dotted, line width=1pt] (6.2,-2.5)--(6.8,-2.5);
		\draw[dotted, line width=1pt] (1.2,-1.5)--(1.8,-1.5);
		\draw[dotted, line width=1pt] (8.2,-1.5)--(8.8,-1.5);
		\node at (-1,-0.5) {$= \left\{\begin{aligned} \\ \\ \\ \\ \\ \end{aligned}\right.$};
		\node at (12.5,1) {$n$ : even,};
		\node at (12.5,-2) {$n$ : odd,};
		\node at (-13,0) {};
	\end{tikzpicture} \\
	\begin{tikzpicture}[x=5mm,y=5mm]
		\draw (0,0)--(13.7,0)--(13.7,1)--(0,1)--cycle;
		\draw (0,1)--(0,2)--(15,2)--(15,1)--(13.7,1);
		\node at (3.1,0.5) {$b_n - (n-1) + 1$};
		\node at (3.5,1.5) {$b_{n-1} - (n-1) + 1$};
		\node at (10.5,0.5) {$b_n - (n-1) + 1$};
		\node at (11.5,1.5) {$b_{n-1} - (n-1) + 1$};
		\draw[dotted, line width=1pt] (6.5,0.5)--(7.1,0.5);
		\draw[dotted, line width=1pt] (7.2,1.5)--(7.8,1.5);
		
		\draw (0,-3)--(13.7,-3)--(13.7,-2)--(0,-2)--cycle;
		\draw (0,-2)--(0,-1)--(15,-1)--(15,-2)--(13.7,-2);
		\node at (3.1,-2.5) {$b_n - (n-1) + 1$};
		\node at (3.5,-1.5) {$b_{n-1} - (n-1) + 1$};
		\node at (10.5,-2.5) {$b_n - (n-1) + 1$};
		\node at (11.5,-1.5) {$b_{n-1} - (n-1) + 1$};
		\draw[dotted, line width=1pt] (6.5,-2.5)--(7.1,-2.5);
		\draw[dotted, line width=1pt] (7.2,-1.5)--(7.8,-1.5);
		\node at (-1,-0.5) {$= \left\{\begin{aligned} \\ \\ \\ \\ \\ \end{aligned}\right.$};
		\node at (17,1) {$n$ : even,};
		\node at (17,-2) {$n$ : odd.};
		\node at (-13,0) {};
	\end{tikzpicture} \\
	This completes the proof of the equation in the case $k=n-1$.
	
  Let $k \leq n-2$.
  The proof is divided into two cases according as $k$ is even or odd. 
	First, assume that $k$ is even.
	Note that the following equality holds by Lemma \ref{lem1}.
	\begin{multline}
		\left.\mathsf{pr}_{k+3,\overline{k}} \circ \mathsf{pr}_{k+2,\overline{k+1}} \circ \mathsf{\Psi}_{k+2} \left( T|_{k,\overline{k}} \right) \right| _{k,k+3}
		\\
		= \left.\left( \mathsf{Res}_{k,k+1;\overline{k+1},\overline{k}}\,\,\mathsf{\Psi}_{k+2}\left(T|_{k,\overline{k}}\right) \right)\right| _{\overline{k+1} \to k+2,\,\,\overline{k} \to k+3}.
	\end{multline}
	Here, 
	since $T$ is $\widehat{\mathfrak{g}}$-dominant, we have 
	$\mathsf{Res}_{k,k+3} \,\, \mathsf{\Psi}_k(T) = \mathsf{\Psi}_k(T)|_{k,k+3}$. 
	Therefore, we see that 
	\begin{equation}
		\begin{aligned}
			\mathsf{Res}_{k,k+3} \,\, \mathsf{\Psi}_k(T|_{k,\overline{k}}) & = \mathsf{\Psi}_k(T|_{k,\overline{k}})|_{k,k+3}                                                                                                                                                                                                                                 \\
			                                           & = \left.\mathsf{pr}_{k,\overline{k}} \circ \mathsf{\Psi}_{k+1} \left( T|_{k,\overline{k}} \right) \right|_{k,k+3}                                                                                                                             \\
			                                           & = \left.\mathsf{pr}_{k,k+3} \circ \mathsf{pr}_{k+3,\overline{k}} \circ \mathsf{\Psi}_{k+1} \left( T|_{k,\overline{k}} \right) \right|_{k,k+3}                                                                               \\
			                                           & = \mathsf{pr}_{k,k+3} \left(\left.\mathsf{pr}_{k+3,\overline{k}} \circ \mathsf{\Psi}_{k+1} \left( T|_{k,\overline{k}} \right) \right|_{k,k+3}  \right)              \\
																								 & = \mathsf{pr}_{k,k+3} \left(\left.\mathsf{pr}_{k+3,\overline{k}} \circ \mathsf{pr}_{k+2,\overline{k+1}} \circ \mathsf{\Psi}_{k+2} \left( T|_{k,\overline{k}} \right) \right|_{k,k+3}  \right)              \\
																								 & = \mathsf{pr}_{k,k+3} \left(\left.\left( \mathsf{Res}_{k,k+1;\overline{k+1},\overline{k}}\,\,\mathsf{\Psi}_{k+2}\left(T|_{k,\overline{k}}\right)\right) \right|_{\overline{k+1} \to k+2,\,\,\overline{k} \to k+3}  \right)              \\
																								 & = \mathsf{pr}_{k,k+3} \left(\left.\left( \mathsf{Res}_{k,k+1;\overline{k+1},\overline{k}}\,\,T\right) \right|_{\overline{k+1} \to k+2,\,\,\overline{k} \to k+3}  \right),
		\end{aligned}
	\end{equation}
	where the last equality follows since $\mathsf{\Psi}_{k+2}(T|_{k,\overline{k}})|_{k,k+1;\overline{k+1},\overline{k}} = T|_{k,k+1;\overline{k+1},\overline{k}}$.
	Since $T$ is a $\widehat{\mathfrak{g}}$-dominant tableau,
	$\mathsf{Res}_{k,k+1;\overline{k+1},\overline{k}}\,\,T$ is $\widehat{\mathfrak{g}}_k$-dominant.
	Hence, we deduce that
	\begin{multline}
		\mathsf{P}^{A\mathrm{II}}_{k,k+3} \left(\mathsf{Res}_{k,k+3} \,\, \mathsf{\Psi}_k (T|_{k,\overline{k}})\right)
		\\ = \mathsf{P}^{A\mathrm{II}}_{k,k+3} \left(\mathsf{pr}_{k,k+3} \left(\left.\left(\mathsf{Res}_{k,k+1;\overline{k+1},\overline{k}}\,\,T\right)\right|_{\overline{k+1} \to k+2, \,\, \overline{k} \to k+3}\right)\right)
	\end{multline}
	\begin{tikzpicture}[x=5mm,y=5mm]
		\draw (0,0)--(8,0)--(8,1)--(0,1)--cycle;
		\draw (0,1)--(0,2)--(10,2)--(10,1)--(8,1);
		\node at (1.2,0.5) {$k+2$};
		\node at (1.2,1.5) {$k+1$};
		\node at (6.8,0.5) {$k+2$};
		\node at (8.8,1.5) {$k+1$};
		\draw[dotted, line width=1pt] (2.6,0.5)--(3.2,0.5);
		\draw[dotted, line width=1pt] (4.8,0.5)--(5.4,0.5);
		\draw[dotted, line width=1pt] (2.6,1.5)--(3.2,1.5);
		\draw[dotted, line width=1pt] (6.8,1.5)--(7.4,1.5);
		\node at (-0.8,1) {$=$};
		\node at (-12,1) {};
	\end{tikzpicture} \\
	
	\vspace{0.5pt}
	
	\begin{tikzpicture}[x=5mm,y=5mm]
		\draw (0,0)--(13.7,0)--(13.7,1)--(0,1)--cycle;
		\draw (0,1)--(0,2)--(15,2)--(15,1)--(13.7,1);
		\node at (2.5,0.5) {$b_{k+1} - k + 1$};
		\node at (2.1,1.5) {$b_k - k + 1$};
		\node at (11.2,0.5) {$b_{k+1} - k + 1$};
		\node at (12.9,1.5) {$b_k - k + 1$};
		\draw[dotted, line width=1pt] (5.2,0.5)--(5.8,0.5);
		\draw[dotted, line width=1pt] (8,0.5)--(8.6,0.5);
		\draw[dotted, line width=1pt] (4.4,1.5)--(5,1.5);
		\draw[dotted, line width=1pt] (9.9,1.5)--(10.5,1.5);
		\node at (-0.8,1) {$=$};
		\node at (-11.1,0) {};
		\node at (15.5,0.5) {,};
	\end{tikzpicture} \\
	where the second equality follows from Proposition \ref{n=2thm} (2).
	Also, 
	the induction hypothesis implies that 
	\begin{center}
		\begin{tikzpicture}[x=5mm,y=5mm]
			\draw (0,0)--(15,0)--(15,1)--(0,1)--cycle;
			\draw (0,3)--(18,3)--(18,4)--(0,4)--cycle;
			\draw (0,4)--(0,5)--(19,5)--(19,4)--(18,4);
			\draw (0,1)--(0,3);
			\draw (15,1)--(16,1)--(16,2)--(17,2)--(17,3);
			\draw[dotted, line width=1pt] (8,1.7)--(8,2.3);
			\node (a) at (1.5,0.5) {$b_n - k$};
			\node (b) at (1.8,3.5) {$b_{k+2} - k$};
			\node (c) at (1.8,4.5) {$b_{k+1} - k$};
			\node (d) at (13.5,0.5) {$b_n - k$};
			\node (e) at (16.2,3.5) {$b_{k+2} - k$};
			\node (f) at (17.2,4.5) {$b_{k+1} - k$};
			\draw[dotted, line width=1pt] (3,0.5)--(3.6,0.5);
			\draw[dotted, line width=1pt] (11.4,0.5)--(12,0.5);
			\draw[dotted, line width=1pt] (3.7,3.5)--(4.3,3.5);
			\draw[dotted, line width=1pt] (3.7,4.5)--(4.3,4.5);
			\draw[dotted, line width=1pt] (13.7,3.5)--(14.3,3.5);
			\draw[dotted, line width=1pt] (14.7,4.5)--(15.3,4.5);
			\node at (-6,2.5) {$\mathsf{P}^{A\mathrm{II}}_{k+1,\overline{k+1}}\left(\mathsf{\Psi}_{k+1}\left(T|_{k+1,\overline{k+1}}\right)\right) = $};
			\node at (19.5,2) {.};
		\end{tikzpicture}
	\end{center}
	 Since $\mathsf{Res}_{k+2,\overline{k}}\,\,\mathsf{\Psi}_k(T|_{k,\overline{k}})$
	is the tableau obtained from $\mathsf{\Psi}_{k+1}(T|_{k+1,\overline{k+1}})$
        by increasing all entries by $1$, 
	we get 
	\begin{center}
		\begin{tikzpicture}[x=5mm,y=5mm]
			\draw (0,0)--(15,0)--(15,1)--(0,1)--cycle;
			\draw (0,3)--(18,3)--(18,4)--(0,4)--cycle;
			\draw (0,4)--(0,5)--(19,5)--(19,4)--(18,4);
			\draw (0,1)--(0,3);
			\draw (15,1)--(16,1)--(16,2)--(17,2)--(17,3);
			\draw[dotted, line width=1pt] (8,1.7)--(8,2.3);
			\node (a) at (2.3,0.5) {$b_n - k + 1$};
			\node (b) at (2.5,3.5) {$b_{k+2} - k + 1$};
			\node (c) at (2.5,4.5) {$b_{k+1} - k + 1$};
			\node (d) at (12.9,0.5) {$b_n - k + 1$};
			\node (e) at (15.5,3.5) {$b_{k+2} - k + 1$};
			\node (f) at (16.5,4.5) {$b_{k+1} - k + 1$};
			\draw[dotted, line width=1pt] (9.8,0.5)--(10.4,0.5);
			\draw[dotted, line width=1pt] (4.8,0.5)--(5.4,0.5);
			\draw[dotted, line width=1pt] (5.3,3.5)--(5.9,3.5);
			\draw[dotted, line width=1pt] (5.3,4.5)--(5.9,4.5);
			\draw[dotted, line width=1pt] (12.1,3.5)--(12.7,3.5);
			\draw[dotted, line width=1pt] (13.1,4.5)--(13.7,4.5);
			\node at (-5.7,2.5) {$\mathsf{P}^{A\mathrm{II}}_{k+2,\overline{k}}\left(\mathsf{Res}_{k+2,\overline{k}}\,\,\mathsf{\Psi}_k\left(T|_{k,\overline{k}}\right)\right) = $};
			\node at (19.5,2) {.};
		\end{tikzpicture}
	\end{center}
	Hence, by Lemma \ref{kLGP} (2), 
	the proof of the equation is completed in the case that $k$ is even. 

	Next,
	assume that $k$ is odd.
	Note that
	\begin{multline}
		\left.\mathsf{pr}_{k+3,\overline{k}} \circ \mathsf{pr}_{k+2,\overline{k+1}} \circ \mathsf{\Psi}_{k+2} \left(T|_{k,\overline{k}} \right) \right| _{k,k+3}
		\\ =
		\left. \left( \mathsf{Res}_{k,k+1;\overline{k+1},\overline{k}} \,\, \mathsf{\Psi}_{k+2}\left( T|_{k,\overline{k}} \right) \right) \right| _{\overline{k+1} \to k+2,\,\,\overline{k} \to k+3}.
	\end{multline}
	Here, since $T$ is $\widehat{\mathfrak{g}}$-dominant, we have 
	$\mathsf{Res}_{k,k+3} \,\, \mathsf{\Psi}_k(T) = \mathsf{\Psi}_k(T)|_{k,k+3}$.
	Therefore, we see that 
	\[
		\begin{aligned}
			\mathsf{Res}_{k,k+3} \,\, \mathsf{\Psi}_k (T|_{k,\overline{k}})
			 & =\left.\mathsf{\Psi}_k\left(T|_{k,\overline{k}}\right)\right|_{k,k+3}                                                                                                                                                          \\
			 & =\left.\mathsf{pr}_{k+1,\overline{k}} \circ \mathsf{\Psi}_{k+1}\left(T|_{k,\overline{k}}\right) \right|_{k,k+3}                                                                                                           \\
			 & =\left.\mathsf{pr}_{k+1,\overline{k}} \circ \mathsf{pr}_{k+1,\overline{k+1}} \circ \mathsf{\Psi}_{k+2}\left(T|_{k,\overline{k}}\right) \right|_{k,k+3}                                                                        \\
		   & =\left.\mathsf{pr}_{k+1,k+3} \circ \mathsf{pr}_{k+3,\overline{k}} \circ \mathsf{pr}_{k+1,k+2} \circ \mathsf{pr}_{k+2,\overline{k+1}} \circ \mathsf{\Psi}_{k+2}\left(T|_{k,\overline{k}}\right) \right|_{k,k+3} \\
			 & =\mathsf{pr}_{k+1,k+3}\circ\mathsf{pr}_{k+1,k+2} \left(\mathsf{pr}_{k+3,\overline{k}} \circ \mathsf{pr}_{k+2,\overline{k+1}} \circ \mathsf{\Psi}_{k+2}(T|_{k,\overline{k}})|_{k,k+3} \right) \\
			 & =\mathsf{pr}_{k+1,k+3}\circ\mathsf{pr}_{k+1,k+2} \left(\left(\mathsf{Res}_{k,k+1;\overline{k+1},\overline{k}}\,\,\mathsf{\Psi}_{k+2}(T|_{k,\overline{k}})\right)|_{\overline{k+1}\to k+2, \,\, \overline{k} \to k+3}\right) \\
			 & =\mathsf{pr}_{k+1,k+3}\circ\mathsf{pr}_{k+1,k+2} \left(\left(\mathsf{Res}_{k,k+1;\overline{k+1},\overline{k}}\,\,T|_{k,\overline{k}}\right)|_{\overline{k+1}\to k+2, \,\, \overline{k} \to k+3} \right).
		\end{aligned}
	\]
	Since $T$ is $\widehat{\mathfrak{g}}$-dominant,
	$\mathsf{Res}_{k,k+1;\overline{k+1},\overline{k}}\,\,T$ is $\widehat{\mathfrak{g}}_k$-dominant.
	Hence, by Proposition \ref{n=2thm} (2), we deduce that 
	\[
		\mathsf{P}^{A\mathrm{II}}_{k,k+3} \left(\mathsf{Res}_{k,k+3} \,\, \mathsf{\Psi}_k (T|_{k,\overline{k}})\right) = \mathsf{P}^{A\mathrm{II}}_{k,k+3} \left(\mathsf{pr}_{k+1,k+3}\circ\mathsf{pr}_{k+1,k+2}\left(\left.\left(\mathsf{Res}_{k,k+1;\overline{k+1},\overline{k}}\,\,T\right)\right|_{\overline{k+1}\to k+2, \,\, \overline{k} \to k+3}\right)\right)
	\]
	\begin{tikzpicture}[x=5mm,y=5mm]
		\draw (0,0)--(8,0)--(8,1)--(0,1)--cycle;
		\draw (0,1)--(0,2)--(10,2)--(10,1)--(8,1);
		\node at (1.2,0.5) {$k+2$};
		\node at (0.5,1.5) {$k$};
		\node at (6.8,0.5) {$k+2$};
		\node at (9.5,1.5) {$k$};
		\draw[dotted, line width=1pt] (2.6,0.5)--(3.2,0.5);
		\draw[dotted, line width=1pt] (4.8,0.5)--(5.4,0.5);
		\draw[dotted, line width=1pt] (1.2,1.5)--(1.8,1.5);
		\draw[dotted, line width=1pt] (8.2,1.5)--(8.8,1.5);
		\node at (-0.8,1) {$=$};
		\node at (-10.6,1) {};
	\end{tikzpicture} \\
	
	\vspace{0.5pt}
	
	\begin{tikzpicture}[x=5mm,y=5mm]
		\draw (0,0)--(13.7,0)--(13.7,1)--(0,1)--cycle;
		\draw (0,1)--(0,2)--(15,2)--(15,1)--(13.7,1);
		\node at (2.5,0.5) {$b_{k+1} - k + 1$};
		\node at (2.1,1.5) {$b_k - k + 1$};
		\node at (11.2,0.5) {$b_{k+1} - k + 1$};
		\node at (12.9,1.5) {$b_k - k + 1$};
		\draw[dotted, line width=1pt] (5.2,0.5)--(5.8,0.5);
		\draw[dotted, line width=1pt] (8,0.5)--(8.6,0.5);
		\draw[dotted, line width=1pt] (4.4,1.5)--(5,1.5);
		\draw[dotted, line width=1pt] (9.9,1.5)--(10.5,1.5);
		\node at (-0.8,1) {$=$};
		\node at (-9.7,1) {};
		\node at (15.5,0.5) {.};
	\end{tikzpicture} \\
	Also, by the same argument as in the case that $k$ is even, 
	we obtain 
	\begin{center}
		\begin{tikzpicture}[x=5mm,y=5mm]
			\draw (0,0)--(15,0)--(15,1)--(0,1)--cycle;
			\draw (0,3)--(18,3)--(18,4)--(0,4)--cycle;
			\draw (0,4)--(0,5)--(19,5)--(19,4)--(18,4);\
			\draw (0,1)--(0,3);
			\draw (15,1)--(16,1)--(16,2)--(17,2)--(17,3);
			\draw[dotted, line width=1pt] (8,1.7)--(8,2.3);
			\node (a) at (2.3,0.5) {$b_n - k + 1$};
			\node (b) at (2.5,3.5) {$b_{k+2} - k + 1$};
			\node (c) at (2.5,4.5) {$b_{k+1} - k + 1$};
			\node (d) at (12.9,0.5) {$b_n - k + 1$};
			\node (e) at (15.5,3.5) {$b_{k+2} - k + 1$};
			\node (f) at (16.5,4.5) {$b_{k+1} - k + 1$};
			\draw[dotted, line width=1pt] (9.8,0.5)--(10.4,0.5);
			\draw[dotted, line width=1pt] (4.8,0.5)--(5.4,0.5);
			\draw[dotted, line width=1pt] (5.3,3.5)--(5.9,3.5);
			\draw[dotted, line width=1pt] (5.3,4.5)--(5.9,4.5);
			\draw[dotted, line width=1pt] (12.1,3.5)--(12.7,3.5);
			\draw[dotted, line width=1pt] (13.1,4.5)--(13.7,4.5);
			\node at (-5.8,2.5) {$\mathsf{P}^{A\mathrm{II}}_{k+2,\overline{k}}\left(\mathsf{Res}_{k+2,\overline{k}}\,\,\mathsf{\Psi}_k\left(T|_{k,\overline{k}}\right)\right) = $};
			\node at (19.5,2) {.};
		\end{tikzpicture}
	\end{center}
	Hence, by Lemma \ref{kLGP} (2), with $k$, $k+1$, $\ldots$, $\overline{k}$ in place of $1$, $2$, $\ldots$, $2n$, respectively, 
  the proof of the equation is completed in the case that $k$ is odd. 
	This proves the proposition. 
\end{proof}

\begin{prop} We have 
	\begin{equation}
		\begin{aligned}
			& \mathsf{\Phi} \left( \mathsf{SST}_{2n}^{\text{$\widehat{\mathfrak{g}}$-$\mathrm{dom}$}}(\lambda) \right) \supset \mathsf{SST}_{2n}^{\text{$\mathfrak{k}$-$\mathrm{hw}$}}(\lambda),&
			& \mathsf{\Psi} \left( \mathsf{SST}_{2n}^{\text{$\widehat{\mathfrak{g}}$-$\mathrm{dom}$}}(\lambda) \right) \supset \mathsf{SST}_{2n}^{\text{$\mathfrak{k}$-$\mathrm{lw}$}}(\lambda).
		\end{aligned}
	\end{equation}
\end{prop}

\begin{proof}
	We prove the assertion of the proposition by induction on $n$.
	The assertion in the case $n=2$ follows from Proposition \ref{n=2thm}.
	Assume that $n \geq 3$.
	For the first inclusion, let us take
	$S \in \mathsf{SST}_{2n}^{\text{$\mathfrak{k}$-hw}}(\lambda)$.
	We will prove that
	$\mathsf{\Phi}^{-1}(S) \in \mathsf{SST}_{2n}^{\text{$\widehat{\mathfrak{g}}$-dom}}(\lambda)$.
	Note that
	\begin{equation}
		\left.\mathsf{pr}_{4,\overline{1}} \circ \mathsf{pr}_{3,\overline{2}} \left(\mathsf{\Phi}_3 \circ \mathsf{\Phi}^{-1}(S)\right) \right| _{1,4}
		=
		\left.\left( \mathsf{Res}_{1,2;\overline{2},\overline{1}}\,\,\mathsf{\Phi}_3 \circ \mathsf{\Phi}^{-1}(S) \right)\right| _{\overline{2} \to 3,\,\,\overline{1} \to 4}.
	\end{equation}
	From this equality, we see that 
	\begin{equation}
		\begin{aligned}
			\mathsf{Res}_{1,4}\,\,S & = \mathsf{Res}_{1,4}\,\,\mathsf{\Phi}\circ\mathsf{\Phi}^{-1}(S)                                                                                                                                         \\
			                               & = \left.\mathsf{\Phi} \circ \mathsf{\Phi}^{-1} (S)\right|_{1,4}                                                                                                                                         \\
			                               & = \left.\mathsf{pr}_{1,\overline{1}} \circ \mathsf{\Phi}_2 \circ \mathsf{\Phi}^{-1} (S) \right|_{1,4}                                                                                            \\
			                               & = \left.\mathsf{pr}_{1,4} \circ \mathsf{pr}_{4,\overline{1}} \circ \mathsf{pr}_{3,\overline{2}} \circ \mathsf{\Phi}_3 \circ \mathsf{\Phi}^{-1} (S)\right|_{1,4}                    \\
			                               & = \mathsf{pr}_{1,4} \left( \left. \mathsf{pr}_{4,\overline{1}} \circ \mathsf{pr}_{3,\overline{2}} \left( \mathsf{\Phi}_3 \circ \mathsf{\Phi}^{-1}(S) \right) \right|_{1,4} \right) \\
			                               & = \mathsf{pr}_{1,4} \left( \left.\left( \mathsf{Res}_{1,2;\overline{2},\overline{1}}\,\,\mathsf{\Phi}_3 \circ \mathsf{\Phi}^{-1}(S) \right)\right| _{\overline{2} \to 3,\,\,\overline{1} \to 4} \right) \\
			                               & = \mathsf{pr}_{1,4} \left( \left.\left( \mathsf{Res}_{1,2;\overline{2},\overline{1}}\,\,\mathsf{\Phi}^{-1}(S) \right)\right| _{\overline{2} \to 3,\,\,\overline{1} \to 4} \right).
		\end{aligned}
	\end{equation}
	Therefore, we get
	\begin{equation}
		\mathsf{Res}_{1,2;\overline{2},\overline{1}}\,\,\mathsf{\Phi}^{-1}(S)
		=
		\left.\mathsf{pr}_{1,4}^{-1} \left( \mathsf{Res}_{1,4}\,\,S \right)\right|_{3 \to \overline{2},\,\,4\to\overline{1}}.
	\end{equation}
	Since $S$ is a $\mathfrak{k}$-highest weight tableau,
	it follows from Lemma \ref{kLGP} (1) that 
	\begin{center}
		\begin{tikzpicture}[x=5mm,y=5mm]
			\draw (0,0)--(8,0)--(8,1)--(0,1)--cycle;
			\draw (0,1)--(0,2)--(10,2)--(10,1)--(8,1);
			\node at (0.5,0.5) {$3$};
			\node at (0.5,1.5) {$2$};
			\node at (7.5,0.5) {$3$};
			\node at (9.5,1.5) {$2$};
			\draw[dotted, line width=1pt] (1.2,0.5)--(1.8,0.5);
			\draw[dotted, line width=1pt] (6.2,0.5)--(6.8,0.5);
			\draw[dotted, line width=1pt] (1.2,1.5)--(1.8,1.5);
			\draw[dotted, line width=1pt] (8.2,1.5)--(8.8,1.5);
			\node at (-3.5,1) {$\mathsf{P}^{A\mathrm{II}}_{1,4} \left( \mathsf{Res}_{1,4}\,\, S \right)  =$};
			\node at (-4.45,1) {};
			\node at (10.5,0.5) {.};
		\end{tikzpicture}
	\end{center}
	Hence,
	$\mathsf{Res}_{1,2;\overline{2},\overline{1}}\,\,\mathsf{\Phi}^{-1}(S)$
	is $\widehat{\mathfrak{g}}_1$-dominant by Proposition \ref{n=2thm} (1).
	Also, we have 
	\begin{equation}
		\mathsf{\Phi}^{-1}(S)|_{2,\overline{2}}
		=
		\left.\mathsf{\Phi}_2^{-1} \circ \mathsf{pr}_{1,\overline{1}}^{-1} (S) \right| _{2,\overline{2}}
		=
		\mathsf{\Phi}_2^{-1} \left(\left.\mathsf{pr}_{1,\overline{1}}^{-1}(S)\right|_{2,\overline{2}}\right).
	\end{equation}
	Since $\mathsf{pr}_{1,\overline{1}}^{-1}(S)$ 
	is the tableau obtained from $\mathsf{Res}_{3,\overline{1}}\,\,S$
  by decreasing all entries by $1$, 
	we deduce that 
	\begin{center}
		\begin{tikzpicture}[x=5mm,y=5mm]
			\draw (0,0)--(15,0)--(15,1)--(0,1)--cycle;
			\draw (0,3)--(18,3)--(18,4)--(0,4)--cycle;
			\draw (0,4)--(0,5)--(19,5)--(19,4)--(18,4);
			\draw (0,1)--(0,3);
			\draw (15,1)--(16,1)--(16,2)--(17,2)--(17,3);
			\draw[dotted, line width=1pt] (8,1.7)--(8,2.3);
			\node (a) at (1.4,0.5) {$a_n - 1$};
			\node (b) at (1.4,3.5) {$a_3 - 1$};
			\node (c) at (1.4,4.5) {$a_2 - 1$};
			\node (d) at (13.5,0.5) {$a_n - 1$};
			\node (e) at (16.6,3.5) {$a_3 - 1$};
			\node (f) at (17.6,4.5) {$a_2 - 1$};
			\draw[dotted, line width=1pt] (3,0.5)--(3.6,0.5);
			\draw[dotted, line width=1pt] (11.3,0.5)--(11.9,0.5);
			\draw[dotted, line width=1pt] (3,3.5)--(3.6,3.5);
			\draw[dotted, line width=1pt] (14.4,3.5)--(15,3.5);
			\draw[dotted, line width=1pt] (3,4.5)--(3.6,4.5);
			\draw[dotted, line width=1pt] (15.4,4.5)--(16,4.5);
			\node at (-3.7,2.5) {$\mathsf{P}^{A\mathrm{II}}_{2,\overline{2}} \left( \mathsf{pr}_{1,\overline{1}}^{-1}(S) \right)  =$};
		\end{tikzpicture}
		
		\vspace{15pt}
		
		\begin{tikzpicture}[x=5mm,y=5mm]
			\draw (0,0)--(15,0)--(15,1)--(0,1)--cycle;
			\draw (0,3)--(18,3)--(18,4)--(0,4)--cycle;
			\draw (0,4)--(0,5)--(19,5)--(19,4)--(18,4);
			\draw (0,1)--(0,3);
			\draw (15,1)--(16,1)--(16,2)--(17,2)--(17,3);
			\draw[dotted, line width=1pt] (8,1.7)--(8,2.3);
			\node (a) at (1.8,0.5) {$b_{n-1} + 1$};
			\node (b) at (1.4,3.5) {$b_2 + 1$};
			\node (c) at (1.4,4.5) {$b_1 + 1$};
			\node (d) at (13.2,0.5) {$b_{n-1} + 1$};
			\node (e) at (16.6,3.5) {$b_2 + 1$};
			\node (f) at (17.6,4.5) {$b_1 + 1$};
			\draw[dotted, line width=1pt] (3.5,0.5)--(4.1,0.5);
			\draw[dotted, line width=1pt] (10.8,0.5)--(11.4,0.5);
			\draw[dotted, line width=1pt] (3,3.5)--(3.6,3.5);
			\draw[dotted, line width=1pt] (14.4,3.5)--(15,3.5);
			\draw[dotted, line width=1pt] (3,4.5)--(3.6,4.5);
			\draw[dotted, line width=1pt] (15.4,4.5)--(16,4.5);
			\node at (-1,2.5) {$=$};
			\node at (-7.6,2.5) {};
			\node at (19.5,2) {.};
		\end{tikzpicture}
	\end{center}
	Here, observe that 
	$\mathsf{\Phi}_2$
	is just 
	$\mathsf{\Psi}_1$ for 
	$\mathfrak{gl}_{2(n-1)}(\mathbb{C})$, with all the subscripts of promotion operators        increased by $1$. 
	From this observation, together with the induction hypothesis, we conclude that 
	$\mathsf{\Phi}^{-1}(S)|_{2,\overline{2}}$
	is a $\widehat{\mathfrak{g}}_{\geq2}$-dominant tableau.
	Hence it follows from Lemma \ref{gLGP} that 
	$\mathsf{\Phi}^{-1}(S) \in \mathsf{SST}_{2n}^{\text{$\widehat{\mathfrak{g}}$-dom}}(\lambda)$, as desired. 
	
	Next, for the second inclusion
	let us take an arbitrary 
	$S' \in \mathsf{SST}_{2n}^{\text{$\mathfrak{k}$-lw}}(\lambda)$.
	We will prove that $\mathsf{\Psi}^{-1}(S') \in \mathsf{SST}_{2n}^{\text{$\widehat{\mathfrak{g}}$-dom}}(\lambda)$.
	Note that
	\begin{equation}
		\left.\mathsf{pr}_{4,\overline{1}} \circ \mathsf{pr}_{3,\overline{2}} \left(\mathsf{\Psi}_3 \circ \mathsf{\Psi}^{-1}(S')\right) \right| _{1,4}
		=
		\left.\left( \mathsf{Res}_{1,2;\overline{2},\overline{1}}\,\,\mathsf{\Psi}_3 \circ \mathsf{\Phi}^{-1}(S') \right)\right| _{\overline{2} \to 3,\,\,\overline{1} \to 4}.
	\end{equation}
	From this equality, we see that 
	\begin{equation}
		\begin{aligned}
			\mathsf{Res}_{1,4}\,\,S' & = \mathsf{Res}_{1,4}\,\,\mathsf{\Psi}\circ\mathsf{\Psi}^{-1}(S')                                                                                                                                                                 \\
			                                & = \left.\mathsf{\Psi} \circ \mathsf{\Psi}^{-1} (S')\right|_{1,4}                                                                                                                                                                 \\
			                                & = \left.\mathsf{pr}_{2,\overline{1}} \circ \mathsf{\Psi}_2 \circ \mathsf{\Psi}^{-1} (S') \right|_{1,4}                                                                                                                    \\
			                                & = \left.\mathsf{pr}_{2,4} \circ \mathsf{pr}_{4,\overline{1}} \circ \mathsf{pr}_{2,\overline{2}} \circ \mathsf{\Psi}_3 \circ \mathsf{\Psi}^{-1} (S')\right|_{1,4}                                            \\
			                                & = \left.\mathsf{pr}_{2,4} \circ \mathsf{pr}_{4,\overline{1}} \circ \mathsf{pr}_{2,3} \circ \mathsf{pr}_{3,\overline{2}} \circ \mathsf{\Psi}_3 \circ \mathsf{\Psi}^{-1} (S')\right|_{1,4}                    \\
			                                & = \left.\mathsf{pr}_{2,4} \circ \mathsf{pr}_{2,3} \circ \mathsf{pr}_{4,\overline{1}} \circ \mathsf{pr}_{3,\overline{2}} \circ \mathsf{\Psi}_3 \circ \mathsf{\Psi}^{-1} (S')\right|_{1,4}                    \\
			                                & = \mathsf{pr}_{2,4} \circ \mathsf{pr}_{2,3} \left( \left. \mathsf{pr}_{4,\overline{1}} \circ \mathsf{pr}_{3,\overline{2}} \left( \mathsf{\Psi}_3 \circ \mathsf{\Psi}^{-1}(S') \right) \right|_{1,4} \right) \\
			                                & = \mathsf{pr}_{2,4} \circ \mathsf{pr}_{2,3} \left( \left.\left( \mathsf{Res}_{1,2;\overline{2},\overline{1}}\,\,\mathsf{\Psi}_3 \circ \mathsf{\Psi}^{-1}(S') \right)\right| _{\overline{2} \to 3,\,\,\overline{1} \to 4} \right) \\
			                                & = \mathsf{pr}_{2,4} \circ \mathsf{pr}_{2,3} \left( \left.\left( \mathsf{Res}_{1,2;\overline{2},\overline{1}}\,\,\mathsf{\Psi}^{-1}(S') \right)\right| _{\overline{2} \to 3,\,\,\overline{1} \to 4} \right).
		\end{aligned}
	\end{equation}
	Therefore, we obtain 
	\begin{equation}
		\mathsf{Res}_{1,2;\overline{2},\overline{1}}\,\,\mathsf{\Psi}^{-1}(S')
		=
		\left.\mathsf{pr}_{2,3}^{-1} \circ\mathsf{pr}_{2,4}^{-1} \left( \mathsf{Res}_{1,4}\,\,S' \right)\right|_{3 \to \overline{2},\,\,4\to\overline{1}}.
	\end{equation}
	Since
	$S'$ is a $\mathfrak{k}$-lowest weight tableau,
	it follows from Lemma \ref{kLGP} (2) that 
	\begin{center}
		\begin{tikzpicture}[x=5mm,y=5mm]
			\draw (0,0)--(8,0)--(8,1)--(0,1)--cycle;
			\draw (0,1)--(0,2)--(10,2)--(10,1)--(8,1);
			\node at (0.5,0.5) {$4$};
			\node at (0.5,1.5) {$1$};
			\node at (7.5,0.5) {$4$};
			\node at (9.5,1.5) {$1$};
			\draw[dotted, line width=1pt] (1.2,0.5)--(1.8,0.5);
			\draw[dotted, line width=1pt] (6.2,0.5)--(6.8,0.5);
			\draw[dotted, line width=1pt] (1.2,1.5)--(1.8,1.5);
			\draw[dotted, line width=1pt] (8.2,1.5)--(8.8,1.5);
			\node at (-3.5,1) {$\mathsf{P}^{A\mathrm{II}}_{1,4} \left( \mathsf{Res}_{1,4}\,\, S' \right)  =$};
			\node at (-4.45,1) {};
			\node at (10.5,0.5) {.};
		\end{tikzpicture}
	\end{center}
	Hence, 
	$\mathsf{Res}_{1,2;\overline{2},\overline{1}}\,\,\mathsf{\Psi}^{-1}(S')$
	is $\widehat{\mathfrak{g}}_1$-dominant by Proposition \ref{n=2thm} (2).
	Also, we have 
	\begin{equation}
		\mathsf{\Psi}^{-1}(S')|_{2,\overline{2}}
		=
		\left.\mathsf{\Psi}_2^{-1} \circ \mathsf{pr}_{2,\overline{1}}^{-1} (S') \right| _{2,\overline{2}}
		=
		\mathsf{\Psi}_2^{-1} \left(\left.\mathsf{pr}_{2,\overline{1}}^{-1}(S')\right|_{2,\overline{2}}\right).
	\end{equation}
	Since $\mathsf{pr}_{2,\overline{1}}^{-1}(S')$
	is the tableau obtained from $\mathsf{Res}_{3,\overline{1}}\,\,S'$ 
  by decreasing all entries by $1$, 
	we deduce that 
	\begin{center}
		\begin{tikzpicture}[x=5mm,y=5mm]
			\draw (0,0)--(15,0)--(15,1)--(0,1)--cycle;
			\draw (0,3)--(18,3)--(18,4)--(0,4)--cycle;
			\draw (0,4)--(0,5)--(19,5)--(19,4)--(18,4);
			\draw (0,1)--(0,3);
			\draw (15,1)--(16,1)--(16,2)--(17,2)--(17,3);
			\draw[dotted, line width=1pt] (8,1.7)--(8,2.3);
			\node (a) at (1.4,0.5) {$b_n - 1$};
			\node (b) at (1.4,3.5) {$b_3 - 1$};
			\node (c) at (1.4,4.5) {$b_2 - 1$};
			\node (d) at (13.5,0.5) {$b_n - 1$};
			\node (e) at (16.6,3.5) {$b_3 - 1$};
			\node (f) at (17.6,4.5) {$b_2 - 1$};
			\draw[dotted, line width=1pt] (3,0.5)--(3.6,0.5);
			\draw[dotted, line width=1pt] (11.3,0.5)--(11.9,0.5);
			\draw[dotted, line width=1pt] (3,3.5)--(3.6,3.5);
			\draw[dotted, line width=1pt] (14.4,3.5)--(15,3.5);
			\draw[dotted, line width=1pt] (3,4.5)--(3.6,4.5);
			\draw[dotted, line width=1pt] (15.4,4.5)--(16,4.5);
			\node at (-3.7,2.5) {$\mathsf{P}^{A\mathrm{II}}_{2,\overline{2}} \left( \mathsf{pr}_{2,\overline{1}}^{-1}(S') \right)  =$};
		\end{tikzpicture}
		
		\vspace{15pt}
		
		\begin{tikzpicture}[x=5mm,y=5mm]
			\draw (0,0)--(15,0)--(15,1)--(0,1)--cycle;
			\draw (0,3)--(18,3)--(18,4)--(0,4)--cycle;
			\draw (0,4)--(0,5)--(19,5)--(19,4)--(18,4);
			\draw (0,1)--(0,3);
			\draw (15,1)--(16,1)--(16,2)--(17,2)--(17,3);
			\draw[dotted, line width=1pt] (8,1.7)--(8,2.3);
			\node (a) at (1.8,0.5) {$a_{n-1} + 1$};
			\node (b) at (1.4,3.5) {$a_2 + 1$};
			\node (c) at (1.4,4.5) {$a_1 + 1$};
			\node (d) at (13.2,0.5) {$a_{n-1} + 1$};
			\node (e) at (16.6,3.5) {$a_2 + 1$};
			\node (f) at (17.6,4.5) {$a_1 + 1$};
			\draw[dotted, line width=1pt] (3.5,0.5)--(4.1,0.5);
			\draw[dotted, line width=1pt] (10.8,0.5)--(11.4,0.5);
			\draw[dotted, line width=1pt] (3,3.5)--(3.6,3.5);
			\draw[dotted, line width=1pt] (14.4,3.5)--(15,3.5);
			\draw[dotted, line width=1pt] (3,4.5)--(3.6,4.5);
			\draw[dotted, line width=1pt] (15.4,4.5)--(16,4.5);
			\node at (-1,2.5) {$=$};
			\node at (-7.6,2.5) {};
			\node at (19.5,2) {.};
		\end{tikzpicture}
	\end{center}
	From this, by the same argument as for $\mathsf{\Phi}$, we find that 
	$\mathsf{\Psi}^{-1}(S')|_{2,\overline{2}}$
	is $\widehat{\mathfrak{g}}_{\geq2}$-dominant. 
	Therefore, by Lemma \ref{gLGP}, we conclude that 
	$\mathsf{\Psi}^{-1}(S') \in \mathsf{SST}_{2n}^{\text{$\widehat{\mathfrak{g}}$-dom}}(\lambda)$, as desired. 
  This proves the proposition. 
\end{proof}

  Also,
	by equations (\ref{5.13a}) and (\ref{4.14}) 
	in the case that $\lambda=\varepsilon_1$,
	we can verify that for each $T \in \mathsf{SST}_{2n}(\lambda)$, equalities 
(\ref{5.19}) hold. 
Hence, we have established the following. 
\begin{prop}
	The statement of Theorem \ref{KeyProposition} also holds true for $n \geq3$.
\end{prop}
Thus, the proof of the Naito--Sagaki conjecture is now complete.

\end{document}